\documentclass[leqno]{amsart}
\usepackage[left=1in, right=1in, top=1in, bottom=1in]{geometry}
\usepackage{boilerplate}
\usepackage{mathrsfs}
\usepackage{tikz-cd}
\usepackage{mathdots}
\usetikzlibrary{arrows,arrows.meta}
\usepackage[mathscr]{euscript}
\usepackage{algorithm}
\usepackage{algpseudocode}
\usepackage{hyperref}
\hypersetup{
    colorlinks,
    citecolor=black,
    filecolor=black,
    linkcolor=black,
    urlcolor=black
}
\usepackage{enumitem}
\usepackage{subcaption}

\makeatletter
\renewcommand\subsection{\leftskip 0pt\@startsection{subsection}{2}{\z@}%
                                     {-3.25ex\@plus -1ex \@minus -.2ex}%
                                     {1.5ex \@plus .2ex}%
                                     {\normalfont\normalsize\bfseries}}
\makeatother

\def\mfp{\underline{\mathbb{F}_p}}
\def\mft{\underline{\mathbb{F}_2}}
\def\uzp{\underline{\mathbb{Z}_p}}

\def\KR{\mathrm{KR}}
\def\K{\mathrm{K}}
\def\L{\mathrm{L}}
\def\GW{\mathrm{GW}}

\DeclareFontFamily{U}{cbgreek}{}
\DeclareFontShape{U}{cbgreek}{m}{n}{
        <-6>    grmn0500
        <6-7>   grmn0600
        <7-8>   grmn0700
        <8-9>   grmn0800
        <9-10>  grmn0900
        <10-12> grmn1000
        <12-17> grmn1200
        <17->   grmn1728
      }{}
\DeclareFontShape{U}{cbgreek}{bx}{n}{
        <-6>    grxn0500
        <6-7>   grxn0600
        <7-8>   grxn0700
        <8-9>   grxn0800
        <9-10>  grxn0900
        <10-12> grxn1000
        <12-17> grxn1200
        <17->   grxn1728
      }{}

\DeclareRobustCommand{\qoppa}{%
  \text{\usefont{U}{cbgreek}{\normalorbold}{n}\symbol{21}}%
}
\makeatletter
\newcommand{\normalorbold}{%
  \ifnum\pdf@strcmp{\math@version}{bold}=\z@ bx\else m\fi
}
\makeatother

\setlist[enumerate]{label*=\arabic*.}

\title{On the equivalence of two theories of real cyclotomic spectra}

\author{J.D. Quigley}
\address{
Department of Mathematics \\
Cornell University \\
Ithaca, NY, U.S.A.
}
\email{jdq27@cornell.edu}

\author{Jay Shah}
\address{Fachbereich Mathematik und Informatik, WWU Münster, 48149 M\"{u}nster, Germany}
\email{jayhshah@gmail.com}

\begin{document}

\tikzcdset{arrow style=tikz, diagrams={>=stealth}}

\begin{abstract} 
We give a new formula for real topological cyclic homology that refines the fiber sequence formula discovered by Nikolaus and Scholze for topological cyclic homology to one involving genuine $C_2$-spectra. To accomplish this, we give a new definition of the $\infty$-category of real cyclotomic spectra that replaces the usage of genuinely equivariant dihedral spectra with the parametrized Tate construction. We then define an $\infty$-categorical version of H{\o}genhaven's $\mathrm{O}(2)$-orthogonal cyclotomic spectra, construct a forgetful functor relating the two theories, and show that this functor restricts to an equivalence between full subcategories of appropriately bounded-below objects. As an application, we compute the real topological cyclic homology of perfect $\mathbb{F}_p$-algebras for all primes $p$. 
\end{abstract}

\date{\today}
\maketitle

\tableofcontents

\section{Introduction}
\subsection{Motivation}

%  Grothendieck--Witt groups of discrete rings \cite{KM71} arise in enumerative algebraic geometry \cite{WW19}, real algebraic geometry \cite[\S 1.2.2]{Bal05}, and arithmetic geometry via the Milnor Conjecture (cf. the introduction of \cite{CDH+20a} for some applications).
Algebraic K-theory, Grothendieck--Witt theory (i.e., hermitian K-theory), and L-theory are fundamental invariants in algebra, geometry, and topology. For discrete rings or schemes, algebraic K-groups encode important arithmetic invariants like Picard and Brauer groups \cite{Tat76} and special values of Dedekind zeta functions \cite{Lic73, Qui72, Qui73}. Grothendieck--Witt groups naturally arise in motivic homotopy theory \cite{Mor12} and enumerative algebraic geometry \cite{WW19}, and are also related to special values of Dedekind zeta functions \cite{BK05, BKSO15,Kylling2020}. L-groups encode geometric information, with applications towards the Borel and Novikov Conjectures in geometric topology \cite[\S 1.5-1.6]{LR05}. For ring spectra, algebraic K-theory recovers Waldhausen's $A$-theory \cite[\S VI.8]{MR97h:55006}, which also has applications in geometric topology \cite{WJR:stabparhcob}. Moreover, the algebraic K-theory of ring spectra is closely connected to chromatic homotopy theory via the redshift philosophy, as discussed in \cite{AKQ19, HW20, LMMT20, Yua21}.

% \footnote{In this paper, the term \emph{$G$-spectra} will be synonymous with genuine $G$-spectra.}
In \cite{HM13}, Hesselholt and Madsen defined a genuine $C_2$-spectrum called \emph{real algebraic K-theory} $\KR(\sC,D,\eta)$ for any exact category with duality, such that its underlying spectrum $\KR(\sC,D,\eta)^e$ is algebraic K-theory $\K(\sC)$ and its categorical $C_2$-fixed points spectrum $\KR(\sC,D,\eta)^{C_2}$ is Grothendieck--Witt theory $\GW(\sC,D,\eta)$. They also conjectured that the geometric $C_2$-fixed points of real algebraic K-theory identify with a version of Ranicki's L-theory in the example of a discrete ring. In his lecture course on L-theory and surgery \cite{Lur13}, Lurie defined L-theory in the generality of \emph{Poincar{\'e} $\infty$-categories} $(\sC, \qoppa)$. In \cite{CDH+20b}, Calm{\`e}s et. al. introduced the hermitian Q-construction to define the (non-connective) Grothendieck--Witt theory and real algebraic K-theory of a Poincar{\'e} $\infty$-category, and they established the sought-after connection with L-theory by showing that $\Phi^{C_2} \KR(\sC, \qoppa) \simeq \L(\sC, \qoppa)$. The isotropy separation sequence from equivariant stable homotopy theory then takes the form
% \label{Eqn:FundamentalSequence}
\begin{equation*}
\K(\sC)_{hC_2} \to \GW(\sC, \qoppa) \to \L(\sC,\qoppa).
\end{equation*}
This fundamental fiber sequence provides a tight connection between these three invariants which can be leveraged to make new computations. This was used in \cite{CDH+20c}, for instance, to access the Grothendieck--Witt groups of the integers via their more accessible algebraic K-groups and L-groups. 
% Historically, K-groups and L-groups have been computed using very different techniques. 

The algebraic K-groups of ring spectra (or more generally, stable $\infty$-categories) can be accessed using trace maps
\[ \K(R) \to \TC(R) \to \THH(R) \]
into topological cyclic homology and topological Hochschild homology. The Dundas--Goodwillie--McCarthy Theorem \cite{DGM12} implies that for nilpotent extensions $f: R \to S$ of connective $E_1$-ring spectra, there is an isomorphism $\K_*(f) \cong \TC_*(f)$. In certain cases (cf. \cite{HM97}), one can even show that the trace map $\K(R) \to \tau_{\geq 0} \TC(R)$ is an equivalence after $p$-completion.

The classical approach to computing topological cyclic homology $\TC(R)$ involves using powerful, but quite sophisticated, tools from genuine equivariant stable homotopy theory. An insight of Nikolaus and Scholze \cite{NS18} is that when $\THH(R)$ is bounded-below, one can instead compute $\TC(R)$ using only Borel equivariant stable homotopy theory. This more recent perspective has been implemented in a myriad of applications to algebraic K-theory; we refer readers to the excellent survey articles of Hesselholt--Nikolaus \cite{HN20} and Mathew \cite{Mat21} for further discussion.

On the other hand, the L-groups of discrete rings are usually computed using algebraic and geometric methods. For example, let $M$ be a compact $n$-dimensional manifold of dimension at least five. The surgery exact sequence
\[ \cdots \to \cN_{n+1}(M \times I) \to \L_{n+1}^s(\ZZ[\pi_1(M)]) \to \cS^{\top}(M) \to \cN_n(M) \to \L_n^s(\ZZ[\pi_1(M)]) \]
relates the symmetric L-theory of $\ZZ[\pi_1(M)]$ to the topological structure set $\cS^{\top}(M)$ and the groups of normal invariants $\cN_*(M \times I)$ and $\cN_*(M)$. These are accessible using algebraic methods, as well as geometric methods like controlled topology (cf. \cite{LR05}). 

These different approaches to computing K- and L-groups have led to quite different results. For instance, the K- and L-theoretic Novikov Conjectures state that the assembly maps
$$H_*(BG; K(\ZZ)) \otimes_\ZZ \QQ \to K_*(\ZZ[G]) \otimes_\ZZ \QQ,$$
$$H_*(BG; L(\ZZ)) \otimes_\ZZ \QQ \to L_*(\ZZ[G]) \otimes_\ZZ \QQ,$$
respectively, are injective. B{\"o}kstedt--Hsiang--Madsen \cite{BHM93} showed that the K-theoretic Novikov Conjecture holds for groups $G$ with $H_s(G;\ZZ)$ finite for all $s \geq 1$; combined with older work of Culler--Vogtmann \cite{CM86}, this implied the K-theoretic Novikov Conjecture for groups of outer automorphisms of free groups nearly three decades ago. On the other hand, the L-theoretic Novikov Conjecture for outer automorphism groups of free groups was only proven recently using geometric methods by Bestvina--Guirardel--Horbez \cite{BGH17}. 

We are thus motivated to study real topological cyclic homology (and relatedly, real cyclotomic spectra) in order to facilitate the development of trace methods for real algebraic K-theory. This was initiated in the work of Hesselholt--Madsen \cite{HM13}, who defined real topological Hochschild homology. Dotto \cite{Dot12, Dot16} and Dotto--Ogle \cite{DO19} then studied various trace maps out of real algebraic K-theory, while H{\o}genhaven \cite{Hog16} defined real topological cyclic homology and made some initial computations. Dotto--Moi--Patchkoria--Reeh \cite{DMPR17} pushed the study of real topological Hochschild homology further and made some seminal computations, which Dotto--Moi--Patchkoria \cite{DMP19, DMP21} leveraged to study real topological restriction homology and real topological cyclic homology. 

In forthcoming work \cite{HarpazNikolausShah}, Harpaz, Nikolaus, and the second author define the real topological Hochschild homology (including its real cyclotomic structure) of Poincar{\'e} $\infty$-categories and construct the real cyclotomic trace map
$$\mr{trc}_{\RR}: \KR(\sC, \qoppa) \to \TCR(\sC, \qoppa)$$
as a $C_2$-equivariant refinement of the usual cyclotomic trace. They then affirm a conjecture of Nikolaus by identifying the fiber of $\Phi^{C_2} \mr{trc}_{\RR}$ with quadratic L-theory $\L^q(\sC,\qoppa)$ under a suitable bounded-below hypothesis on $(\sC, \qoppa)$. Together with Wall's results on the rigidity of quadratic L-theory \cite[Prop.~2.3.7]{CDH+20c}, this promotes the Dundas--Goodwillie--McCarthy theorem to hermitian K-theory and thereby justifies real trace methods as a valid computational methodology.

\subsection{Main results}

In this paper, we study and relate two approaches to the theory of real cyclotomic spectra. Our two theories generalize the theories of cyclotomic spectra developed by Hesselholt--Madsen \cite{HM03}, Blumberg--Mandell \cite{BM12, BM16}, and Nikolaus--Scholze \cite{NS18} by incorporating involutive structures. Further, our theory of genuine real cyclotomic spectra provides an $\infty$-categorical version of H{\o}genhaven's theory of $\mathrm{O}(2)$-orthogonal cyclotomic spectra \cite{Hog16}. More precisely, we define $C_2$-equivariant refinements of the $\infty$-categories of genuine and Borel\footnote{Nikolaus and Scholze write ``naive'' instead of ``Borel''.} cyclotomic spectra introduced in \cite{NS18}. 

Our main result (\cref{Thmx:Equiv}) is that the genuine and Borel notions coincide for objects whose underlying spectrum is bounded below. We also extend many useful results about cyclotomic spectra to the $C_2$-equivariant setting. These include results about multiplicative structure on real topological Hochschild homology (\cref{MT:C2EooTHR}), corepresentability of real topological cyclic homology (\cref{Thmx:Corep}), and various formulas for computing real topological cyclic homology (\cref{Thmx:FiberSeq}). As an application, we compute the real topological cyclic homology of perfect $\mathbb{F}_p$-algebras (\cref{Thmx:TCRFp}). 

\subsubsection{Recollections on cyclotomic spectra}

Let $S^1$ denote the circle group and $\mu_n \leq S^1$ the cyclic subgroup of order $n$. Let $\Sp^G$ denote the $\infty$-category of genuine $G$-spectra and $\Sp^{hG} \coloneq \Fun(BG,\Sp)$ the $\infty$-category of Borel $G$-spectra. Let $\Sp^{S^1}_{\cF}$ denote the $\infty$-category of $\cF$-complete genuine $S^1$-spectra for $\cF$ the family of finite subgroups.

\begin{dfn}[{\cite[Def. II.1.1 and II.3.3]{NS18}}] $ $
\begin{enumerate}
% \item A \emph{genuine cyclotomic spectrum} is a genuine $S^1$-spectrum $X$, together with equivalences of genuine $S^1$-spectra $\Phi^{\mu_p}X \xrightarrow{\simeq} X$ for each prime $p$. The \emph{$\infty$-category of genuine $p$-cyclotomic spectra} is the equalizer
% \[ \begin{tikzcd}
% \CycSp^{\mr{gen}} := \Eq(\Sp^{S^1} \ar[r," \prod_p \Phi^{\mu_p}",shift left=1] \ar[r,swap,shift right=1,"id"] & \Sp^{S^1}).
% \end{tikzcd} \]

\item A \emph{Borel cyclotomic spectrum} is a spectrum $X$ with $S^1$-action, together with $S^1$-equivariant maps $\varphi_p: X \to X^{t\mu_p}$ for each prime $p$. The \emph{$\infty$-category of Borel cyclotomic spectra} is the lax equalizer
\[
\begin{tikzcd}
\CycSp := \LEq(\Sp^{hS^1} \ar[rr,"id",shift left=1] \ar[rr,swap,shift right=1," \prod_p (-)^{t\mu_p}"] & &\prod_{p} \Sp^{hS^1}).
\end{tikzcd}
\]

\item The \emph{$\infty$-category of genuine cyclotomic spectra} is
\[ \CycSp^{\mr{gen}} \coloneq (\Sp^{S^1}_{\cF})^{h \NN_{>0}}, \]
where $\NN_{>0}$ is the multiplicative monoid of positive integers acting on $\Sp^{S^1}_{\cF}$ by the geometric fixed points functors $\Phi^{\mu_n}$. A \emph{genuine cyclotomic spectrum} is then an $\cF$-complete genuine $S^1$-spectrum $X$, together with equivalences $X \simeq \Phi^{\mu_n} X$ and coherence data witnessing their homotopy commutativity.
\end{enumerate}

\end{dfn}

\begin{exm}
The topological Hochschild homology of any $E_1$-ring spectrum is both a genuine and Borel cyclotomic spectrum. More generally, the topological Hochschild homology of any stable $\infty$-category is a Borel cyclotomic spectrum \cite{ArbeitsgemeinschaftTC_Nikolaus}.
\end{exm}

Evidently, a genuine cyclotomic spectrum \emph{a priori} constitutes more data than a Borel cyclotomic spectrum. Given a genuine cyclotomic spectrum $[X, \{ X \xrightarrow{\simeq} \Phi^{\mu_n} X \}]$, the forgetful functor $\sU: \Sp^{S^1} \to \Sp^{hS^1}$ defines a Borel $S^1$-spectrum $\sU(X)$ with $p$th cyclotomic structure map defined as the composite $\varphi_p : X \simeq \Phi^{\mu_p} X \to X^{t\mu_p}$. This assignment extends to a forgetful functor
\begin{equation}\label{Eqn:ForgetfulFunctor}
\sU: \CycSp^{\mr{gen}} \to \CycSp.
\end{equation}
Nikolaus and Scholze proved the following remarkable result:

\begin{thm}[{\cite[Thm. 1.4]{NS18}}]
The functor $\sU : \CycSp^{\mr{gen}} \to \CycSp$ restricts to an equivalence between the full subcategories of those objects whose underlying spectrum is bounded-below.
\end{thm}

\begin{rem}
If $R$ is an $E_1$-ring spectrum which is bounded-below, then its topological Hochschild homology $\THH(R)$ is also bounded-below. In particular, the topological Hochschild homology spectra of discrete rings, schemes, and connective ring spectra are all bounded-below.
\end{rem}

The equivalence between bounded-below genuine and Borel cyclotomic spectra gives rise to an important computational tool, the \emph{fiber sequence formula for topological cyclic homology}:

\begin{thm}[{\cite[Cor.~1.5]{NS18}}]
Let $R$ be a connective ring spectrum. There is a natural fiber sequence
\begin{equation}\label{Eqn:TCFiber}
\TC(R) \to \THH(R)^{hS^1} \xrightarrow{(\varphi_p^{hS^1} - \can)_{p}} \prod_{p} (\THH(R)^{t\mu_p})^{hS^1},
\end{equation}
where $\varphi_p^{hS^1}$ is the homotopy fixed points of the cyclotomic structure map and $\can$ is the composite
$$\can : \THH(R)^{hS^1} \simeq (\THH(R)^{h\mu_p})^{h(S^1/\mu_p)} \simeq  (\THH(R)^{h\mu_p})^{hS^1} \to (\THH(R)^{t\mu_p})^{hS^1}$$
where the middle equivalence comes from the $p$-th power map $S^1/\mu_p \cong S^1$. Moreover, the right-most term identifies with the profinite completion of $\THH(R)^{tS^1}$.
\end{thm}

\begin{rem}
In \cite[Sec. IV]{NS18}, Nikolaus and Scholze apply the fiber sequence formula to compute $\TC$ of spherical monoid rings and $\TC(\FF_p)$. In \cite{KNYouTube}, Nikolaus explains how to modify this method for computing $\TC(\FF_p)$ to compute $\TC$ of all perfect $\FF_p$-algebras.
\end{rem}

% The key insight of Nikolaus--Scholze is that under reasonable boundedness hypotheses, the theory of cyclotomic spectra can be formulated using Borel equivariant homotopy theory instead of genuine equivariant stable homotopy theory. 

\subsubsection{Real cyclotomic spectra}

In \cite{Hog16}, H{\o}genhaven defined $\mathrm{O}(2)$-orthogonal cyclotomic spectra. The primary example of interest is the real topological Hochschild homology of an $E_\sigma$-algebra in genuine $C_2$-spectra. 

Our first contribution to the theory of real cyclotomic spectra is the following $\infty$-categorical version of H{\o}genhaven's definition. Let $\mathrm{O}(2) = S^1 \rtimes C_2$ denote the orthogonal group, $\mu_p = \mu_p \times e \leq \mathrm{O}(2)$, and $D_{2p^n} = \mu_{p^n} \rtimes C_2 \leq \mathrm{O}(2)$. 

\begin{dfn}[{\cref{dfn:GenRealCycSp} and \cref{dfn:realintegralgen}}]
$ $
A \emph{genuine real $p$-cyclotomic spectrum} is a genuine $D_{2p^{\infty}}$-spectrum $X$, together with an equivalence $\Phi^{\mu_p} X \xrightarrow{\simeq} X$ in $\Sp^{D_{2p^{\infty}}}$. The \emph{$\infty$-category of genuine real $p$-cyclotomic spectra} is the equalizer
\[
\begin{tikzcd}
\RCycSp^{\mr{gen}}_p \coloneq \Eq(\Sp^{D_{2p^{\infty}}} \ar[r,"id",shift left=1] \ar[r,swap,shift right=1,"\Phi^{\mu_p}"] & \Sp^{D_{2p^{\infty}}}).
\end{tikzcd}
\]

Let $\cF$ denote the family of finite subgroups of $\mathrm{O}(2)$ and $\Sp^{\mathrm{O}(2)}_{\cF}$ the $\infty$-category of $\cF$-complete genuine $\mathrm{O}(2)$-spectra. Let $\NN_{>0}$ act on $\Sp^{\mathrm{O}(2)}_{\cF}$ by the geometric fixed points functors $\Phi^{\mu_n}$. The \emph{$\infty$-category of genuine real cyclotomic spectra} is
\[ \RCycSp^{\mr{gen}} \coloneq (\Sp^{\mathrm{O}(2)}_{\cF})^{h \NN_{>0}}. \]
% A \emph{genuine real cyclotomic spectrum} is a genuine $\mathrm{O}(2)$-spectrum $X$, together with an equivalence $\Phi^{\mu_p} X \xrightarrow{\simeq} X$ in $\Sp^{\mathrm{O}(2)}$ for each prime $p$. The \emph{$\infty$-category of genuine real cyclotomic spectra} is the equalizer
% \[
% \begin{tikzcd}
% \RCycSp^{\mr{gen}} := \Eq(\Sp^{\mathrm{O}(2)} \ar[r," \prod_p \Phi^{\mu_p}",shift left=1] \ar[r,swap,shift right=1,"id"] & \Sp^{\mathrm{O}(2)}).
% \end{tikzcd}
% \]

\end{dfn}

\begin{exm}
H{\o}genhaven shows that the real topological Hochschild homology $\THR(R,D)$ of a ring $R$ with anti-involution $D$ is a genuine real cyclotomic spectrum.
\end{exm}

Our definition of genuine real cyclotomic spectra is the obvious $\mathrm{O}(2)$-equivariant analog of Nikolaus and Scholze's definition of genuine cyclotomic spectra. There are exact, colimit-preserving, and symmetric monoidal functors
\begin{align*}
\mr{triv}_{\RR}^{\mr{gen}} & : \Sp^{C_2} \to \RCycSp^{\mr{gen}} \\
\mr{triv}_{\RR,p}^{\mr{gen}} & : \Sp^{C_2} \to \RCycSp_p^{\mr{gen}}
\end{align*}
which equip a genuine $C_2$-spectrum with `trivial' genuine real ($p$-)cyclotomic structure (\cref{cnstr:GenuineTrivialFunctor}). 

\begin{dfn}[{\cref{dfn:ClassicalTCR}}]
The \emph{genuine $p$-typical real topological cyclic homology} functor
$$\TCR^{\mr{gen}}(-,p): \RCycSp^{\mr{gen}}_p \to \Sp^{C_2}$$
is the right adjoint to $\mr{triv}_{\RR,p}^{\mr{gen}}$. Likewise, $\TCR^{\mr{gen}}$ is the right adjoint to $\triv_{\RR}^{\mr{gen}}$.
\end{dfn}

% \begin{rem}
% It is expected, but has not been proven to the authors' knowledge, that an analog of the Dundas--Goodwillie--McCarthy Theorem relates real topological cyclic homology and real algebraic K-theory. Progress in this direction has been made in the work of Dotto \cite{Dot16} and forthcoming work of Harpaz, Nikolaus, and the second author. 
% \end{rem}

\begin{rem}
We are unaware of a definition of \emph{integral} real topological cyclic homology $\TCR^{\mr{gen}}$ in previous literature. Following Goodwillie \cite{Goo86, DGM12}, Nikolaus and Scholze define integral topological cyclic homology as the pullback
\begin{equation} \label{eqn:integral_TC}
\begin{tikzcd}
\TC^{\mr{gen}}(X) \ar[r] \ar[d] & X^{hS^1} \ar[d] \\
\prod_{p} \TC^{\mr{gen}}(X,p)_p^\wedge \ar[r] & \prod_{p} (X_p^{\wedge})^{hS^1}.
\end{tikzcd}
\end{equation}
The obvious analog of this diagram does not type-check if we work with real topological cyclic homology: if $X \in \Sp^{\mathrm{O}(2)}$, then $\TCR^{\mr{gen}}(X,p)$ is a genuine $C_2$-spectrum but $X^{hS^1}$ is a nonequivariant spectrum. Put differently, taking only $S^1$-homotopy fixed points of the underlying spectrum discards too much genuine equivariant information.
\end{rem}

The issue with homotopy fixed points raised in the previous remark is the main conceptual obstacle to developing a Borel equivariant theory of real cyclotomic spectra. We would like a theory of real cyclotomic spectra in which the genuine $S^1$-equivariant structure is replaced by a Borel $S^1$-equivariant structure, but the genuine $C_2$-equivariant structure is preserved. This partially Borel, partially genuine equivariant structure was our chief motivation in developing the theory of the parametrized Tate construction in \cite{QS21a}. To make this precise, we used the following notions from parametrized higher category theory:

\begin{dfn}
Let $\sO_G$ be the orbit category of a group $G$. A \emph{$G$-$\infty$-category}, resp. \emph{$G$-space} is a cocartesian, resp. left fibration $C \to \sO_G^{\op}$. A \emph{$G$-functor} $F: C \to D$ is then a morphism over $\sO_G^{\op}$ that preserves cocartesian edges. We let $\Fun_G(C,D)$ denote the $\infty$-category of $G$-functors.
\end{dfn}

\begin{exm}
We have the $G$-$\infty$-categories $\underline{\Spc}^G$ and $\underline{\Sp}^G$ of $G$-spaces and (genuine) $G$-spectra. Their respective fibers over an orbit $U \cong G/H$ are given by the $\infty$-categories $\Spc^H \simeq \Fun(\sO_H^{\op}, \Spc)$ and $\Sp^H$ of $H$-spaces and $H$-spectra, and the cocartesian edges encode the functoriality of restriction and conjugation.
\end{exm}

% Our main definition in \cite{QS21a} was the following:
% [{\cite[Def.~3.18]{QS21a}}]
\begin{dfn}
Let $\psi = [N \to G \to G/N]$ be a group extension with $G/N$ finite, and let $B_{G/N}^\psi N \subseteq \sO_G^{\op}$ be the full subcategory of $G$-orbits which are $N$-free, viewed as $G/N$-$\infty$-category via the quotient map. The \emph{$\infty$-category of $G/N$-spectra with $\psi$-twisted $N$-action} is
$$\Sp^{h_{G/N} N}  \coloneq \Fun_{G/N}(B_{G/N}^\psi N , \underline{\Sp}^{G/N}).$$
If $\psi$ is the group extension for a semidirect product, we will instead write $B^t_{G/N} N$.

The \emph{parametrized homotopy orbits} $(-)_{h_{G/N} N}$ and \emph{parametrized fixed points} $(-)^{h_{G/N} N}$ are then the $G$-colimit and $G$-limit functors, i.e., the left and right adjoints to the constant diagram functor
\[ \delta: \Sp^{G/N} \to \Fun_{G/N}(B_{G/N}^\psi N , \underline{\Sp}^{G/N}) \]
given by restriction along the structure map $B_{G/N}^\psi N \to \sO_{G/N}^{\op}$.
\end{dfn}

\begin{ntn}
In this paper, we write $\Sp^{h_{C_2} S^1}$ and $\Sp^{h_{C_2} \mu_{p^n}}$ for the $\infty$-categories associated to the semidirect products $\mathrm{O}(2) = S^1 \rtimes C_2$ and $D_{2p^n} = \mu_{p^n} \rtimes C_2$.
\end{ntn}

\begin{exm} \label{exm:unpacking_C2_functor}
Let $X \in \Sp^{h_{C_2} S^1}$. Under the canonical isomorphism $\mu_2 \cong W_{\mathrm{O}(2)} C_2$, we have
$$B^t_{C_2} S^1 = [B \mu_2 \to B S^1]$$
as a presheaf on $\sO_{C_2}$, with residual $C_2$-action on $BS^1$ such that $(BS^1)_{h C_2} \simeq B \mathrm{O}(2)$. Then $X$ amounts to the data of a tuple
\[ [X^1 \in \Sp^{h \mathrm{O}(2)}, \: X^{\phi C_2} \in \Sp^{h \mu_2}, \: X^{\phi C_2} \to (X^1)^{t C_2} \in \Sp^{h \mu_2}]. \]
In general, one may describe a $G/N$-spectrum with $\psi$-twisted $N$-action in terms of its $H$-geometric fixed points for $H \in \Gamma_N$ via the fully faithful embedding $j_!$ into $\Sp^G$ described below.\footnote{We only considered this for finite groups $G$ in \cite{QS21a}, but one may understand the case $\Sp^{h_{C_2} S^1}$ directly in terms of the recollement obtained by applying \cite[Lem.~2.37]{Sha21} with $K = B^t_{C_2} S^1$ to the $C_2$-recollement $(\underline{\Sp}^{h C_2}, \underline{\Sp}^{\Phi C_2})$ on $\underline{\Sp}^{C_2}$ (\cite[Rem.~2.33]{QS21a}).}
\end{exm}

% \begin{exm}
% % \item When $\psi = [N \xrightarrow{=} N \to e]$, we recover the category of $N$-Borel spectra, i.e.
% % $$\Sp^N_{N\textnormal{-Borel}} \simeq \Sp^{hN} = \Fun(BN,\Sp).$$
% % Taking $N = S^1$, we recover the category $\Sp^{hS^1}$ used to define Borel cyclotomic spectra. 
% When $\psi = [S^1 \to \mathrm{O}(2) \to C_2]$, we obtain the $\infty$-category
% $$\Sp^{h_{C_2}S^1} = \Fun_{C_2}(B^t_{C_2} S^1, \underline{\Sp}^{C_2})$$
% of $C_2$-spectra with twisted $S^1$-action. The objects in this category contain the minimal amount of data we need to define Borel real cyclotomic spectra.

% Restricting to subgroups $D_{2p^n} \leq \mathrm{O}(2)$, to obtain
% $$\Sp^{h_{C_2} \mu_{p^n}} = \Fun_{C_2}(B^t_{C_2} \mu_{p^n}, \underline{\Sp}^{C_2}).$$
% \end{enumerate}

% \end{exm}

As discussed above, the Tate construction $(-)^{t\mu_p} : \Sp^{hS^1} \to \Sp^{hS^1}$ used to define Borel cyclotomic spectra must be refined to encode the genuine equivariant structure relevant to real cyclotomic spectra. In \cite{QS21a}, we studied the \emph{parametrized Tate construction}
$$(-)^{t_{G/N}N} : \Sp^{h_{G/N} N} \to \Sp^{G/N}$$
associated to any extension $[N \to G \to G/N]$. It is this refinement of the Tate construction which will appear in our definition of Borel real cyclotomic spectra. 

More precisely, we showed in \cite[Thm. A]{QS21a} that when $G$ is finite there exists a symmetric monoidal restriction functor
$$j^*: \Sp^{G} \to \Fun_{G/N}(B_{G/N}^\psi N, \underline{\Sp}^{G/N})$$
that participates in an adjoint triple
\[ \begin{tikzcd}[column sep=6ex]
\Fun_{G/N}(B_{G/N}^\psi N, \underline{\Sp}^{G/N}) \ar[shift left = 3, hookrightarrow]{r}{j_!} \ar[shift right = 3, hookrightarrow]{r}[swap]{j_*} & \Sp^{G} \ar{l}[description]{j^*}
\end{tikzcd} \]
in which $j_!$ and $j_*$ are fully faithful and embed as the $\Gamma_N$-torsion and $\Gamma_N$-complete G-spectra, respectively, where $\Gamma_N = \{ H \leq G: N \cap H = 1 \}$ is the $G$-family of $N$-free subgroups.

% The functors in the adjoint triple above participate in a \emph{recollement}. As in \cite{QS21a}, we will assume knowledge of the theory of recollements in this paper and refer to \cite{Sha21} as our primary reference. Given a stable $\infty$-category $\sX$ decomposed by a stable recollement $(\sU, \sZ)$, we will generically label the recollement adjunctions as
% \[ \begin{tikzcd}[column sep=4em]
% \sU \ar[hookrightarrow, shift left=2]{r}{j_!} \ar[hookrightarrow, shift right=4]{r}[swap]{j_*} & \sX \ar[shift left=1]{l}[description]{j^*} \ar[shift left=2]{r}{i^*} \ar[shift right=1, hookleftarrow]{r}[swap, description]{i_*} \ar[shift right=4]{r}[swap]{i^!} & \sZ.
% \end{tikzcd} \]
% Here $j^* i_* = 0$ determines the directionality of the recollement.

\begin{dfn}[{\cite[Def.~1.6]{QS21a}}] \label{dfn:param_tate}
Suppose $G$ is a finite group. The \emph{parametrized Tate construction} 
$$(-)^{t_{G/N} N} : \Fun_{G/N}(B_{G/N}^\psi N, \underline{\Sp}^{G/N}) \to \Sp^{G/N}$$
is the composite lax symmetric monoidal functor
\[ \begin{tikzcd}[column sep=4em]
\Fun_{G/N}(B_{G/N}^\psi N, \underline{\Sp}^{G/N}) \ar[hookrightarrow]{r}{j_*} &  \Sp^{G} \ar{r}{- \wedge \widetilde{E \Gamma_N}} & \Sp^{G} \ar{r}{\Psi^N} & \Sp^{G/N}
\end{tikzcd} \]
% $$\Fun_{G/N}(B_{G/N}^\psi N, \underline{\Sp}^{G/N}) \overset{j_*}{\hookrightarrow} \Sp^{G} \xrightarrow{- \wedge \widetilde{E \Gamma_N}} \Sp^{G} \xrightarrow{\Psi^N} \Sp^{G/N},$$
where
\begin{enumerate}
\item $j_*$ is the embedding mentioned above,
\item $\widetilde{E\Gamma_N}$ is the cofiber of the map ${E\Gamma_N}_+ \to S^0$, and
\item $\Psi^N$ is the categorical $N$-fixed points.\footnote{We write this as $\Psi^N$ to distinguish it from the spectrum-valued functor of categorical $N$-fixed points $(-)^N: \Sp^{G} \to \Sp$.}
\end{enumerate}

The parametrized homotopy orbits (resp. parametrized homotopy fixed points) may be identified as the composite $(-)_{h_{G/N}N} \simeq \Psi^N \circ j_!$ (resp. $(-)^{h_{G/N}N} \simeq \Psi^N \circ j_*$), after which we obtain the norm cofiber sequence
\[ (-)_{h_{G/N}N} \to (-)^{h_{G/N}N} \to (-)^{t_{G/N}N}. \]
\end{dfn}

\begin{rem}
In \cite{QS21a}, using the theory of parametrized norm maps associated to weakly $1$-ambidextrous morphisms (\cite[Constr.~4.1.8]{HL13}) of $G/N$-spaces, we showed how to promote $(-)^{t_{G/N} N}$ to a functor respecting residual actions. In particular, under the isomorphism $\mathrm{O}(2)/\mu_p \cong \mathrm{O}(2)$, we have $(-)^{t_{C_2} \mu_p}$ as an endofunctor of $\Sp^{h_{C_2} S^1}$.
% \[ (-)^{t_{C_2} \mu_p}: \Fun_{C_2}(B^t_{C_2} S^1, \underline{\Sp}^{C_2} ) \to \Fun_{C_2}(B^t_{C_2} S^1, \underline{\Sp}^{C_2}). \]
\end{rem}

\begin{rem}
We did not discuss genuine $G$-spectra for general compact Lie groups $G$ in \cite{QS21a}. Nonetheless, by using the theory of parametrized assembly maps, \cite[Thm. D]{QS21a} shows that the parametrized Tate construction $(-)^{t_{G/N} N}$ can be extended to the case where the group extension $\psi$ has the normal subgroup $N$ as a compact Lie group (with $G/N$ still finite). In particular, we will use $(-)^{t_{C_2}S^1}$ in our analysis of integral real cyclotomic spectra.
\end{rem}

\begin{rem}
The properties of the parametrized Tate construction needed in this paper are summarized in \cite[Intro.]{QS21a}. These properties will be unsurprising for readers familiar with the ordinary Tate construction, as discussed in the work of Greenlees--May \cite{GM95} and Nikolaus--Scholze \cite{NS18}. 
\end{rem}

With the appropriate category of partially Borel, partially genuine equivariant spectra identified, we make the following definition:

\begin{dfn}[{\cref{dfn:RealCycSp}, \ref{dfn:C2CategoryOfCyclotomicSpectra}, and \ref{dfn:realintegral}}]
$ $
A \emph{Borel real $p$-cyclotomic spectrum} is a $C_2$-spectrum $X$ with twisted $S^1$-action, together with a twisted $S^1$-equivariant map $\varphi_p : X \to X^{t_{C_2}\mu_p}$. The \emph{$\infty$-category of Borel real $p$-cyclotomic spectra} is the lax equalizer
\[
\begin{tikzcd}
\RCycSp_p \coloneq \LEq( \Sp^{h_{C_2}  \mu_{p^{\infty}}} \ar[rr,"id",shift left=1] \ar[rr,swap,shift right=1," (-)^{t_{C_2}\mu_p}"] & & \Sp^{h_{C_2} \mu_{p^{\infty}}} ).
\end{tikzcd}
\]

A \emph{Borel real cyclotomic spectrum} is a $C_2$-spectrum $X$ with twisted $S^1$-action, together with twisted $S^1$-equivariant maps $\varphi_p : X \to X^{t_{C_2}\mu_p}$ for each prime $p$. The \emph{$\infty$-category of Borel real cyclotomic spectra} is the lax equalizer
\[
\begin{tikzcd}
\RCycSp \coloneq \LEq( \Sp^{h_{C_2} S^1} \ar[rr,"id",shift left=1] \ar[rr,swap,shift right=1," \prod_p (-)^{t_{C_2}\mu_p}"] & & \prod_p \Sp^{h_{C_2} S^1} ).
\end{tikzcd}
\]

\end{dfn}

\begin{rem}
Both $\RCycSp$ and $\RCycSp_p$ admit the richer structure of a $C_2$-$\infty$-category; we let
\[ \underline{\RCycSp} \coloneq [\res: \RCycSp \to \CycSp], \]
where $\res$ is $C_2$-equivariant for the trivial action on the source and the residual $C_2$-action on the target, and likewise for $\underline{\RCycSp}_p$ (cf. \cref{dfn:C2CategoryOfCyclotomicSpectra} and \cref{dfn:realintegral}, where these definitions are formulated using the $C_2$-parametrized version of the lax equalizer and the $C_2$-endofunctors $(-)^{\underline{t}_{C_2} \mu_p}$).
% This extra structure is used to prove many of the results stated later in the Introduction.
% In particular, we establish $C_2$-symmetric monoidal structures on these $C_2$-$\infty$-categories.
\end{rem}

% Our primary motivation for studying real cyclotomic structures is understanding real topological Hochschild homology. In forthcoming work of Harpaz, Nikolaus, and the second author, it is shown that the real topological Hochschild homology of a Poincar{\'e} $\infty$-category admits a Borel real cyclotomic structure. In this paper, we restrict attention to $C_2$-$E_\infty$-algebras. \textcolor{red}{reference for definition of $G$-$E_\infty$-algebra?}

\cite{HarpazNikolausShah} will contain a construction of the real topological Hochschild homology of a Poincar{\'e} $\infty$-category as a Borel real cyclotomic spectrum. For the computations we do in this paper, it suffices to make the following more restrictive definition:

\begin{dfn}[{\cref{dfn:THR}}]
Let $A$ be a $C_2$-$E_\infty$-algebra in $\underline{\Sp}^{C_2}$. The \emph{real topological Hochschild homology $\THR(A)$} is the $C_2$-tensor of $A$ with the $C_2$-space $S^\sigma$ in the $C_2$-$\infty$-category of $C_2$-$E_\infty$-algebras in $\underline{\Sp}^{C_2}$, i.e.
$$\THR(A) := S^\sigma \odot A \in \underline{\CAlg}_{C_2}(\underline{\Sp}^{C_2}).$$
\end{dfn}
% $\underline{\CAlg}_{C_2}(\underline{\Sp}^{C_2})$ 

In the following theorem, proven in \cref{sec:thr}, we use the $C_2$-symmetric monoidal structure on $\underline{\RCycSp}$ that is defined using the unique lax $C_2$-symmetric monoidal structure on $(-)^{\underline{t}_{C_2}\mu_p}$ given by \cite[Thm.~E]{QS21a}.

\begin{thmx}\label{MT:C2EooTHR}
Let $A$ be a $C_2$-$E_\infty$-algebra in $\underline{\Sp}^{C_2}$. Its real topological Hochschild homology $\THR(A)$ admits the structure of a $C_2$-$E_\infty$-algebra in $\underline{\RCycSp}$. 
\end{thmx}

\begin{rem}
Real topological Hochschild homology and its (genuine) real cyclotomic structure have been defined more generally for ring spectra with anti-involution in \cite{DMPR17} and \cite{Hog16}. The new contribution of \cref{MT:C2EooTHR} is deducing that $\THR(A)$ admits the structure of a $C_2$-$E_\infty$-algebra in real cyclotomic spectra; to our knowledge, this genuine equivariant multiplicative structure on $\THR(A)$ has not been discussed in previous work. 
\end{rem}

\begin{exm}
If $\underline{M}$ is a $C_2$-Tambara functor, then its Eilenberg--MacLane $C_2$-spectrum $H\underline{M}$ is a $C_2$-$E_\infty$-algebra in $\underline{\Sp}^{C_2}$ by \cite[Thm. 5.1]{Ull13b}. Subsequently, we get that $\THR(H\underline{M})$ is a $C_2$-$E_{\infty}$-algebra in $\underline{\RCycSp}$.
 % where $\underline{\FF_p}$ is the constant $C_2$-Tambara functor on $\FF_p$. 
\end{exm}

As in the genuine setting, there is an exact, colimit-preserving, and symmetric monoidal functor
$$\mr{triv}_{\RR,p} : \Sp^{C_2} \to \RCycSp_p$$
that equips a $C_2$-spectrum with a `trivial' real $p$-cyclotomic structure (\cref{cnstr:trivialFunctor}), and likewise for $\RCycSp$. We define real topological cyclic homology using this functor:

\begin{dfn}[{\cref{dfn:newTCR}}]
The \emph{$p$-typical real topological cyclic homology} functor 
$$\TCR(-,p) : \RCycSp_p \to \Sp^{C_2}$$
is the right adjoint to $\mr{triv}_{\RR,p}$. Similarly, $\TCR$ is the right adjoint to $\mr{triv}_{\RR}$.

If $X = \THR(R)$, we write $\TCR(R) \coloneq \TCR(\THR(R))$ and likewise for $\TCR(R,p)$.
\end{dfn}

\begin{rem}
Under the decomposition \eqref{eq:pullback_integralrcyc} of $\RCycSp$, $\TCR(X)$ identifies as the pullback
\[
\begin{tikzcd}
\TCR(X) \ar[r] \ar[d] & X^{h_{C_2}S^1} \ar[d] \\
\prod_{p} \TCR(X,p)_p^\wedge \ar[r] & \prod_{p} (X_p^\wedge)^{h_{C_2}S^1}
\end{tikzcd}
\]
in analogy to \eqref{eqn:integral_TC}.
% As far as we are aware, this amounts to the first definition of integral real topological cyclic homology. 
\end{rem}

A fundamental observation of Blumberg--Mandell \cite{BM16} in their study of cyclotomic spectra is that genuine $p$-typical topological cyclic homology is corepresentable by the unit in the category of genuine $p$-cyclotomic spectra. Nikolaus--Scholze prove an analogous corepresentability result for $p$-typical topological cyclic homology in the Borel equivariant setting (cf. \cite[Rem. II.6.10]{NS18}). 

We show that both the genuine and Borel variants of $p$-typical real topological cyclic homology are corepresentable. Since real topological cyclic homology is naturally a $C_2$-spectrum, corepresentability must be understood $C_2$-equivariantly. We give a general discussion of $G$-corepresentability around \cref{dfn:Gcorepresentable}. 

\begin{thmx}[{\cref{prp:C2representabilityOfTC} and \ref{prp:classicalTCRcorepresentable}}]\label{Thmx:Corep}
$\TCR(-,p)$ and $\TCR^{\mr{gen}}(-,p)$ are $C_2$-corepresentable by the units in $\RCycSp_p$ and $\RCycSp_p^{\mr{gen}}$, respectively. 
\end{thmx}

\begin{rem}
The same reasoning behind \cref{Thmx:Corep} applies to show that $\TCR$ and $\TCR^{\mr{gen}}$ are $C_2$-corepresentable by the unit in $\RCycSp$ and $\RCycSp^{\mr{gen}}$, respectively.
\end{rem}

Since $\RCycSp^{\mr{gen}}_p$ is defined as an equalizer and $\RCycSp_p$ is defined as a lax equalizer, we obtain limit formulas for their mapping $C_2$-spectra. Combined with the $C_2$-corepresentability of $\TCR^{\mr{gen}}(-,p)$ and $\TCR(-,p)$, we obtain several formulas for real topological cyclic homology. 

To fully describe these formulas in the genuine setting, we need to introduce the restriction and Frobenius maps. For $X \in \RCycSp^{\mr{gen}}_p$, we have maps of genuine $C_2$-spectra
$$R,F : \Psi^{\mu_{p^n}} X \to \Psi^{\mu_{p^{n-1}}} X,$$
which can be used to define
$$\TRR(X,p) := \lim_{n,R} \Psi^{\mu_{p^n}} X,$$
$$\TFR(X,p) := \lim_{n,F} \Psi^{\mu_{p^n}} X.$$

\begin{rem}
The $C_2$-spectrum $\TRR(X,p)$ is studied extensively in the work of Dotto--Moi--Patchkoria \cite{DMP21}, where they relate it to the Witt vectors for Tambara functors. 
\end{rem}

\begin{thmx}[{Propositions \ref{prp:fiberSequenceGenuineRealCyc}, \ref{prp:TCRfiberSequence}, and \ref{Prp:ProfiniteCompletion}}]\label{Thmx:FiberSeq}
$ $ 
\begin{enumerate}

\item Let $X$ be a genuine real $p$-cyclotomic spectrum. There are natural fiber sequences of $C_2$-spectra
$$\TCR^{\mr{gen}}(X,p) \to \TFR(X,p) \xrightarrow{ id - R} \TFR(X,p),$$
$$\TCR^{\mr{gen}}(X,p) \to \TRR(X,p) \xrightarrow{ id  - F} \TRR(X,p),$$
and a natural equivalence of $C_2$-spectra
$$\TCR^{\mr{gen}}(X,p) \simeq \lim_{n,F,R} \Psi^{\mu_{p^n}} X.$$

\item Let $X$ be a Borel real $p$-cyclotomic spectrum. There is a natural fiber sequence of $C_2$-spectra
$$\TCR(X,p) \to X^{h_{C_2} \mu_{p^{\infty}}} \xrightarrow{ \varphi_p^{h_{C_2} \mu_{p^{\infty}} } - \can} (X^{t_{C_2}\mu_p})^{h_{C_2} \mu_{p^{\infty}}}.$$

Moreover, $X^{t_{C_2} \mu_{p^{\infty}}} \simeq (X^{t_{C_2}\mu_p})^{h_{C_2} \mu_{p^{\infty}}}$ if the underlying spectrum of $X$ is bounded-below.

\item Let $X$ be a Borel real cyclotomic spectrum. There is a natural fiber sequence of $C_2$-spectra
$$\TCR(X) \to X^{h_{C_2}S^1} \xrightarrow{(\varphi_p^{h_{C_2}S^1} - \can)_{p}} \prod_{p} (X^{t_{C_2}\mu_p})^{h_{C_2}S^1}.$$

Moreover, the rightmost term identifies with the profinite completion of $X^{t_{C_2}S^1}$ if the underlying $C_2$-spectrum of $X$ is bounded-below.
\end{enumerate}
% If the underlying $C_2$-spectrum of $X$ is bounded-below, then the right-most terms of (2) and (3) may be identified with the $p$-completion of $X^{t_{C_2} \mu_{p^{\infty}}}$, resp. the profinite completion of $X^{t_{C_2}S^1}$.
\end{thmx}

\begin{rem}
The formula for $\TCR^{\mr{gen}}(X,p)$ recovers the definition given by H{\o}genhaven in \cite{Hog16}. Therefore, our definition of $\TCR^{\mr{gen}}(-,p)$ as the right adjoint to $\mr{triv}_{\RR,p}^{\mr{gen}}$ coincides with previous definitions. 
\end{rem}

The main theorem of this work is that for suitably bounded-below objects, the $\infty$-category of genuine real cyclotomic spectra is equivalent to the $\infty$-category of Borel real cyclotomic spectra. We compare these theories using real versions of the forgetful functor \eqref{Eqn:ForgetfulFunctor} and its $p$-typical variant
\begin{align*}
\sU_\RR:& \RCycSp^{\mr{gen}} \to \RCycSp, \\
\sU_{\RR,p}:& \RCycSp^{\mr{gen}}_p \to \RCycSp_p.
\end{align*}
In fact, $\sU$ and $\sU_{\RR}$ sit together in a forgetful $C_2$-functor
\begin{align*}
\underline{\sU}_\RR:& \underline{\RCycSp}^{\mr{gen}} \to \underline{\RCycSp}
\end{align*}
as the fibers over $C_2/1$ and $C_2/C_2$, respectively, and $\underline{\sU}_\RR$ comes endowed with the structure of a $C_2$-symmetric monoidal functor; similarly for $\underline{\sU}_{\RR,p}$. Furthermore, the functors $\triv_{\RR,p}$ and $\triv_{\RR}$ promote to $C_2$-symmetric monoidal functors, so $\TCR(-,p)$ and $\TCR(-)$ promote to lax $C_2$-symmetric monoidal functors. We then have the following theorem as a real version of the main theorem of \cite{NS18}:

\begin{thmx}[{Theorems \ref{thm:MainTheoremEquivalenceBddBelow} and \ref{thm:integral_comparison}}]\label{Thmx:Equiv}\label{Thmx:TCREquiv}
$ $
\begin{enumerate}

\item The functors $\sU_{\RR,p}$ and $\sU_\RR$ restrict to equivalences between the full subcategories on those objects whose underlying spectrum is bounded-below. 

\item Let $X$ be a genuine real ($p$-)cyclotomic spectrum whose underlying spectrum is bounded-below. There is a canonical equivalence of $C_2$-spectra
$$\TCR^{\mr{gen}}(X,p) \simeq \TCR(\sU_{\RR,p}(X),p)$$ 
that respects $C_2$-equivariant multiplicative structures, and similarly for the integral theories. In particular, if $X$ is $\THR$ of a $C_2$-$E_{\infty}$-algebra, then this is an equivalence of $C_2$-$E_{\infty}$-algebras.
\end{enumerate}
\end{thmx}

\begin{rem}
We emphasize that the only restriction appearing in \cref{Thmx:Equiv} is that the underlying spectrum, as opposed to the underlying $C_2$-spectrum, is bounded-below. This weaker hypothesis is crucial for applications: the homotopy groups $\pi_*^{C_2} \THR(R)$ are generally \emph{not} bounded-below.
\end{rem}

As in \cite{NS18}, the main input needed for the proof of \cref{Thmx:Equiv} is the \emph{dihedral Tate orbit lemma}:
\begin{thmx}[{\cref{lem:dihedralTOLEven} and \ref{lem:dihedralTOLOdd}}]\label{Thmx:TO}
The functor given by the composite
$$\Sp^{h_{C_2} \mu_{p^2}} \xrightarrow{(-)_{h_{C_2}\mu_p}} \Sp^{h_{C_2} \mu_{p}} \xrightarrow{(-)^{t_{C_2}\mu_p}} \Sp^{C_2}$$
evaluates to $0$ on those objects $X$ whose underlying spectrum is bounded-below.
\end{thmx}

\begin{rem}
As a corollary of \cref{Thmx:TO}, we show that for any $D_{2p^n}$-spectrum $X$ with $\Phi^{\mu_{p^k}} X$ bounded-below as a spectrum for all $0 \leq k \leq n-1$, the $C_2$-spectrum of $\mu_{p^n}$-categorical fixed points $\Psi^{\mu_{p^n}}X$ can be computed as an iterated pullback of parametrized homotopy fixed points and Tate constructions (\cref{lm:SameLemmaNS}). Along the way to the proof of \cref{Thmx:Equiv}, we categorify this relation by establishing a decomposition of the $\infty$-category of suitably bounded-below $D_{2p^n}$-spectra as an iterated pullback (\cref{cor:iterated_pullback_descr}); compare \cite[Rem.~II.4.8]{NS18}.
% The categorical fixed points $\Psi^{\mu_{p^n}}X$ are relevant, in particular, for computing real topological restriction homology $TRR(X) := \lim \Psi^{\mu_{p^n}}X$ of any real cyclotomic spectrum $X$, cf. \cite[Thm. D]{DMP19}. 
\end{rem}

Taken in conjuction, \cref{Thmx:FiberSeq} and \cref{Thmx:TCREquiv} provide a powerful new tool for studying real topological cyclic homology. In this paper, we give a sample computation by using the fiber sequence formula to prove:
% The parametrized homotopy fixed points $X^{h_{C_2}S^1}$ and parametrized Tate construction $X^{t_{C_2}S^1}$ are accessible by spectral sequences discussed in \cite{GM95, QS21a}. We apply these spectral sequences to make the following computation:

\begin{thmx}[{Theorems \ref{thm:odd_primary_computation}, \ref{Thm:TCRF2}, and \ref{thm:extn_perfect_algebras}}]\label{Thmx:TCRFp}
Let $p$ be a prime and let $H\underline{M}$ denote the Eilenberg--MacLane $C_2$-spectrum for the constant $C_2$-Mackey functor $\underline{M}$. There is an equivalence of $C_2$-spectra
$$\TCR(H\underline{\FF_p}) \simeq H\underline{\ZZ_p} \oplus \Sigma^{-1} H\underline{\ZZ_p}.$$

More generally, let $R$ be a perfect $\FF_p$-algebra. Then there is an equivalence of $C_2$-spectra
$$\TCR(H\underline{R}) \simeq H \underline{\ker(1-F)} \oplus \Sigma^{-1}H\underline{\coker(1-F)},$$
where $F: W(R;p) \to W(R;p)$ is the $p$-typical Witt vector Frobenius.
\end{thmx}

\begin{rem}
In our previous work \cite{QS19}, we proved \cref{Thmx:TCRFp} for odd primes. Since then, Dotto--Moi--Patchkoria \cite[Thm. D]{DMP21} have identified $\TCR^{\mr{gen}}(H\underline{k}, p)$ for every perfect field $k$ of characteristic $p>0$. In their computation, they compute the geometric fixed points of $\THR(H\underline{k})^{\mu_{p^n}}$ for all $n$ \cite[Thm.~4.6]{DMP21} and use them (along with their previous results in \cite{DMP19}) to compute $\TCR^{\mr{gen}}(H\underline{k},p)$ as an inverse limit, cf. part (1) of \cref{Thmx:FiberSeq}. Our computation of $\TCR(H\underline{R},p)$  bypasses these finite-level computations.
 % and instead computes $\TCR(H\underline{\FF_p},p)$ via the fiber sequence in part (2) of \cref{Thmx:FiberSeq}; the other terms in the fiber sequence are analyzed using various spectral sequences. 

% Although \cref{Thmx:TCRFp} is only stated for $H\underline{\FF_p}$, its proof can actually be extended to show for any perfect $\FF_p$-algebra $k$ that 
% $$\TCR^{\mr{gen}}(H\underline{k},p) \simeq H\underline{Z}_p \vee \Sigma^{-1}H\underline{\coker(1-F)},$$
% where $F: W(k;p) \to W(k;p)$ is the $p$-typical Witt vector Frobenius. We discuss this further in \cref{SS:Perf}. 
\end{rem}

\begin{rem}
Using our new results on the multiplicative structure of $\THR$ and $\TCR$, it is possible to identify $\THR(H\underline{\FF_p})$ as a $C_2$-$E_\infty$-algebra in real cyclotomic spectra along the lines of \cite[Cor. IV.4.13]{NS18}. We will not pursue this identification further here, but it would be interesting to explore its implications. In particular, this identification should be relevant for arithmetic applications of real cyclotomic spectra; cf. the proof of \cite[Thm. 8.17]{BMSII}. 
\end{rem}

\begin{rem}
Using \cref{Thmx:FiberSeq} and \cref{Thmx:TCREquiv}, it is also possible to recover H{\o}genhaven's computation of $\TCR^{\mr{gen}}(\SS[\Gamma])$, the real topological cyclic homology of spherical group rings with anti-involution defined by inverting group elements. In fact, we can obtain a similar identification of $\TCR^{\mr{gen}}(\SS[\Omega^\sigma M])$, where $\Omega^\sigma M = \Map(S^\sigma,M)$ is the $C_2$-space of maps from the one-point compactification of the real sign representation of $C_2$ into $M$. This identification has also been obtained by Dotto--Moi--Patchkoria \cite[Thm. B]{DMP21}. 
\end{rem}

\subsection{Conventions, notation, and prerequisites}

\begin{enumerate}[leftmargin=4ex,label*=\arabic*.]
\item For an $\infty$-category $C$, we denote its mapping spaces by $\Map_C(-,-)$ or $\Map(-,-)$ if $C$ is understood. Likewise, if $C$ is a stable $\infty$-category, then we denote its mapping spectra by $\map_C(-,-)$ or $\map(-,-)$.
\item As in \cite{QS21a}, we will assume knowledge of the theory of recollements in this paper and refer to \cite{Sha21} as our primary reference. Given a stable $\infty$-category $\sX$ decomposed by a stable recollement $(\sU, \sZ)$, we will generically label the recollement adjunctions as
\[ \begin{tikzcd}[column sep=4em]
\sU \ar[hookrightarrow, shift left=2]{r}{j_!} \ar[hookrightarrow, shift right=4]{r}[swap]{j_*} & \sX \ar[shift left=1]{l}[description]{j^*} \ar[shift left=2]{r}{i^*} \ar[shift right=1, hookleftarrow]{r}[swap, description]{i_*} \ar[shift right=4]{r}[swap]{i^!} & \sZ.
\end{tikzcd} \]
Here $j^* i_* = 0$ determines the directionality of the recollement.

In addition, we will make use of the notion of $G$-recollement \cite[Def.~2.43]{Sha21} in the case $G = C_2$.

\item The conventions, notation, and terminology from our companion paper \cite{QS21a} extend to this work. See \cite[\S 2.1]{QS21a} for our conventions on equivariant stable homotopy theory, see \cite[\S 1]{QS21a} and the introduction of \cite[\S 4]{QS21a} for the relevant background from parametrized higher category theory,\footnote{See also \cite[App. A]{QS19}.} and see \cite[\S 5.1]{QS21a} for the relevant background from parametrized higher algebra. In particular, for $G$ a finite group, the reader should be aware of the concepts of:

\begin{enumerate}[leftmargin=6ex,label*=\arabic*.]

\item The $G$-$\infty$-categories $\underline{\Spc}^G$ and $\underline{\Sp}^G$ of $G$-spaces and $G$-spectra;

\item $G$-functor categories and the notation $\underline{\Fun}_G(-,-)$ for the internal hom in $G$-$\infty$-categories;

\item $G$-adjunctions \cite[\S 8]{Exp2} (or see \cite[Def.~1.33]{QS21a});

\item $G$-(co)products \cite[Def.~5.10]{Exp2} (or see \cite[Exm.~1.32]{QS21a}), $G$-(co)limits \cite[\S 5]{Exp2}, and $G$-Kan extensions \cite[\S 10]{Exp2}: given a functor of small\footnote{A $G$-$\infty$-category $I \to \sO_G^{\op}$ is \emph{small} if $I$ is small as an $\infty$-category.} $G$-$\infty$-categories $f: I \to J$ and a $G$-cocomplete $G$-$\infty$-category $C$, the restriction $G$-functor
\[ \phi^*: \underline{\Fun}_G(J, C) \to \underline{\Fun}_G(I,C) \]
admits a left $G$-adjoint $\phi_!$ given by (pointwise) $G$-left Kan extension, and similarly for $G$-complete $G$-$\infty$-categories and $G$-right Kan extension.

\item $G$-semiadditivity as the condition under which $G$-coproducts and $G$-products coincide \cite[Def.~4.14]{QS21a}, and $G$-stable $G$-$\infty$-categories defined as fiberwise stable $G$-semiadditive $G$-$\infty$-categories \cite{Exp4};

\item $G$-symmetric monoidal $\infty$-categories and $G$-commutative algebras \cite[\S 5.1]{QS21a}.
\end{enumerate}

% \item For an $\infty$-category $C$, we let $\Ar(C) = \Fun(\Delta^1,C)$ denote the $\infty$-category of arrows in $C$.
%Arrow category is used without definition in QS21a

%\item For a simplicial set $S$, let $S^{\sharp}$ be the marked simplicial set with all edges marked, and let $S^{\flat}$ be the marked simplicial set with only the degenerate edges marked. For a (locally) cocartesian fibration $p: C \to S$, let $\leftnat{C}$ be the marked simplicial set with the (locally) $p$-cocartesian edges marked, and let $\leftnat{\Lambda^n_0}$, $\leftnat{\Delta^n}$ indicate that the edge $\{0,1\}$ is marked. Dually, we may consider $\rightnat{C}$ if $p$ is a cartesian fibration, and $\rightnat{\Lambda^n_n}$, $\rightnat{\Delta^n}$ with the edge $\{n-1,n\}$ marked.
%This notation is never used later. 

% \item Constructions made internal to an $\infty$-category, such as limits and colimits, are necessarily homotopy invariant, so we will typically suppress the adjective `homotopy' in our discussion. We also suppress routine arguments that concern the homotopy invariance of constructions involving $\infty$-categories that are made in simplicial sets or marked simplicial sets.

% \item Let $C = (C, \otimes, 1)$ be a symmetric monoidal $\infty$-category. If $C$ is stable, then we require that the tensor product on $C$ is exact separately in each variable, and if $C$ is presentable, then we require that the tensor product on $C$ commutes with colimits separately in each variable.

\item When working with Mackey functors, we write $\underline{A}$ for the Burnside Mackey functor, $\underline{B}$ for the constant Mackey functor on a group $B$, and $\underline{I}$ for the kernel of the augmentation $\underline{A} \to \underline{\mathbb{Z}}$. We will usually write $\underline{M}$ for a generic $C_2$-Mackey functor. If $S$ is a $G$-set, we write $\underline{M}_S$ for the Mackey functor defined by $\underline{M}_S(T) = \underline{M}(S \times T)$. 

% \item Our grading conventions for $RO(C_2)$-graded homotopy groups appear in \cref{Cvn:Grading}. 

% \item If $C$ is a closed symmetric monoidal $\infty$-category (e.g., $C$ is presentably symmetric monoidal), then we typically denote its internal hom by $F_C(-,-)$ or just $F(-,-)$.
% \item Let $F: C \to D$ be a functor between two symmetric monoidal $\infty$-categories. Then we say that $F$ is \emph{lax monoidal} if $F$ lifts to the structure of a functor $F^{\otimes}: C^{\otimes} \to D^{\otimes}$ of $\infty$-operads (so $F^{\otimes}$ is a functor over $\Fin_{\ast}$ that preserves inert edges). If $F^{\otimes}$ moreover preserves all cocartesian edges, then we say that $F$ is \emph{symmetric monoidal} or simply \emph{monoidal}.
% \item Let $C$ and $D$ be symmetric monoidal $\infty$-categories and let $\adjunct{L}{C}{D}{R}$ be an adjunction. Then $L \dashv R$ is \emph{monoidal} if $L$ is symmetric monoidal (so then $R$ is necessarily lax monoidal).

%\item In contrast to the introduction, we will typically denote the smash product of $G$-spectra (and related $\infty$-categories) by the symbol $\otimes$ instead of $\wedge$.
%We don't use the old smash product notation anywhere. 

%\item For a finite group $G$, we let $\FF_G$ be the category of finite $G$-sets and $\sO_G \subset \FF_G$ be the full subcategory on the nonempty transitive $G$-sets.
%This notation is already used in QS21a
\end{enumerate}

% \begin{rem} The full subcategory of $\sO_G$ spanned by the orbits $\{G/H : H \leq G \}$ constitutes a skeleton of $\sO_G$, where given a finite nonempty transitive $G$-set $U$, a choice of basepoint $b \in U$ specifies an isomorphism $U \cong G/H$ with $H = \{ h \in G : h \cdot b = b \}$ and $b \mapsto 1 H$. To avoid some basepoint technicalities, we opt for the basepoint-free definition of $\sO_G$. Note that we may always pass to a skeleton of $\sO_G$ when checking \emph{conditions} that involve $\sO_G$ in some way -- e.g., to check if a natural transformation of presheaves on $\sO_G$ is an equivalence, it suffices to check on orbits.
% \end{rem}

\subsection{Acknowledgments}

This work is an expansion of \cite[Secs. 6-7]{QS19}. The main changes are as follows:
\begin{enumerate}[leftmargin=4ex]
	\item We give a systematic development of $C_2$-symmetric monoidal structures on real cyclotomic spectra. We establish that the real topological Hochschild homology of a $C_2$-$E_\infty$-algebra in $C_2$-spectra is a $C_2$-$E_\infty$-algebra in real cyclotomic spectra.
	\item We include results on integral real cyclotomic spectra, which in part rely on the new results in \cite{QS21a} on the parametrized Tate construction for compact Lie groups. 
	\item We compute $\TCR(H\underline{\FF}_p)$ for all primes, instead of just odd primes, and we compute $\TCR$ for all perfect $\FF_p$-algebras. We replace the spectral sequence approach for odd primes of \cite[\S 7.5]{QS19} with a direct approach involving geometric fixed points.
\end{enumerate}

The authors thank Mark Behrens, Andrew Blumberg, Emanuele Dotto, Jeremy Hahn, Achim Krause, Kristian Moi, Thomas Nikolaus, Irakli Patchkoria, Dylan Wilson, and Mingcong Zeng for helpful discussions. The authors were partially supported by NSF grant DMS-1547292. J.S. was also funded by the Deutsche Forschungsgemeinschaft (DFG, German Research Foundation) under Germany’s Excellence Strategy EXC 2044–390685587, Mathematics Münster: Dynamics–Geometry– Structure.

\section{Two theories of real \texorpdfstring{$p$}{p}-cyclotomic spectra}\label{Sec:RealCyc}

In this section, we define $\infty$-categories $\RCycSp_p$ and $\RCycSp^{\mr{gen}}_p$ of \emph{Borel} and \emph{genuine} real $p$-cyclotomic spectra (\cref{dfn:RealCycSp} and \cref{dfn:GenRealCycSp}). For each $\infty$-category, we define the corresponding theories $\TCR(-,p)$ and $\TCR^{\mr{gen}}(-,p)$ of $p$-typical real topological cyclic homology (\cref{dfn:newTCR} and \cref{dfn:ClassicalTCR}) as functors to $\Sp^{C_2}$ right adjoint to functors that endow $C_2$-spectra with trivial real $p$-cyclotomic structure (\cref{cnstr:trivialFunctor} and \cref{cnstr:GenuineTrivialFunctor}), show $\TCR(-,p)$ and $\TCR^{\mr{gen}}(-,p)$ are $C_2$-corepresentable (\cref{dfn:Gcorepresentable}, \cref{prp:C2representabilityOfTC}, and \cref{prp:classicalTCRcorepresentable}), and thereby deduce fiber sequence formulas (\cref{prp:TCRfiberSequence} and \cref{prp:fiberSequenceGenuineRealCyc}) that in the case of $\TCR^{\mr{gen}}(-,p)$ recovers the more standard definition in terms of maps $R$ and $F$. To begin our study, we need to fix a few conventions regarding the dihedral groups and dihedral spectra.

% \begin{rec}[{\cite[Exm.~3.26]{QS21a}}]
% We have a stable symmetric monoidal recollement
% \begin{equation} \label{eq:dihedral_recollement}
% \begin{tikzcd}[row sep=4ex, column sep=10ex, text height=1.5ex, text depth=0.5ex]
% \Fun_{C_2}(B^t_{C_2} \mu_{p^n}, \underline{\Sp}^{C_2}) \ar[shift right=1,right hook->]{r}[swap]{j_{\ast} = \sF^{\vee}_b[\mu_{p^n}]} & \Sp^{D_{2p^n}} \ar[shift right=2]{l}[swap]{j^{\ast} = \sU_b[\mu_{p^n}]} \ar[shift left=2]{r}{i^{\ast} = \Phi^{\mu_p}} & \Sp^{D_{2p^{n-1}}} \ar[shift left=1,left hook->]{l}{i_{\ast}},
% \end{tikzcd}
% \end{equation}
% which coincides with the recollement on $\Sp^{D_{2p^n}}$ determined by the family $\Gamma_{\mu_{p^n}} = \{ H \leq D_{2p^n}: H \cap \mu_{p^n} = 1 \}$ of $\mu_{p^n}$-free subgroups (\cite[Thm.~3.23 and Rem.~3.24]{QS21a}).
% \end{rec}

\begin{setup} \label{setup:Dihedral} Let $\mathrm{O}(2)$ denote the group of $2 \times 2$ orthogonal matrices, and regard the circle group $S^1$ as the subgroup $\mathrm{SO}(2) \subset \mathrm{O}(2)$. We fix, once and for all, a splitting of the determinant $\det: \mathrm{O}(2) \to C_2 \cong \{ \pm 1 \}$ by choosing $\sigma \in \mathrm{O}(2)$ to be the $\det =-1$ matrix given by
\[
   \sigma =
  \left[ {\begin{array}{cc}
   0 & 1 \\
   1 & 0 \\
  \end{array} } \right].
\]
This exhibits $\mathrm{O}(2)$ as the semidirect product $S^1 \rtimes C_2$ for $C_2 = \angs{\sigma} \subset \mathrm{O}(2)$, where $C_2$ acts on $S^1$ by complex conjugation, i.e., inversion. For $0 \leq n \leq \infty$, let $\mu_{p^n} \subset S^1$ be the subgroup of $p^n$th roots of unity, and let $D_{2p^{n}} \subset \mathrm{O}(2)$ be the subgroup $\mu_{p^n} \rtimes C_2$. For $n<\infty$, we let $\{ x_n \}$ denote a compatible system of generators for $\mu_{p^n}$ (so $x_n = x_{n+1}^p$ for all $n \geq 0$), and we also set $x = x_n$ if there is no ambiguity about the ambient group.\footnote{We will only refer to these elements in our proof of the dihedral Tate orbit lemma in \cref{sec:dihedral_tate_orbit_lemma}.} When considering restriction functors $\Sp^{D_{2p^n}} \to \Sp^{D_{2p^m}}$ for $m \leq n$, we always choose restriction to be with respect to the inclusion $D_{2p^m} \subset D_{2p^n}$ induced by $\mu_{p^m} \subset \mu_{p^n}$. Then we define
\[ \Sp^{D_{2p^{\infty}}} \coloneq \lim_n \Sp^{D_{2p^n}} \]
to be the inverse limit taken along these restriction functors. Since the restriction functors are symmetric monoidal and colimit preserving, we may take the inverse limit in $\CAlg(\Pr^{L,\st})$, and $\Sp^{D_{2p^{\infty}}}$ is then a stable presentable symmetric monoidal $\infty$-category. 

 We also define the $C_2$-space $B^t_{C_2} \mu_{p^{\infty}}$ to be the full subcategory of $\sO_{D_{2p^{\infty}}}^{\op}$ on the $\mu_{p^{\infty}}$-free orbits.\footnote{For an infinite group $G$ like $D_{2p^{\infty}}$, we let $\sO_{G}$ be the category of non-empty transitive $G$-sets, which need not be finite.} Note that an orbit $D_{2p^n}/H$ is $\mu_{p^n}$-free if and only if $H=1$ or $H = \angs{\sigma z}$ for some $z \in \mu_{p^n}$, where $H$ then has order $2$ since $(\sigma z)^2 = z^{-1} z = 1$. Indeed, supposing $H \neq 1$, since $H \cap \mu_{p^n} = 1$, if $\sigma z$ and $\sigma z'$ are two elements in $H$, then we must have $(\sigma z)(\sigma z') = z^{-1} z' = 1$, so $z = z'$. It follows that with respect to the induction functors $\sO_{D_{2p^m}} \to \sO_{D_{2p^n}}$ induced by the above inclusions, we have
\[ \colim_n B^t_{C_2} \mu_{p^n} \xto{\simeq} B^t_{C_2} \mu_{p^{\infty}}    \]
as a filtered colimit of $C_2$-spaces. Therefore, we obtain an equivalence
\[ \Fun_{C_2}(B^t_{C_2} \mu_{p^{\infty}}, \ul{\Sp}^{C_2}) \xto{\simeq} \lim_n \Fun_{C_2}(B^t_{C_2} \mu_{p^n}, \ul{\Sp}^{C_2}).  \]
Note also that the restriction functors are colimit preserving and symmetric monoidal with respect to the pointwise monoidal structure, so the equivalence may be taken in $\CAlg(\Pr^{L,\st})$.

Next, recall from \cite[Exm.~3.26]{QS21a} that we have a stable symmetric monoidal recollement
\begin{equation} \label{eq:dihedral_recollement}
\begin{tikzcd}[row sep=4ex, column sep=10ex, text height=1.5ex, text depth=0.5ex]
\Fun_{C_2}(B^t_{C_2} \mu_{p^n}, \underline{\Sp}^{C_2}) \ar[shift right=1,right hook->]{r}[swap]{j_{\ast} = \sF^{\vee}_b[\mu_{p^n}]} & \Sp^{D_{2p^n}} \ar[shift right=2]{l}[swap]{j^{\ast} = \sU_b[\mu_{p^n}]} \ar[shift left=2]{r}{i^{\ast} = \Phi^{\mu_p}} & \Sp^{D_{2p^{n-1}}} \ar[shift left=1,left hook->]{l}{i_{\ast}},
\end{tikzcd}
\end{equation}
which coincides with the recollement on $\Sp^{D_{2p^n}}$ determined by the family $\Gamma_{\mu_{p^n}} = \{ H \leq D_{2p^n}: H \cap \mu_{p^n} = 1 \}$ of $\mu_{p^n}$-free subgroups. Moreover, by \cite[Cor.~3.25]{QS21a} we have a strict morphism (\cite[Def.~2.6]{Sha21}) of  recollements
\begin{equation*}
\begin{tikzcd}[row sep=4ex, column sep=8ex, text height=1.5ex, text depth=0.5ex]
 \Fun_{C_2}(B^t_{C_2} \mu_{p^n}, \ul{\Sp}^{C_2}) \ar{d}{\res} & \Sp^{D_{2p^n}} \ar{d}{\res} \ar{l}[swap]{\sU_b[\mu_{p^n}]} \ar{r}{\Phi^{\mu_p}} &  \Sp^{D_{2p^{n-1}}} \ar{d}{\res} \\
\Fun_{C_2}(B^t_{C_2} \mu_{p^m}, \ul{\Sp}^{C_2}) & \Sp^{D_{2p^{m}}} \ar{l}[swap]{\sU_b[\mu_{p^{m}}]} \ar{r}{\Phi^{\mu_p}} &  \Sp^{D_{2p^{m-1}}}
\end{tikzcd}
\end{equation*}
for all $0< m \leq n < \infty$. By \cite[Cor.~2.38]{Sha21}, passage to inverse limits defines a stable symmetric monoidal recollement
\begin{equation} \label{eq:dihedral_recollement_infinite}
\begin{tikzcd}[row sep=4ex, column sep=8ex, text height=1.5ex, text depth=0.5ex]
\Fun_{C_2}(B^t_{C_2} \mu_{p^{\infty}}, \ul{\Sp}^{C_2}) \ar[shift right=1,right hook->]{r}[swap]{j_{\ast} = \sF^{\vee}_b } & \Sp^{D_{2p^{\infty}}} \ar[shift right=2]{l}[swap]{j^{\ast} = \sU_b} \ar[shift left=2]{r}{i^{\ast} = \Phi^{\mu_p}} & \Sp^{D_{2p^{\infty}}} \ar[shift left=1,left hook->]{l}{i_{\ast}}
\end{tikzcd}
\end{equation}
where we implicitly use the isomorphism $D_{2p^{\infty}}/\mu_p \cong D_{2p^{\infty}}$ induced by the $p$th power map for $\mu_{p^{\infty}}/\mu_p \cong \mu_{p^{\infty}}$ to regard $\Phi^{\mu_p}$ as an endofunctor. By also using the compatibility of restriction with categorical fixed points, we obtain the lax symmetric monoidal endofunctor $\Psi^{\mu_p}$ of $\Sp^{D_{2p^{\infty}}}$, and we retain the relation $\Psi^{\mu_p} \circ i_{\ast} \simeq \id$. Now consider the fiber sequence of functors
% \[ j^{\ast} \Psi^{\mu_p} j_! \to j^{\ast} \Psi^{\mu_p} j_{\ast} \to j^{\ast} \Psi^{\mu_p} i_{\ast} i^{\ast} j_{\ast} \simeq j^{\ast} \Phi^{\mu_p} j_{\ast}. \]
\[ j^{\ast} \Psi^{\mu_p} j_! \to j^{\ast} \Psi^{\mu_p} j_{\ast} \to j^{\ast} \Psi^{\mu_p} i_{\ast} i^{\ast} j_{\ast} \simeq j^{\ast} \Phi^{\mu_p} j_{\ast}. \]
By the same argument as in \cite[Thm.~4.28]{QS21a}, this fiber sequence is equivalent to
\[ (-)_{h_{C_2} \mu_p} \xto{\Nm} (-)^{h_{C_2} \mu_p} \to (-)^{t_{C_2} \mu_p} \]
where the parametrized norm map is that associated to the weakly $1$-ambidextrous morphism of $C_2$-spaces
$$B^t_{C_2} \mu_{p^{\infty}} \to B^t_{C_2} (\mu_{p^{\infty}} / \mu_p) \simeq B^t_{C_2} \mu_{p^{\infty}}$$
with fiber $B^t_{C_2} \mu_p$ (see \cref{cnv:basepoint} for our basepoint convention). We may then use this identification to endow the natural transformation $(-)^{h_{C_2} \mu_p} \to (-)^{t_{C_2} \mu_p}$ of endofunctors of $\Sp^{h_{C_2} \mu_{p^{\infty}}}$ with the structure of a lax symmetric monoidal functor. 
\end{setup}

\begin{cvn} \label{cnv:basepoint} We will regard $B^t_{C_2} \mu_{p^n}$ as a based $C_2$-space via the functor $\sO_{C_2}^{\op} = B^t_{C_2} (1) \to B^t_{C_2} \mu_{p^n}$ induced by $\angs{\sigma} \subset D_{2p^n}$. We then say that for an object $X \in \Fun_{C_2}(B^t_{C_2} \mu_{p^n}, \ul{\Sp}^{C_2})$, evaluation on the $C_2$-basepoint yields the \emph{underlying $C_2$-spectrum} of $X$, and further restriction via $\res^{C_2}: \Sp^{C_2} \to \Sp$ yields the \emph{underlying spectrum} of $X$. Note that if $X = \sU_b[\mu_{p^n}](Y)$, then its underlying $C_2$-spectrum is $\res^{D_{2p^n}}_{\angs{\sigma}} (Y)$.
\end{cvn}

\begin{rem} \label{rem:DihedralBasepoints} Note that for $y,z \in D_{2p^n}$, $y^{-1} (\sigma z) y = \sigma y^2 z$. Therefore, for $p$ odd and all $0 \leq n \leq \infty$, any two subgroups $\angs{\sigma z}$ and $\angs{\sigma z'}$ of $D_{2p^n}$ are conjugate. In contrast, for $p=2$ and $1<n<\infty$, there are two conjugacy classes of order $2$ subgroups $H$ with $H \cap \mu_{p^n} = 1$, with representatives $\angs{\sigma}$ and $\angs{\sigma x}$. However, for $p=2$ and $n=\infty$, we again have a single conjugacy class since we can take square roots for $z \in \mu_{2^{\infty}}$.

It follows that $B^t_{C_2} \mu_{p^n}$ is a connected $C_2$-space for $p$ odd and any $n$ or for $p=2$ and $n \in \{ 0,1,\infty \}$, but its fiber over $C_2/C_2$ splits into two components when $p=2$ and $1<n<\infty$.
\end{rem}

We also will use the following notation for the (lax) equalizer of two endofunctors.

\begin{dfn}[{\cite[Def.~II.1.4]{NS18}}] Suppose $F$ and $F'$ are endofunctors of an $\infty$-category $C$. Define the \emph{lax equalizer} of $F$ and $F'$ to be the pullback
\[ \begin{tikzcd}[row sep=4ex, column sep=6ex, text height=1.5ex, text depth=0.5ex]
\LEq_{F:F'}(C) \ar{r} \ar{d} & \Ar(C) \ar{d}{(\ev_0,\ev_1)} \\
C \ar{r}{(F,F')} & C \times C.
\end{tikzcd} \]
Define the \emph{equalizer} $\Eq_{F:F'}(C) \subset \LEq_{F:F'}(C)$ to be the full subcategory on objects $[x,F(x) \xto{\phi} F'(x)]$ where $\phi$ is an equivalence.
\end{dfn}

\subsection{Borel real \texorpdfstring{$p$}{p}-cyclotomic spectra}\label{SS:BRCyc}

\begin{dfn} \label{dfn:RealCycSp} A \emph{real $p$-cyclotomic spectrum} is a $C_2$-spectrum $X$ with a twisted $\mu_{p^{\infty}}$-action, together with a twisted $\mu_{p^{\infty}}$-equivariant map $\varphi: X \to X^{t_{C_2} \mu_p}$. The $\infty$-category of \emph{real $p$-cyclotomic spectra} is then
$$\RCycSp_p := \LEq_{\id:t_{C_2} \mu_p}(\Fun_{C_2}(B^t_{C_2} \mu_{p^{\infty}}, \underline{\Sp}^{C_2})). $$
\end{dfn}

Such objects might be more accurately called \emph{Borel} real $p$-cyclotomic spectra, but we follow \cite{NS18} in our choice of terminology.

\begin{rem} We will sometimes abuse notation and refer to $X$ itself as the real $p$-cyclotomic spectrum, leaving the map $\varphi$ implicit.
\end{rem}

 We have the same conclusion as \cite[Cor.~II.1.7]{NS18} for $\RCycSp_p$, with the same proof.

\begin{prp} \label{prp:RCycSpPresentable} \label{prp:forgetful_conservative} $\RCycSp_p$ is a presentable stable $\infty$-category, and the forgetful functor $$\RCycSp_p \to \Sp^{C_2}$$ is conservative, exact, and creates colimits and finite limits.
\end{prp}
\begin{proof} The endofunctor $t_{C_2} \mu_p$ is exact and accessible as the cofiber of functors that admit adjoints, or as the composite $\sU_b \circ \Phi^{\mu_p} \circ \sF^{\vee}_b$ as noted in \cref{setup:Dihedral}. By \cite[Prop.~II.1.5]{NS18}, $\RCycSp_p$ is stable and presentable and the forgetful functor $$\RCycSp_p \to \Fun_{C_2}(B^t_{C_2} \mu_{p^{\infty}}, \ul{\Sp}^{C_2})$$ is colimit-preserving and exact. It is also obviously conservative, and since the $C_2$-space $B^t_{C_2} \mu_{p^{\infty}}$ has connected fibers over $C_2/C_2$ and $C_2/1$, the further forgetful functor to $\Sp^{C_2}$ is also conservative, exact, and colimit-preserving. Finally, any conservative functor between presentable $\infty$-categories that preserves $K$-indexed (co)limits necessarily also creates $K$-indexed (co)limits.
\end{proof}

% Let $\Sp^{D_{2p^{\infty}}} = \lim_n \Sp^{D_{2p^n}}$, where the inverse limit is taken along the restriction functors. The geometric fixed points functors $\Phi^{\mu_p}: \Sp^{D_{2p^{n+1}}} \to \Sp^{D_{2p^{n}}}$ commute with the restriction functors, so we obtain an endofunctor $\Phi^{\mu_p}$ of $\Sp^{D_{2p^{\infty}}}$.
\begin{cnstr}[Symmetric monoidal structure on $\RCycSp_p$] \label{con:smc_rcyc} Recall from \cite[IV.2.1]{NS18} that if $C$ is a symmetric monoidal $\infty$-category, $F$ is a symmetric monoidal functor, and $G$ is a lax symmetric monoidal functor, then $\LEq_{F:G}(C)$ acquires a symmetric monoidal structure by forming the pullback of $\infty$-operads\footnote{Here, the cotensor with $\Delta^1$ is taken relative to $\Fin_{\ast}$.}
% \footnote{This is analogous to how we defined the canonical symmetric monoidal structure on a recollement. Note again that the cotensor with $\Delta^1$ is taken relative to $\Fin_{\ast}$, and also that the righthand vertical map is induced by cotensoring with $\partial \Delta^1 \subset \Delta^1$.}
\[ \begin{tikzcd}[row sep=4ex, column sep=6ex, text height=1.5ex, text depth=0.5ex]
\LEq_{F:G}(C)^{\otimes} \ar{r} \ar{d} & (C^{\otimes})^{\Delta^1} \ar{d}{(\ev_0,\ev_1)} \\
C^{\otimes} \ar{r}{(F^{\otimes},G^{\otimes})} & C^{\otimes} \times_{\Fin_{\ast}} C^{\otimes}.
\end{tikzcd} \]
Let us then endow $\RCycSp_p$ with the symmetric monoidal structure given by taking $t_{C_2} \mu_p$ to have the lax symmetric monoidal structure as indicated in \cref{setup:Dihedral}.
\end{cnstr}

\begin{cnstr}[Trivial real $p$-cyclotomic structure] \label{cnstr:trivialFunctor} We construct an exact and colimit-preserving symmetric monoidal functor
\[ \mr{triv}_{\RR,p}: \Sp^{C_2} \to \RCycSp_p \]
that endows a $C_2$-spectrum with the structure of a real $p$-cyclotomic spectrum in a `trivial' way. Consider the maps of $C_2$-spaces
\[ \sO_{C_2}^{\op} \xto{\iota} B^t_{C_2} \mu_{p^{\infty}} \xto{\pi} B^t_{C_2} (\mu_{p^{\infty}}/\mu_p) \simeq B^t_{C_2} \mu_{p^{\infty}} \xto{p} \sO_{C_2}^{\op} \]
and the associated restriction functors
\[ \Sp^{C_2} \xto{p^{\ast}} \Fun_{C_2}(B^t_{C_2} \mu_{p^{\infty}}, \ul{\Sp}^{C_2}) \xto{\pi^{\ast}} \Fun_{C_2}(B^t_{C_2} \mu_{p^{\infty}}, \ul{\Sp}^{C_2}) \xto{\iota^{\ast}} \Sp^{C_2}. \]
Because $\pi^{\ast} p^{\ast} \simeq p^{\ast}$, by adjunction we obtain a natural transformation $p^{\ast} \to \pi_{\ast} p^{\ast} = (-)^{h_{C_2} \mu_p} \circ p^{\ast}$. Then let
\[ \lambda_{\RR,p}: p^{\ast} \to (-)^{h_{C_2} \mu_p} \circ p^{\ast} \to (-)^{t_{C_2} \mu_p} \circ p^{\ast} \]
be the composite natural transformation. Note that since $p^{\ast}$ and $\pi^{\ast}$ are symmetric monoidal, the adjoint natural transformation $p^{\ast} \to \pi_{\ast} p^{\ast}$ is canonically lax symmetric monoidal. With the lax symmetric monoidal structure on $(-)^{h_{C_2} \mu_p} \to (-)^{t_{C_2} \mu_p}$ as in \cref{setup:Dihedral}, $\lambda_{\RR,p}$ acquires the structure of a lax symmetric monoidal transformation. We then define $\mr{triv}_{\RR,p}$ to be the functor determined by the data of $p^{\ast}$ and $\lambda_{\RR,p}$.

Finally, note that since $\iota^{\ast} \circ p^{\ast} \simeq \id$, the composite of $\mr{triv}_{\RR,p}$ and the forgetful functor to $\Sp^{C_2}$ is also homotopic to the identity. By \cref{prp:RCycSpPresentable}, we deduce that $\mr{triv}_{\RR,p}$ is exact and preserves colimits.
\end{cnstr}

\begin{dfn}\label{dfn:newTCR} The \emph{$p$-typical real topological cyclic homology} functor $$\TCR(-,p): \RCycSp_p \to \Sp^{C_2}$$ is the right adjoint to the trivial functor $\mr{triv}_{\RR,p}$ of \cref{cnstr:trivialFunctor}.
\end{dfn}

We would like to say that $\TCR(-,p)$ is corepresentable by the unit in $\RCycSp_p$. However, because $\TCR(-,p)$ is valued in $C_2$-spectra, any such representability result must be understood in the $C_2$-sense. We now digress to give a general account of $G$-corepresentability.

\begin{cnstr}[$G$-mapping spectrum] \label{cnstr:MappingSpectrum} Suppose $C \to \sO^{\op}_G$ is a $G$-$\infty$-category. In \cite[Def.~10.2]{Exp1}, Barwick et al. defined the $G$-mapping space $G$-functor
\[ \ul{\Map}_C: C^{\vop} \times_{\sO_G^{\op}} C \to \ul{\Spc}^G.\footnote{Given a cocartesian fibration $C \to S$ classified by a functor $F: S \to \Cat_{\infty}$, the \emph{vertical opposite} $C^{\vop} \to S$ is a cocartesian fibration classified by the postcomposition of $F$ with the automorphism of $\Cat_{\infty}$ that takes the opposite. Barwick--Glasman--Nardin give an explicit construction of this at the level of simplicial sets in \cite{BGN}.} \]
Informally, this sends an object $(x,y)$ over $G/H$ to the $H$-space determined by $\Map_{C_{G/K}}(\res^H_K x, \res^H_K y)$ varying over subgroups $K \leq H$. In \cite[Cor.~11.9]{Exp2} (taking $F$ there to be the identity on $C$), the second author showed that for any $x \in C_{G/G}$, the $G$-functors
\[ \ul{\Map}_C(x,-): C \to \ul{\Spc}^G, \quad \ul{\Map}_C(-,x): C^{\vop} \to \ul{\Spc}^G \]
preserve $G$-limits, with $G$-limits in $C^{\vop}$ computed as $G$-colimits in $C$. Now suppose $C$ is $G$-stable, let $\Fun^{\lex}_G(-,-)$ denote the full subcategory on those $G$-functors that preserve finite $G$-limits, and let $\ul{\Fun}^{\lex}_G(-,-)$ denote the full $G$-subcategory (i.e., sub-cocartesian fibration) of $\ul{\Fun}_G(-,-)$ that over the fiber $G/H$ is given by $\Fun_H^{\lex}(-,-)$. In \cite[Thm.~7.4]{Exp4}, Nardin proved\footnote{We obtain our formulation involving $G$-left-exact functors from his using that $C$ is $G$-stable.} that the $G$-functor $\Omega^{\infty}: \ul{\Sp}^G \to \ul{\Spc}^G$ induces equivalences
\begin{align*} \Omega^{\infty}_{\ast}: & \Fun^{\lex}_G(C,\ul{\Sp}^G) \xto{\simeq} \Fun^{\lex}_G(C,\ul{\Spc}^G), \\
\Omega^{\infty}_{\ast}: & \ul{\Fun}^{\lex}_G(C,\ul{\Sp}^G) \xto{\simeq} \ul{\Fun}^{\lex}_G(C,\ul{\Spc}^G).
\end{align*}
In particular, for fixed $x \in C_{G/G}$ (that selects a cocartesian section $x: \sO_G^{\op} \to C^{\vop}$), we may lift the $G$-mapping space $G$-functor $\ul{\Map}_C(x,-): C \to \ul{\Spc}^G$ to a $G$-mapping spectrum $G$-functor
\[ \ul{\map}_C(x,-): C \to \ul{\Sp}^G. \]
Moreover, as in the non-parametrized setting, the $G$-mapping space $G$-functor is adjoint to a $G$-functor $$C^{\vop} \to \ul{\Fun}^{\lex}_G(C, \ul{\Spc}^G) \subset \ul{\Fun}_G(C,\ul{\Spc}^G),$$ which we may lift to $\ul{\Fun}^{\lex}_G(C,\ul{\Sp}^G)$ via $\Omega^{\infty}_{\ast}$ and then adjoint over to obtain
\[ \ul{\map}_C(-,-): C^{\vop} \times_{\sO_G^{\op}} C \to \ul{\Sp}^G. \]
Note that $\ul{\map}_C(-,-)$ continues to transform $G$-colimits into $G$-limits in the first variable and to preserve $G$-limits in the second variable.

Finally, by restriction to the fiber over $G/H$, we obtain the $H$-mapping spectrum functor for $C_{G/H}$:
\[ \ul{\map}_{C}(-,-): C_{G/H}^{\op} \times C_{G/H} \to \Sp^H, \]
which lifts the $H$-mapping space functor for $C_{G/H}$
\[ \ul{\Map}_{C}(-,-): C_{G/H}^{\op} \times C_{G/H} \to \Spc^H \]
through $\Omega^{\infty}: \Sp^H \to \Spc^H$. Thus, for $x,y \in C_{G/H}$, we may compute $\Omega^{\infty}$ of the categorical fixed points as
\[ \Omega^{\infty}(\ul{\map}_{C}(x,y)^K) \simeq \ul{\Map}_{C}(x,y)(H/K). \]
In particular, taking the fiber over $G/G$, for all $G/H$ the diagram
\[ \begin{tikzcd}[row sep=4ex, column sep=10ex, text height=1.5ex, text depth=0.75ex]
C_{G/G}^{\op} \times C_{G/G} \ar{r}{\ul{\map}_{C}(-,-)} \ar{d}[swap]{((\res^G_H)^{\op}, \res^G_H)} & \Sp^G \ar{r}{\Omega^{\infty}} \ar{d}{\Psi^{H}} & \Spc^G \ar{d}{\ev_{G/H}} \\
C_{G/H}^{\op} \times C_{G/H} \ar{r}{\map_{C_{G/H}}(-,-)} & \Sp \ar{r}{\Omega^{\infty}} & \Spc
\end{tikzcd} \]
is homotopy commutative, where the top horizontal composite is $\ul{\Map}_{C}(-,-)$ and the bottom horizontal composite is $\Map_{C_{G/H}}(-,-)$.

%Should we have this?
% We will suppress the subscript $C$ and simply write $\ul{\map}(-,-)$ in case there is no ambiguity.
% and we may thereby think of $\ul{\map}_{C_{G/G}}(x,y)$ as globalizing the individual deloopings of the values of $\ul{\Map}_{C_{G/G}}(x,y)$ on orbits.
\end{cnstr}

% the comparison map
% \[ \ul{\Map}_{D}(L x,y) \xto{R_{\ast}} \ul{\Map}_{C}(R L x,R y) \xto{\eta^{\ast}} \ul{\Map}_{C}(x, R y)  \]
% is an equivalence of $G$-spaces. 
\begin{lem} \label{lem:adjunctionMappingSpacesEquivalence} Suppose $C$ and $D$ are $G$-$\infty$-categories and $\adjunct{L}{C}{D}{R}$ is a $G$-adjunction. Then we have natural equivalences in $\ul{\Spc}^G$
\[ \ul{\Map}_{C}(L x, y) \simeq \ul{\Map}_{D}(x, R y).  \]
If $C$ and $D$ are also $G$-stable, then we have natural equivalences in $\ul{\Sp}^G$
\[ \ul{\map}_{C}(L x, y) \simeq \ul{\map}_{D}(x, R y). \]
\end{lem}
\begin{proof} For a $G$-adjunction, we have a unit transformation $\eta: \id \to R L$ such that $\eta$ cover the identity in $\sO_G^{\op}$ \cite[Prop.~7.3.2.1(2)]{HA}. We then obtain the comparison map
\[ \ul{\Map}_D(L x,y) \xto{R_{\ast}} \ul{\Map}_C(R L x,R y) \xto{\eta^{\ast}} \ul{\Map}_C(R L x, y)  \]
in $\ul{\Spc}^G$. Because we may restrict to subgroups $H$ of $G$, without loss of generality it suffices to consider the case where $x \in C_{G/G}$ and $y \in D_{G/G}$, so the comparison map is a map of $G$-spaces. For an orbit $G/H$, let $x' = \res^G_H x \in C_{G/H}$ and $y' = \res^G_H y \in D_{G/H}$. Then on $G/H$ this map evaluates to
\[ \Map_{D_{G/H}}(L_{H} x', y') \xto{(R_H)_{\ast}} \Map_{C_{G/H}}(R_H L_H x',R_H y') \xto{\eta_H^{\ast}} \Map_{C_{G/H}}(x', R_H y'), \]
which implements the equivalence of mapping spaces for the adjunction $$\adjunct{L_H}{C_{G/H}}{D_{G/H}}{R_H}$$ between the fibers over $G/H$. The conclusion then follows. Finally, the subsequent claim about $\ul{\map}$ follows by reduction to $F$ in the same manner, where instead of using the jointly conservative family of evaluation functors at orbits $G/H$ to detect equivalences in $\Spc^G$, we use the categorical fixed points functors ranging over all subgroups $H \leq G$ to detect equivalences in $\Sp^G$.
\end{proof}

\begin{prp} \label{prp:GenericRepresentabilityByUnit} Suppose $C$ is a $G$-stable $G$-$\infty$-category and we have a $G$-adjunction $$\adjunct{L}{\ul{\Sp}^G}{C}{R}.$$  Then for all $c \in C_{G/H} \subset C$, if we let $S^0$ denote the unit of $\Sp^H$, then we have a natural equivalence
\[ R(c) \simeq \ul{\map}_{C}(L(S^0),c). \]
Consequently, we have equivalences of functors
\[ R_{G/H}(-) \simeq \ul{\map}_{C}(L(S^0),-): C_{G/H} \to \Sp^H. \]
Moreover, if we let $S^0: \sO^{\op}_G \to (\ul{\Sp}^G)^{\vop}$ also denote the cocartesian section that selects each $S^0$, then we have an equivalence of $G$-functors $$R(-) \simeq \ul{\map}_{C}(L(S^0),-): C \to \underline{\Sp}^{G}.$$
% \[ R_H(c) \simeq \ul{\map}_{C_{G/H}}(L_H(1),c). \]
\end{prp}
\begin{proof} Without loss of generality we may suppose $H=G$, so we let $S^0$ be the unit for $\Sp^G$. Note that since the identity on $\ul{\Sp}^G$ lifts the $G$-mapping space $G$-functor $\ul{\Map}_G(S^0,-): \ul{\Sp}^G \to \ul{\Spc}^G$ through $\Omega^{\infty}$, we have that $\ul{\map}_{\ul{\Sp}^G}(S^0,-) \simeq \id$. Then for $c \in C_{G/G}$, using \cref{lem:adjunctionMappingSpacesEquivalence} we have the equivalences
\[ R(c) \simeq \ul{\map}_{\ul{\Sp}^G}(S^0,R(c)) \simeq \ul{\map}_{C}(L(S^0),c). \]
The naturality of these equivalences in $c$ then imply the remaining statements.
\end{proof}

\begin{dfn} \label{dfn:Gcorepresentable} In the situation of \cref{prp:GenericRepresentabilityByUnit}, we say that $R$ and $R_{G/G}$ are \emph{$G$-corepresentable} by $L(S^0)$, for the unit $S^0 \in \Sp^G$.
\end{dfn}

\begin{exm} \label{exm:LimitRepresentability} Let $p: K \to \sO_{G}^{\op}$ be a (small) $G$-$\infty$-category and consider the adjunction
\[ \adjunct{p^{\ast}}{\ul{\Sp}^G}{\ul{\Fun}_G(K,\ul{\Sp}^G)}{p_{\ast}} \]
where $p_{\ast}$ takes the $G$-limit. Then $p_{\ast}$ is $G$-corepresentable by the constant $G$-diagram at the unit. With respect to the pointwise symmetric monoidal structure on a $G$-functor $\infty$-category (\cite[Def.~A.1]{QS21a}), this is the unit in $\Fun_G(K,\ul{\Sp}^G)$.
\end{exm}

We want to apply \cref{prp:GenericRepresentabilityByUnit} to prove that $p$-typical real topological cyclic homology is $C_2$-corepresentable by $\mr{triv}_{\RR,p}(S^0)$, the $C_2$-sphere spectrum endowed with the trivial real $p$-cyclotomic structure, which by symmetric monoidality of $\mr{triv}_{\RR,p}$ is also the unit in $\RCycSp_p$. For this, we need to refine $\RCycSp_p$ to a $C_2$-stable $C_2$-$\infty$-category and to refine $\TCR(-,p)$ to a $C_2$-right adjoint. We first extend the definition of lax equalizer to the parametrized setting.

\begin{dfn} Let $C$ be a $G$-$\infty$-category, and suppose $F$ and $F'$ are $G$-endofunctors of $C$. Let $\Ar_G(C) \coloneq \sO_G^{\op} \times_{\Ar(\sO^{\op}_G)} \Ar(C)$ be the $G$-$\infty$-category of arrows in $C$ and define the \emph{$G$-lax equalizer} to be the pullback of $G$-$\infty$-categories
\[ \begin{tikzcd}[row sep=4ex, column sep=4ex, text height=1.5ex, text depth=0.25ex]
\ul{\LEq}_{F:F'}(C) \ar{r} \ar{d} & \Ar_G(C) \ar{d}{(\ev_0,\ev_1)} \\
C \ar{r}{(F,F')} & C \times_{\sO^{\op}_G} C.
\end{tikzcd} \]
Note that for all $G$-orbits $U$, we have an isomorphism of simplicial sets $\ul{\LEq}_{F:F'}(C)_U \cong \LEq_{F_U:F'_U}(C_U)$, and for all morphisms $\alpha: U \to V$ of $G$-orbits, the restriction functor
$$\alpha^*: \LEq_{F_V:F'_V}(C_V) \to \LEq_{F_U:F'_U}(C_U)$$
sends $[x,F_V(x) \xto{\phi} F'_V(x)]$ to $[\alpha^{\ast} x, \alpha^{\ast}(\phi)]$. Define the \emph{$G$-equalizer} $\ul{\Eq}_{F:F'}(C) \subset \ul{\LEq}_{F:F'}(C)$ to be the full $G$-subcategory that fiberwise is given by $\Eq_{F_U:F'_U}(C_U)$.
\end{dfn}

We have the following parametrized analog of \cite[Lem.~II.1.5(iii)]{NS18}.

\begin{lem} \label{lem:ParamLaxEqualizerStable} If $C$ is a $G$-$\infty$-category that admits finite $G$-limits and $F$,$G$ are $G$-left exact endofunctors, then $\ul{\LEq}_{F:G}(C)$ admits finite $G$-limits and the forgetful functor $\ul{\LEq}_{F:G}(C) \to C$ preserves finite $G$-limits. If $C$ is moreover $G$-stable, then $\ul{\LEq}_{F:G}(C)$ is $G$-stable.
\end{lem}
\begin{proof} We already know that the fibers of $\ul{\LEq}_{F:G}(C)$ admit finite limits or are stable and the pushforward functors are left-exact or exact, given our respective hypotheses. For the first statement, it thus suffices to show that for all maps of $G$-orbits $f:U \to V$, the restriction functor $f^{\ast}: \LEq_{F_V:G_V}(C_V) \to \LEq_{F_U:G_U}(C_U)$ admits a right adjoint $f_{\ast}$ computed by postcomposing by the right adjoint in $C$, and moreover that these adjunctions satisfy the Beck-Chevalley condition. Given the adjunction $\adjunct{f^{\ast}}{C_V}{C_U}{f_{\ast}}$, let $\cM \to \Delta^1$ be the bicartesian fibration that encodes this adjunction. Then since $F$ and $G$ commute with both $f^{\ast}$ and $f_{\ast}$, we obtain induced endofunctors $F_f$ and $G_f$ of $\cM$ over $\Delta^1$ that preserve both cocartesian and cartesian edges and restrict to $F_U$, $F_V$ and $G_U$, $G_V$ on the fibers. Therefore, the $\Delta^1$-lax equalizer $\ul{\LEq}_{F_f:G_f}(\cM) \to \Delta^1$ is again a bicartesian fibration that encodes the adjunction between $\LEq_{F_V:G_V}(C_V)$ and $\LEq_{F_U:G_U}(C_U)$, so $f^{\ast}$ as a functor on lax equalizers admits a right adjoint computed by postcomposition by $f_{\ast}: C_U \to C_V$. For the Beck-Chevalley condition, suppose a pullback square of finite $G$-sets
\[ \begin{tikzcd}[row sep=4ex, column sep=4ex, text height=1.5ex, text depth=0.25ex]
U \times_V W \ar{r}{f} \ar{d}{g} & W \ar{d}{g} \\
U \ar{r}{f} & V
\end{tikzcd} \]
where without loss of generality we may suppose $U,V,W$ are $G$-orbits. We need to show the natural transformations
\[ (\eta: f^{\ast} g_{\ast} \Rightarrow g_{\ast} f^{\ast}): \LEq_{F_W:G_W}(C_W) \to \LEq_{F_U:G_U}(C_U) \]
are equivalences. However, since the forgetful functor $\LEq_{F_U:G_U}(C_U) \to C_U$ detects equivalences, this follows from the Beck-Chevalley conditions assumed on $C$ itself. For the second statement, ambidexterity for the adjunctions $f^{\ast} \dashv f_{\ast}$ on $C$ promotes to the same for $\ul{\LEq}_{F:G}(C)$, using the same methods.
\end{proof}

The restriction functors $\Sp^{D_{2p^n}} \to \Sp^{D_{2p^m}}$ extend to $C_2$-functors $$\sO_{C_2}^{\op} \times_{\sO_{D_{2p^n}}^{\op}} \Sp^{D_{2p^n}} \to \sO_{C_2}^{\op} \times_{\sO_{D_{2p^m}}^{\op}} \Sp^{D_{2p^m}}$$
given on the fiber $C_2/1$ by the restriction functors $\Sp^{\mu_{2p^n}} \to \Sp^{\mu_{2p^m}}$. Taking the inverse limit, we obtain a $C_2$-stable $C_2$-$\infty$-category $\Sp^{D_{2p^{\infty}}}_{C_2}$. By taking inverse limits of the dihedral $C_2$-stable $C_2$-recollements of \cite[Exm.~3.26]{QS21a} along the restriction $C_2$-functors, we obtain a $C_2$-stable $C_2$-recollement\footnote{Note that we abuse notation here in letting $j^*, j_*, i_*$ denote both the $C_2$-functors and their fiberwise restrictions. However, we do wish to distinguish between $\underline{\Phi}^{\mu_p}$ and $\Phi^{\mu_p}$ in what follows.}
\begin{equation} \label{eq:C2_dihedral_recollement_infinite}
\begin{tikzcd}[row sep=4ex, column sep=8ex, text height=1.5ex, text depth=0.5ex]
\ul{\Fun}_{C_2}(B^t_{C_2} \mu_{p^{\infty}}, \ul{\Sp}^{C_2}) \ar[shift right=1,right hook->]{r}[swap]{j_{\ast}} & \Sp^{D_{2p^{\infty}}}_{C_2} \ar[shift right=2]{l}[swap]{j^{\ast}} \ar[shift left=2]{r}{\underline{\Phi}^{\mu_p}} & \Sp^{D_{2p^{\infty}}}_{C_2} \ar[shift left=1,left hook->]{l}{i_{\ast}}
\end{tikzcd}
\end{equation}
that over $C_2/C_2$ restricts to the recollement \eqref{eq:dihedral_recollement_infinite} and whose fiber over $C_2/1$ is the recollement
\[ \begin{tikzcd}[row sep=4ex, column sep=8ex, text height=1.5ex, text depth=0.5ex]
\Fun(B \mu_{p^{\infty}}, \Sp) \ar[shift right=1,right hook->]{r}[swap]{j_{\ast}} & \Sp^{\mu_{p^{\infty}}} \ar[shift right=2]{l}[swap]{j^{\ast}} \ar[shift left=2]{r}{\Phi^{\mu_p}} & \Sp^{\mu_{p^{\infty}}} \ar[shift left=1,left hook->]{l}{i_{\ast}}.
\end{tikzcd} \]

Observe that the $C_2$-exact $C_2$-endofunctor $j^{\ast} \underline{\Phi}^{\mu_p} j_{\ast}$ restricts over the fiber $C_2/C_2$ to $t_{C_2} \mu_p$ and over the fiber $C_2/1$ to $t \mu_p$. We let $\ul{t}_{C_2} \mu_p = j^{\ast} \underline{\Phi}^{\mu_p} j_{\ast}$.

\begin{dfn} \label{dfn:C2CategoryOfCyclotomicSpectra} The \emph{$C_2$-$\infty$-category of real $p$-cyclotomic spectra} is
\[ \ul{\RCycSp}_p := \ul{\LEq}_{\id:\underline{t}_{C_2} \mu_p}(\ul{\Fun}_{C_2}(B^t_{C_2} \mu_{p^{\infty}}, \ul{\Sp}^{C_2})). \]
\end{dfn}

% (see \cite[Prop.~IV.4.14]{NS18} for the integral version of this adjunction)
Observe that the fiber of $\ul{\RCycSp}_p$ over $C_2/C_2$ is $\RCycSp_p$, and the fiber over $C_2/1$ is the $\infty$-category of $p$-cyclotomic spectra $\CycSp_p$ as defined in \cite[Def.~II.1.6(ii)]{NS18}. Also, by repeating \cref{cnstr:trivialFunctor} with $C_2$-$\infty$-categories, we obtain a $C_2$-functor
\[ \ul{\mr{triv}}_{\RR,p}: \ul{\Sp}^{C_2} \to \ul{\RCycSp}_p \]
that restricts over $C_2/C_2$ to $\mr{triv}_{\RR,p}$ and over $C_2/1$ to the trivial functor $\mr{triv}_p: \Sp \to \CycSp_p$ whose right adjoint is by definition $p$-typical topological cyclic homology $\TC(-,p)$. By the dual of \cite[Prop.~7.3.2.6]{HA}, the fiberwise right adjoints refine to the structure of a relative right adjoint $$\ul{\TCR}(-,p): \ul{\RCycSp}_p \to  \ul{\Sp}^{C_2} $$ to $\ul{\mr{triv}}_{\RR,p}$. Moreover, because the composite $$\ul{\Sp}^{C_2} \to \ul{\RCycSp}_p \to \ul{\Fun}_{C_2}(B^t_{C_2} \mu_{p^{\infty}}, \ul{\Sp}^{C_2})$$ of $\ul{\mr{triv}}_{\RR,p}$ and the forgetful $C_2$-functor is $C_2$-left exact, by \cref{lem:ParamLaxEqualizerStable} we get that $\ul{\mr{triv}}_{\RR,p}$ itself is $C_2$-left exact. Therefore, $\ul{\TCR}(-,p)$ preserves cocartesian edges, i.e., is a $C_2$-functor, and is thus $C_2$-right adjoint to $\ul{\mr{triv}}_{\RR,p}$.

\begin{prp} \label{prp:C2representabilityOfTC} $\ul{\TCR}(-,p)$ and $\TCR(-,p)$ are $C_2$-corepresentable by the unit.
\end{prp}
\begin{proof} Since $\ul{\RCycSp}_p$ is $C_2$-stable by \cref{lem:ParamLaxEqualizerStable}, this follows immediately from \cref{prp:GenericRepresentabilityByUnit} and the above discussion.
\end{proof}

We now apply \cref{prp:C2representabilityOfTC} to derive an equalizer formula for $\TCR(-,p)$. We first explain how to compute $G$-mapping spaces and spectra in a limit of $G$-$\infty$-categories and then in the $G$-lax equalizer, analogous to \cite[II.1.5(ii)]{NS18}.

\begin{lem} \label{lem:LimitOfGMappingSpaces} Let $C_{\bullet}: K \to \Cat_{\infty}^G$ be a diagram of $G$-$\infty$-categories, and let $C = \lim C_{\bullet}$ be the limit. Let $x,y \in C_{G/G}$. Then the natural comparison map
\[ \ul{\Map}_{C}(x,y) \to \lim_{i \in K} \left(\ul{\Map}_{C_i}(x_i,y_i) \right)\]
is an equivalence of $G$-spaces. Furthermore, if $C_{\bullet}$ is a diagram of $G$-stable $G$-$\infty$-categories and $G$-exact $G$-functors, then the natural comparison map
\[ \ul{\map}_{C}(x,y) \to \lim_{i \in K} \left(\ul{\map}_{C_i}(x_i,y_i) \right) \]
is an equivalence of $G$-spectra.
\end{lem}
\begin{proof} By either evaluation at $G/H$ or taking $H$-categorical fixed points and using the commutative diagram at the end of \cref{cnstr:MappingSpectrum}, we may reduce to the known non-parametrized statements.
\end{proof}

% , and let $C_0 = C_{G/G}$, $F_0 = F_{G/G}$, and $F'_0 = F'_{G/G}$
\begin{lem} \label{lem:EqualizerMappingSpaces} Suppose $C$ is a $G$-$\infty$-category and $F,F'$ are $G$-endofunctors of $C$. Let $X = [x,\phi: F(x) \to F'(x)]$ and $Y=[y,\psi: F(y) \to F'(y)]$ be two objects in $\LEq_{F_{G/G}:F'_{G/G}}(C_{G/G})$. Then we have a natural equivalence of $G$-spaces
\[ \ul{\Map}_{\ul{\LEq}_{F:G}(C)}(X,Y) \simeq \mr{eq} \left( 
\begin{tikzcd}[row sep=4ex, column sep=6ex, text height=1.5ex, text depth=0.25ex]
\ul{\Map}_{C}(x,y) \ar[shift left=1]{r}{\psi_{\ast} \circ F} \ar[shift right=1]{r}[swap]{\phi^{\ast} \circ F'} & \ul{\Map}_{C}(F(x),F'(y)).
\end{tikzcd} \right) \]
If $C$ is $G$-stable and $F,F'$ are $G$-exact, then we have a natural fiber sequence of $G$-spectra
\[ \ul{\map}_{\ul{\LEq}_{F:G}(C)}(X,Y) \to \ul{\map}_{C}(x,y) \xtolong{\psi_{\ast}  F - \phi^{\ast} F'}{2} \ul{\map}_{C}(F(x),F'(y)).  \]
\end{lem}
\begin{proof} In view of \cref{lem:LimitOfGMappingSpaces}, the same arguments as in the proof of \cite[II.1.5.(ii)]{NS18} apply to produce the formulas.
\end{proof}

\begin{dfn} For $X \in \Fun_{C_2}(B^t_{C_2} \mu_{p^{\infty}}, \ul{\Sp}^{C_2})$, we define the \emph{canonical} map
\[ \can_p: X^{h_{C_2} \mu_{p^{\infty}}} \simeq (X^{h_{C_2} \mu_p})^{h_{C_2} \mu_{p^{\infty}}} \to (X^{t_{C_2} \mu_p})^{h_{C_2} \mu_{p^{\infty}}} \]
where for the first equivalence we use that $B^t_{C_2} (\mu_{p^{\infty}}/\mu_p) \simeq B^t_{C_2} \mu_{p^{\infty}}$ as before.
\end{dfn}

\begin{prp} \label{prp:TCRfiberSequence} Let $[X,\varphi:X \to X^{t_{C_2} \mu_p}]$ be a real $p$-cyclotomic spectrum. Then we have a natural fiber sequence of $C_2$-spectra
\[ \TCR(X,p) \to X^{h_{C_2} \mu_{p^\infty}} \xtolong{\varphi^{h_{C_2} \mu_{p^{\infty}}} -\can_p}{2} (X^{t_{C_2} \mu_p})^{h_{C_2} \mu_{p^\infty}}. \]
\end{prp}
\begin{proof} We mimic the proof of \cite[Prop.~II.1.9]{NS18}. Let $C = \ul{\Fun}_{C_2}(B^t_{C_2} \mu_{p^{\infty}}, \ul{\Sp}^{C_2})$, let $S^0 \in C$ be the unit (i.e., $S^0$ with trivial action), and note that by \cref{exm:LimitRepresentability}, $$\ul{\map}_{C}(S^0,X) \simeq X^{h_{C_2} \mu_{p^{\infty}}}.$$
The claim then follows from \cref{prp:C2representabilityOfTC} and \cref{lem:EqualizerMappingSpaces}. In more detail, if we let $\lambda = \lambda_{\RR,p}: S^0 \to (S^0)^{t_{C_2} \mu_p}$ denote the trivial real $p$-cyclotomic structure map as in \cref{cnstr:trivialFunctor}, then given a map $f: S^0 \to X$ in $C$, we have a commutative diagram (again in $C$)
\[ \begin{tikzcd}[row sep=4ex, column sep=8ex, text height=1.5ex, text depth=0.5ex]
S^0 \ar{r} \ar{rd}[swap]{\lambda} & (S^0)^{h_{C_2} \mu_p} \ar{r}{f^{h_{C_2} \mu_p}} \ar{d} & X^{h_{C_2} \mu_p } \ar{d} \\
& (S^0)^{t_{C_2} \mu_p} \ar{r}{f^{t_{C_2} \mu_p}} & X^{t_{C_2} \mu_p}.
\end{tikzcd} \]
Therefore, the composite map
\[ \ul{\map}_C(S^0, X) \xto{t_{C_2} \mu_p} \ul{\map}_C((S^0)^{t_{C_2} \mu_p}, X^{t_{C_2} \mu_p}) \xto{\lambda^{\ast}} \ul{\map}_C((S^0), X^{t_{C_2} \mu_p}) \]
is homotopic to $\can_p$. It is then clear that the desired fiber sequence is given by \cref{lem:EqualizerMappingSpaces}.
\end{proof}

\subsection{Genuine real \texorpdfstring{$p$}{p}-cyclotomic spectra}
\label{section:genRCycSp}

\begin{dfn} \label{dfn:GenRealCycSp} A \emph{genuine real $p$-cyclotomic spectrum} is a $D_{2p^{\infty}}$-spectrum $X$, together with an equivalence $\Phi^{\mu_p} X \xto{\simeq} X$ in $\Sp^{D_{2p^{\infty}}}$. The $\infty$-category of \emph{genuine real $p$-cyclotomic spectra} is then
\[ \RCycSp^{\mathrm{gen}}_p := \Eq_{\Phi^{\mu_p}:\id}(\Sp^{D_{2p^{\infty}}}). \]
\end{dfn}

\begin{rem} \label{rem:EqualizerSwap} Let $F_i: C \to D$, $i=0,1$ be two functors and let us temporarily revert to the notation of \cite[Def.~II.1.4]{NS18} for (lax) equalizers. Let $J = (a \twoarrows b)$, so equalizers in an $\infty$-category $\cE$ are limits over diagrams $J \to \cE$ (\cite[\S 4.4.3]{HTT}). Let $q: J \to \Cat_{\infty}$ be the diagram that sends one arrow to $F_0$ and the other to $F_1$. Because
\[  \mr{eq}(\begin{tikzcd}[row sep=4ex, column sep=4ex, text height=1.5ex, text depth=0.25ex]
C \ar[shift left=1]{r}{F_0} \ar[shift right=1]{r}[swap]{F_1} & D
\end{tikzcd}) \coloneq \lim_J q \simeq \Eq(\begin{tikzcd}[row sep=4ex, column sep=4ex, text height=1.5ex, text depth=0.25ex]
C \ar[shift left=1]{r}{F_0} \ar[shift right=1]{r}[swap]{F_1} & D
\end{tikzcd}) \]
and the former expression is symmetric in $F_i$, we have an equivalence 
\[ \Eq(\begin{tikzcd}[row sep=4ex, column sep=4ex, text height=1.5ex, text depth=0.25ex]
C \ar[shift left=1]{r}{F_0} \ar[shift right=1]{r}[swap]{F_1} & D
\end{tikzcd}) \simeq \Eq(\begin{tikzcd}[row sep=4ex, column sep=4ex, text height=1.5ex, text depth=0.25ex]
C \ar[shift left=1]{r}{F_1} \ar[shift right=1]{r}[swap]{F_0} & D
\end{tikzcd}). \]
These equivalences are already implicit in \cite{NS18}, but in more detail, let $\cX \to J$ be the cocartesian fibration classified by $q$, so $\lim_J q \simeq \Sect(\cX) \coloneq \Fun^{\cocart}_{/J}(J,\cX)$. Let $\alpha_i: \Delta^1 \to J$, $i = 0,1$ be the two arrows in $J$, with $q \alpha_i$ selecting $F_i$, and let $\cX_{\alpha_i}$ be the pullback. Then $\ev_a: \Sect(\cX_{\alpha_i}) \to C$ is a trivial fibration, and choosing a section, the composite
\[ C \xto{\simeq} \Sect(\cX_{\alpha_i}) \xto{\ev_b} D \]
is homotopic to $F_i$. Therefore, if we let $\iota: D \to \Ar(D)$ denote the identity section, we have a homotopy commutative diagram
\[ \begin{tikzcd}[row sep=4ex, column sep=4ex, text height=1.5ex, text depth=0.25ex]
\Sect(\cX) \ar{r}{\ev_b} \ar{d}{\ev_a} & D \ar{r}{\iota} & \Ar(D) \ar{d}{(\ev_0,\ev_1)} \\
C \ar{rr}{(F,G)} & & D \times D
\end{tikzcd} \]
and an induced functor $\Sect(\cX) \to \LEq(\begin{tikzcd}[row sep=4ex, column sep=4ex, text height=1.5ex, text depth=0.25ex]
C \ar[shift left=1]{r}{F_0} \ar[shift right=1]{r}[swap]{F_1} & D
\end{tikzcd})$, which is fully faithful by comparing the formulas for mapping spaces in the limit over $J$ and in the lax equalizer. Because the essential image of $\iota$ consists of the equivalences in $\Ar(D)$, it follows that the essential image of $\Sect(\cX)$ is $\Eq(\begin{tikzcd}[row sep=4ex, column sep=4ex, text height=1.5ex, text depth=0.25ex]
C \ar[shift left=1]{r}{F_0} \ar[shift right=1]{r}[swap]{F_1} & D
\end{tikzcd})$. Repeating the analysis with $F_0$ and $F_1$ exchanged, we obtain a zig-zag of equivalences
\[ \begin{tikzcd}[row sep=4ex, column sep=4ex, text height=1.5ex, text depth=0.25ex]
\Eq(C \ar[shift left=1]{r}{F_1} \ar[shift right=1]{r}[swap]{F_0} & D) & \Sect(\cX) \ar{r}{\simeq} \ar{l}[swap]{\simeq} & \Eq(C \ar[shift left=1]{r}{F_0} \ar[shift right=1]{r}[swap]{F_1} & D).
\end{tikzcd} \]

It follows that in defining genuine real $p$-cyclotomic spectra, the choice of direction of the equivalence $\Phi^{\mu_p} X \simeq X$ is immaterial. Thus, in lieu of \cref{dfn:GenRealCycSp} we could have let
\[ \RCycSp^{\mr{gen}}_p = \Eq_{\id:\Phi^{\mu_p}}(\Sp^{D_{2p^{\infty}}}). \]
This definition is more convenient when comparing to \cref{dfn:RealCycSp}, whereas \cref{dfn:GenRealCycSp} is more suitable for defining the structure maps $R$ (\cref{dfn:ParamRStructureMaps}) that define the term $\TRR_p$ in the fiber sequence formula for $\TCR^{\mr{gen}}(-,p)$ of \cref{prp:fiberSequenceGenuineRealCyc}.
\end{rem}
%We need to change notation for TR, or change R to \mathbb{R}

\begin{prp} $\RCycSp^{\mr{gen}}_p$ is a stable presentable symmetric monoidal $\infty$-category such that the forgetful functor to $\Sp^{D_{2p^{\infty}}}$ is conservative, creates colimits and finite limits, and is symmetric monoidal.
\end{prp}
\begin{proof} Because $\Phi^{\mu_p}$ is colimit-preserving and symmetric monoidal, we may lift the equalizer diagram to $\CAlg(\Pr^{L,\st})$. Since limits there are computed as for the underlying $\infty$-categories, the claim then follows, with conservativity proven as for the lax equalizer.
\end{proof}

%Alter inflation notation at some point to conform with other conventions
% Let $\inf^n: \Sp^{C_2} \to \Sp^{D_{2p^n}}$ denote inflation with respect to the quotient map $D_{2p^n} \to C_2$. 
\begin{cnstr} \label{cnstr:InflationToInfinity} Since for all $0 \leq m<n<\infty$, the diagram
\[ \begin{tikzcd}[row sep=4ex, column sep=6ex, text height=1.5ex, text depth=0.25ex]
\Sp^{C_2} \ar{r}{\inf^{\mu_{p^n}}} \ar{rd}[swap]{\inf^{\mu_{p^m}}} & \Sp^{D_{2p^n}} \ar{d}{\res}  \\
&  \Sp^{D_{2p^m}}
\end{tikzcd} \]
commutes in $\CAlg(\Pr^{L,\st})$, we may define an exact, colimit-preserving, and symmetric monoidal functor
\[ \inf{}^{\mu_{p^{\infty}}}: \Sp^{C_2} \to \Sp^{D_{2p^{\infty}}} \]
as the inverse limit of the functors $\inf^{\mu_{p^n}}$. Let $$\Psi^{\mu_{p^{\infty}}}: \Sp^{D_{2p^{\infty}}} \to \Sp^{C_2}$$ denote its right adjoint, and also write $$\Psi^{\mu_{p^{n}}}: \Sp^{D_{2p^{\infty}}} \to \Sp^{C_2}$$ for the composite of the restriction to $\Sp^{D_{2p^n}}$ and $\Psi^{\mu_{p^{n}}}$. Recall that for a diagram $\overline{C}_{\bullet}: K^{\lhd} \to \Pr^L$, if we write $C$ for its value on the cone point $v$ and $D = \lim_K (C_{\bullet})$, then for the induced adjunction
\[ \adjunct{L}{C}{D}{R} \]
we may compute $R$ in terms of the description of $D$ as an $\infty$-category of cocartesian sections as follows:
\begin{itemize} \item[($\ast$)] Let $\overline{\cX} \to K^{\lhd}$ be the presentable fibration classified by $\overline{C}_{\bullet}$, with restriction $\cX \to K$. Let $p: \cX \subset \overline{\cX} \to \overline{\cX}_v \simeq C$ be the \emph{cartesian} pushforward to the fiber over the initial object $v \in K^{\lhd}$. Then the functor
\[ p_{\ast}: D \simeq \Sect(\cX) \to C \]
obtained via postcomposition by $p$ is homotopic to $R$.
\end{itemize}

See \cite[\S 2]{BehrensShah} for a reference. To specialize to our situation, we note that $\Psi^{\mu_{p^n}}: \Sp^{D_{2p^n}} \to \Sp^{C_2}$ applied to the unit map $X \to \ind \res X$ for the restriction-induction adjunction $$\adjunct{\res}{\Sp^{D_{2p^n}}} {\Sp^{D_{2p^{n-1}}}} {\ind}$$ defines the map $F: \Psi^{\mu_{p^n}}(X) \to \Psi^{\mu_{p^{n-1}}}(X)$ of $C_2$-spectra that lifts the map $F: X^{\mu_{p^n}} \to X^{\mu_{p^{n-1}}}$ of spectra given by inclusion of fixed points.  Using the formula above, we conclude that 
\[ \Psi^{\mu_{p^{\infty}}}(X) \xto{\simeq} \lim_{n,F} \Psi^{\mu_{p^n}}(X). \]
% Since $\Psi^{\mu_{p^{\infty}}}$ is corepresentable by the unit $S^0$, we may also write
% \[ F(S^0, X) \simeq \lim_n \Psi^{\mu_{p^n}}(X) \]
\end{cnstr}

\begin{rem} For $n< \infty$, the functors $\Psi^{\mu_{p^n}}$ all commute with colimits, but the inverse limit $\Psi^{\mu_{p^{\infty}}}$ does not commute with colimits in general.
\end{rem}

\begin{cnstr} \label{cnstr:GenuineTrivialFunctor} Because the diagram
\[ \begin{tikzcd}[row sep=4ex, column sep=6ex, text height=1.5ex, text depth=0.25ex]
\Sp^{C_2} \ar{r}{\inf^{\mu_{p^n}}} \ar{rd}[swap]{\inf^{\mu_{p^{n-1}}}} & \Sp^{D_{2p^n}} \ar{d}{\Phi^{\mu_p}}  \\
&  \Sp^{D_{2p^{n-1}}}
\end{tikzcd} \]
commutes for all $0<n<\infty$, we have an equivalence $\Phi^{\mu_p} \inf^{\mu_{p^{\infty}}} \simeq \inf^{\mu_{p^{\infty}}}$ in $\CAlg(\Pr^{L,\st})$. Therefore, $\inf^{\mu_{p^{\infty}}}$ lifts to the equalizer of $\id$ and $\Phi^{\mu_p}$ to define an exact, colimit-preserving, and symmetric monoidal functor
\[ \mr{triv}^{\mr{gen}}_{\RR,p}: \Sp^{C_2} \to \RCycSp^{\mr{gen}}_p. \]

\end{cnstr}

\begin{dfn} \label{dfn:ClassicalTCR} The \emph{classical $p$-typical real topological cyclic homology} functor
$$\TCR^{\mr{gen}}(-,p): \RCycSp^{\mr{gen}}_p \to \Sp^{C_2} $$
is the right adjoint to $\mr{triv}^{\mr{gen}}_{\RR,p}$.
\end{dfn}

As with $\TCR(-,p)$, we can prove that (the $C_2$-refinement of) $\TCR^{\mr{gen}}(-,p)$ is $C_2$-corepresentable by the unit and thereby deduce a fiber sequence formula for the functor. Recall that in the course of formulating \cref{dfn:C2CategoryOfCyclotomicSpectra}, we extended $\Sp^{D_{2p^{\infty}}}$ and $\Phi^{\mu_p}$ to a $C_2$-stable $C_2$-$\infty$-category $\Sp^{D_{2p^{\infty}}}_{C_2}$ with $\underline{\Phi}^{\mu_p}$ as a $C_2$-endofunctor.

\begin{dfn} The \emph{$C_2$-$\infty$-category of genuine real $p$-cyclotomic spectra} is 
\[ \ul{\RCycSp}^{\mr{gen}}_p := \ul{\Eq}_{\underline{\Phi}^{\mu_p}:\id}(\Sp^{D_{2p^{\infty}}}_{C_2}). \]
\end{dfn}

Note that the fiber of $\ul{\RCycSp}^{\mr{gen}}_p$ over $C_2/C_2$ is $\RCycSp^{\mr{gen}}_p$, and the fiber over $C_2/1$ is $\CycSp^{\mr{gen}}_p$ as defined in \cite[Def.~II.3.1]{NS18}.

\begin{lem}$\ul{\RCycSp}^{\mr{gen}}_p$ is $C_2$-stable.
\end{lem}
\begin{proof}
Since $\Sp^{D_{2p^{\infty}}}_{C_2}$ is $C_2$-stable and $\underline{\Phi}^{\mu_p}$ is $C_2$-exact, by \cref{lem:ParamLaxEqualizerStable} we see that $\ul{\LEq}_{\underline{\Phi}^{\mu_p}:\id}(\Sp^{D_{2p^{\infty}}}_{C_2})$ is $C_2$-stable. Note that the $G$-equalizer as a full $G$-subcategory of the $G$-lax equalizer of $F$ and $F'$ is closed under finite $G$-limits if $F$ and $F'$ are $G$-left exact. Thus, it follows that $\ul{\RCycSp}^{\mr{gen}}_p$ is also $C_2$-stable.
\end{proof}

By repeating \cref{cnstr:InflationToInfinity} and \cref{cnstr:GenuineTrivialFunctor} in the $C_2$-sense, we construct $C_2$-functors
\begin{align*} \ul{\inf}^{\mu_{p^{\infty}}} &: \ul{\Sp}^{C_2} \to \Sp^{D_{2p^{\infty}}}_{C_2}, \\
\ul{\mr{triv}}^{\mr{gen}}_{\RR,p} &: \ul{\Sp}^{C_2} \to \ul{\RCycSp}^{\mr{gen}}_p, \end{align*}
which are $C_2$-exact in view of the commutativity of the diagrams
\[ \begin{tikzcd}[row sep=4ex, column sep=6ex, text height=1.5ex, text depth=0.5ex]
\Sp^{C_2} \ar{r}{\inf^{\mu_{p^n}}} & \Sp^{D_{2p^n}} \\
\Sp \ar{r} \ar{u}{\ind^{C_2}} \ar{r}{\inf^{\mu_{p^n}}} & \Sp^{\mu_{p^n}} \ar{u}[swap]{\ind^{D_{2p^n}}_{\mu_{p^n}} }
\end{tikzcd} \]
for all $n \geq 0$. Since they are also fiberwise colimit preserving, it follows that they are $C_2$-colimit preserving. Therefore, we obtain $C_2$-right adjoints
\begin{align*} \ul{\Psi}^{\mu_{p^{\infty}}} &: \Sp^{D_{2p^{\infty}}}_{C_2} \to \ul{\Sp}^{C_2}, \\
\ul{\TCR}^{\mr{gen}}(-,p) &: \ul{\RCycSp}^{\mr{gen}}_p \to \ul{\Sp}^{C_2}. \end{align*}

\begin{prp} \label{prp:classicalTCRcorepresentable} $\ul{\TCR}^{\mr{gen}}(-,p)$ and $\TCR^{\mr{gen}}(-,p)$ are $C_2$-corepresentable by the unit.
\end{prp}
\begin{proof} This follows by applying \cref{prp:GenericRepresentabilityByUnit} to the $C_2$-adjunction $\ul{\mr{triv}}^{\mr{gen}}_{\RR,p} \dashv \ul{\TCR}^{\mr{gen}}(,-p)$ and using that $\mr{triv}^{\mr{gen}}_{\RR,p}$ is symmetric monoidal.
\end{proof}

\begin{cnstr}[Structure maps $R$] \label{dfn:ParamRStructureMaps} First note that we have natural transformations $$\Psi^{\mu_{p^n}} \to \Psi^{\mu_{p^{n-1}}} \Phi^{\mu_p}$$ defined via applying $\Psi^{\mu_{p^n}}$ to the unit map $\id \to i_{\ast} i^{\ast} \simeq i_{\ast} \Phi^{\mu_p}$ of the recollement on $\Sp^{D_{2p^n}}$ with closed part $\Sp^{D_{2p^{n-1}}}$, using again that $\Psi^{\mu_{p^n}} i_{\ast} \simeq \Psi^{\mu_{p^{n-1}}}$.

Let $[X, \alpha: \Phi^{\mu_p} X \xto{\simeq} X]$ be a genuine real $p$-cyclotomic spectrum. For all $0 < n < \infty$, we define natural maps of $C_2$-spectra
\[ R: \Psi^{\mu_{p^n}} X \to \Psi^{\mu_{p^{n-1}}} \Phi^{\mu_p} (X) \xto{\simeq} \Psi^{\mu_{p^{n-1}}} (X) \]
to be the composite of the above map and $\alpha$. Note that $R$ lifts the maps of spectra $$R: X^{\mu_{p^n}} \to (\Phi^{\mu_p} (X))^{\mu_{p^{n-1}}} \xto{\simeq} X^{\mu_{p^{n-1}}}$$ for the underlying $p$-cyclotomic spectrum (cf. the discussion prior to \cite[Def.~II.4.4]{NS18}). Note also that the diagram
\[ \begin{tikzcd}[row sep=4ex, column sep=4ex, text height=1.5ex, text depth=0.25ex]
\Psi^{\mu_{p^n}} X \ar{r} \ar{d}{F} & \Psi^{\mu_{p^{n-1}}} \Phi^{\mu_p} (X) \ar{r}{\simeq} \ar{d}{F} & \Psi^{\mu_{p^{n-1}}} (X) \ar{d}{F} \\
\Psi^{\mu_{p^{n-1}}} X \ar{r} & \Psi^{\mu_{p^{n-2}}} \Phi^{\mu_p} (X) \ar{r}{\simeq} & \Psi^{\mu_{p^{n-2}}} (X)
\end{tikzcd} \]
commutes for all $1<n<\infty$, with the maps $F$ defined as in \cref{cnstr:InflationToInfinity}. By taking the inverse limit along the maps $F$, the maps $R$ then induce a map
\[ R: \lim_{n,F}\Psi^{\mu_{p^n}}(X) \to \lim_{n,F}\Psi^{\mu_{p^n}}(X). \]
On the other hand, taking the inverse limit along the maps $R$, the maps $F$ induce a map
\[ F:  \lim_{n,R}\Psi^{\mu_{p^n}}(X) \to \lim_{n,R}\Psi^{\mu_{p^n}}(X).\]
\end{cnstr}
% for restriction and induction defined with respect to $D_{2p^{n-1}} \subset D_{2p^{n}}$, the composite
% \[ R: \Psi^{\mu_{p^n}}(\ind \res X) \to \Psi^{\mu_{p^{n-1}}} \Phi^{\mu_p} (\ind \res X) \xto{\simeq} \Psi^{\mu_{p^{n-1}}} (\ind \res X) \]
% is homotopic to
% \[ R: \Psi^{\mu_{p^{n-1}}}(X) \to \Psi^{\mu_{p^{n-2}}} \Phi^{\mu_p} (X) \xto{\simeq} \Psi^{\mu_{p^{n-2}}} (X) \]
% for the map $R$ defined with respect to $n-1$. Therefore, 

\begin{dfn} For a genuine real $p$-cyclotomic spectrum $[X, \Phi^{\mu_p} X \xto{\simeq} X]$, let 
\begin{align*} \TRR(X,p) &:= \lim_{n,R} \Psi^{\mu_{p^n}}(X), \\
\TFR(X,p) &:= \lim_{n,F} \Psi^{\mu_{p^n}}(X).
\end{align*}
\end{dfn}

\begin{prp} \label{prp:fiberSequenceGenuineRealCyc} Let $[X, \Phi^{\mu_p} X \xto{\simeq} X]$ be a genuine real $p$-cyclotomic spectrum. We have natural fiber sequences of $C_2$-spectra
\begin{align*} \TCR^{\mr{gen}}(X,p) &\to \TFR(X,p) \xtolong{\id - R}{1.5} \TFR(X,p), \\
\TCR^{\mr{gen}}(X,p) &\to \TRR(X,p) \xtolong{\id - F}{1.5} \TRR(X,p),
\end{align*}
and a natural equivalence of $C_2$-spectra
\[ \TCR^{\mr{gen}}(X,p) \simeq \lim_{n, F, R} \Psi^{\mu_{p^n}}(X) \coloneq \lim_{J_{\infty}} \Psi^{\mu_{p^n}}(X).  \]
\end{prp}

Before giving the proof, we define the category $J_{\infty}$ and prove a few necessary results about it.

\begin{dfn} \label{dfn:EqualizerInfinity} Let $J_{\infty}$ be the category freely generated by
\[ \begin{tikzcd}[row sep=4ex, column sep=4ex, text height=1.5ex, text depth=0.25ex]
\cdots \ar[shift left=1]{r}{\alpha_{n+1}} \ar[shift right=1]{r}[swap]{\beta_{n+1}} & n+1 \ar[shift left=1]{r}{\alpha_n} \ar[shift right=1]{r}[swap]{\beta_n} & n \ar[shift left=1]{r}{\alpha_{n-1}} \ar[shift right=1]{r}[swap]{\beta_{n-1}} & \cdots \ar[shift left=1]{r}{\alpha_1} \ar[shift right=1]{r}[swap]{\beta_1} & 1 \ar[shift left=1]{r}{\alpha_0} \ar[shift right=1]{r}[swap]{\beta_0} & 0 \:,
\end{tikzcd} \]
modulo the relation $\beta_n \circ \alpha_{n+1} = \alpha_n \circ \beta_{n+1}$ for all $n \geq 0$. More concretely, the objects of $J_{\infty}$ are non-negative integers, there are no morphisms $n \to n+k$ for $k>0$, there is only the identity $n \to n$, and morphisms $n+k \to n$, $k > 0$ are in bijection with non-empty sieves in $[k]$, where we attach to $S \subset [k]$ the composition $$\beta_{n} \cdots \beta_{n+l-1} \alpha_{n+l} \cdots \alpha_{n+k-1}$$ for $l = \max(S)$ (so if $l=0$, we have $\alpha_n \cdots \alpha_{n+k-1}$, and if $l=k$, we have $\beta_n \cdots \beta_{n+k-1}$).
\end{dfn}

\begin{rem} \label{rem:spineOfEqualizerInfinity} Let $\pi: J_{\infty} \to \ZZ_{\geq 0}^{\op}$ be the functor that sends $n$ to $n$, and $\alpha_n, \beta_n$ to $n+1 \to n$. For $n \geq m$, let $[n:m] \subset \ZZ_{\geq 0}^{\op}$ denote the full subcategory on integers $n \geq k \geq m$, and let $$J_{[n:m]} = J_{\infty} \times_{\ZZ_{\geq 0}^{\op}} [n:m].$$
We claim that the square
\[ \begin{tikzcd}[row sep=4ex, column sep=4ex, text height=1.5ex, text depth=0.25ex]
J_{[2:1]} \ar{r} \ar{d} & J_{[2:0]} \ar{d} \\
J_{[3:1]} \ar{r} & J_{[3:0]}
\end{tikzcd} \]
is a homotopy pushout square of $\infty$-categories. Indeed, for clarity write $a<b<c$ for the vertices of $\Delta^2$, and let $q: J_{[3:0]} \to \Delta^2$ be the functor that sends $3$ to $a$, $2,1$ to $b$, and $0$ to $c$, and maps in the obvious way. The claim amounts to showing that $q$ is a flat inner fibration (\cite[Def.~B.3.1]{HA}), for which we may use the criterion of \cite[Prop.~B.3.2]{HA}. Suppressing subscripts of morphisms in $J_{[3:0]}$ for clarity, we need to check that for morphisms $$\gamma \in \{ \beta^3, \beta^2 \alpha, \beta \alpha^2, \alpha^3 \} \in \Hom(3,0),$$
the resulting category $(J_{[2:1]})_{3/ /1}$ of factorizations of $\gamma$ through $J_{[2:1]}$ is weakly contractible. For $\gamma = \delta \circ \epsilon$ with the domain of $\delta$ equal to $i=1,2$, write $[\delta|\epsilon]_i$ for the object in $(J_{[2:1]})_{3/ /1}$. If $\gamma = \beta^3$, then $(J_{[2:1]})_{3/ /1}$ is given by $[\beta^2|\beta]_2 \to [\beta|\beta^2]_1$, so is weakly contractible, and likewise for $\gamma = \alpha^3$. If $\gamma = \beta \alpha^2$, then $(J_{[2:1]})_{3/ /1}$ is the category
\[ \begin{tikzcd}[row sep=4ex, column sep=4ex, text height=1.5ex, text depth=0.25ex]
\left[\beta \alpha | \alpha \right]_2 \ar{d} \ar{rd} & \left[ \alpha^2 | \beta\right]_2 \ar{d} \\
\left[ \beta| \alpha^2 \right]_1 & \left[ \alpha| \beta \alpha \right]_1,
\end{tikzcd} \]
using that $\beta \alpha = \alpha \beta$ and always writing maps as $\beta^i \alpha^j$. Thus, $(J_{[2:1]})_{3/ /1}$ is weakly contractible, and likewise for $\gamma = \beta \alpha^2$, proving the claim. Continuing this line of reasoning, we see that the cofibration
\[ \left( J_{[2:0]} \bigcup_{J_{[2:1]}} J_{[3:1]} \bigcup_{J_{[3:2]}} J_{[4:2]} \bigcup_{J_{[4:3]}} \cdots \right) \to J_{\infty} \]
is a categorical equivalence. Therefore, for an $\infty$-category $C$ and two diagrams $\ZZ_{\geq 0}^{\op} \to C$ written as
\begin{align*} \cdots \xto{F} X_2 \xto{F} X_1 \xto{F} X_0, \\
\cdots \xto{R} X_2 \xto{R} X_1 \xto{R} X_0,
\end{align*}
to extend this data to a diagram $J_{\infty} \to C$ that sends $\alpha$ to $F$ and $\beta$ to $R$, we only need to supply the data of commutative squares in $C$
\[ \begin{tikzcd}[row sep=4ex, column sep=4ex, text height=1.5ex, text depth=0.25ex]
X_{n+2} \ar{r}{F} \ar{d}{R} & X_{n+1} \ar{d}{R} \\
X_{n+1} \ar{r}{F} & X_n 
\end{tikzcd} \]
for all $n \geq 0$.
\end{rem}

\begin{lem} \label{lm:abstractEqualizerExchangeFormula} \begin{enumerate}[leftmargin=*] \item Let $p, p': J_{\infty} \to B \NN$ be the functors determined by sending $\alpha_n$ to $0$ and $\beta_n$ to $1$, resp. $\beta_n$ to $0$ and $\alpha_n$ to $1$. Then $p$ and $p'$ are cartesian fibrations classified by the functor $(B \NN)^{\op} \simeq B \NN \to \Cat \subset \Cat_{\infty}$ that sends the unique object to $\ZZ_{\geq 0}^{\op}$ and $1$ to $s^{\op}$, the (opposite of the) successor endofunctor (cf. \cref{ntn:nonnegativeintegers}).
\item Suppose $C$ is an $\infty$-category and $X_{\bullet}: J_{\infty} \to C$ is a diagram. Denote all maps $(X_{\bullet})(\alpha_n)$ by $F$ and all maps $(X_{\bullet})(\beta_n)$ by $R$. Then assuming the limits exist in $C$, we have equivalences
\[ \begin{tikzcd}[row sep=4ex, column sep=4ex, text height=1.5ex, text depth=0.25ex]
\mr{eq}(\lim_{n,R} X_n \ar[shift left=1]{r}{\id} \ar[shift right=1]{r}[swap]{F} & \lim_{n,R} X_n) \simeq \lim_{J_{\infty}} (X_{\bullet}) \simeq \mr{eq}(\lim_{n,F} X_n \ar[shift left=1]{r}{\id} \ar[shift right=1]{r}[swap]{R} & \lim_{n,F} X_n),
\end{tikzcd} \]
where we also write $F$ and $R$ for the induced maps on the limits.
\end{enumerate} 
\end{lem}
\begin{proof} (1): To show that $p$ is a cartesian fibration, it suffices to show that $\beta_n$ is a $p$-cartesian edge for all $n \geq 0$. For this, suppose given a map $f:m \to n$ in $J_{\infty}$ such that $p(f)$ factors as $\ast \xto{k} \ast \xto{1} \ast$, i.e., $p(f) \geq 1$. Then we must have $m \geq n+1$ and $f \neq \alpha_n \cdots \alpha_{m-1}$, so $f$ factors uniquely through $\beta_n$ and the claim is proven. The case of $p'$ is identical. Finally, the description of the resulting action of $\NN$ on the fiber $\ZZ_{\geq 0}^{\op}$ is clear in view of the commutative diagram
\[ \begin{tikzcd}[row sep=4ex, column sep=4ex, text height=1.5ex, text depth=0.25ex]
0 & 1 \ar{l}[swap]{\alpha_0} & 2 \ar{l}[swap]{\alpha_1} & 3 \ar{l}[swap]{\alpha_2} & \cdots \ar{l}[swap]{\alpha_3} \\
1 \ar{u}{\beta_0} & 2 \ar{l}{\alpha_1} \ar{u}{\beta_2} & 3 \ar{l}{\alpha_2} \ar{u}{\beta_2} & 4 \ar{l}{\alpha_3} \ar{u}{\beta_3} & \cdots \ar{l}{\alpha_4}
\end{tikzcd} \]
and similarly with the roles of $\alpha_{\bullet}$ and $\beta_{\bullet}$ exchanged.

(2): Factoring $J_{\infty} \to \ast$ through the cartesian fibration $p$ and using the transitivity of right Kan extensions, we get that 
$$\lim_{J_{\infty}} (X_{\bullet}) \simeq \lim_{B \NN} \lim_{n, F} X_n, $$
where $\NN$ acts on $\lim_{n, F} X_n$ via $R$. But the limit over $B\NN$ is computed also as the equalizer of $\id$ and $R$, so we deduce the equivalence 
\[ \begin{tikzcd}[row sep=4ex, column sep=4ex, text height=1.5ex, text depth=0.25ex]
\lim_{J_{\infty}} (X_{\bullet}) \simeq \mr{eq}(\lim_{n,F} X_n \ar[shift left=1]{r}{\id} \ar[shift right=1]{r}[swap]{R} & \lim_{n,F} X_n).
\end{tikzcd} \]
Doing the same with $p'$ shows the other equivalence.
\end{proof}

\begin{proof}[Proof of \cref{prp:fiberSequenceGenuineRealCyc}] Let $\alpha: \Phi^{\mu_p} X \xto{\simeq} X$ denote the structure map. Let $C = \Sp^{D_{2p^{\infty}}}_{C_2}$. For $X \in \Sp^{D_{2p^{\infty}}}$, we have that $\Psi^{\mu_{p^{\infty}}}(X) \simeq \ul{\map}_{C}(S^0,X)$ as the $C_2$-right adjoint to $\ul{\inf}^{\mu_{p^{\infty}}}$, and $\Psi^{\mu_{p^{\infty}}}(X) \simeq \lim_{n,F} \Psi^{\mu_{p^n}}(X)$ as we saw in \cref{cnstr:InflationToInfinity}. Also, the $G$-equalizer is a full $G$-subcategory of the $G$-lax equalizer, so $G$-mapping spaces and spectra may be computed as in \cref{lem:EqualizerMappingSpaces} if we take $F = \Phi^{\mu_p}$, $F'=\id$. We thus obtain a fiber sequence with the objects as in the first fiber sequence in the statement, and it remains to identify the maps. Because $\Phi^{\mu_p}(S^0) \simeq S^0$, one of the maps in that fiber sequence is homotopic to $\id$. On the other hand, we claim that
\[ \Phi^{\mu_p}: \ul{\map}_C(S^0,X) \to \ul{\map}_C(S^0,\Phi^{\mu_p} X) \]
is homotopic to the map $$\lim_{n,F} \Psi^{\mu_{p^n}}(X) \to \lim_{n,F} \Psi^{\mu_{p^n}} \Phi^{\mu_p}(X)$$ induced by taking the limit of the natural transformations $\Psi^{\mu_{p^n}} X \to \Psi^{\mu_{p^{n-1}}} \Phi^{\mu_p} X$ defined in \cref{dfn:ParamRStructureMaps}. Indeed, since the functor $\Phi^{\mu_p}$ is obtained as the inverse limit of functors $\Phi^{\mu_p}: \Sp^{D_{2p^n}} \to \Sp^{C_2}$, the map $\Phi^{\mu_p}$ of $C_2$-spectra is also obtained as the inverse limits of maps
\[\Phi^{\mu_p}: \ul{\map}_{C_n}(S^0,X_n) \to \ul{\map}_{C_{n-1}}(S^0,\Phi^{\mu_p} X_n) \]
where $C_n = \sO_{C_2}^{\op} \times_{\sO_{D_{2p^n}}^{\op}} \ul{\Sp}^{D_{2p^{n}}}$ and $X_n$ is the restriction of $X$ to $\Sp^{D_{2p^{n}}}$. But with respect to the $C_2$-adjunction $\Phi^{\mu_p} \dashv i_{\ast}$ and the resulting equivalence
\[ \ul{\map}_{C_{n-1}}(S^0,\Phi^{\mu_p} X_n) \simeq \ul{\map}_{C_{n}}(S^0, i_{\ast} \Phi^{\mu_p} X_n),  \]
we may identify this map as given by $\ul{\map}_{C_{n}}(S^0,-) \simeq \Psi^{\mu_{p^n}}$ on the unit for $X_n$, which is the map of \cref{dfn:ParamRStructureMaps}. It follows that the composite
\[ \ul{\map}_C(S^0,X) \xto{\Phi^{\mu_p}} \ul{\map}_C(S^0,\Phi^{\mu_p} X) \xto{\alpha_{\ast}} \ul{\map}_C(S^0, X) \]
is homotopic to $R$, and we deduce the first fiber sequence.

Because the maps $F$ and $R$ commute, by \cref{rem:spineOfEqualizerInfinity} the $F$ and $R$ maps extend to define a diagram $J_{\infty} \to \Sp^{C_2}$. Then by \cref{lm:abstractEqualizerExchangeFormula}, we deduce the last equivalence and second fiber sequence.
\end{proof}

\begin{rem} Although they allude to the corepresentability of $\TC^{\mr{gen}}$ in the introduction \cite[p.~207]{NS18}, Nikolaus and Scholze choose to \emph{define} $\TC^{\mr{gen}}(-,p)$ via the fiber sequence \cite[Def.~II.4.4]{NS18}
\[ \TC^{\mr{gen}}(X,p) \to \mr{TR}(X,p) \xtolong{\id - F}{1} \mr{TR}(X,p). \]
The $C_2$-corepresentability of $\ul{\TCR}^{\mr{gen}}(-,p)$, or simple repetition of the proof of \cref{prp:classicalTCRcorepresentable}, immediately implies that $\TC^{\mr{gen}}(-,p)$ is corepresentable by the unit. Alternatively, one may deduce this from results of Blumberg-Mandell \cite{BM16} and the comparison \cite[Thm.~II.3.7]{NS18} as noted in \cite[Rem.~II.6.10]{NS18}.
\end{rem}

\subsection{\texorpdfstring{$C_2$}{C2}-symmetric monoidal structures} \label{sec:C2_smc}
% Recall (e.g., from \cite[\S 5.1]{QS21a}) that for a finite group $G$ the Hill--Hopkins--Ravenel norm functors endow the $G$-$\infty$-category $\underline{\Sp}^G$ with a $G$-symmetric monoidal structure in the sense of \cite[Def.~5.6]{QS21a}. 
In this subsection we explain how to endow $\underline{\RCycSp}_p$ and $\underline{\RCycSp}^{\mr{gen}}_p$ with $C_2$-symmetric monoidal structures so that $\TCR(-,p)$ and $\TCR^{\mr{gen}}(-,p)$ are lax $C_2$-symmetric monoidal functors (and so in particular preserve $C_2$-$E_{\infty}$-algebras).

To begin with, we showed in \cite[Thm.~E]{QS21a} that the parametrized Tate functor $(-)^{t_{C_2} \mu_p}: \Fun_{C_2}(B^t_{C_2} \mu_p, \underline{\Sp}^{C_2}) \to \Sp^{C_2}$ considered together with the natural transformation $(-)^{h_{C_2} \mu_p} \Rightarrow (-)^{t_{C_2} \mu_p}$ uniquely refines to a lax $C_2$-symmetric monoidal functor
$$(-)^{\underline{t}_{C_2} \mu_p}: \underline{\Fun}_{C_2}(B^t_{C_2} \mu_p, \underline{\Sp}^{C_2}) \to \underline{\Sp}^{C_2}$$
under $(-)^{\underline{h}_{C_2} \mu_p}$, where we endow the source with the pointwise $C_2$-symmetric monoidal structure of \cite[Exm.~5.13]{QS21a}. In fact, by the same logic $(-)^{\underline{t}_{C_2} \mu_p}$ considered as an $C_2$-endofunctor of $\underline{\Fun}_{C_2}(B^t_{C_2} \mu_{p^{\infty}}, \underline{\Sp}^{C_2})$ uniquely admits a lax $C_2$-symmetric monoidal structure. Writing the analogous pullback square of $C_2$-$\infty$-operads to that in \cref{con:smc_rcyc}, we thereby obtain a $C_2$-symmetric monoidal structure on $\underline{\RCycSp}_p$ such that the forgetful functor $\underline{\RCycSp}_p \to \underline{\Sp}^{C_2}$ is $C_2$-symmetric monoidal.

We now note that $\underline{\RCycSp}_p$ is fiberwise presentable by \cref{prp:RCycSpPresentable} and \cite[Cor.~II.1.7]{NS18} and $C_2$-stable by \cref{lem:ParamLaxEqualizerStable}, hence is $C_2$-bicomplete. Using that the forgetful $C_2$-functor to $\underline{\Sp}^{C_2}$ is $C_2$-colimit preserving and fiberwise conservative (by \cref{lem:ParamLaxEqualizerStable}, it is $C_2$-exact, and it is fiberwise conservative and fiberwise colimit preserving by \cref{prp:forgetful_conservative} and \cite[Cor.~II.1.7]{NS18}, so $C_2$-colimit preserving), we deduce that the $C_2$-symmetric monoidal structure on $\underline{\RCycSp}_p$ inherits the property of being $C_2$-distributive \cite[Def.~5.16]{QS21a} from $\underline{\Sp}^{C_2}$. We thus see that $\underline{\RCycSp}_p$ is \emph{$C_2$-presentably symmetric monoidal} in the sense of being a $C_2$-commutative algebra in the $C_2$-symmetric monoidal $\infty$-category $\underline{\Pr}^{L, \st}_{C_2}$ of $C_2$-presentable $C_2$-stable $C_2$-$\infty$-categories \cite[\S 3.4]{nardin}. The unit map then furnishes a unique $C_2$-colimit preserving $C_2$-exact $C_2$-symmetric monoidal functor $\underline{\Sp}^{C_2} \to \underline{\RCycSp}_p$, specified on underlying $C_2$-$\infty$-categories by the unique $C_2$-colimit preserving $C_2$-exact $C_2$-functor that sends $S^0 \in \Sp^{C_2}$ to the unit in $\RCycSp_p$.\footnote{In fact, the second author shows in \cite{paramalg} that for any $G$-symmetric monoidal $\infty$-category $C$, $\CAlg_G(C)$ has an initial object given by the unit in $C$.} Since this $C_2$-functor identifies with $\underline{\mr{triv}}_{\RR,p}$, we thus obtain a unique $C_2$-symmetric monoidal refinement of $\underline{\mr{triv}}_{\RR,p}$. It follows that $\underline{\TCR}(-,p)$, being $C_2$-right adjoint to $\underline{\mr{triv}}_{\RR,p}$, refines to a lax $C_2$-symmetric monoidal functor.

Next, we note that the Hill--Hopkins--Ravenel norms endow the $C_2$-$\infty$-category $\Sp^{D_{2p^n}}_{C_2}$ with $C_2$-symmetric monoidal structure, where the $C_2$-indexed multiplication $\otimes_{C_2}$ is furnished by the norm functor $N^{D_{2p^n}}_{\mu_{p^n}}: \Sp^{\mu_{p^n}} \to \Sp^{D_{2p^n}}$; indeed, we define the requisite functor $\Span(\FF_{C_2}) \to \Cat_{\infty}$ by precomposing the $D_{2p^n}$-symmetric monoidal $\infty$-category
$$(\Sp^{D_{2p^n}})^{\otimes}: \Span(\FF_{D_{2p^n}}) \to \Cat_{\infty}$$
by the functor $\Span(\FF_{C_2}) \to \Span(\FF_{D_{2p^n}})$ that endows a $C_2$-set with $D_{2p^n}$-action via the quotient map $D_{2p^n} \to D_{2p^n}/\mu_{p^n} \cong C_2$. The restriction functors $\Sp^{D_{2p^{n+1}}}_{C_2} \to \Sp^{D_{2p^n}}_{C_2}$ are compatible with the $C_2$-symmetric monoidal structure, and hence passing to the inverse limit we endow $\Sp^{D_{2p^\infty}}_{C_2}$ with a $C_2$-symmetric monoidal structure. Moreover, $\Sp^{D_{2p^\infty}}_{C_2}$ inherits the property of being $C_2$-presentably symmetric monoidal from the $\Sp^{D_{2p^n}}_{C_2}$ (which themselves inherit this property from the $D_{2p^n}$-distributivity of $\underline{\Sp}^{D_{2p^n}}$ as a $D_{2p^n}$-symmetric monoidal $\infty$-category).

Now observe that the $C_2$-endofunctor $\underline{\Phi}^{\mu_p}$ is $C_2$-symmetric monoidal by virtue of the geometric fixed points functor commuting with the Hill--Hopkins--Ravenel norm functors (cf. the discussion under \cite[Def.~2.5]{QS21a}). Indeed, the commutative square
\[ \begin{tikzcd}
B \mu_{p^{n+1}} \ar{r} \ar{d} & B(\mu_{p^{n+1}}/\mu_p) \ar{d} \\ 
B D_{2p^{n+1}} \ar{r} & B (D_{2p^{n+1}}/\mu_p)
\end{tikzcd} \]
of finite groupoids transports under the functor $\SH^{\otimes}: \Span(\Gpd_{\fin}) \to \CAlg(\Cat^{\sift}_{\infty})$ of \cite[Def.~2.5]{QS21a} to the commutative square of $\infty$-categories
\[ \begin{tikzcd}
\Sp^{\mu_{p^n+1}} \ar{r}{\Phi^{\mu_p}} \ar{d}[swap]{N^{D_{2p^{n+1}}}_{\mu_{p^{n+1}}}} & \Sp^{\mu_{p^n}} \ar{d}{N^{D_{2p^n}}_{\mu_{p^n}}} \\ 
\Sp^{D_{2p^{n+1}}} \ar{r}{\Phi^{\mu_p}} & \Sp^{D_{2p^n}},
\end{tikzcd} \]
and this extends to a natural transformation $(\underline{\Phi}^{\mu_p})^{\otimes}: (\Sp^{D_{2p^{n+1}}}_{C_2})^{\otimes} \Rightarrow (\Sp^{D_{2p^{n}}}_{C_2})^{\otimes}$ via precomposing $\SH^{\otimes}$ by the map
\[ h: \Delta^1 \times \Span(\FF_{C_2}) \to \Span(\Gpd_{\fin}) \]
that restricts on $\{i\} \times \Span(\FF_{C_2})$, $i \in \{0,1\}$ to the map $X \mapsto X//D_{2p^{n+1-i}}$ (where $X//G$ denotes the action groupoid), so $\SH^{\otimes} \circ h_i \simeq (\Sp^{D_{2p^{n+1-i}}}_{C_2})^{\otimes}$, and sends each $[(0,X) \rightarrow (1,X)]$ to the functor $X//_{D_{2p^{n+1}}} \to X//_{D_{2p^n}}$ induced by the quotient map $D_{2p^{n+1}} \to D_{2p^n} \cong D_{2p^{n+1}}/\mu_p$ covering the quotient map to $C_2$. Passage to the inverse limit then endows the $C_2$-endofunctor $\underline{\Phi}^{\mu_p}$ with its $C_2$-symmetric monoidal structure.

Since $\underline{\Phi}^{\mu_p}$ is also $C_2$-colimit preserving, the equalizer diagram of $C_2$-$\infty$-categories defining $\underline{\RCycSp}^{\mr{gen}}_p$ lifts to the $\infty$-category $\CAlg_{C_2}(\underline{\Pr}^{L, \st}_{C_2})$ of $C_2$-commutative algebras in $C_2$-presentable $C_2$-stable $C_2$-$\infty$-categories. Therefore, using that the forgetful functor creates limits $\underline{\RCycSp}^{\mr{gen}}_p$ is again $C_2$-presentably symmetric monoidal. Finally, the $C_2$-functors $\underline{\inf}^{\mu_{p^n}}: \underline{\Sp}^{C_2} \to \Sp^{D_{2p^n}}_{C_2}$ refine to $C_2$-symmetric monoidal functors $\underline{\Sp}^{C_2} \to \Sp^{D_{2p^n}}_{C_2}$ by virtue of \cite[Rem.~2.7]{QS21a} applied to the pullback square of finite groupoids
\[ \begin{tikzcd}
B \mu_{p^n} \ar{r} \ar{d} & B D_{2p^n} \ar{d} \\ 
\ast \ar{r} & B (D_{2p^n} / \mu_{p^n}) \cong B C_2
\end{tikzcd} \]
(and extending to natural transformations over $\Span(\FF_{C_2})$ by the same methods as before), and hence $\underline{\inf}^{\mu_{p^{\infty}}}$ and $\underline{\mr{triv}}^{\mr{gen}}_{\RR,p}$ refine to $C_2$-symmetric monoidal functors, whose $C_2$-right adjoints $\underline{\Psi}^{\mu_{p^{\infty}}}$ and $\underline{\TCR}^{\mr{gen}}(-,p)$ then acquire lax $C_2$-symmetric monoidal structures.

\section{Comparison of the theories}
\label{section:ComparisonSection}

Let $[X, \alpha: \Phi^{\mu_p} X \xto{\simeq} X]$ be a genuine real $p$-cyclotomic spectrum. From \cref{setup:Dihedral}, consider the recollement \eqref{eq:dihedral_recollement_infinite}
\[ \begin{tikzcd}[row sep=4ex, column sep=8ex, text height=1.5ex, text depth=0.5ex]
\Fun_{C_2}(B^t_{C_2} \mu_{p^{\infty}}, \ul{\Sp}^{C_2}) \ar[shift right=1,right hook->]{r}[swap]{\sF^{\vee}_b } & \Sp^{D_{2p^{\infty}}} \ar[shift right=2]{l}[swap]{\sU_b} \ar[shift left=2]{r}{\Phi^{\mu_p}} & \Sp^{D_{2p^{\infty}}} \ar[shift left=1,left hook->]{l}{i_{\ast}}
\end{tikzcd} \]
and the morphism induced by the unit of $\sU_b \dashv \sF^{\vee}_b$
$$ \beta: \sU_b \Phi^{\mu_p} (X) \to  \sU_b \Phi^{\mu_p} \sF^{\vee}_b \sU_b (X) \simeq (\sU_b X)^{t_{C_2} \mu_p}. $$

Choosing an inverse $\alpha^{-1}$, let $\varphi = \beta \circ (\sU_b \alpha^{-1}): \sU_b X \to (\sU_b X)^{t_{C_2} \mu_p}$. Then $[\sU_b X, \varphi]$ is a real $p$-cyclotomic spectrum. More generally, the lax symmetric monoidal natural transformation
\[ \sU^b \Phi^{\mu_p} \to \sU_b \Phi^{\mu_p} \sF^{\vee}_b \sU_b \simeq (-)^{t_{C_2} \mu_p} \circ \sU_b \]
defines a symmetric monoidal functor
\[ \Eq_{\id:\Phi^{\mu_p}}(\Sp^{D_{2p^{\infty}}}) \to \RCycSp_p \]
via the universal property of the lax equalizer (as explained in \cite[Prop.~II.3.2]{NS18}), and precomposition with the (symmetric monoidal) equivalence of \cref{rem:EqualizerSwap} defines a symmetric monoidal functor
\[ \sU_{\RR}: \RCycSp^{\mr{gen}}_p \to \RCycSp_p \]
that lifts the functor $\sU_b: \Sp^{D_{2p^{\infty}}} \to \Fun_{C_2}(B^t_{C_2} \mu_{p^{\infty}}, \ul{\Sp}^{C_2})$ through the functors that forget the structure maps.

\begin{rem} \label{rem:C2symm_forgetful}
Considering instead the $C_2$-recollement \eqref{eq:C2_dihedral_recollement_infinite}
\[ \begin{tikzcd}[row sep=4ex, column sep=8ex, text height=1.5ex, text depth=0.5ex]
\underline{\Fun}_{C_2}(B^t_{C_2} \mu_{p^{\infty}}, \ul{\Sp}^{C_2}) \ar[shift right=1,right hook->]{r}[swap]{\underline{\sF}^{\vee}_b } & \Sp^{D_{2p^{\infty}}}_{C_2} \ar[shift right=2]{l}[swap]{\underline{\sU}_b} \ar[shift left=2]{r}{\underline{\Phi}^{\mu_p}} & \Sp^{D_{2p^{\infty}}}_{C_2} \ar[shift left=1,left hook->]{l}{i_{\ast}}
\end{tikzcd} \]
we obtain a fiberwise symmetric monoidal $C_2$-functor
\[ \underline{\sU}_{\RR}: \underline{\RCycSp}^{\mr{gen}}_p \to \underline{\RCycSp}_p \]
given over $C_2/C_2$ by $\sU_{\RR}$ and over $C_2/1$ by the forgetful functor $\CycSp^{\mr{gen}}_p \to \CycSp_p$ as defined by Nikolaus and Scholze. In fact, we can prove that $\underline{\sU}_{\RR}$ is $C_2$-symmetric monoidal. This will follow by the same reasoning as above after we note:
\begin{enumerate}
\item $\underline{\Phi}^{\mu_{p}}$ is $C_2$-symmetric monoidal as was explained in \cref{sec:C2_smc}.
\item At each finite level $n$, the right adjoint inclusion $\Sp^{D_{2p^n}}_{C_2} \subset \Sp^{D_{2p^{n+1}}}_{C_2}$ is that of a $C_2$-$\otimes$-ideal \cite[Def.~5.24]{QS21a} by the same argument as \cite[Exm.~5.26]{QS21a}. Passing to inverse limits, we get that $i_*$ is the inclusion of a $C_2$-$\otimes$-ideal. Therefore, the projection $C_2$-functor $\underline{\sU}_b$ to the $C_2$-Verdier quotient is $C_2$-symmetric monoidal by \cite[Thm.~5.28]{QS21a}.
\end{enumerate}
\end{rem}

\begin{dfn} Let $[X,\varphi]$ be a real $p$-cyclotomic spectrum. Then $[X,\varphi]$ is \emph{underlying bounded-below} if the underlying spectrum of $X$ is bounded-below.

Similarly, we say that a genuine real $p$-cyclotomic spectrum $[X,\alpha]$ is \emph{underlying bounded-below} if the underlying spectrum of $X$ is bounded-below.\footnote{This implies that the underlying spectra of $\Phi^{\mu_{p^n}}(X)$ are bounded-below for all $n \geq 0$.}
\end{dfn}

Let us restate the main theorem of this paper from the introduction.

\begin{thm}[See \cref{thm:MainTheoremRestated}] \label{thm:MainTheoremEquivalenceBddBelow} $\sU_{\RR}$ restricts to an equivalence on the full subcategories of underlying bounded-below objects.
\end{thm}

\begin{cor} \label{cor:TCRFormulasEquivalent} Let $X$ be a genuine real $p$-cyclotomic spectrum that is underlying bounded-below. Then we have a canonical equivalence
\[ \TCR^{\mr{gen}}(X,p) \simeq \TCR(\sU_{\RR} X,p). \]
Moreover, for any (reduced) $C_2$-$\infty$-operad $\sO^{\otimes}$, if $X$ is an $\sO$-algebra in $\underline{\RCycSp}^{\mr{gen}}_p$, then this equivalence is that of $\sO$-algebras in $\underline{\Sp}^{C_2}$. In particular, if $X$ is a $C_2$-$E_{\infty}$-algebra, then the above equivalence is that of $C_2$-$E_{\infty}$-algebras.
\end{cor}
\begin{proof} Because $\sU_{\RR}$ preserves the unit (which is bounded-below), the first claim follows immediately from \cref{thm:MainTheoremEquivalenceBddBelow} and the $C_2$-corepresentability of $\TCR(-,p)$ and $\TCR^{\mr{gen}}(-,p)$ by the unit (\cref{prp:C2representabilityOfTC} and \cref{prp:classicalTCRcorepresentable}). The second claim then follows in view of the lax $C_2$-symmetric monoidality of $\TCR(-,p)$ and $\TCR^{\mr{gen}}(-,p)$ and \cref{rem:C2symm_forgetful}.
\end{proof}

As with the comparison theorem for $p$-cyclotomic spectra \cite[Thm.~II.6.3]{NS18}, the key computation that establishes \cref{thm:MainTheoremEquivalenceBddBelow} is a dihedral extension of the Tate orbit lemma, to which we turn first. Using this, we indicate how the ``decategorified'' version of \cref{thm:MainTheoremEquivalenceBddBelow} in the form of \cref{cor:TCRFormulasEquivalent} follows by the same arguments as in \cite[\S II.4]{NS18}. We then leverage the reconstruction theorem of Ayala--Mazel-Gee--Rozenblyum describing $G$-spectra in terms of their geometric fixed points\footnote{We recall this as \cref{rec:reconstruction} below.} in order to prove \cref{thm:MainTheoremEquivalenceBddBelow}, proceeding in two stages. First, we obtain a comparison result at ``finite level'' (\cref{cor:iterated_pullback_descr}) by inputting the dihedral Tate orbit lemma together with a relative version of \cite[Thm.~2.42]{QS21a} into the abstract equivalence of \cite[Thm.~4.15]{Sha21} between $1$-generated and extendable objects in a $[n]$-stratified stable $\infty$-category. This identifies an underlying bounded-below $D_{2p^n}$-spectrum $X$ with a list of $C_2$-spectra with parametrized action\footnote{Here, we abuse notation and let $\Phi^{\mu_{p^k}} X$ also denote its underlying $C_2$-spectrum with twisted $\mu_{p^{n-k}}$-action. We will later instead write $L[\zeta]_k X$ to disambiguate (\cref{eq:dihedral_localization}).}
\[ [\Phi^{1} X, \: \Phi^{\mu_p} X, \: ..., \: \Phi^{\mu_{p^n}} X], \quad \Phi^{\mu_{p^k}} X \in \Fun_{C_2}(B^t_{C_2} \mu_{p^{n-k}}, \underline{\Sp}^{C_2}), \]
along with (twisted equivariant) maps $\Phi^{\mu_{p^{k+1}}} X \to (\Phi^{\mu_{p^k}} X)^{t_{C_2} \mu_{p}}$ for $0 \leq k < n$. Passing to infinity, we then identify an underlying bounded-below $D_{2p^{\infty}}$-spectrum $X$ with a list of $C_2$-spectra with twisted $\mu_{p^{\infty}}$-action and maps thereof
\[ [\Phi^{1} X, \: \Phi^{\mu_p} X \to (\Phi^{1} X)^{t_{C_2} \mu_p}, \: ..., \: \Phi^{\mu_{p^{n+1}}} X \to (\Phi^{\mu_{p^n}} X)^{t_{C_2} \mu_p}, \: ... \: ], \]
with respect to which the endofunctor $\Phi^{\mu_p}$ of $\Sp^{D_{2p^{\infty}}}$ identifies with the operation that shifts once to the left. Using an abstract result on the equalizer of the identity and the shift endofunctor (\cref{lem:LaxEqualizerGenericEquivalence}), we may then deduce \cref{thm:MainTheoremEquivalenceBddBelow}.

Throughout, we will use that the parametrized Tate functor
$$(-)^{t_{C_2} \mu_p}: \Fun_{C_2}(B^t_{C_2} \mu_p, \ul{\Sp}^{C_2}) \to \Sp^{C_2}$$
identifies with the gluing functor of the $\Gamma_{\mu_p}$-recollement on $\Sp^{D_{2p}}$ for $\Gamma_{\mu_p}$ the $\mu_p$-free $D_{2p}$-family. By \cite[Thm.~F]{QS21a}, we then have a formula for the geometric fixed points of $X^{t_{C_2} \mu_p}$. Namely:

\begin{enumerate}[leftmargin=*]
\item \cite[Exm.~2.51]{QS21a} Suppose $p = 2$, so $D_4 = C_2 \times \mu_2$ is the Klein four-group and $\Gamma_{\mu_2} = \{ 1, C_2, \Delta \}$ for $\Delta$ the diagonal subgroup. Then the data of a $C_2$-functor $X : B^t_{C_2} \mu_2 \to \underline{\Sp}^{C_2}$ amounts to a tuple of objects
\[ [X^1 \in \Sp^{h D_4}, \: X^{\phi C_2} \in \Sp^{h D_4/C_2}, \: X^{\phi \Delta} \in \Sp^{h D_4/\Delta} ] \]
and maps
\[ [\alpha: X^{\phi C_2} \to (X^1)^{tC_2} \in \Sp^{h D_4/C_2}, \: \beta: X^{\phi \Delta} \to (X^1)^{t \Delta} \in  \Sp^{h D_4/\Delta} ]. \]
If we then let $Y$ be the pullback of
\[ \begin{tikzcd}[column sep=4em]
& (X^{\phi C_2})^{t D_4/C_2} \times (X^{\phi \Delta})^{t D_4/\Delta} \ar{d}{\alpha^{tD_4/C_2} \times \beta^{tD_4/\Delta}} \\ 
(X^1)^{\tau D_4} \ar{r}{(\can, \can)} & (X^1)^{t C_2 tD_4/C_2} \times (X^1)^{t \Delta tD_4/\Delta }
\end{tikzcd} \]
and let $\gamma: Y \xto{\pr} (X^1)^{\tau D_4} \xto{\can} ((X^1)^{t \mu_2})^{t C_2}$ (under the canonical isomorphism $C_2 \cong D_4/\mu_2$), then
\[ X^{t_{C_2} \mu_2} = [(X^1)^{t \mu_2} \in \Sp^{h C_2}, \: Y \xto{\gamma} ((X^1)^{t \mu_2})^{t C_2} \in \Sp]. \]
Note also that if $X$ is restricted from a $C_2$-functor $X': B^t_{C_2} \mu_{p^2} \to \underline{\Sp}^{C_2}$, then $X^{\phi C_2} \simeq X^{\phi \Delta}$ and we may write $Y$ instead as
\begin{equation} \label{eq:geometric_fixed_points_formula}
(X^1)^{\tau D_4} \times_{\bigoplus_{\mu_2} (X^1)^{t C_2 t \mu_2}} \bigoplus_{\mu_2} (X^{\phi C_2})^{t \mu_2} 
\end{equation}
where $\bigoplus_{\mu_2}: \Sp \to \Sp^{h \mu_2}$ denotes the $\mu_2$-induction functor.

\item \cite[Exm.~2.53]{QS21a} Suppose $p$ is odd, so up to conjugacy $\Gamma_{\mu_p}$ consists of the subgroups $1$ and $C_2$. Then the data of a $C_2$-functor $X : B^t_{C_2} \mu_p \to \underline{\Sp}^{C_2}$ amounts to a tuple
\[ [X^1 \in \Sp^{h D_{2p}}, \: X^{\phi C_2} \to (X^1)^{tC_2} \in \Sp], \]
and $X^{t_{C_2}\mu_p}$ is given by the tuple $[(X^1)^{t \mu_p} \in \Sp^{h C_2}, \: 0 \in \Sp]$.
\end{enumerate}

\begin{rem}
As we explain below, our techniques also apply to give a new proof of the main result of Nikolaus and Scholze on $p$-cyclotomic spectra (\cite[Thm.~II.6.3]{NS18}), which is in particular independent of the machinery of ``coalgebras for endofunctors'' developed in \cite[II.5]{NS18}. On the other hand, a straightforward adaptation of their setup to the real cyclotomic context would suffice to prove \cref{thm:MainTheoremRestated} from the dihedral Tate orbit lemma (but wouldn't apply to prove \cref{cor:iterated_pullback_descr}). Our main motivation for undertaking this separate line of argument was to give a uniform outlook on the comparison theorems at finite and infinite levels.

In a different direction, we also note that computational techniques around the reconstruction theorem (e.g., \cite[Thm.~F]{QS21a}) play a new and essential role in our proof of the dihedral Tate orbit lemma. So, one cannot logically avoid usage of the reconstruction theorem in any case.
\end{rem}

% \begin{rem}\label{rmk:necessityofparam}
% We pause to highlight some crucial roles of parametrized higher category theory in the dihedral Tate orbit lemma:

% \begin{enumerate}

% \item The statement of the lemma uses the parametrized homotopy orbits and parametrized Tate construction as defined in \cite[\S 4.2]{QS21a}.

% \item Our proof hinged on the identification of the geometric fixed points of the parametrized Tate construction, which we only know how to access using \cite[Cor.~2.49]{QS21a} which relied heavily on parametrized techniques. This is one of the main advantages of working in the parametrized setting, and in fact, all of our attempts to prove the dihedral Tate orbit lemma through a more straightforward adaptation of Nikolaus and Scholze's proof of the Tate orbit lemma \cite[Lem. I.2.1]{NS18} were unsuccessful. 

% \end{enumerate}
% \end{rem}

\subsection{The dihedral Tate orbit lemma} \label{sec:dihedral_tate_orbit_lemma}

In \cite[\S I.2]{NS18}, Nikolaus and Scholze prove the \emph{Tate orbit lemma}: for a Borel $C_{p^2}$-spectrum $X$ that is bounded-below, the spectrum $(X_{h C_p})^{t_{C_p}}$ vanishes \cite[Lem.~I.2.1]{NS18}. In this subsection, we give a dihedral refinement of the Tate orbit lemma (\cref{lem:dihedralTOLEven} for $p=2$ and \cref{lem:dihedralTOLOdd} for $p$ odd). As a corollary, we then deduce that $\TCR^{\mr{gen}}(-,p)$ is computed by the fiber sequence formula for $\TCR(-,p)$ on bounded-below genuine real $p$-cyclotomic spectra (\cref{cor:decategorifiedEasy}).

\begin{dfn}[{\cite[\S 6]{BarwickGlasmanShah}}] The \emph{homotopy $t$-structure} on $\Sp^G$ is the $t$-structure \cite[Def.~1.2.1.1]{HA} determined by the pair of full subcategories $\Sp^G_{\geq 0}$, $\Sp^G_{\leq 0}$ of $G$-spectra $X$ such that $X^H$ is connective, resp. coconnective for all subgroups $H \leq G$.

A $G$-spectrum $X$ is \emph{bounded-below} if $X$ is bounded-below in the homotopy $t$-structure on $\Sp^G$, i.e., for all subgroups $H \leq G$, $X^H$ is bounded-below.
\end{dfn}

\begin{rem}[{\cite[Exm.~6.3]{BarwickGlasmanShah}}] The heart of the homotopy $t$-structure on $\Sp^G$ is the category of abelian group-valued Mackey functors on finite $G$-sets. In addition, the homotopy $t$-structure on $\Sp^G$ is accessible \cite[Def.~1.4.4.12]{HA} and left and right complete \cite[\S 1.2.1]{HA}.
\end{rem}

\begin{dfn} \label{dfn:sliceBoundedBelow} A $G$-spectrum $X$ is \emph{slice bounded-below} if for all subgroups $H \leq G$, $X^{\phi H}$ is bounded-below.
\end{dfn}

\begin{rem} By \cite[Thm.~A]{HillYarnall}, a $G$-spectrum $X$ is slice bounded-below in the sense of \cref{dfn:sliceBoundedBelow} if and only if it is slice $n$-connective for some $n>-\infty$ in the sense of the slice filtration \cite[\S 4]{HHR}.
\end{rem}

 % We recall the following well-known result, which we will use without further comment.
When $G = C_{p^n}$, there is no distinction between bounded-below and slice bounded-below $G$-spectra.

\begin{lem} \label{lem:BddBelowEqualsSliceBddBelow} Suppose $X \in \Sp^{C_{p^n}}$. Then $X$ is bounded-below if and only if $X$ is slice bounded-below.
\end{lem}
\begin{proof} We proceed by induction on $n$. The base case $n=0$ is trivial. Let $n>0$ and suppose we have proven the lemma for $C_{p^{n-1}}$. Let $X \in \Sp^{C_{p^n}}$ and consider the recollement
\[ \begin{tikzcd}[row sep=4ex, column sep=6ex, text height=1.5ex, text depth=0.25ex]
\Fun(B C_{p^n},\Sp) \ar[shift right=1,right hook->]{r}[swap]{j_{\ast}} & \Sp^{C_{p^n}} \ar[shift right=2]{l}[swap]{j^{\ast}} \ar[shift left=2]{r}{i^{\ast} = \Phi^{C_p}} & \Sp^{C_{p^{n-1}}} \simeq \Sp \ar[shift left=1,left hook->]{l}{i_{\ast}}
\end{tikzcd} \]
from which we obtain the fiber sequence $(X^1)_{h C_{p^n}} \to X^{C_{p^n}} \to (\Phi^{C_p} X)^{C_{p^{n-1}}}$ as in \cite[Prop.~II.2.13]{NS18}. By the inductive hypothesis, we may suppose that both $X^{C_{p^k}}$ and $X^{\phi C_{p^k}}$ are bounded-below for all $0 \leq k < n$. Then noting that $(\Phi^{C_p} X)^{\phi C_{p^k}} \simeq X^{\phi C_{p^{k+1}}}$, we deduce from the fiber sequence that $X^{C_{p^n}}$ is bounded-below if and only if $X^{\phi C_{p^n}}$ is bounded-below.
\end{proof}

Note that the restriction functors $\res^G_H: \Sp^G \to \Sp^H$ are $t$-exact with respect to the homotopy $t$-structures. Consequently, we can make the following definition.

\begin{dfn} Let $\underline{\Sp}^G_{\geq n}, \underline{\Sp}^G_{\leq m} \subset \underline{\Sp}^G$ be the full $G$-subcategories defined fiberwise over $G/H$ on objects $X \in \Sp^H_{\geq n}$, resp. $X \in \Sp^H_{\leq m}$.
\end{dfn}

\begin{lem} The inclusions $\underline{\Sp}^G_{\geq n} \subset \underline{\Sp}^G$, resp. $\underline{\Sp}^G_{\leq m} \subset \underline{\Sp}^G$ admit right $G$-adjoints $\tau_{\geq n}$, resp. left $G$-adjoints $\tau_{\leq m}$.
\end{lem}
\begin{proof} These adjunctions exist fiberwise, so we deduce both statements from the $t$-exactness of the restriction and induction functors using \cite[Prop.~7.3.2.6]{HA} and \cite[Prop.~7.2.3.11]{HA}.
\end{proof}

\begin{rem} For a $G$-$\infty$-category $K$, we have an induced `pointwise' $t$-structure on $\Fun_G(K, \ul{\Sp}^G)$ determined by $\Fun_G(K, \ul{\Sp}^G_{\geq 0})$ and $\Fun_G(K, \ul{\Sp}^G_{\leq 0})$.
\end{rem}

% Postnikov limits in $\Sp^G$ commute with $K$-indexed colimits for any $G$-$\infty$-category $K$.
\begin{lem} \label{lem:TateConvergence} Let $K$ be a $G$-$\infty$-category and $f: K \to \underline{\Sp}^G$ a $G$-functor. Then the canonical maps
\[ \begin{tikzcd}[row sep=4ex, column sep=4ex, text height=1.5ex, text depth=0.25ex]
\colim^G_K f \ar{r} & \lim_{n} \colim^G_K \tau_{\leq n} f
\end{tikzcd}, \quad 
\begin{tikzcd}[row sep=4ex, column sep=4ex, text height=1.5ex, text depth=0.25ex]
\colim_{n} \lim^G_K \tau_{\geq -n} f \ar{r} & \lim^G_K f
\end{tikzcd} \]
are equivalences. Consequently, if $X: B_{G/N}^{\psi} N \to \ul{\Sp}^{G/N}$ is a $G/N$-spectrum with $\psi$-twisted $N$-action, then the canonical maps
\[ X^{t_{G/N} N} \to \lim_n (\tau_{\leq n} X)^{t_{G/N} N}, \quad \colim_n (\tau_{\geq -n} X)^{t_{G/N} N} \to X^{t_{G/N} N} \]
are equivalences.
% Consequently, if the norm map $\colim^G_K p \to \lim^G_K p$ is defined, the cofiber commutes with Postnikov limits and colimits.
\end{lem}
\begin{proof} For the first equivalence, using the cofiber sequences $\tau_{>n} \to \id \to \tau_{\leq n}$, it suffices to show that $\lim_{n} \colim^G_K \tau_{> n} f \simeq 0$. But this follows by completeness of the homotopy $t$-structure on $\Sp^G$, since the inclusion $\underline{\Sp}^G_{>n} \subset \underline{\Sp}^G$ preserves $G$-colimits as a left $G$-adjoint \cite[Cor.~8.7]{Exp2}. The second equivalence is proven by a dual argument. The final two equivalences then follow from the first two in view of the defining fiber sequence
\[ X_{h_{G/N} N} \to X^{h_{G/N} N} \to X^{t_{G/N} N} \]
and the commutativity of parametrized orbits with colimits and parametrized fixed points with limits.
\end{proof}

For applying the next lemma, note that we have a canonical lax symmetric monoidal natural transformations $(-)^{t G} \to (-)^{\tau G}$ and $(-)^{t G} \to ((-)^{t N})^{t(G/N)}$, defined via the universal property of the Verdier quotient \cite[\S1.3]{NS18}.

% Then the Tate spectrum $X^{tG}$ is $p$-nilpotent.
% \item Suppose that $X$ is bounded. Then the proper Tate spectrum $X^{\tau G}$ is $p$-nilpotent. 
% \item Suppose that $X$ is bounded and let $N$ be a normal subgroup of $G$. Then the iterated Tate spectrum $(X^{tN})^{t(G/N)}$ is $p$-nilpotent.
% \end{enumerate}
% Now let $F = (-)^{tG}$, $(-)^{\tau G}$, or $((-)^{tN})^{t(G/N)}$ and suppose that $X$ is bounded-below. Then $F(X)$ is $p$-complete, and the map $F(X) \to F(X^{\wedge}_{p})$ is an equivalence.
\begin{lem} \label{lem:TateNilpotent} Let $G$ be a finite $p$-group, $X$ a Borel $G$-spectrum, and $(-)^{t G} \to F(-)$ a lax symmetric monoidal natural transformation.
\begin{enumerate}
\item Suppose that $X$ is bounded. Then $F(X)$ is $p$-nilpotent.
\item Suppose that $X$ is bounded-below. Then $F(X)$ is $p$-complete, and the map $F(X) \to F(X^{\wedge}_{p})$ is an equivalence.
\end{enumerate}
\end{lem}
\begin{proof} First suppose that $F = (-)^{t G}$ itself. Then the same proof as in \cite[Lem.~I.2.9]{NS18} applies: for (1), we reduce to $X = HM$ by induction on the Postnikov tower and use that the order of $G$ annihilates Tate cohomology $\widehat{H}^{\ast}(G;M)$, and (2) then follows, using that $p$-complete spectra are closed under limits. For the general situation, again we may reduce to the case $X = H M$. Then because $F(-)$ is a lax symmetric monoidal functor, $F(HM)$ is a $F(H \ZZ)$-module, so it suffices to show that $F(H \ZZ)$ is $p$-nilpotent. For this, via the lax symmetric monoidal natural transformation $(-)^{t G} \to F(-)$, we obtain an $E_{\infty}$-map $(H \ZZ)^{tG} \to F(H \ZZ)$, and because $(H \ZZ)^{tG}$ is $p$-nilpotent, we deduce that $F(H \ZZ)$ is $p$-nilpotent.
\end{proof}

%Restructure this
% By \cite[Thm.~3.23]{QS21a} and \cite[Thm.~4.28]{QS21a}, we may identify $$(-)^{t_{C_2} \mu_p}: \Fun_{C_2}(B^t_{C_2} \mu_p, \ul{\Sp}^{C_2}) \to \Sp^{C_2}$$ with the gluing functor of the $\Gamma_{\mu_p}$-recollement on $\Sp^{D_{2p}}$.
% We thus obtain the following corollary of \cref{lem:TateNilpotent}.

% Let $X$ be a $C_2$-spectrum with twisted $\mu_q$-action whose underlying spectrum $X^1$ is bounded-below. Then $X$ is $p$-complete if and only if $X^1$, $X^{\phi C_2}$, and $X^{\phi \Delta}$ are $p$-complete.

\begin{lem} \label{lem:pcompletion}
Let $X$ be a $C_2$-spectrum whose underlying spectrum $X^1$ is bounded-below. Then $X$ is $p$-complete if and only if $X^1$ and $X^{\phi C_2}$ are $p$-complete.
\end{lem}
\begin{proof}
In terms of the recollement 
\[ \begin{tikzcd}[row sep=4ex, column sep=6ex, text height=1.5ex, text depth=0.25ex]
\Fun(B C_2,\Sp) \ar[shift right=1,right hook->]{r}[swap]{j_{\ast}} & \Sp^{C_2} \ar[shift right=2]{l}[swap]{j^{\ast}} \ar[shift left=2]{r}{i^{\ast} = \Phi^{C_2}} & \Sp \ar[shift left=1,left hook->]{l}{i_{\ast}}
\end{tikzcd} \]
we will prove that the $p$-completion of $X$ is computed by the factorwise $p$-completion of $X^1$ and $X^{\phi C_2}$. For this, by \cref{lem:recollement_limit} it suffices to show that the restriction of $(-)^{t C_2}$ to the full subcategory of bounded-below Borel $C_2$-spectra commutes with the functor $(-)^{\wedge}_p$ of $p$-completion. We consider two cases:
\begin{enumerate}
\item[(i)] Suppose $p$ is odd. Then $(-)^{t C_2}$ vanishes identically on $p$-complete spectra and $((-)^{tC_2})^{\wedge}_p \simeq 0$ as a module over $(\SS^{t C_2})^{\wedge}_p \simeq 0$. Thus, $(-)^{t C_2}$ commutes with $(-)^{\wedge}_p$ unconditionally.
\item[(ii)] Suppose $p = 2$. Then the restriction of $(-)^{t C_2}$ to the full subcategory of bounded-below Borel $C_2$-spectra inverts $2$-adic equivalences and is valued in $2$-complete spectra, so $(-)^{t C_2}$ commutes with $(-)^{\wedge}_2$.
\end{enumerate}
% For (2), note that for every $C_2$-basepoint $\iota: \sO^{\op}_{C_2} \to B^t_{C_2} \mu_q$, the restriction functor $\iota^*:\Sp^{h_{C_2} \mu_q} \to \Sp^{C_2}$ preserves limits, and ranging over the two distinct $C_2$-basepoints they are jointly conservative. (2) then follows from (1).
\end{proof}

\begin{lem} \label{lem:recollement_limit}
Let $\sX$ be a stable $\infty$-category and $(\sU, \sZ)$ a stable recollement on $\sX$ with gluing functor $\phi: \sU \to \sZ$. Suppose $\overline{p}: K^{\lhd} \to \sX$ is such that $j^* \overline{p}$ and $\phi j^* \overline{p}$ are limit diagrams. Then $\overline{p}$ is a limit diagram if and only if $i^* \overline{p}$ is a limit diagram.
\end{lem}
\begin{proof}
This follows immediately from the recollement fiber sequence $i^! \to i^* \to i^* j_* j^* = \phi j^*$ \cite[Obs.~2.18]{Sha21} and the joint conservativity of $j^*$ and $i^*$. 
\end{proof}

\begin{wrn}
For $X \in \Sp^{C_2}$, it may be that $(X^{\wedge}_2)^{\phi C_2} \not\simeq (X^{\phi C_2})^{\wedge}_2$. For example, one may take $X = j_* E$ for any $E \in \Sp^{h C_2}$ such that $(E^{t C_2})^{\wedge}_2 \not\simeq (E^{\wedge}_2)^{t C_2}$ (such as $E = \mr{KU}$ with trivial $C_2$-action).
\end{wrn}

% Consider the recollement
% \[ \begin{tikzcd}[column sep=4em]
% \Fun(B C_2, \Sp) \ar[hookrightarrow, shift left=2]{r}{j_!} \ar[hookrightarrow, shift right=4]{r}[swap]{j_*} & \Sp^{C_2} \ar[shift left=1]{l}[description]{j^*} \ar[shift left=2]{r}{i^*} \ar[shift right=1, hookleftarrow]{r}[swap, description]{i_*} \ar[shift right=4]{r}[swap]{i^!} & \Sp.
% \end{tikzcd} \]
% and the associated fiber sequence $j_! j^* X = j_!(X^1) \to X \to i_* i^* X = i_* (X^{\phi C_2})$. Since $i_*$ preserves limits, $(i_* X^{\phi C_2})^{\wedge}_p \simeq i_* ((X^{\phi C_2})^{\wedge}_p)$. Since $j^*$ preserves limits, $j^*(X^{\wedge}_p) \simeq (X^1)^{\wedge}_p$ (and the $p$-completion of a Borel $C_2$-spectrum is computed by $p$-completing the underlying spectrum). 

 % Suppose that for all choices of $C_2$-basepoints $\iota: \sO_{C_2}^{\op} \to B^t_{C_2} \mu_p$, $\iota^{\ast} X$ is bounded-below as a $C_2$-spectrum (cf. \cref{rem:DihedralBasepoints}).\footnote{In particular, if $X$ arises as the restriction of a $C_2$-spectrum with twisted $\mu_{p^{\infty}}$-action, then this bounded-below condition is equivalent to stipulating that the underlying $C_2$-spectrum is bounded-below.}
  % (i.e., $S^0/p$-local in $\Sp^{C_2}$)

We may then leverage \cref{lem:TateNilpotent} and \cref{lem:pcompletion} to prove:

\begin{cor} \label{cor:ParamTatePcomplete} Let $X$ be a $C_2$-spectrum with twisted $\mu_p$-action.
\begin{enumerate}
\item Suppose that for all choices of $C_2$-basepoints $\iota: \sO_{C_2}^{\op} \to B^t_{C_2} \mu_p$, $\iota^{\ast} X$ is bounded-below as a $C_2$-spectrum.\footnote{In particular, if $X$ arises as the restriction of a $C_2$-spectrum with twisted $\mu_{p^{\infty}}$-action, then this bounded-below condition is equivalent to stipulating that the underlying $C_2$-spectrum is bounded-below.} Then $X^{t_{C_2} \mu_p}$ is $p$-complete.
\item For all primes $q \neq p$, $(X^{t_{C_2} \mu_p})^{\wedge}_q \simeq 0$. 
\item For all primes $q \neq p$, if $X$ is $q$-local then $X^{t_{C_2} \mu_p} \simeq 0$ .
\end{enumerate}
\end{cor}
\begin{proof} To prove (1), by \cref{lem:pcompletion} it suffices to prove that $(X^{t_{C_2} \mu_p})^1$ and $(X^{t_{C_2} \mu_p})^{\phi C_2}$ are $p$-complete. This follows from the explicit description of these spectra given in \cite[Exm.~2.51 and 2.53]{QS21a} together with \cref{lem:TateNilpotent}. The same formulas also prove (3) since $X$ is $q$-local if and only if its geometric fixed points are $q$-local.

Finally, (2) then follows since $X^{t_{C_2} \mu_p}$ is a module over the $C_2$-spectrum $\SS^{t_{C_2} \mu_p}$, which is $p$-complete by (1) and hence is annihilated by $q$-completion for all primes $q \neq p$.\footnote{In fact, the Segal conjecture implies that $\SS^{t_{C_2} \mu_2} \simeq \SS^{\wedge}_2$; cf. \cite[Exm.~5.35]{QS21a}.}
% Note that a $C_2$-spectrum $E$ is $p$-complete if its geometric fixed points are $p$-complete, by reference to the usual fracture square. The claim then follows from \cref{lem:TateNilpotent} and \cite[Exm.~2.51]{QS21a} (for $p=2$) or \cite[Exm.~2.53]{QS21a} (for $p$ odd).
\end{proof}

We now turn to our dihedral refinement of the Tate orbit lemma. In the proofs of \cref{lem:dihedralTOLEven} and \cref{lem:dihedralTOLOdd}, we let $x = x_{2}$ be a generator for $\mu_{p^2}$ (cf. \cref{setup:Dihedral}).

% [Dihedral Tate orbit lemma for $p=2$] 
\begin{lem} \label{lem:dihedralTOLEven} The functor given by the composite
\[ \begin{tikzcd}[row sep=4ex, column sep=8ex, text height=2ex, text depth=0.75ex]
\Fun_{C_2}(B^t_{C_2} \mu_4, \underline{\Sp}^{C_2}) \ar{r}{(-)_{h_{C_2} \mu_2}} & \Fun_{C_2}(B^t_{C_2} \mu_2, \underline{\Sp}^{C_2}) \ar{r}{(-)^{t_{C_2} \mu_2}} & \Sp^{C_2}
\end{tikzcd} \]
evaluates to $0$ on those objects $X$ such that the underlying spectrum $X^1$ is bounded-below.
\end{lem}

For the proof, we first need the following lemma on $\Phi^{C_2}$ as a $C_2$-functor. Let $\underline{\Sp}^{\Phi C_2} \coloneq [\Sp^{C_2} \to \ast]$ and write $\underline{\Phi}^{C_2}: \underline{\Sp}^{C_2} \to \underline{\Sp}^{\Phi C_2}$ for the $C_2$-functor given by $\Phi^{C_2}$ on the fiber over $C_2/C_2$ and the zero functor on the fiber over $C_2/1$.

\begin{lem} \label{lem:GeometricFixedPointsPreservesParametrizedColimits} The $C_2$-functor $\underline{\Phi}^{C_2}: \underline{\Sp}^{C_2} \to \underline{\Sp}^{\Phi C_2}$ preserves $C_2$-colimits, so for every $C_2$-functor $f: I \to J$, the diagram
\[ \begin{tikzcd}[row sep=4ex, column sep=8ex, text height=1.5ex, text depth=0.5ex]
\Fun_{C_2}(I, \underline{\Sp}^{C_2}) \ar{r}{f_!} \ar{d}{(\underline{\Phi}^{C_2})_*} & \Fun_{C_2}(J,\underline{\Sp}^{C_2}) \ar{d}{(\underline{\Phi}^{C_2})_*} \\
\Fun_{C_2}(I, \underline{\Sp}^{\Phi C_2}) \ar{r}{f_!} \ar{d}{\simeq} & \Fun_{C_2}(J, \underline{\Sp}^{\Phi C_2}) \ar{d}{\simeq} \\
\Fun(I_{C_2/C_2}, \Sp) \ar{r}{(f_{C_2/C_2})_!} & \Fun(J_{C_2/C_2},\Sp)
\end{tikzcd} \]
commutes, where $f_!$ denotes $C_2$-left Kan extension along $f$ and $(f_{C_2/C_2})_!$ denotes left Kan extension along $f_{C_2/C_2}$.
\end{lem}
\begin{proof} By \cite[Rem.~2.23]{QS21a}, $\underline{\Phi}^{C_2}$ is a $C_2$-left adjoint and hence preserves $C_2$-colimits \cite[Cor.~8.7]{Exp2}, so the upper square commutes. By definition, the $C_2$-left Kan extension $f_!$ is left adjoint to restriction along $f$. Since $({\underline{\Sp}^{\Phi C_2}})_{C_2/1} \simeq \ast$, we have the vertical equivalences of the lower square under which restriction along $f$ is identified with restriction along $f_{C_2/C_2}$. This implies the commutativity of the lower square.
\end{proof}

\begin{exm} \label{exm:geometric_fixed_points_and_parametrized_colimits}
Let $X \in \Sp^{h_{C_2} S^1}$. By \cref{lem:GeometricFixedPointsPreservesParametrizedColimits}, $(X_{h_{C_2} \mu_2})^{\phi C_2} \simeq (X^{\phi C_2})_{h \mu_2}$ and $(X_{h_{C_2} \mu_p})^{\phi C_2} \simeq X^{\phi C_2}$ for $p$ odd.
\end{exm}

\begin{rem} \label{rem:UpperShriekPreservesParametrizedLimits}
Consider instead the $C_2$-right adjoint $i^!: \underline{\Sp}^{C_2} \to \underline{\Sp}^{\Phi C_2}$ given as part of the structure of the $C_2$-stable $C_2$-recollement. Then for every $C_2$-functor $f: I \to J$, we have a commutative diagram
\[ \begin{tikzcd}
\Fun_{C_2}(I, \underline{\Sp}^{C_2}) \ar{r}{f_*} \ar{d}{(i^!)_*} & \Fun_{C_2}(J,\underline{\Sp}^{C_2}) \ar{d}{(i^!)_*} \\ 
\Fun(I_{C_2/C_2}, \Sp) \ar{r}{f_*} & \Fun(J_{C_2/C_2},\Sp)
\end{tikzcd} \]
where $f_*$ denotes $C_2$-right Kan extension along $f$. Together with the recollement fiber sequence $i^! \to i^* \to i^* j_* j^*$, this relation permits one to access the geometric fixed points of the parametrized fixed points.

For example, let $X \in \Sp^{h_{C_2} S^1}$, so $i^! X = \fib(X^{\phi C_2} \to (X^e)^{t C_2})$. Then $i^!(X^{h_{C_2} \mu_2}) \simeq (i^! X)^{h \mu_2}$. Moreover, if $X^e = 0$, then $i^! \simeq i^* = \Phi^{C_2}$, so $(X^{h_{C_2} \mu_2})^{\phi C_2} \simeq (X^{\phi C_2})^{h \mu_2}$.
\end{rem}

\begin{proof}[Proof of \cref{lem:dihedralTOLEven}] In view of the compatiblity of the functors with restriction as described by the commutative diagram
\[ \begin{tikzcd}[row sep=4ex, column sep=8ex, text height=2ex, text depth=0.75ex]
\Fun_{C_2}(B^t_{C_2} \mu_4, \underline{\Sp}^{C_2}) \ar{r}{(-)_{h_{C_2} \mu_2}} \ar{d} & \Fun_{C_2}(B^t_{C_2} \mu_2, \underline{\Sp}^{C_2}) \ar{r}{(-)^{t_{C_2} \mu_2}} \ar{d} & \Sp^{C_2} \ar{d} \\
\Fun(B D_8, \Sp) \ar{r}{(-)_{h \mu_2}} & \Fun(B D_4, \Sp) \ar{r}{(-)^{t \mu_2}} & \Fun(B C_2, \Sp) 
\end{tikzcd} \]

we have that $((X_{h_{C_2} \mu_2})^{t_{C_2} \mu_2})^1 \simeq ((X^1)_{h \mu_2})^{t \mu_2}$, which vanishes by the Tate orbit lemma for bounded-below $\mu_4$-Borel spectra \cite[Lem.~I.2.1]{NS18}. Thus, it suffices to show that $((X_{h_{C_2} \mu_2})^{t_{C_2} \mu_2})^{\phi C_2} \simeq 0$. Let $\rho = \rho_{\mu_2}: B^t_{C_2} \mu_4 \to B^t_{C_2} \mu_2$ be as in \cite[Lem.~4.35(2)]{QS21a}, so the $C_2$-orbits functor $(-)_{h_{C_2} \mu_2}$ is $C_2$-left Kan extension along $\rho$. By \cref{lem:GeometricFixedPointsPreservesParametrizedColimits}, we have a commutative diagram
\[ \begin{tikzcd}[row sep=4ex, column sep=8ex, text height=2ex, text depth=0.75ex]
\Fun_{C_2}(B^t_{C_2} \mu_4, \underline{\Sp}^{C_2}) \ar{r}{(-)_{h_{C_2} \mu_2}} \ar{d}{(\underline{\Phi}^{C_2})_*} & \Fun_{C_2}(B^t_{C_2} \mu_2, \underline{\Sp}^{C_2}) \ar{d}{(\underline{\Phi}^{C_2})_*} \\
\Fun( (B^t_{C_2} \mu_4)_{C_2/C_2}, \Sp) \ar{r}{\rho'_!} & \Fun((B^t_{C_2} \mu_2)_{C_2/C_2}, \Sp)
\end{tikzcd} \]
where the bottom horizontal functor is left Kan extension along the restriction $\rho'$ of $\rho$ to the fiber $(B^t_{C_2} \mu_4)_{C_2/C_2}$. Picking $\angs{\sigma}$ and $\angs{\sigma x}$ as representatives of their respective conjugacy classes of subgroups in $D_8$, we have 
\begin{align*}
(B^t_{C_2} \mu_4)_{C_2/C_2} \simeq B W_{D_8} \angs{\sigma} \bigsqcup B W_{D_8} \angs{\sigma x} \simeq B C_2 \sqcup B C_2, \\
(B^t_{C_2} \mu_2)_{C_2/C_2} \simeq B W_{D_4} \angs{\sigma} \bigsqcup B W_{D_4} \angs{\sigma x} \simeq B C_2 \sqcup B C_2.
\end{align*}

Note that $\rho$ sends the generator $x^2 \in W_{D_8} \angs{\sigma} \cong C_2$ to $1 \in W_{D_4} \angs{\sigma} \cong C_2$ and likewise for $W_{D_8} \angs{\sigma x}$. Therefore, $\rho'$ may be identified with the map $BC_2 \bigsqcup BC_2 \to \ast \bigsqcup \ast \to BC_2 \bigsqcup BC_2$, and we see that
\[ (X_{h_{C_2} \mu_2})^{\phi \angs{\sigma x^i}} \simeq \ind^{C_2}((X^{\phi \angs{\sigma x^i}})_{h C_2}), \quad i = 0, 1. \]

As for the functor $(-)^{t_{C_2} \mu_2}$, given $Y \in \Fun_{C_2}(B^t_{C_2} \mu_2, \underline{\Sp}^{C_2})$, by \cite[Exm.~2.51]{QS21a} we have that
\[ (Y^{t_{C_2} \mu_2})^{\phi C_2} \simeq (Y^1)^{\tau D_4} \times_{((Y^1)^{t \angs{\sigma} t C_2} \times (Y^1)^{t \angs{\sigma x} t C_2})} ((Y^{\phi \angs{\sigma}})^{t C_2} \times (Y^{\phi \angs{\sigma x}})^{t C_2}). \]

Using that $(-)^{t C_2}$ vanishes on $C_2$-induced objects, we deduce that
\begin{equation} \label{eq:orbit_lemma}
((X_{h_{C_2} \mu_2})^{t_{C_2} \mu_2})^{\phi C_2} = \fib((X^1_{h \mu_2})^{\tau D_4} \to (X^1_{h \mu_2})^{t \angs{\sigma} tC_2} \times (X^1_{h \mu_2})^{t \angs{\sigma x} tC_2} ).
\end{equation}

Thus, the terms $X^{\phi \angs{\sigma}}$ and $X^{\phi \angs{\sigma x}}$ are irrelevant for the computation, in the sense that the counit map $j_! j^{\ast} X = X \otimes E {D_8}_+ \to X$ for the adjunction
\[ \adjunct{j_!}{\Fun(B D_8, \Sp)}{\Fun_{C_2}(B^t_{C_2} \mu_4, \underline{\Sp}^{C_2})}{j^\ast} \]
is sent to an equivalence under $(((-)_{h_{C_2} \mu_2})^{t_{C_2} \mu_2})^{\phi C_2}$. We may therefore extend our hypothesis that $X^1 = j^{\ast} X$ is bounded-below to further suppose that $X$ is bounded-below with respect to the homotopy $t$-structure on $\Sp^{C_2}$. Then by \cref{lem:TateConvergence}, the cofiber sequence
\[ (-)_{h_{C_2} \mu_4} \to ((-)_{h_{C_2} \mu_2})^{h_{C_2} \mu_2} \to  ((-)_{h_{C_2} \mu_2})^{t_{C_2} \mu_2}, \]
and induction up the Postnikov tower of $X$, we reduce to the case of $X = j_! HM$ for $M$ a $\ZZ[D_8]$-module. Moreover, in view of the fiber sequence \eqref{eq:orbit_lemma} and \cref{lem:TateNilpotent}, $((X_{h_{C_2} \mu_2})^{t_{C_2} \mu_2})^{\phi C_2}$ is $2$-complete. Thus, to show vanishing we may further suppose that $M$ is a $\FF_2[D_8]$-module.

Let us now consider the $\FF_2[D_8]$-free resolution of $M$ from \cite[\S IV.2, p.~129]{AdemMilgram} (and with all signs suppressed since $2=0$), given by taking the total complex of the bicomplex

\[ \begin{tikzcd}[row sep=4ex, column sep=4ex, text height=1.5ex, text depth=0.5ex]
\vdots \ar{d}{\sigma+1} & \vdots \ar{d}{\sigma x +1} & \vdots \ar{d}{\sigma+1} & \vdots \ar{d}{\sigma x +1} &  \\
M[D_8] \ar{d}{\sigma+1} & M[D_8] \ar{l}{x+1} \ar{d}{\sigma x + 1} & M[D_8] \ar{l}{\Sigma_x} \ar{d}{\sigma+1} & M[D_8] \ar{d}{\sigma x +1} \ar{l}{x+1} & \cdots \ar{l}{\Sigma_x} \\
M[D_8] \ar{d}{\sigma+1} & M[D_8] \ar{l}{x+1} \ar{d}{\sigma x + 1} & M[D_8] \ar{l}{\Sigma_x} \ar{d}{\sigma+1} & M[D_8] \ar{d}{\sigma x +1} \ar{l}{x+1} & \cdots \ar{l}{\Sigma_x} \\
M[D_8] \ar{d}{\sigma+1} & M[D_8] \ar{l}{x+1} \ar{d}{\sigma x + 1} & M[D_8] \ar{l}{\Sigma_x} \ar{d}{\sigma+1} & M[D_8] \ar{d}{\sigma x +1} \ar{l}{x+1} & \cdots \ar{l}{\Sigma_x} \\
M[D_8] & M[D_8] \ar{l}{x+1} & M[D_8] \ar{l}{\Sigma_x} & M[D_8] \ar{l}{x+1} & \cdots \ar{l}{\Sigma_x}
\end{tikzcd} \]
where $\Sigma_x = 1+x+x^2+x^3$. Application of the functor $(-)/\mu_2$ to this bicomplex yields the bicomplex of $\FF_2[D_4]$-modules
\[ \begin{tikzcd}[row sep=4ex, column sep=4ex, text height=1.5ex, text depth=0.5ex]
\vdots \ar{d}{\sigma+1} & \vdots \ar{d}{\sigma x +1} & \vdots \ar{d}{\sigma+1} & \vdots \ar{d}{\sigma x +1} &  \\
M[D_4] \ar{d}{\sigma+1} & M[D_4] \ar{l}{x+1} \ar{d}{\sigma x + 1} & M[D_4] \ar{l}{0} \ar{d}{\sigma+1} & M[D_4] \ar{d}{\sigma x +1} \ar{l}{x+1} & \cdots \ar{l}{0} \\
M[D_4] \ar{d}{\sigma+1} & M[D_4] \ar{l}{x+1} \ar{d}{\sigma x + 1} & M[D_4] \ar{l}{0} \ar{d}{\sigma+1} & M[D_4] \ar{d}{\sigma x +1} \ar{l}{x+1} & \cdots \ar{l}{0} \\
M[D_4] \ar{d}{\sigma+1} & M[D_4] \ar{l}{x+1} \ar{d}{\sigma x + 1} & M[D_4] \ar{l}{0} \ar{d}{\sigma+1} & M[D_4] \ar{d}{\sigma x +1} \ar{l}{x+1} & \cdots \ar{l}{0} \\
M[D_4] & M[D_4] \ar{l}{x+1} & M[D_4] \ar{l}{0} & M[D_4] \ar{l}{x+1} & \cdots \ar{l}{0}
\end{tikzcd} \]
whose total complex is quasi-isomorphic to $M_{h \mu_2}$ in the derived category of $\FF_2[D_4]$ (crucially, we use that $2=0$ to see that $(\Sigma_x)/\mu_2 = 0$). Let $F^n(M_{\mu_2})$ be the total complex obtained by truncating the bicomplex to the first $2n$ columns, viewed in the derived category. Because of the zero maps that appear horizontally in the bicomplex, we have retractions $r_n: M_{\mu_2} \to F^n(M_{\mu_2})$ splitting the natural inclusions such that
\begin{enumerate}
    \item The induced map $M_{h \mu_2} \to \lim_{n} F^n(M_{h \mu_2})$ is an equivalence.
    \item The connectivity of the fiber of $M_{h \mu_2} \to F^n(M_{h \mu_2})$ goes to $\infty$ as $n \to \infty$. 
\end{enumerate}

% Because $(j_! M)_{h_{C_2} \mu_2} \simeq j_! M_{h \mu_2}$ viewed as an object in $\Sp^{C_2}$ (forgetting the action) is a Borel torsion object, we have that $((j_! M)_{h_{C_2} \mu_2})^{C_2} \simeq M_{h \mu_4}$, so the same connectivity hypotheses apply to $(j_! M)_{h_{C_2} \mu_2} \to F^n(M_{h \mu_2})$ with respect to the homotopy t-structure on $\Sp^{C_2}$.
Moreover, in view of the commutative diagram
\[ \begin{tikzcd}[row sep=4ex, column sep=4ex, text height=1.5ex, text depth=0.25ex]
\Fun(B D_8, \Sp) \ar{r}{j_!} \ar{d}{(-)_{h \mu_2}} & \Fun_{C_2}(B^t_{C_2} \mu_4, \underline{\Sp}^{C_2}) \ar{d}{(-)_{h_{C_2} \mu_2}} \\
\Fun(B D_4, \Sp) \ar{r}{j_!} & \Fun_{C_2}(B^t_{C_2} \mu_2, \underline{\Sp}^{C_2})
\end{tikzcd} \]
we obtain a filtration $j_! (F^n(M_{\mu_2})$ of $(j_! M)_{h_{C_2} \mu_2} \simeq j_!(M_{h \mu_2})$ such that
\begin{enumerate}
    \item The induced map $j_! (M_{h \mu_2}) \to \lim_{n} j_! (F^n(M_{h \mu_2}))$ is an equivalence. For this, to commute $j_!$ past the inverse limit we use that
    \[ (\lim_{n} F^n(M_{h \mu_2}))^{t \mu_2} \simeq \lim_n F^n(M_{h \mu_2})^{t \mu_2} \]
    in view of the increasing connectivity of the fibers.
    \item The $C_2$-connectivity of the fiber of $j_! (M_{h \mu_2}) \to j_! (F^n(M_{h \mu_2}))$ goes to $\infty$ as $n \to \infty$.\footnote{A priori, when considering $C_2$-connectivity of the underlying object in $\Sp^{C_2}$ of a $C_2$-functor $B^t_{C_2} \mu_2 \to \underline{\Sp}^{C_2}$, we must consider all $C_2$-basepoints of $B^t_{C_2} \mu_2$. However, because the objects in question are Borel-torsion, any choice of $C_2$-basepoint yields the same object.} For this, note that for Borel-torsion objects $E \in \Sp^{C_2}$, $E^{C_2} \simeq E_{h C_2}$ since $E^{\phi C_2} \simeq 0$, so the connectivity of $E^{C_2}$ is bounded-below by that of $E$.
\end{enumerate}

Now by \cref{lem:TateConvergence} applied to $(-)^{t_{C_2} \mu_2}$, in order to show that $(((j_! M)_{h_{C_2} \mu_2})^{t_{C_2} \mu_2})^{\phi C_2} \simeq 0$ it suffices to consider the vanishing of the functor $((j_!(-))^{t_{C_2} \mu_2})^{\phi C_2}$ on the filtered quotients $F^{n+1}/F^n(M_{\mu_2})$. For this, we observe that the alternating vertical columns of the bicomplex are free resolutions of $M[D_4/\angs{\sigma}]$ and $M[D_4/\angs{\sigma x}]$, respectively. Therefore, the filtered quotients $F^{n+1}/F^n(M_{\mu_2})$ are extensions of objects induced from proper subgroups of $D_4$, and are thus annihilated by $(-)^{\tau D_4}$, $(-)^{t \angs{\sigma} tC_2}$, and $(-)^{t \angs{\sigma x} tC_2}$.
\end{proof}

In contrast to \cref{lem:dihedralTOLEven}, the proof of the dihedral Tate orbit lemma at an odd prime is far simpler.

% [Dihedral Tate orbit lemma for $p$ odd]
\begin{lem} \label{lem:dihedralTOLOdd} Let $p$ be an odd prime. The functor given by the composite
\[ \begin{tikzcd}[row sep=4ex, column sep=8ex, text height=2ex, text depth=0.75ex]
\Fun_{C_2}(B^t_{C_2} \mu_{p^2}, \underline{\Sp}^{C_2}) \ar{r}{(-)_{h_{C_2} \mu_p}} & \Fun_{C_2}(B^t_{C_2} \mu_p, \underline{\Sp}^{C_2}) \ar{r}{(-)^{t_{C_2} \mu_p}} & \Sp^{C_2}
\end{tikzcd} \]
evaluates to $0$ on those objects $X$ such that the underlying spectrum $X^1$ is bounded-below.
\end{lem}
\begin{proof} As in the proof of \cref{lem:dihedralTOLEven}, one has the commutative diagram
\[ \begin{tikzcd}[row sep=4ex, column sep=8ex, text height=2ex, text depth=0.75ex]
\Fun_{C_2}(B^t_{C_2} \mu_{p^2}, \underline{\Sp}^{C_2}) \ar{r}{(-)_{h_{C_2} \mu_p}} \ar{d} & \Fun_{C_2}(B^t_{C_2} \mu_p, \underline{\Sp}^{C_2}) \ar{r}{(-)^{t_{C_2} \mu_p}} \ar{d} & \Sp^{C_2} \ar{d} \\
\Fun(B D_{2 p^2}, \Sp) \ar{r}{(-)_{h \mu_p}} & \Fun(B D_{p^2}, \Sp) \ar{r}{(-)^{t \mu_p}} & \Fun(B C_2, \Sp) ,
\end{tikzcd} \]

so $((X_{h_{C_2} \mu_p})^{t_{C_2} \mu_p})^1 \simeq ((X^1)_{h \mu_p})^{t \mu_p}$, which vanishes by the Tate orbit lemma for bounded-below $\mu_{p^2}$-Borel spectra \cite[Lem.~I.2.1]{NS18}. Then by \cite[Exm.~2.53]{QS21a},  we have that $((Y)^{t_{C_2} \mu_p})^{\phi C_2} \simeq 0$ for all $Y$ and thus $((X_{h_{C_2} \mu_p})^{t_{C_2} \mu_p})^{\phi C_2} \simeq 0$ unconditionally.
\end{proof}

% \begin{rem} When $p$ is odd, the restriction of the $C_2$-functor $B^t_{C_2} \mu_{p^2} \to B^t_{C_2} \mu_{p}$ to the fiber over $C_2/C_2$ is equivalent to the trivial map $\ast \to \ast$, so by \cref{lem:GeometricFixedPointsPreservesParametrizedColimits}, for $X \in \Fun_{C_2}(B^t_{C_2} \mu_{p^2}, \underline{\Sp}^{C_2})$ we have that $(X_{h_{C_2} \mu_p})^{\phi \angs{\sigma}} \simeq X^{\phi \angs{\sigma}}$.
% \end{rem}

We now prove a few corollaries of the dihedral Tate orbit lemma. These results are all obvious analogs of those in \cite[\S II.4]{NS18}.

\begin{lem} \label{lem:BoundedBelowEquivalence} Suppose $X$ is a $C_2$-spectrum with twisted $\mu_{p^n}$-action whose underlying spectrum $X^1$ is bounded-below. Then the canonical map of \cite[Prop.~4.36]{QS21a}
$$ X^{t_{C_2} \mu_{p^n}} \to (X^{t_{C_2} \mu_p})^{h_{C_2} \mu_{p^{n-1}}} $$
is an equivalence of $C_2$-spectra.
\end{lem}
\begin{proof} We mimic the proof of \cite[Lem.~II.4.1]{NS18}. Note that $X_{h_{C_2} \mu_{p^{n-1}}}$ has bounded-below underlying spectrum $(X^1)_{h \mu_{p^{n-1}}}$. By the dihedral Tate orbit lemma, we see that the norm map
\[ X_{h_{C_2} \mu_{p^n}} \simeq (X_{h_{C_2} \mu_{p^{n-1}}})_{h_{C_2} \mu_p} \to (X_{h_{C_2} \mu_{p^{n-1}}})^{h_{C_2} \mu_p} \]
is an equivalence. By induction, it follows that the norm map
\[ X_{h_{C_2} \mu_{p^n}} \to (X_{h_{C_2} \mu_p})^{h_{C_2} \mu_{p^{n-1}}} \]
is an equivalence. Therefore, the left and middle vertical maps in the commutative diagram
\[ \begin{tikzcd}[row sep=4ex, column sep=4ex, text height=1.5ex, text depth=0.25ex]
X_{h_{C_2} \mu_{p^n}} \ar{r} \ar{d} & X^{h_{C_2} \mu_{p^n}} \ar{r} \ar{d} & X^{t_{C_2} \mu_{p^n}} \ar{d} \\
(X_{h_{C_2} \mu_p})^{h_{C_2} \mu_{p^{n-1}}} \ar{r} & (X^{h_{C_2} \mu_p})^{h_{C_2} \mu_{p^{n-1}}} \ar{r} & (X^{t_{C_2} \mu_p})^{h_{C_2} \mu_{p^{n-1}}}
\end{tikzcd} \]
are equivalences, so the right vertical map is also an equivalence.
\end{proof}

\begin{dfn} For $X \in \Fun_{C_2}(B^t_{C_2} \mu_{p^{\infty}}, \ul{\Sp}^{C_2})$, let
$$ X^{t_{C_2} \mu_{p^{\infty}}} := \lim_n X^{t_{C_2} \mu_{p^{n}}} $$
where the inverse limit is taken along the maps
$$ X^{t_{C_2} \mu_{p^{n}}} \to (X^{t_{C_2} \mu_{p^{n-1}}})^{h_{C_2} \mu_p} \to X^{t_{C_2} \mu_{p^{n-1}}}. $$
\end{dfn}

% Note that for $m \leq n$, the canonical map $X^{t_{C_2} \mu_{p^n}} \to X^{t_{C_2} \mu_{p^m}}$ is twisted $(\mu_{p^n}/\mu_{p^m})$-equivariant with respect to the trivial action on the source and residual action on the target. Passing to the inverse limit, this yields a canonical map
% $$ X^{t_{C_2} \mu_{p^{\infty}}} \to (X^{t_{C_2} \mu_p})^{h_{C_2} \mu_{p^{\infty}}}.$$

\begin{cor} \label{cor:TateToFixedPointsEquivalence} Suppose that $X$ is a $C_2$-spectrum with twisted $\mu_{p^{\infty}}$-action whose underlying spectrum is bounded-below. Then the canonical map $$X^{t_{C_2} \mu_{p^{\infty}}} \to (X^{t_{C_2} \mu_p})^{h_{C_2} \mu_{p^{\infty}}}$$ is an equivalence.
\end{cor}
\begin{proof} The map in question is the inverse limit of the maps $X^{t_{C_2} \mu_{p^n}} \to (X^{t_{C_2} \mu_p})^{h_{C_2} \mu_{p^{n-1}}}$, which by \cref{lem:BoundedBelowEquivalence} are equivalences under our assumption on $X$.
\end{proof}

\begin{rem} \label{rem:Tate_pcomplete} Using \cref{cor:ParamTatePcomplete}(1), if we further suppose that the underlying $C_2$-spectrum of $X$ is bounded-below, then $(X^{t_{C_2} \mu_p})^{h_{C_2} \mu_{p^n}}$ is $p$-complete for all $0 \leq n \leq \infty$, and hence $X^{t_{C_2} \mu_{p^n}}$ is also $p$-complete for all $1 \leq n \leq \infty$.
\end{rem}

The following lemma extends \cite[II.4.5-7]{NS18}.

\begin{lem} \label{lm:SameLemmaNS} Suppose $X$ is a $D_{2p^n}$-spectrum.
\begin{enumerate} \item We have a natural pullback square of $C_2$-spectra
\[ \begin{tikzcd}[row sep=4ex, column sep=4ex, text height=1.5ex, text depth=0.25ex]
\Psi^{\mu_{p^n}} X \ar{r} \ar{d} & \Psi^{\mu_{p^{n-1}}} (\Phi^{\mu_p} X) \ar{d} \\
X^{h_{C_2} \mu_{p^n}} \ar{r} & X^{t_{C_2} \mu_{p^n}}.
\end{tikzcd} \]
\item Suppose in addition that the underlying spectrum of $X$ is bounded-below. Then we have a natural pullback square of $C_2$-spectra
\[ \begin{tikzcd}[row sep=4ex, column sep=4ex, text height=1.5ex, text depth=0.25ex]
\Psi^{\mu_{p^n}} X \ar{r} \ar{d} & \Psi^{\mu_{p^{n-1}}} (\Phi^{\mu_p} X) \ar{d} \\
X^{h_{C_2} \mu_{p^n}} \ar{r} & (X^{t_{C_2} \mu_{p}})^{h_{C_2} \mu_{p^{n-1}}}.
\end{tikzcd} \] 
\item Suppose in addition that the underlying spectra of
\[ X, \Phi^{\mu_p} X, \Phi^{\mu_{p^2}} X, \cdots, \Phi^{\mu_{p^{n-1}}} X \]
are all bounded-below. Then we have a natural limit diagram of $C_2$-spectra
\[ \begin{tikzcd}[row sep=4ex, column sep=4ex, text height=1.5ex, text depth=0.25ex]
\Psi^{\mu_{p^n}} X \ar{rrrr} \ar{dddd} & & & & \Phi^{\mu_{p^n}} X \ar{d} \\
 & & & (\Phi^{\mu_{p^{n-1}}} X)^{h_{C_2} \mu_p} \ar{r} \ar{d}  &  (\Phi^{\mu_{p^{n-1}}} X)^{t_{C_2} \mu_p} \\
 & & (\Phi^{\mu_{p^{2}}} X)^{h_{C_2} \mu_{p^{n-2}}} \ar{r} \ar{d} & \cdots \\
 & (\Phi^{\mu_p} X)^{h_{C_2} \mu_{p^{n-1}}} \ar{r} \ar{d} & ((\Phi^{\mu_p} X)^{t_{C_2} \mu_p})^{h_{C_2} \mu_{p^{n-2}}} \\
X^{h_{C_2} \mu_{p^n}} \ar{r} & (X^{t_{C_2} \mu_p})^{h_{C_2} \mu_{p^{n-1}}}
\end{tikzcd} \]
\end{enumerate}
\end{lem}
\begin{proof} In view of \cite[Thm.~4.28]{QS21a}, the first pullback square arises from applying $\Psi^{\mu_{p^n}}$ to the fracture square for the $\Gamma_{\mu_{p^n}}$-recollement on $\Sp^{D_{2p^n}}$. The second pullback square then follows by \cref{lem:BoundedBelowEquivalence}, and the last limit diagram follows by induction on $n$.
\end{proof}

We may now equate the fiber sequences for $\TCR(-,p)$ (\cref{prp:TCRfiberSequence}) and $\TCR^{\mr{gen}}(-,p)$ (\cref{prp:fiberSequenceGenuineRealCyc}) in the bounded-below situation, giving a direct proof of \cref{cor:TCRFormulasEquivalent}.

%(\cref{cor:TCRFormulasEquivalent})
\begin{cor} \label{cor:decategorifiedEasy} Let $X$ be a genuine real $p$-cyclotomic spectrum whose underlying spectrum is bounded-below. Then there is a canonical and natural fiber sequence
\[ \TCR^{\mr{gen}}(X,p) \to X^{h_{C_2} \mu_{p^{\infty}}} \xtolong{\varphi^{h_{C_2} \mu_{p^{\infty}}} - \can}{2} (X^{t_{C_2} \mu_p})^{h_{C_2} \mu_{p^{\infty}}} \]
and thus an equivalence $\TCR^{\mr{gen}}(X,p) \simeq \TCR(\sU_{\RR} X,p)$.
\end{cor}
\begin{proof} Using \cref{lm:SameLemmaNS}, we may transcribe the proof of \cite[Thm.~II.4.10]{NS18} into the $C_2$-parametrized setting to prove the claim, with no change of detail.
\end{proof}

\subsection{The comparison at finite level}

\begin{dfn} Let $\Sp^{C_{p^n}}_{bb} \subset \Sp^{C_{p^n}}$ be the full subcategory on those $C_{p^n}$-spectra that are bounded-below, or equivalently, slice bounded-below (\cref{lem:BddBelowEqualsSliceBddBelow}).

Let $\Sp^{D_{2p^n}}_{ubb} \subset \Sp^{D_{2p^n}}$ be the full subcategory on those $D_{2p^n}$-spectra whose underlying $\mu_{p^n}$-spectrum is bounded-below.
\end{dfn}

In this subsection, we give an iterated pullback decomposition of $\Sp^{D_{2p^n}}_{ubb}$ (\cref{cor:iterated_pullback_descr}) that categorifies the staircase limit diagram in \cref{lm:SameLemmaNS}. We also take the opportunity to give a similar iterated pullback decomposition of $\Sp^{C_{p^n}}_{bb}$, along the lines described in \cite[Rem.~II.4.8]{NS18}.\footnote{We write $C_{p^n}$ instead of $\mu_{p^n}$ here in adherence to \cite{NS18}.} Our main tool in achieving this will be \cite[Thm.~4.15]{Sha21} in conjunction with \cite[Thm.~2.46]{QS21a}. We first recall the theorem of Ayala, Mazel-Gee, and Rozenblyum on reconstructing a $G$-spectrum from its geometric fixed points, referring the reader to \cite[\S 2.3]{QS21a} or \cite{AMGRb} for a detailed treatment.

\begin{rec}[Reconstruction theorem for {$G$}-spectra] \label{rec:reconstruction}
Let $G$ be a finite group and let $\fS[G]$ denote its subconjugacy category (\cite[Def.~2.14]{QS21a}), which is the nerve of a preordered set. We write $\fS$ if $G$ is understood. For a subgroup $H \leq G$, let $\sF_{\leq H}$ and $\sF_{< H}$ denote the $G$-family of subgroups that are subconjugate, resp. properly subconjugate to $H$. The \emph{$G$-geometric locus} (\cite[Def.~2.37]{QS21a})
\[ \Sp^{G}_{\locus{\phi}} \subset \Sp^G \times \fS \]
is the full subcategory on objects $(X,H)$ such that $X$ is both $\sF_{\leq H}$-complete and $(\sF_{<H})^{-1}$-local.\footnote{See \cite[Constr.~2.17]{QS21a} or \cite{MATHEW2017994} for this terminology.} We then have that the projection $p: \Sp^{G}_{\locus{\phi}} \to \fS$ is a locally cocartesian fibration such that $(\Sp^{G}_{\locus{\phi}})_H \simeq \Fun(B W_G H, \Sp)$ and over a subconjugacy relation $H \leq K$, the associated pushforward functor is the \emph{generalized Tate construction} $\tau^K_H$. Consider now the barycentric subdivision $\sd(\fS)$ as a locally cocartesian fibration over $\fS$ via the functor that takes the maximum. By \cite[Thm.~2.46]{QS21a}, we then have that the functor
\begin{equation} \label{eqn:reconstruction}
\Theta: \Fun^{\cocart}_{/\fS}(\sd(\fS),  \Sp^{G}_{\locus{\phi}}) \xto{\simeq} \Sp^G,
\end{equation}
given by postcomposing with the projection $\Sp^{G}_{\locus{\phi}} \to \Sp^G$ and taking the limit, is an equivalence.
\end{rec}

For our intended application, we will need a \emph{relative} version of the geometric locus construction. For a $G$-family $\cF$, let $\Sp^{h \sF}$ and $\Sp^{\Phi \sF}$ denote the full subcategories of $\Sp^G$ on the $\cF$-complete and $\cF^{-1}$-local objects, respectively (\cite[Notn.~2.18]{QS21a}).

\begin{dfn} \label{dfn:relativeGeometricLocus} Suppose $\pi: \fS \to \Delta^n$ is a surjective functor. For $0 \leq k \leq n$, let $\fS_{\pi \leq k}, \fS_{\pi<k} \subset \fS$ be the sieves containing subgroups $H$ such that $\pi(H) \leq k$, resp. $\pi(H) < k$ (by convention, $\fS_{\pi<0} = \emptyset$). Let
$$L[\pi]_k: \Sp^G \to \Sp^{h \fS_{\pi \leq k}} \cap \Sp^{\Phi \fS_{\pi<k}}$$
denote the localization functors.

Define the \emph{$\pi$-relative geometric locus} $\Sp^G_{\locus{\phi,\pi}} \subset \Sp^G \times \Delta^n$ be the full subcategory on $(X,k)$ such that $X$ is $\fS_{\pi \leq k}$-complete and $\fS_{\pi<k}^{-1}$-local. For $0\leq i \leq j \leq n$, define the \emph{$\pi$-relative generalized Tate construction} to be the composite
\[ \tau[\pi]^j_i: \Sp^{h \fS_{\pi \leq i}} \cap \Sp^{\Phi \fS_{\pi<i}} \into \Sp^G \to \Sp^{h \fS_{\pi \leq j}} \cap \Sp^{\Phi \fS_{\pi<j}} \]
of the inclusion and localization functors.

Let $[i:j] \subset \Delta^n$ denote the full subcategory on the vertices $i$ through $j$, so $\Delta^{j-i} \cong [i:j]$. Define the comparison functor
\[ \Theta'[\pi]_{i,j}: \Fun^{\cocart}_{/[i:j]}(\sd([i:j]), [i:j] \times_{\Delta^n} \Sp^G_{\locus{\phi,\pi}} ) \to \Fun(\sd([i:j]), \Sp^G) \xto{\lim} \Sp^G. \]
As in \cite[Constr.~2.43]{QS21a}, the essential image of $\Theta'[\pi]_{i,j}$ lies in $\Sp^{h \fS_{\pi \leq j}} \cap \Sp^{\Phi \fS_{\pi < i}}$. Let $\Theta[\pi]_{i,j}$ denote the comparison functor with this codomain, and also write $\Theta[\pi] = \Theta[\pi]_{0,n}$.
\end{dfn}

\begin{vrn} \label{vrn:ReconstructionEquivalence} We have the following variants of the results in \cite[\S 2.3]{QS21a}. 
\begin{enumerate}
\item $\Sp^G_{\locus{\phi,\pi}} \to \Delta^n$ is a locally cocartesian fibration such that the pushforward functors are given by $\tau[\pi]^j_i$.
\item For all $0 \leq i \leq j \leq n$, 
$$\Theta[\pi]_{i,j}: \Fun^{\cocart}_{/[i:j]}(\sd([i:j]), [i:j] \times_{\Delta^n} \Sp^G_{\locus{\phi,\pi}} ) \xto{\simeq} \Sp^{h \fS_{\pi \leq j}} \cap \Sp^{\Phi \fS_{\pi < i}}$$
is an equivalence of $\infty$-categories.
\item Let $0 \leq i \leq j <k \leq n$, so $[i:j]$, $[j+1:k]$ is a sieve-cosieve decomposition of $[i:k]$. Then we have a strict morphism of stable recollements through equivalences
\[ \begin{tikzcd}[row sep=4ex, column sep=8ex, text height=1.5ex, text depth=0.5ex]
\Fun^{\cocart}_{/[i:j]}(\sd([i:j]), [i:j] \times_{\Delta^n} \Sp^G_{\locus{\phi,\pi}}) \ar[shift left=4, left hook->]{d} \ar[shift right=4, left hook->]{d} \ar{r}{\Theta[\pi]_{i,j}}[swap]{\simeq} & \Sp^{h \fS_{\pi \leq j}} \cap \Sp^{\Phi \fS_{\pi < i}} \ar[shift left=4, left hook->]{d} \ar[shift right=4, left hook->]{d} \\
\Fun^{\cocart}_{/[i:k]}(\sd([i:k]), [i:k] \times_{\Delta^n} \Sp^G_{\locus{\phi,\pi}}) \ar{u} \ar{d} \ar{r}{\Theta[\pi]_{i,k}}[swap]{\simeq} & \Sp^{h \fS_{\pi \leq k}} \cap \Sp^{\Phi \fS_{\pi < i}}  \ar{u} \ar{d} \\
\Fun^{\cocart}_{/[j+1:k]}(\sd([j+1:k]), [j+1:k] \times_{\Delta^n} \Sp^G_{\locus{\phi,\pi}} )  \ar{r}{\Theta[\pi]_{j+1,k}}[swap]{\simeq} \ar[shift right=4, right hook->]{u} & \Sp^{h \fS_{\pi \leq k}} \cap \Sp^{\Phi \fS_{\pi \leq j}} \ar[shift right=4, right hook->]{u}
\end{tikzcd} \]
In particular, for any $i \leq l \leq j$, the composite
\[ \Fun^{\cocart}_{/[i:j]}(\sd([i:j]), [i:j] \times_{\Delta^n} \Sp^G_{\locus{\phi,\pi}}) \to \Sp^{h \fS_{\pi \leq j}} \cap \Sp^{\Phi \fS_{\pi < i}} \to \Sp^{h \fS_{\pi \leq l}} \cap \Sp^{\Phi \fS_{\pi < l}} \]
is homotopic to evaluation at $l \in [i:j] \subset \sd[i:j]$.
\end{enumerate}
\end{vrn}

% Specializing to the situation of interest, we have the following definition, which exploits a key self-similarity property of the dihedral groups.

% \begin{dfn} \label{def_zeta_map} Given a subgroup $H \subset D_{2p^n}$, let $\zeta(H) \geq 0$ be the integer such that $H \cap \mu_{p^n} = \mu_{p^{\zeta(H)}}$, and let $\zeta: \fS[D_{2 p^n}] \to \Delta^n$ denote the resulting map. Note that if $H$ is subconjugate to $K$, then $\zeta(H) \leq \zeta(K)$, so $\zeta$ defines a functor.
% \end{dfn}

Specializing to the situation of interest with $G = D_{2p^n}$, we have:

\begin{dfn} \label{def_zeta_map}
Let $\zeta: \fS[D_{2 p^n}] \to \fS[\mu_{p^n}] \cong \Delta^n$ be given by $\zeta(H) = H \cap \mu_{p^n}$.\footnote{Under the isomorphism $\fS[\mu_{p^n}] \cong \Delta^n$, $\zeta(H)$ is the unique integer such that $H \cap \mu_{p^n} = \mu_{p^{\zeta(H)}}$.}
\end{dfn}

Note that for all $0 \leq k \leq n$, the cosieve $\fS_{\zeta \geq k}$ equals $\fS[D_{2p^n}]_{\geq \mu_{p^k}}$, so
$$\Sp^{\Phi \fS_{\zeta<k}} \simeq \Sp^{D_{2p^n}/\mu_{p^k}} \simeq \Sp^{D_{2p^{n-k}}}$$
and then
\begin{equation} \label{eq:stratum_equivalence}
\Fun_{C_2}(B^t_{C_2} \mu_{p^{n-k}}, \underline{\Sp}^{C_2}) \simeq  \Sp^{h \fS_{\zeta \leq k}} \cap \Sp^{\Phi \fS_{\zeta < k}},
\end{equation}
in view of the recollement \eqref{eq:dihedral_recollement}.
% \[ \begin{tikzcd}[row sep=4ex, column sep=10ex, text height=1.5ex, text depth=0.5ex]
% \Fun_{C_2}(B^t_{C_2} \mu_{p^m}, \underline{\Sp}^{C_2}) \ar[shift right=1,right hook->]{r}[swap]{j_{\ast} = \sF^{\vee}_b[\mu_{p^m}]} & \Sp^{D_{2p^m}} \ar[shift right=2]{l}[swap]{j^{\ast} = \sU_b[\mu_{p^m}]} \ar[shift left=2]{r}{i^{\ast} = \Phi^{\mu_p}} & \Sp^{D_{2p^{m-1}}} \ar[shift left=1,left hook->]{l}{i_{\ast}}
% \end{tikzcd} \]
% specified by the $\mu_{p^m}$-free $D_{2p^m}$-family. See \cite[Exm.~3.26]{QS21a}.
\begin{ntn} \label{eq:dihedral_localization}
We write
\begin{equation*} 
 L[\zeta]_k: \Sp^{D_{2 p^n}} \to \Fun_{C_2}(B^t_{C_2} \mu_{p^{n-k}}, \underline{\Sp}^{C_2})
 \end{equation*}
for the composite of $\Phi^{\mu_{p^k}}: \Sp^{D_{2 p^n}} \to \Sp^{D_{2 p^{n-k}}}$ and $\sU_b[\mu_{p^{n-k}}]$.\footnote{This identifies with the functor $L[\zeta]_k$ of \cref{dfn:relativeGeometricLocus} under the equivalence \eqref{eq:stratum_equivalence}.}
\end{ntn}

\begin{dfn} The \emph{$C_2$-generalized Tate functors}
\[ \tau_{C_2} \mu_{p^k}: \Fun_{C_2}(B^t_{C_2} \mu_{p^n}, \underline{\Sp}^{C_2}) \to \Fun_{C_2}(B^t_{C_2} \mu_{p^{n-k}}, \underline{\Sp}^{C_2}) \]
are defined as in \cref{dfn:relativeGeometricLocus} with respect to $\zeta$ under the equivalence \eqref{eq:stratum_equivalence}.
\end{dfn}

\begin{rem}
The $C_2$-generalized Tate functors are compatible with restriction in the sense that for all $k \leq m \leq n$, the diagram
\[ \begin{tikzcd}
\Fun_{C_2}(B^t_{C_2} \mu_{p^n}, \underline{\Sp}^{C_2}) \ar{r}{\tau_{C_2} \mu_{p^k}} \ar{d}{\res} & \Fun_{C_2}(B^t_{C_2} \mu_{p^{n-k}}, \underline{\Sp}^{C_2}) \ar{d}{\res} \\
\Fun_{C_2}(B^t_{C_2} \mu_{p^m}, \underline{\Sp}^{C_2}) \ar{r}{\tau_{C_2} \mu_{p^k}} & \Fun_{C_2}(B^t_{C_2} \mu_{p^{m-k}}, \underline{\Sp}^{C_2})
\end{tikzcd} \]
commutes. See \cite[Rem.~2.39]{QS21a} for the proof of the same property for the ordinary generalized Tate functors; the reasoning in this case is the same.
\end{rem}

To apply the (dihedral) Tate orbit lemma, we need to re-express the functors $\tau C_{p^n}$ and $\tau_{C_2} \mu_{p^n}$ in terms of more familiar functors. For expositional purposes, we deal with these cases separately, although the statement for $\tau_{C_2} \mu_{p^n}$ logically implies that for $\tau C_{p^n}$. We note at the outset that we already know $\tau C_p \simeq t C_p$ and $\tau_{C_2} \mu_p \simeq t_{C_2} \mu_p$.

\begin{lem} \label{lm:identifyGenTate} Suppose $X$ is a Borel $C_{p^n}$-spectrum. For $1 < k \leq n$, we have $C_{p^{n-k}}$-equivariant equivalences
\[ X^{\tau C_{p^k}} \simeq X^{h C_p \tau C_{p^{k-1}}} \simeq X^{h C_{p^2} \tau C_{p^{k-2}}} \simeq \cdots \simeq X^{h C_{p^{k-1}} t C_p} \]
with respect to which the canonical map $X^{\tau C_{p^k}} \to X^{t C_p \tau C_{p^{k-1}}}$ fits into the fiber sequence
\[ (X_{h C_p})^{\tau C_{p^{k-1}}} \to X^{h C_p \tau C_{p^{k-1}}} \to X^{t C_p \tau C_{p^{k-1}}}.  \]
% obtained by application of $\tau C_{p^{k-1}}$ to the usual norm cofiber sequence.
\end{lem}
\begin{proof} Consider the recollement
\[ \begin{tikzcd}[row sep=4ex, column sep=6ex, text height=1.5ex, text depth=0.25ex]
\Sp^{h C_{p^n}} \ar[shift right=1,right hook->]{r}[swap]{j_{\ast}} & \Sp^{C_{p^n}} \ar[shift right=2]{l}[swap]{j^{\ast}} \ar[shift left=2]{r}{i^{\ast} \simeq \Phi^{C_p}} & \Sp^{C_{p^{n-1}}} \ar[shift left=1,left hook->]{l}{i_{\ast}}
\end{tikzcd} \]
and the associated fiber sequence $j_! \to j_{\ast} \to i_{\ast} i^{\ast} j_{\ast}$. By \cite[Lem.~4.35]{QS21a} applied to $C_p \trianglelefteq C_{p^n} \trianglelefteq C_{p^n}$, we see that $\Psi^{C_p}(j_! X)$ is $C_{p^{n-1}}$-Borel torsion and $\Psi^{C_p}(j_{\ast} X) \simeq j'_{\ast}( X^{h C_p})$ for $j'_{\ast}: \Sp^{h C_{p^{n-1}}} \to \Sp^{C_{p^{n-1}}}$. Therefore, the fiber sequence of $C_{p^{n-1}}$-spectra
\[ \Psi^{C_p}(j_! X) \to \Psi^{C_p}(j_{\ast} X) \to \Psi^{C_p}(i_{\ast} i^{\ast} j_{\ast} X) \simeq \Phi^{C_p} (j_{\ast} X) \]
yields the fiber sequence of underlying Borel $C_{p^{n-1}}$-spectra
\[ X_{h C_p} \to X^{h C_p} \to X^{t C_p} \]
and, applying $\phi^{C_{p^i}}:\Sp^{C_{p^{n-1}}} \to \Fun(B C_{p^{n-1-i}}, \Sp)$ for $0 < i \leq n-1$, the fiber sequence of Borel $(C_{p^{n-1-i}})$-spectra
\[ 0 \to X^{h C_p \tau C_{p^i}} \to X^{\tau C_{p^{i+1}}}. \]
We thereby deduce the equivalence $X^{h C_p \tau C_{p^i}} \simeq X^{\tau C_{p^{i+1}}}$, and the remaining equivalences follow by replacing $X$ by $X^{h C_{p^k}}$.

Next, we map the fiber sequence in $\Sp^{C_{p^{n-1}}}$ to its Borel completion
\[ \begin{tikzcd}[row sep=4ex, column sep=4ex, text height=1.5ex, text depth=0.25ex]
\Psi^{C_p}(j_! X) \ar{r} \ar{d} & \Psi^{C_p}(j_{\ast} X) \ar{r} \ar{d}{\simeq} &  \Phi^{C_p} (j_{\ast} X) \ar{d} \\
(j'_{\ast} {j'}^{\ast}) (\Psi^{C_p}(j_! X))  \ar{r} & (j'_{\ast} {j'}^{\ast}) (\Psi^{C_p}(j_{\ast} X)) \ar{r} & (j'_{\ast} {j'}^{\ast}) (\Phi^{C_p} (j_{\ast} X)).
\end{tikzcd} \]
We note that by definition, the canonical map $X^{\tau C_{p^k}} \to X^{t C_p \tau C_{p^{k-1}}}$ is obtained as $\phi^{C_{p^{k-1}}}$ of the unit map $\Phi^{C_p} (j_{\ast} X) \to (j'_{\ast} {j'}^{\ast}) (\Phi^{C_p} (j_{\ast} X))$. Because the middle map is an equivalence, the fiber of the righthand map is canonically equivalent to the cofiber of the lefthand map. But because $(j_! X)^{C_p}$ is Borel-torsion with $X_{h C_p}$ as its underlying Borel $C_{p^{n-1}}$-spectrum, $\phi^{C_{p^{k-1}}}$ of that cofiber is definitionally $(X_{ hC_p})^{\tau C_{p^{k-1}}}$.
\end{proof}

 % Suppose that $X \in \Sp^{h C_{p^n}}$ is bounded-below.
\begin{cor} \label{cor:CanonicalFunctorsEquivalences} Suppose that $X$ is a bounded-below Borel $C_{p^n}$-spectrum. Then the canonical map
\[ X^{\tau C_{p^k}} \to X^{t C_p \tau C_{p^{k-1}}} \]
is an equivalence for all $1 < k \leq n$.
\end{cor}
\begin{proof} By \cref{lm:identifyGenTate}, we may equivalently show that $(X_{h C_p})^{\tau C_{p^{k-1}}} \simeq 0$. We proceed by induction on $k$ (and prove the claim for all $n \geq k$ with $k$ fixed since we may ignore residual action for vanishing). For the base case $k=2$, the fiber of the canonical map is $(X_{h C_p})^{t C_p}$, which vanishes by the Tate orbit lemma. Now suppose $k>2$ and $(Y_{h C_p})^{\tau C_{p^i}} \simeq 0$ for all bounded-below $Y \in \Sp^{h C_{p^m}}$, $m \geq i+1$ and $1 \leq i<k-1$. By \cref{lm:identifyGenTate}, we have a fiber sequence
\[ ((X_{h C_{p}})_{h C_p})^{\tau C_{p^{k-2}}} \to (X_{h C_p})^{\tau C_{p^{k-1}}} \to (X_{h C_p})^{t C_p \tau C_{p^{k-2}}}. \]
Since $X_{h C_{p}}$ remains bounded-below, the inductive hypothesis ensures that the left term vanishes, and the right term vanishes by the Tate orbit lemma again. We conclude that the middle term vanishes.
\end{proof}

% $X \in \Fun_{C_2}(B^t_{C_2} \mu_{p^n}, \underline{\Sp}^{C_2})$
\begin{lem} \label{lem:dihedralIdentifyGenTate} Suppose $X$ is a $C_2$-spectrum with twisted $\mu_{p^n}$-action. For $1 < k \leq n$, we have twisted $\mu_{p^{n-k}}$-equivariant equivalences
\[ X^{\tau_{C_2} \mu_{p^k}} \simeq X^{h_{C_2} \mu_p \tau_{C_2} \mu_{p^{k-1}}} \simeq X^{h_{C_2} \mu_{p^2} \tau_{C_2} \mu_{p^{k-2}}} \simeq \cdots \simeq X^{h_{C_2} \mu_{p^{k-1}} t_{C_2} \mu_p}, \]
with respect to which the canonical map $$X^{\tau_{C_2} \mu_{p^k}} \to X^{t_{C_2} \mu_p \tau_{C_2} \mu_{p^{k-1}}}$$ fits into the fiber sequence
\[ (X_{h_{C_2} \mu_p})^{\tau_{C_2} \mu_{p^{k-1}}} \to X^{h_{C_2} \mu_p \tau_{C_2} \mu_{p^{k-1}}} \to X^{t_{C_2} \mu_p \tau_{C_2} \mu_{p^{k-1}}}.  \]
\end{lem}
\begin{proof} The strategy of the proof is the same as that of \cref{lm:identifyGenTate}, where we instead consider the recollement \eqref{eq:dihedral_recollement}
\[ \begin{tikzcd}[row sep=4ex, column sep=6ex, text height=1.5ex, text depth=0.25ex]
\Fun_{C_2}(B^t_{C_2} \mu_{p^n}, \underline{\Sp}^{C_2}) \ar[shift right=1,right hook->]{r}[swap]{j_{\ast}} & \Sp^{D_{2 p^n}} \ar[shift right=2]{l}[swap]{j^{\ast}} \ar[shift left=2]{r}{i^{\ast} \simeq \Phi^{{\mu}_p}} & \Sp^{D_{2p^{n-1}}} \ar[shift left=1,left hook->]{l}{i_{\ast}}.
\end{tikzcd} \]
Let $\Gamma = \Gamma_{\mu_{p^{n-1}}}$ be the $\mu_{p^{n-1}}$-free $D_{2p^{n-1}}$-family, and let 
\[ j'_{\ast} = \sF_b^{\vee}[\mu_{p^{n-1}}]: \Fun_{C_2}(B^t_{C_2} \mu_{p^{n-1}}, \underline{\Sp}^{C_2}) \to \Sp^{D_{2 p^{n-1}}}. \]
Applying \cite[Lem.~4.35]{QS21a} to the functor $\Psi^{\mu_p}: \Sp^{D_{2 p^n}} \to \Sp^{D_{2 p^{n-1}}}$, we see that in the fiber sequence
\[ \Psi^{\mu_p}(j_! X) \to \Psi^{\mu_p}(j_{\ast} X) \to \Phi^{\mu_p}(j_{\ast} X), \]
$\Psi^{\mu_p}(j_! X)$ is $\Gamma$-torsion and $\Psi^{\mu_p}(j_{\ast} X) \simeq j'_{\ast}(X^{h_{C_2} \mu_p})$. As before, for $0 \leq i \leq n-1$ let
\[ L[\zeta]_i: \Sp^{D_{2 p^{n-1}}} \to \Fun_{C_2}(B^t_{C_2} \mu_{p^{n-1-i}}, \underline{\Sp}^{C_2}) \]
be the localization functor. Then $L[\zeta]_0$ of the fiber sequence yields
\[ X_{h_{C_2} \mu_p} \to X^{h_{C_2} \mu_p} \to X^{t_{C_2} \mu_p} \]
whereas for $0< i\leq n-1$, $L[\zeta]_i$ of the fiber sequence yields
\[ 0 \to (X^{h_{C_2} \mu_p})^{\tau_{C_2} \mu_{p^i}} \xto{\simeq} X^{\tau_{C_2} \mu_{p^{i+1}}}  \]
from which we deduce the string of equivalences in the statement. Finally, if we map the fiber sequence of $D_{2p^{n-1}}$-spectra into its $\Gamma$-completion, then as in the proof of \cref{lm:identifyGenTate} we obtain the fiber sequence as in the statement.
\end{proof}

% Suppose that $X \in \Fun_{C_2}(B^t_{C_2} \mu_{p^n}, \underline{\Sp}^{C_2})$ is bounded-below.
\begin{cor} \label{cor:dihedralTOLyieldsEquivs} Suppose $X$ is a $C_2$-spectrum with twisted $\mu_{p^n}$-action whose underlying spectrum is bounded-below. Then the canonical map
\[ X^{\tau_{C_2} \mu_{p^k}} \to (X^{t_{C_2} \mu_p})^{\tau_{C_2} \mu_{p^{k-1}}} \]
is an equivalence for all $1 < k \leq n$.
\end{cor}
\begin{proof} Using the dihedral Tate orbit lemma and \cref{lem:dihedralIdentifyGenTate}, the proof is the same as that of \cref{cor:CanonicalFunctorsEquivalences}.
\end{proof}

To make effective use of \cref{cor:CanonicalFunctorsEquivalences} and \cref{cor:dihedralTOLyieldsEquivs} in analyzing $\Sp^{C_{p^n}}_{bb}$ and $\Sp^{D_{2p^n}}_{ubb}$, we need to introduce some more notation.

\begin{ntn} Let 
\begin{align*} \Fun^{\cocart}_{/\Delta^n}(\sd(\Delta^n), \Sp^{C_{p^n}}_{\locus{\phi}})_{bb} \subset \Fun^{\cocart}_{/\Delta^n}(\sd(\Delta^n), \Sp^{C_{p^n}}_{\locus{\phi}}) \\
\Fun^{\cocart}_{/\Delta^n}(\sd(\Delta^n), \Sp^{D_{2 p^n}}_{\locus{\phi,\zeta}})_{ubb} \subset \Fun^{\cocart}_{/\Delta^n}(\sd(\Delta^n), \Sp^{D_{2 p^n}}_{\locus{\phi,\zeta}})
\end{align*}
be the full subcategories on functors that evaluate on all singleton strings to bounded-below spectra, resp. underlying bounded-below $C_2$-spectra.
\end{ntn}

\begin{lem} We have equivalences
\begin{align*} \Theta: & \Fun^{\cocart}_{/\Delta^n}(\sd(\Delta^n), \Sp^{C_{p^n}}_{\locus{\phi}})_{bb} \xto{\simeq} \Sp^{C_{p^n}}_{bb}, \\
\Theta[\zeta]: & \Fun^{\cocart}_{/\Delta^n}(\sd(\Delta^n), \Sp^{D_{2 p^n}}_{\locus{\phi,\zeta}})_{ubb} \xto{\simeq} \Sp^{D_{2 p^n}}_{ubb}
\end{align*}
obtained by restriction from the equivalences of \cite[Thm.~2.46]{QS21a} and \cref{vrn:ReconstructionEquivalence}.
\end{lem}
\begin{proof} This follows from the definitions after recalling that $\Phi^{H} \circ \Theta$ and $L[\zeta]_k \circ \Theta[\zeta]$ are homotopic to evaluation at $H$ and $k$, respectively.
\end{proof}

\begin{ntn} Let $(t C_p)_{\bullet}: \Delta^n \to \Cat_{\infty}$ be the functor defined by
\[ \Sp^{h C_{p^n}} \xto{t^{C_p}} \Sp^{h C_{p^{n-1}}} \xto{t^{C_p}} \cdots \xto{t^{C_p}} \Sp^{h C_p} \xto{t^{C_p}} \Sp \]
and let $\Sp^{h C_{p^n}}_{\Tate} \to (\Delta^n)^{\op}$ be the cartesian fibration classified by $(t C_p)_{\bullet}$.

Similarly, let $(t_{C_2} \mu_p)_{\bullet}: \Delta^n \to \Cat_{\infty}$ be the functor defined by
\[ \Fun_{C_2}(B^t_{C_2} \mu_{p^n}, \underline{\Sp}^{C_2}) \xtolong{t_{C_2} {\mu_p}}{1.2} \Fun_{C_2}(B^t_{C_2} \mu_{p^{n-1}}, \underline{\Sp}^{C_2}) \xtolong{t_{C_2} {\mu_p}}{1.2} \cdots \xtolong{t_{C_2} {\mu_p}}{1.2} \Sp^{C_2} \]
and let $\Sp^{h_{C_2} \mu_{p^n}}_{\Tate} \to (\Delta^n)^{\op}$ be the cartesian fibration classified by $(t_{C_2} \mu_p)_{\bullet}$.
\end{ntn}

\begin{dfn} Given a section $X: (\Delta^n)^{\op} \to \Sp^{h C_{p^n}}_{\Tate}$, we say that $X$ is \emph{bounded-below} if $X(k)$ is bounded-below for all $k$. Let
$\Fun_{/(\Delta^n)^{\op}}((\Delta^n)^{\op},\Sp^{h C_{p^n}}_{\Tate})_{bb}$ denote the full subcategory on the bounded-below objects.

Similarly, we say that a section $X:(\Delta^n)^{\op} \to \Sp^{h_{C_2} \mu_{p^n}}_{\Tate}$ is \emph{underlying bounded-below} if $X(k)$ is underlying bounded-below for all $k$. Let $\Fun_{/(\Delta^n)^{\op}}((\Delta^n)^{\op},\Sp^{h_{C_2} \mu_{p^n}}_{\Tate})_{ubb}$ denote the full subcategory on the underlying bounded-below objects.
\end{dfn}

% \begin{dfn} Given a section $X: (\Delta^n)^{\op} \to \Sp^{h C_{p^n}}_{\Tate}$, resp. $X:(\Delta^n)^{\op} \to \Sp^{h_{C_2} \mu_{p^n}}_{\Tate}$, we say that $X$ is \emph{bounded-below}, resp. \emph{underlying bounded-below} if for all $0 \leq k \leq n$, the underlying spectrum of $X(k)$ is bounded-below. Let
% $\Fun_{/(\Delta^n)^{\op}}((\Delta^n)^{\op},\Sp^{h C_{p^n}}_{\Tate})_{bb}$ and $\Fun_{/(\Delta^n)^{\op}}((\Delta^n)^{\op},\Sp^{h_{C_2} \mu_{p^n}}_{\Tate})_{ubb}$ denote the corresponding full subcategories.
% \end{dfn}

We are now ready to prove the main result of this subsection. Its statement involves the notions of $1$-generated and extendable objects in an $[n]$-stratified stable $\infty$-category introduced in \cite[\S 4]{Sha21} as \cite[Def.~4.5]{Sha21} and \cite[Def.~4.12]{Sha21}, respectively. There, we proved the following abstract results:
\begin{enumerate}[leftmargin=*]
\item \cite[Thm.~4.15]{Sha21} Let $C \to \Delta^n$ be a locally cocartesian fibration whose fibers are stable $\infty$-categories and whose pushforward functors are exact. Let $\sd_1(\Delta^n) \subset \sd(\Delta^n)$ be the subposet on strings $\{ k \}$ and $\{ k < k+1 \}$. Then the restriction functor
\[ \gamma_n^*: \Fun^{\cocart}_{/\Delta^n}(\sd(\Delta^n), C) \to \Fun^{\cocart}_{/\Delta^n}(\sd_1(\Delta^n), C) \]
restricts to an equivalence between the full subcategories of $1$-generated objects on the left and extendable objects on the right.
\item \cite[Prop.~4.17]{Sha21} Let $C^{\vee} \to (\Delta^n)^{\op}$ denote the cartesian fibration classified by the composition $C_0 \to C_1 \to ... \to C_n$ of adjacent pushforward functors. Then we have an equivalence
\[ \Fun^{\cocart}_{/\Delta^n}(\sd_1(\Delta^n), C) \simeq \Fun_{/(\Delta^n)^{\op}}((\Delta^n)^{\op}, C^{\vee}), \]
so $\Fun^{\cocart}_{/\Delta^n}(\sd_1(\Delta^n), C)$ decomposes as an iterated pullback
\[ \Ar(C_n) \times_{C_n} \Ar(C_{n-1}) \times_{C_{n-1}} ... \times_{C_2} \Ar(C_1) \times_{C_1} C_0 \]
where the structure maps to the right are given by evaluation at the target and those to the left are given by pushforward functors applied to the source.
\end{enumerate}

\begin{prp} \label{prp:EquivalenceOnBoundedBelowAtFiniteLevel} We have inclusions of full subcategories
\begin{align*} \Fun^{\cocart}_{/\Delta^n}(\sd(\Delta^n), \Sp^{C_{p^n}}_{\locus{\phi}})_{bb} &\subset \Fun^{\cocart}_{/\Delta^n}(\sd(\Delta^n), \Sp^{C_{p^n}}_{\locus{\phi}})_{\gen{1}},  \\
\Fun_{/(\Delta^n)^{\op}}((\Delta^n)^{\op},\Sp^{h C_{p^n}}_{\Tate})_{bb} &\subset \Fun_{/(\Delta^n)^{\op}}((\Delta^n)^{\op},\Sp^{h C_{p^n}}_{\Tate})_{\ext}, \\
\Fun^{\cocart}_{/\Delta^n}(\sd(\Delta^n), \Sp^{D_{2 p^n}}_{\locus{\phi,\zeta}})_{ubb} &\subset \Fun^{\cocart}_{/\Delta^n}(\sd(\Delta^n), \Sp^{D_{2 p^n}}_{\locus{\phi,\zeta}})_{\gen{1}}, \\
\Fun_{/(\Delta^n)^{\op}}((\Delta^n)^{\op},\Sp^{h_{C_2} \mu_{p^n}}_{\Tate})_{ubb} &\subset \Fun_{/(\Delta^n)^{\op}}((\Delta^n)^{\op},\Sp^{h_{C_2} \mu_{p^n}}_{\Tate})_{\ext},
\end{align*}
such that the equivalences of \cite[Thm.~4.15]{Sha21} between $1$-generated and extendable objects restrict to
\begin{align*} \Fun^{\cocart}_{/\Delta^n}(\sd(\Delta^n), \Sp^{C_{p^n}}_{\locus{\phi}})_{bb} & \xto{\simeq} \Fun_{/(\Delta^n)^{\op}}((\Delta^n)^{\op},\Sp^{h C_{p^n}}_{\Tate})_{bb}, \\
\Fun^{\cocart}_{/\Delta^n}(\sd(\Delta^n), \Sp^{D_{2 p^n}}_{\locus{\phi,\zeta}})_{ubb} & \xto{\simeq} \Fun_{/(\Delta^n)^{\op}}((\Delta^n)^{\op},\Sp^{h_{C_2} \mu_{p^n}}_{\Tate})_{ubb}.
\end{align*}
\end{prp}
\begin{proof} The inclusions follow from \cref{cor:CanonicalFunctorsEquivalences}, \cref{cor:dihedralTOLyieldsEquivs}, and \cite[Lem.~4.6]{Sha21}. Matching the (underlying) bounded-below conditions then implies that the equivalence of \cite[Thm.~4.15]{Sha21} restricts as claimed.
\end{proof}

\begin{cor} \label{cor:iterated_pullback_descr}
We have iterated pullback decompositions
\begin{align*}
\Sp^{C_{p^n}}_{bb} & \simeq \Sp^{h C_{p^n}}_{bb} \times_{\Sp^{h C_{p^{n-1}}}} \Ar'(\Sp^{h C_{p^{n-1}}}) \times \cdots \times_{\Sp} \Ar'(\Sp), \\ 
\Sp^{D_{2p^n}}_{ubb} & \simeq \Sp^{h_{C_2} \mu_{p^n}}_{ubb} \times_{\Sp^{h_{C_2} \mu_{p^{n-1}}}} \Ar'(\Sp^{h_{C_2} \mu_{p^{n-1}}}) \times \cdots \times_{\Sp^{C_2}} \Ar'(\Sp^{C_2}),
\end{align*}
where:
\begin{enumerate}
\item[(i)] We take the full subcategory $\Ar'$ of arrows in which the source is (underlying) bounded-below.
\item[(ii)] The structure maps in the fiber product going to the right are $t C_p$, resp. $t_{C_2} \mu_p$.
\item[(iii)] The structure maps in the fiber product going to the left are given by evaluation at the target.
\end{enumerate}
These equivalences are implemented by the functors that forget to geometric fixed points
\begin{align*}
X \in \Sp^{C_{p^n}} & \mapsto [X^{\phi 1}, \: X^{\phi C_p}, \: ..., \: X^{\phi C_{p^n}}], \\ 
X \in \Sp^{D_{2p^n}} & \mapsto [L[\zeta]_0(X), \: L[\zeta]_1(X) , \: ..., \: L[\zeta]_n(X) ]
\end{align*}
and adjacent gluing data $\{ X^{\phi C_{p^k}} \to (X^{\phi C_{p^{k-1}}})^{tC_p} \}$ and $\{ L[\zeta]_k(X) \to L[\zeta]_{k-1}(X)^{t_{C_2} \mu_p} \}$ thereof.
\end{cor}
\begin{proof}
Combine \cref{prp:EquivalenceOnBoundedBelowAtFiniteLevel} and \cite[Prop.~4.17]{Sha21}.
\end{proof}

% \begin{rem}
%  \cref{cor:iterated_pullback_descr}, we see that the data of an underlying bounded-below $D_{2p^n}$-spectrum $X$ is equivalent to a sequence of objects
% \[ \{ L[\zeta]_0(X), \: L[\zeta]_1(X), \: ..., \: L[\zeta]_n(X) \}, \quad L[\zeta]_k(X) \in \Sp^{h_{C_2} \mu_{p^{n-k}}}_{ubb}, \]
% together with twisted $\mu_{p^{n-k}}$-equivariant maps
% \[ L[\zeta]_k(X) \to L[\zeta]_{k-1}(X)^{t_{C_2} \mu_p} \]
% for $1 \leq k \leq n$.
% \end{rem}

\subsection{Exchanging a lax equalizer for an equalizer}

In this subsection, we record an abstract lemma regarding the lax equalizer of the identity and an endofunctor $F:C \to C$ that we will need for the proof of \cref{thm:MainTheoremEquivalenceBddBelow}.

\begin{ntn} \label{ntn:nonnegativeintegers} Let $\ZZ_{\geq 0}$ denote the totally ordered set of non-negative integers regarded as a category, and let $\NN$ denote the monoid of non-negative integers under addition. Let $s: \ZZ_{\geq 0} \to \ZZ_{\geq 0}$ denote the successor functor that sends $n$ to $n+1$.
\end{ntn}

\begin{dfn} \label{dfn:spines} The \emph{spine} $\spine(\Delta^n) \subset \Delta^n$ is the subsimplicial set $\bigcup_{k=0}^{n-1} \{k<k+1\}$. Likewise, the \emph{spine} $\spine(\ZZ_{\geq 0}) \subset \ZZ_{\geq 0}$ is the subsimplicial set $\bigcup_{k=0}^{\infty} \{k<k+1\}$.
% whose nondegenerate simplices consist of all vertices and the $1$-simplices $\{i<i+1\}$, $i \geq 0$, which we refer to as the \emph{spine} of $\ZZ_{\geq 0}$. 
\end{dfn}

\begin{rem} \label{rem:innerAnodyneSpineInclusion} Recall that the spine inclusions of \cref{dfn:spines} are inner anodyne; indeed, a simple inductive argument with inner horn inclusions shows the maps $\spine(\Delta^n) \subset \Delta^n$ are inner anodyne, and it follows that $\spine(\ZZ_{\geq 0}) \subset \ZZ_{\geq 0}$ is inner anodyne by the stability of inner anodyne maps under filtered colimits.
\end{rem}

 % we may also arrange $C \cong \ast \times_{B \NN} \widehat{C}$ for convenience
\begin{cnstr} \label{cnstr:SetupForEndofunctorFibration} Let $C$ be an $\infty$-category and $F: C \to C$ an endofunctor. Let $$\widehat{C} \to B \NN \cong B \NN^{\op}$$ be the cartesian fibration classified by the functor $B \NN \to \Cat_{\infty}$ that deloops the map of monoids $\NN \to \Fun(C,C)$ uniquely specified by $1 \mapsto F$.\footnote{This is the operadic left Kan extension of the functor $\ast \to \Fun(C,C)$ selecting $F$ for the monoidal structure on $\Fun(C,C)$ defined by composition of endofunctors.} Define a structure map $$p: \ZZ_{\geq 0}^{\op} \times \ZZ_{\geq 0} \to B \NN$$ by $p[(n+k,m) \rightarrow (n,m+l)] = k$ and note that $p$ is a cartesian fibration. We will regard any subcategory of $\ZZ_{\geq 0}^{\op} \times \ZZ_{\geq 0}$ as over $B \NN$ via $p$, so $\ZZ_{\geq 0}^{\op} \to B \NN$ is a cartesian fibration whereas $\ZZ_{\geq 0} \to B \NN$ is the constant functor at $\ast$. Precomposition by the successor functor $s$ defines two `shift' functors
\begin{align*} \sh = s^{\ast} &: \Fun(\ZZ_{\geq 0},C) \to \Fun(\ZZ_{\geq 0},C), \\ 
\sh = (s^{\op})^{\ast} &: \Fun_{/B \NN}(\ZZ_{\geq 0}^{\op}, \widehat{C}) \to \Fun_{/B \NN}(\ZZ_{\geq 0}^{\op}, \widehat{C}).
\end{align*}

Let $F_{\ast}$ be the endofunctor of $\Fun(\ZZ_{\geq 0},C)$ defined by postcomposition by $F$. Observe that under the straightening correspondence, $\Fun(\ZZ_{\geq 0}, C) \simeq \Fun_{/ B \NN}(\ZZ_{\geq 0}, \widehat{C})$ (since $C \simeq \ast \times_{B \NN} \widehat{C}$) and $F_{\ast}$ is encoded by the exponentiated cartesian fibration $(\widehat{C})^{\ZZ_{\geq 0}} \to B \NN$. Elaborating upon this, it is easily seen that the lax equalizer $\LEq_{\sh:F_{\ast}}(\Fun(\ZZ_{\geq 0}, C))$ is equivalent to the pullback of the diagram
\[ \begin{tikzcd}[row sep=4ex, column sep=4ex, text height=1.5ex, text depth=0.25ex]
& \Fun_{/B \NN}(\{1 < 0\} \times \ZZ_{\geq 0} , \widehat{C}) \ar{d}{(\ev_1, \ev_0)} \\
\Fun_{/ B \NN}(\ZZ_{\geq 0}, \widehat{C}) \ar{r}{(s^{\ast},\id)} & \Fun_{/B \NN}(\ZZ_{\geq 0}, \widehat{C}) \times \Fun_{/B \NN}(\ZZ_{\geq 0}, \widehat{C}),
\end{tikzcd} \]
since objects of that pullback are equivalent to diagrams
\[ \begin{tikzcd}[row sep=4ex, column sep=4ex, text height=1.5ex, text depth=0.25ex]
X_1 \ar{r}{\beta_0} \ar{d}{\alpha_1} & F(X_0) \ar{d}{F(\alpha_0)} \ar{r} & X_0 \ar{d}{\alpha_0} \\
X_2 \ar{r}{\beta_1} \ar{d}{\alpha_2} & F(X_1) \ar{r} \ar{d}{F(\alpha_1)} & X_1 \ar{d}{\alpha_1} \\
\vdots & \vdots & \vdots
\end{tikzcd} \]
where the labeled arrows are in $C$ and the right horizontal edges are cartesian in $\widehat{C}$ over $1 \in \NN$. Rather than give a complete account of the details, for the subsequent lemma let us abuse notation and instead \emph{define} the expression $\LEq_{\sh:F_{\ast}}(\Fun(\ZZ_{\geq 0}, C))$ to refer to this pullback.
\end{cnstr}

\begin{lem} \label{lem:LaxEqualizerGenericEquivalence} There is an equivalence of $\infty$-categories
\[ \chi: \LEq_{\id:\sh}(\Fun_{/B \NN}(\ZZ_{\geq 0}^{\op}, \widehat{C})) \simeq \LEq_{\sh:F_{\ast}}(\Fun(\ZZ_{\geq 0}, C)) \]
that restricts to an equivalence of $\infty$-categories
\[  \chi_0: \Eq_{\id:\sh}(\Fun_{/B \NN}(\ZZ_{\geq 0}^{\op}, \widehat{C})) \simeq \LEq_{\id:F}(C).  \]
\end{lem}

\begin{proof} Intuitively, the first equivalence $\chi$ exchanges diagrams
\[ \begin{tikzcd}[row sep=4ex, column sep=4ex, text height=1.5ex, text depth=0.25ex]
\cdots \ar{r}{\phi_2} & X_2 \ar{r}{\phi_1} \ar{d}{\alpha_2} & X_1 \ar{r}{\phi_0} \ar{d}{\alpha_1} & X_0 \ar{d}{\alpha_0} \\
\cdots \ar{r}{\phi_3} & X_3 \ar{r}{\phi_2} & X_2 \ar{r}{\phi_1} & X_1
\end{tikzcd} \]
with diagrams
\[ \begin{tikzcd}[row sep=4ex, column sep=4ex, text height=1.5ex, text depth=0.25ex]
X_1 \ar{r}{\phi_0} \ar{d}{\alpha_1} & X_0 \ar{d}{\alpha_0} \\
X_2 \ar{r}{\phi_1} \ar{d}{\alpha_2} & X_1 \ar{d}{\alpha_1} \\
\vdots & \vdots
\end{tikzcd} \]
with one such diagram uniquely determining the other. To make this idea precise, we need to introduce some auxiliary constructions. Given $n \geq 0$, define
\[ \sh^n: \LEq_{\id:\sh}(\Fun_{/B \NN}(\ZZ_{\geq 0}^{\op}, \widehat{C}))  \to \Fun_{/B \NN}(\ZZ_{\geq 0}^{\op} \times \{n<n+1\}, \widehat{C}) \] to be the composite of the projection to $\Fun_{/B \NN}(\ZZ_{\geq 0}^{\op} \times \{0<1\}, \widehat{C})$ and precomposition by the $n$-fold successor functor $\ZZ_{\geq 0}^{\op} \times \{n < n+1 \} \to \ZZ_{\geq 0}^{\op} \times \{0 <1 \} $, $(i,n+j) \mapsto (i+n,j)$. Then form the pullback
\[ \begin{tikzcd}[row sep=4ex, column sep=6ex, text height=1.5ex, text depth=0.25ex]
\LEq^{\infty}_{\id:\sh}(\Fun_{/B \NN}(\ZZ_{\geq 0}^{\op}, \widehat{C})) \ar{r} \ar{d}{\pi'} & \Fun_{/B \NN}(\ZZ_{\geq 0}^{\op} \times \ZZ_{\geq 0}, \widehat{C}) \ar{d} \\
\LEq_{\id:\sh}(\Fun_{/B \NN}(\ZZ_{\geq 0}^{\op}, \widehat{C})) \ar{r}{(\sh^k)} & \prod_{k=0}^{\infty} \Fun_{/B \NN}(\ZZ_{\geq 0}^{\op} \times \{k < k+1 \}, \widehat{C}).
\end{tikzcd} \]
where the lower right object is the iterated fiber product. The righthand vertical map is obtained via precomposition by the inclusion $\ZZ_{\geq 0}^{\op} \times \spine(\ZZ_{\geq 0}) \to \ZZ_{\geq 0}^{\op} \times \ZZ_{\geq 0}$, which is inner anodyne by \cref{rem:innerAnodyneSpineInclusion} and \cite[Prop.~2.3.2.4]{HTT}. Therefore, the vertical maps are trivial fibrations. Similarly, define 
\[ \sh^n: \LEq_{\sh:F_{\ast}}(\Fun(\ZZ_{\geq 0}, C)) \to \Fun_{/B \NN}( \{ n+1<n \} \times \ZZ_{\geq 0}, \widehat{C}) \]
as the composite of the projection to $ \Fun_{/B \NN}( \{ 1<0 \} \times \ZZ_{\geq 0}, \widehat{C})$ and precomposition by the $n$-fold successor functor $\{ n+1<n \} \times \ZZ_{\geq 0} \to \{ 1<0 \} \times \ZZ_{\geq 0}$, $(n+j,i) \mapsto (j,i+n)$. Form the pullback square
\[ \begin{tikzcd}[row sep=4ex, column sep=6ex, text height=1.5ex, text depth=0.25ex]
\LEq^{\infty}_{\sh:F_{\ast}}(\Fun(\ZZ_{\geq 0}, C)) \ar{r} \ar{d}{\pi''} & \Fun_{/B \NN}(\ZZ_{\geq 0}^{\op} \times \ZZ_{\geq 0}, \widehat{C}) \ar{d} \\
\LEq_{\sh:F_{\ast}}(\Fun(\ZZ_{\geq 0}, C)) \ar{r}{(\sh^k)} & \prod_{k=0}^{\infty} \Fun_{/B \NN}(\{k+1<k\} \times \ZZ_{\geq 0}, \widehat{C}).
\end{tikzcd} \]
The righthand vertical map is precomposition by the inner anodyne map $\spine(\ZZ_{\geq 0}^{\op}) \times \ZZ_{\geq 0} \to \ZZ_{\geq 0}^{\op} \times \ZZ_{\geq 0}$, so both vertical maps are trivial fibrations. Next, the product map $\spine(\ZZ_{\geq 0}^{\op}) \times \spine(\ZZ_{\geq 0}) \subset \ZZ_{\geq 0}^{\op} \times \ZZ_{\geq 0}$ is also inner anodyne, and via precomposition we get a trivial fibration
\[ \rho: \Fun_{/B \NN}(\ZZ_{\geq 0}^{\op} \times \ZZ_{\geq 0}, \widehat{C}) \to \Fun_{/B \NN}(\spine(\ZZ_{\geq 0}^{\op}) \times \spine(\ZZ_{\geq 0}), \widehat{C}). \]
Let $\sB$ be the full subcategory of $\Fun_{/B \NN}(\spine(\ZZ_{\geq 0}^{\op}) \times \spine(\ZZ_{\geq 0}), \widehat{C})$ on objects $X_{\bullet,\bullet}$ such that for all $m \geq 0$, $n >0$ we have that
\[ X|_{ \{n+1<n\} \times \{m<m+1\} } = X|_{ \{n<n-1\} \times \{m+1,m+2\} }. \]
By definition, objects of $\LEq^{\infty}_{\id:\sh}(\Fun_{/B \NN}(\ZZ_{\geq 0}^{\op}, \widehat{C}))$ are diagrams $X_{\bullet, \bullet}: \ZZ_{\geq 0}^{\op} \times \ZZ_{\geq 0} \to \widehat{C}$ such that for every $m \geq 0$,
\[ X_{\bullet,\bullet}|_{\ZZ_{\geq 0}^{\op} \times \{m\} } = X_{\bullet,\bullet}|_{\ZZ_{\geq 0}^{\op} \times \{0\}} \circ s^m \text{ and } X_{\bullet,\bullet}|_{\ZZ_{\geq 0}^{\op} \times \{m<m+1\} } = X_{\bullet,\bullet}|_{\ZZ_{\geq 0}^{\op} \times \{0<1\}} \circ s^m,\]
and similarly for $\LEq^{\infty}_{\sh:F_{\ast}}(\Fun(\ZZ_{\geq 0}, C))$. The conditions on edges are implied by those for squares in $\cB$, and the functor $\rho$ thereby restricts to trivial fibrations
\[ \rho', \rho'': \LEq^{\infty}_{\id:\sh}(\Fun_{/B \NN}(\ZZ_{\geq 0}^{\op}, \widehat{C})), \LEq^{\infty}_{\sh:F_{\ast}}(\Fun(\ZZ_{\geq 0}, C)) \to \sB.  \]
Choosing sections of the trivial fibrations $\pi', \rho''$ or $\pi'', \rho'$ then furnishes the equivalence $\chi$. 

For the second equivalence $\chi_0$, let $\Fun^{\simeq}_{/B \NN}(\ZZ_{\geq 0}, \widehat{C}) \subset \Fun_{/B \NN}(\ZZ_{\geq 0}, \widehat{C})$ be the full subcategory on those objects $X_{\bullet}: \ZZ_{\geq 0} \to C$ that send every edge to an equivalence, and form the pullback
\[ \begin{tikzcd}[row sep=4ex, column sep=4ex, text height=1.5ex, text depth=0.25ex]
\LEq_{\sh:F_{\ast}}(\Fun^{\simeq}(\ZZ_{\geq 0}, C)) \ar{r} \ar{d} & \Fun_{/B \NN}(\{1 < 0\} \times \ZZ_{\geq 0} , \widehat{C}) \ar{d}{(\ev_1, \ev_0)} \\
\Fun^{\simeq}_{/ B \NN}(\ZZ_{\geq 0}, \widehat{C}) \ar{r}{(s^{\ast},\id)} & \Fun_{/B \NN}(\ZZ_{\geq 0}, \widehat{C}) \times \Fun_{/B \NN}(\ZZ_{\geq 0}, \widehat{C}),
\end{tikzcd} \]
which defines $\LEq_{\sh:F_{\ast}}(\Fun^{\simeq}(\ZZ_{\geq 0}, C))$ as a full subcategory of $\LEq_{\sh:F_{\ast}}(\Fun(\ZZ_{\geq 0}, C))$. It follows from the definitions that $\chi$ restricts to an equivalence
\[ \chi_0': \Eq_{\id:\sh}(\Fun_{/B \NN}(\ZZ_{\geq 0}^{\op}, \widehat{C})) \simeq \LEq_{\sh:F_{\ast}}(\Fun^{\simeq}(\ZZ_{\geq 0}, C)).  \]

Let $P =( \{ 1 < 0\} \times \ZZ_{\geq -1}) \setminus \{ (0,-1)\} $ and regard it as over $B \NN$ via the projection to $\{1<0\}$. Note by \cite[Lem.~3.18]{Sha21} that the cofibration $\{ (1,-1) \rightarrow (1,0) \} \cup_{(1,0)} (\{1<0\} \times \ZZ_{\geq 0}) \to P$ is a categorical equivalence. Therefore, if we form the pullback
\[ \begin{tikzcd}[row sep=4ex, column sep=6ex, text height=1.5ex, text depth=0.25ex]
\LEq^+_{\sh:F_{\ast}}(\Fun(\ZZ_{\geq 0}, C)) \ar{r} \ar{d} & \Fun_{/B \NN}(P, \widehat{C}) \ar{d}{(\ev_1,\ev_0)} \\
\Fun_{/B \NN}(\ZZ_{\geq 0}, \widehat{C}) \ar{r}{(s^{\ast},\id)} & \Fun_{/B \NN}(\ZZ_{\geq -1}, \widehat{C}) \times \Fun_{/B \NN}(\ZZ_{\geq 0}, \widehat{C}),
\end{tikzcd} \]
precomposition by $\{ 1 < 0\} \times \ZZ_{\geq 0} \subset P$ induces a trivial fibration
\[ \xi: \LEq^+_{\sh:F_{\ast}}(\Fun(\ZZ_{\geq 0}, C)) \to \LEq_{\sh:F_{\ast}}(\Fun(\ZZ_{\geq 0}, C)). \]
Defining $\LEq^+_{\sh:F_{\ast}}(\Fun^{\simeq}(\ZZ_{\geq 0}, C))$ in a similar fashion, we also obtain a trivial fibration
\[ \xi_0: \LEq^+_{\sh:F_{\ast}}(\Fun^{\simeq}(\ZZ_{\geq 0}, C)) \to \LEq_{\sh:F_{\ast}}(\Fun^{\simeq}(\ZZ_{\geq 0}, C)), \]
which is obtained by restricting $\xi$.

We now observe that a functor $f: P \to \widehat{C}$ over $B \NN$ is a relative left Kan extension of its restriction to $P_0 = \{(1,-1) \to (0,0)\}$ if and only if it sends the edges $\{ (1,m) \to (1,m+k) \}$, $m \geq -1$ and $\{0,m) \to (0,m+k) \}$, $m \geq 0$ to equivalences, since each slice category $P_0 \times_P P_{/(i,m)}$ has $(i,-i)$ as a terminal object. Therefore, if we form the pullback
\[ \begin{tikzcd}[row sep=4ex, column sep=6ex, text height=1.5ex, text depth=0.25ex]
\LEq'_{\id:F}(C) \ar{r} \ar{d} & \Fun_{/P_0}(P_0, P_0 \times_{B \NN} \widehat{C}) \ar{d}{(\ev_1, \ev_0)} \\
C \ar{r}{(\id,\id)} & C \times C
\end{tikzcd} \]
the restriction functor induced by $P_0 \subset P$
\[ \LEq^+_{\sh:F_{\ast}}(\Fun^{\simeq}(\ZZ_{\geq 0}, C)) \to \LEq'_{\id:F}(C)\]
is a trivial fibration. Let us now write $\Delta^1 = P_0$ and $\sM = \Delta^1 \times_{B \NN} \widehat{C}$ for clarity. Since the source functor $\Ar(\Delta^1) \to \Delta^1$ is the free cartesian fibration (\cite[Exm.~2.6 or Def.~7.5]{Exp2}) on the identity, we obtain a trivial fibration
\[ \Fun^{\cart}_{/\Delta^1}(\Ar(\Delta^1), \sM) \to \Fun_{/\Delta^1}(\Delta^1, \sM). \]
Moreover, writing $\Ar(\Delta^1) = [00<01<11]$, the square
\[ \begin{tikzcd}[row sep=4ex, column sep=8ex, text height=1.5ex, text depth=0.5ex]
\Fun^{\cart}_{/\Delta^1}(\Ar(\Delta^1), \sM) \ar{r}{\ev|_{[00 <01]}} \ar{d}[swap]{(\ev_{00}, \ev_{11})} & \Ar(C) \ar{d}{(\ev_0, \ev_1)} \\
C \times C \ar{r}{(\id,F)} & C \times C
\end{tikzcd} \]
is homotopy commutative. We thereby obtain an equivalence $\LEq'_{\id:F}(C) \simeq \LEq_{\id:F}(C)$. Chaining together the various equivalences above then produces the desired equivalence $\chi_0$.
\end{proof}

\begin{rem} The equivalence $\chi_0$ of \cref{lem:LaxEqualizerGenericEquivalence} sends an object 
\[ \begin{tikzcd}[row sep=4ex, column sep=4ex, text height=1.5ex, text depth=0.25ex]
\cdots \ar{r}{\phi_2} & X_2 \ar{r}{\phi_1} \ar{d}{\alpha_2}[swap]{\simeq} & X_1 \ar{r}{\phi_0} \ar{d}{\alpha_1}[swap]{\simeq} & X_0 \ar{d}{\alpha_0}[swap]{\simeq} \\
\cdots \ar{r}{\phi_3} & X_3 \ar{r}{\phi_2} & X_2 \ar{r}{\phi_1} & X_1
\end{tikzcd} \]
 to the composite $X_0 \xto{\alpha_0} X_1 \xto{\beta_0} F(X_0)$, where we factor the edge $\phi_0$ through $\beta_0$ and a cartesian edge $F(X_0) \to X_0$.
\end{rem}

% \begin{wrn} Although $\LEq_{\id:F}(C)$ is the $\infty$-category $\mr{CoAlg}_F(C)$ studied in \cite[\S II.5]{NS18}, we will apply \cref{lm:abstractEqualizerExchangeFormula} with $F$ taken to be the endofunctor $t_{C_2} \mu_p$ that defines $\RCycSp_p$.
% \end{wrn}

\subsection{Proof of the main theorem} \label{sec:main_theorem}

We have almost assembled all of the ingredients needed to prove \cref{thm:MainTheoremEquivalenceBddBelow}. In fact, we will also reprove \cite[Thm.~II.6.3]{NS18} by way of illustrating the formal nature of our proof. In order to make effective use of \cref{prp:EquivalenceOnBoundedBelowAtFiniteLevel} in the dihedral situation, we first establish the compatibility of the relative geometric locus construction with restriction and geometric fixed points (cf. \cite[Constrs.~2.54 and 2.55]{QS21a} for the absolute compatibility assertions). 

\begin{ntn} \label{ntn:ConciseRestrictionNotation} For $0 \leq k \leq n$, we have the inclusions $C_{p^k} \subset C_{p^n}$ and $D_{2 p^k} \subset D_{2 p^n}$. Let
\begin{align*} \res^n_k: \Sp^{C_{p^n}} \to \Sp^{C_{p^k}}, \quad & \res^n_k: \Fun(B C_{p^n} ,\Sp) \to \Fun(B C_{p^k} ,\Sp) \\
\res^n_k: \Sp^{D_{2 p^n}} \to \Sp^{D_{2 p^k}}, \quad & \res^n_k: \Fun_{C_2}(B^t_{C_2} \mu_{p^{n}}, \underline{\Sp}^{C_2}) \to \Fun_{C_2}(B^t_{C_2} \mu_{p^{k}}, \underline{\Sp}^{C_2})
\end{align*}
be alternative notation for the restriction functors.
\end{ntn}

\begin{vrn} \label{vrn:DihedralRestrictionGeometricLoci} For $0 \leq k \leq n$ and the inclusion $D_{2 p^k} \subset D_{2 p^n}$ given by $\mu_{p^k} \subset \mu_{p^n}$, we have a commutative diagram
\[ \begin{tikzcd}[row sep=4ex, column sep=4ex, text height=1.5ex, text depth=0.25ex]
\fS[D_{2p^k}] \ar{r}{i} \ar{d}{\zeta} & \fS[D_{2p^n}] \ar{d}{\zeta} \\
\Delta^k \ar{r}{i} & \Delta^n
\end{tikzcd} \]
where the bottom functor is the inclusion of $\Delta^k$ as a sieve. As in \cite[Constr.~2.54]{QS21a}, the restriction functor $\res^n_k: \Sp^{D_{2 p^n}} \to \Sp^{D_{2 p^k}}$ induces a functor 
\[ \res^n_k: \Sp^{D_{2p^n}}_{\locus{\phi,\zeta}} \times_{\Delta^n} \Delta^k \to \Sp^{D_{2p^k}}_{\locus{\phi,\zeta}} \]
that on the fiber over $i \in \Delta^k$ is equivalent to the functor
\[ \res^n_k: \Fun_{C_2}(B^t_{C_2} \mu_{p^{n-i}}, \underline{\Sp}^{C_2}) \to \Fun_{C_2}(B^t_{C_2} \mu_{p^{k-i}}, \underline{\Sp}^{C_2}). \]
Precomposition by $i: \sd(\Delta^k) \to \sd(\Delta^n)$ and postcomposition by $\res^n_k$ yields the functor
\[ \res^n_k: \Fun^{\cocart}_{/\Delta^n}(\sd(\Delta^n), \Sp^{D_{2p^n}}_{\locus{\phi,\zeta}}) \to \Fun^{\cocart}_{/\Delta^k}(\sd(\Delta^k), \Sp^{D_{2p^k}}_{\locus{\phi,\zeta}}). \]
Furthermore, by the same argument as in \emph{loc. cit.} we have a commutative diagram
\[ \begin{tikzcd}[row sep=4ex, column sep=6ex, text height=1.5ex, text depth=0.5ex]
\Fun^{\cocart}_{/\Delta^n}(\sd(\Delta^n), \Sp^{D_{2p^n}}_{\locus{\phi,\zeta}}) \ar{r}{\Theta[\zeta]} \ar{d}[swap]{\res^n_k} & \Sp^{D_{2p^n}} \ar{d}{\res^n_k} \\
\Fun^{\cocart}_{/\Delta^k}(\sd(\Delta^{k}), \Sp^{D_{2p^k}}_{\locus{\phi,\zeta}}) \ar{r}{\Theta[\zeta]} & \Sp^{D_{2p^k}}.
\end{tikzcd} \]
\end{vrn}

\begin{vrn} \label{vrn:DihedralGeometricFixedPoints} For $0 \leq k \leq n$ and the quotient homomorphism $D_{2p^n} \to D_{2p^n}/\mu_{p^k} \cong D_{2p^{n-k}}$, we have a commutative diagram of cosieve inclusions
\[ \begin{tikzcd}[row sep=4ex, column sep=4ex, text height=1.5ex, text depth=0.25ex]
\fS[D_{2p^{n-k}}] \ar{r}{i} \ar{d}{\zeta} & \fS[D_{2p^n}] \ar{d}{\zeta} \\
\Delta^{n-k} \ar{r}{i} & \Delta^n.
\end{tikzcd} \]
As in \cite[Constr.~2.55]{QS21a}, $\Phi^{\mu_{p^k}}$ implements an equivalence
\[ \Sp^{D_{2p^n}}_{\locus{\phi,\zeta}} \times_{\Delta^n} \Delta^{n-k} \simeq \Sp^{D_{2p^{n-k}}}_{\locus{\phi,\zeta}} \]
with respect to which we write $i^{\ast}$ as
\[ \Phi^{\mu_{p^k}}: \Fun^{\cocart}_{/\Delta^n}(\sd(\Delta^n), \Sp^{D_{2p^n}}_{\locus{\phi,\zeta}}) \to \Fun^{\cocart}_{/\Delta^{n-k}}(\sd(\Delta^{n-k}), \Sp^{D_{2p^{n-k}}}_{\locus{\phi,\zeta}}). \]
We then obtain a commutative diagram
\[ \begin{tikzcd}[row sep=4ex, column sep=6ex, text height=1.5ex, text depth=0.5ex]
\Fun^{\cocart}_{/\Delta^n}(\sd(\Delta^n), \Sp^{D_{2p^n}}_{\locus{\phi,\zeta}}) \ar{r}{\Theta[\zeta]} \ar{d}[swap]{\Phi^{\mu_{p^k}}} & \Sp^{D_{2p^{n}}} \ar{d}{\Phi^{\mu_{p^k}}} \\
\Fun^{\cocart}_{/\Delta^{n-k}}(\sd(\Delta^{n-k}), \Sp^{D_{2p^{n-k}}}_{\locus{\phi,\zeta}}) \ar{r}{\Theta[\zeta]} & \Sp^{D_{2p^{n-k}}}.
\end{tikzcd} \]
\end{vrn}

We now consider an axiomatic setup that will handle the $C_{p^{\infty}}$ and $D_{2 p^{\infty}}$-situations simultaneously.

\begin{lem} \label{lm:joinLocallyCocartesian} Suppose $p: C \to S$ is a locally cocartesian fibration.
\begin{enumerate}
\item For any $\infty$-category $T$, $p': C \star T \to S \star T$ is a locally cocartesian fibration.
\item For any locally cocartesian fibration $D \to S$, the restriction functor implements an equivalence
\[ \Fun^{\cocart}_{/S \star T}(C \star T, D \star T) \xto{\simeq} \Fun^{\cocart}_{/S}(C, D). \]
\item Suppose $S \cong S_0 \star S_1$ and let $C_0 = S_0 \times_S C$. Then for any locally cocartesian fibration $D \to S_0$, the restriction functor implements an equivalence 
\[ \Fun^{\cocart}_{/S}(C,D \star S_1) \xto{\simeq} \Fun^{\cocart}_{/S_0}(C_0,D). \]
% Suppose the fibers of $p$ admit terminal objects and the pushforward functors preserve terminal objects. 
\item For any $\infty$-category $T$, the restriction functor
\[ j^{\ast}: \Fun^{\cocart}_{/S \star T}(\sd(S \star T), C \star T) \to \Fun^{\cocart}_{/S}(\sd(S), C) \]
is an equivalence of $\infty$-categories.
\end{enumerate}
\end{lem}
\begin{proof} For (1), first recall that the join is defined by the right Quillen functor $j_{\ast}: s\Set_{/ \partial \Delta^1} \to s\Set_{/\Delta^1}$ for the inclusion $j: \partial \Delta^1 \to \Delta^1$ (cf. \cite[Def~4.1]{Exp2}). Therefore, given two categorical fibrations $X \to A$ and $Y \to B$, $X \star Y \to A \star B$ is a categorical fibration, so in particular $p'$ is a categorical fibration. It is clear that for any edge $e: \Delta^1 \to S \star T$ with image in $S$ or $T$ that the pullback over $e$ of $p'$ is a cocartesian fibration. Suppose $e$ is specified by $e(0) = s \in S$ and $e(1) = t \in T$. Then the pullback over $e$ equals $(C_s)^{\rhd} \to \Delta^1$, which is obviously cocartesian. Thus, $p'$ is locally cocartesian.

For (2), by definition of the join we actually have an isomorphism of simplicial sets
\[ \Fun_{/S \star T}(C \star T, D \star T) \cong \Fun_{/S}(C, D), \]
under which functors preserving locally cocartesian edges are identified with each other. (3) follows by the same argument.

For (4), note that the hypotheses of \cite[Thm.~3.29]{Sha21} are satisfied because the zero category admits all limits, so any functor $F \in \Fun^{\cocart}_{/S \star T}(\sd(S \star T), C \star T)$ is necessarily a $(p \star \id_T)$-right Kan extension of its restriction to $\sd(S \star T)_0$. It follows that $j^{\ast}$ is an equivalence.
% For (4), note that the hypotheses of ? are satisfied because the zero category $\ast$ admits all limits, so $\Fun^{\cocart}_{/S \star T}(\sd(S \star T), C \star T)$ admits a recollement with open part $\Fun^{\cocart}_{/S}(\sd(S), C)$ and closed part $\Fun^{\cocart}_{/T}(\sd(T), T) \simeq \ast$. 
\end{proof}

% where the fibers of $p_n$ are stable $\infty$-categories and the pushforward functors are exact
\begin{cnstr} \label{cnstr:FamilyOfLocallyCocartesianFibrations} Suppose given a set $\{ p_n: C^n \to \Delta^n: n \geq 0\}$ of locally cocartesian fibrations, together with structure maps
\[ r_n: [0:n] \times_{\Delta^{n+1}} C^{n+1} \to C^n \]
over $\Delta^n \cong [0:n]$, where $r_n$ preserves locally cocartesian edges. Then, viewing $\Delta^n \subset \ZZ_{\geq 0}$ as the subcategory $[0:n]$, let
\[ r_n : C^{n+1} \star \ZZ_{>n+1} \to C^{n} \star \ZZ_{>n} \]
also denote the functor over $\ZZ_{\geq 0}$ obtained as in \cref{lm:joinLocallyCocartesian}, and let
\[ C^{\infty} \coloneq \lim_n \left( C^{n} \star \ZZ_{>n} \right) \]
be the locally cocartesian fibration over $\ZZ_{\geq 0}$, with the inverse limit taken over the maps $r_n$.

Suppose further that for all $n \geq 0$, we have functors $i_n: C^{n} \to C^{n+1}$ over the cosieve inclusion $\Delta^{n} \cong [1:n+1] \sub \Delta^{n+1}$ that preserve locally cocartesian edges, such that the commutative square
\[ \begin{tikzcd}[row sep=4ex, column sep=4ex, text height=1.5ex, text depth=0.25ex]
C^{n} \ar{r}{i_n} \ar{d} & C^{n+1} \ar{d} \\
\Delta^{n} \ar{r} & \Delta^{n+1}
\end{tikzcd} \]
is a homotopy pullback, and for all $n>0$, the diagram
\[ \begin{tikzcd}[row sep=4ex, column sep=6ex, text height=1.5ex, text depth=0.5ex]
C^n \ar{r}{i_{n}} \ar{d}[swap]{r_{n-1}} & C^{n+1} \ar{d}{r_n} \\
C^{n-1} \star \{n\} \ar{r}{i_{n-1}} & C^n \star \{n\}
\end{tikzcd} \]
is homotopy commutative (where we denote the various extensions of maps $i_n$ and $r_n$ by the same symbols). By \cref{lm:joinLocallyCocartesian}(4),
\[ \Fun^{\cocart}_{/ \ZZ_{\geq 0}}(\sd(\ZZ_{\geq 0}), C^n \star \ZZ_{>n}) \simeq \Fun^{\cocart}_{/\Delta^n}(\sd(\Delta^n), C^n) \]
under which the maps induced by postcomposing by $r_n$ are identified. Thus, we get that
\[ \Fun^{\cocart}_{/ \ZZ_{\geq 0}}(\sd(\ZZ_{\geq 0}), C^{\infty} ) \simeq \lim_n \Fun^{\cocart}_{/\Delta^n}(\sd(\Delta^n), C^n). \]

Under our assumptions, the diagram
\[ \begin{tikzcd}[row sep=4ex, column sep=4ex, text height=1.5ex, text depth=0.25ex]
\Fun^{\cocart}_{/\Delta^{n+1}}(\sd(\Delta^{n+1}),C^{n+1}) \ar{r}{(i_n)^{\ast}} \ar{d}{(r_n)_{\ast}} & \Fun^{\cocart}_{/\Delta^{n}}(\sd(\Delta^{n}),C^{n}) \ar{d}{(r_{n-1})_{\ast}} \\
\Fun^{\cocart}_{/\Delta^{n}}(\sd(\Delta^{n}),C^{n}) \ar{r}{(i_{n-1})^{\ast}} & \Fun^{\cocart}_{/\Delta^{n-1}}(\sd(\Delta^{n-1}),C^{n-1})
\end{tikzcd} \]
is homotopy commutative, so the maps $(i_n)^{\ast}$ assemble into a natural transformation
\[ i_{\bullet}^{\ast}: \Fun^{\cocart}_{/\Delta^{\bullet+1}}(\sd(\Delta^{\bullet+1}),C^{\bullet+1}) \to \Fun^{\cocart}_{/\Delta^{\bullet}}(\sd(\Delta^{\bullet}),C^{\bullet}). \]
Taking the inverse limit, we obtain an endofunctor $i_{\infty}^{\ast}$ of $\Fun^{\cocart}_{/\ZZ_{\geq 0}}(\sd(\ZZ_{\geq 0}), C^{\infty})$. On the other hand, the successor functor $s: \ZZ_{\geq 0} \to \ZZ_{\geq 0}$ induces a endofunctor $\sd(s)$ of $\sd(\ZZ_{\geq 0})$ that preserves locally cocartesian edges, and thus a `shift' endofunctor $\sh = \sd(s)^{\ast}$ of $\Fun^{\cocart}_{/\ZZ_{\geq 0}}(\sd(\ZZ_{\geq 0}), C^{\infty})$.
\end{cnstr}

\begin{lem} \label{lm:GeometricFixedPointsAsShift} We have an equivalence $\sh \simeq i^{\ast}_{\infty}$.
\end{lem}
\begin{proof} It suffices to check that for all $n \geq 0$, the diagram
\[ \begin{tikzcd}[row sep=4ex, column sep=4ex, text height=1.5ex, text depth=0.25ex]
\Fun^{\cocart}_{/\ZZ_{\geq 0}}(\sd(\ZZ_{\geq 0}), C^{\infty}) \ar{r}{\sh} \ar{d} & \Fun^{\cocart}_{/\ZZ_{\geq 0}}(\sd(\ZZ_{\geq 0}), C^{\infty}) \ar{d} \\
\Fun^{\cocart}_{/\Delta^{n+1}}(\sd(\Delta^{n+1}), C^{n+1}) \ar{r}{i_n^{\ast}} & \Fun^{\cocart}_{/\Delta^n}(\sd(\Delta^n), C^n)
\end{tikzcd} \]
is homotopy commutative. But this follows from the commutativity of the diagram
\[ \begin{tikzcd}[row sep=4ex, column sep=4ex, text height=1.5ex, text depth=0.25ex]
\Delta^n \ar{r} \ar{d} & \Delta^{n+1} \ar{d} \\
\ZZ_{\geq 0} \ar{r}{s} & \ZZ_{\geq 0}
\end{tikzcd} \]
where the upper map is the inclusion of $\Delta^n$ as the cosieve $[1:n+1] \subset \Delta^{n+1}$.
\end{proof}

Next, let $\sd_1(\ZZ_{\geq 0}) \subset \sd(\ZZ_{\geq 0})$ be the subposet on strings $\{ k \}$ and $\{ k< k+1 \}$, and let $$ \gamma_{\infty}^{\ast}: \Fun^{\cocart}_{/\ZZ_{\geq 0}}(\sd(\ZZ_{\geq 0}), C^{\infty}) \to \Fun^{\cocart}_{/\ZZ_{\geq 0}}(\sd_1(\ZZ_{\geq 0}), C^{\infty}) $$ be the functor given by restriction along the inclusion. Parallel to the setup in \cite[Obs.~4.16]{Sha21}, let $t_{\bullet}: \ZZ_{\geq 0} \to \Cat_{\infty}$ be the functor that sends $n$ to the fiber $C^{\infty}_n$ and $[n \to n+1]$ to the pushforward functor $t^{n+1}_n: C^{\infty}_n \to C^{\infty}_{n+1}$, and let $(C^{\infty})^{\vee} \to \ZZ_{\geq 0}^{\op}$ be the cartesian fibration classified by $t_{\bullet}$. Then we may replace the codomain of $\gamma_{\infty}^{\ast}$ as in \cite[Prop.~4.17]{Sha21} to instead write
$$ \gamma_{\infty}^{\ast}: \Fun^{\cocart}_{/\ZZ_{\geq 0}}(\sd(\ZZ_{\geq 0}), C^{\infty}) \to \Fun_{/\ZZ_{\geq 0}^{\op}}(\ZZ_{\geq 0}^{\op}, (C^{\infty})^{\vee}). $$

The functor $\gamma^{\ast}_{\infty}$ clearly commutes with the shift functor $\sh$ defined as $\sd(s)^{\ast}$ on the left and $(s^{\op})^{\ast}$ on the right, so we obtain a functor between the equalizers
\begin{equation} \label{eq:abstract_precomparison_functor}
 \Eq_{\id:\sh}(\Fun^{\cocart}_{/\ZZ_{\geq 0}}(\sd(\ZZ_{\geq 0}), C^{\infty})) \to \Eq_{\id:\sh}(\Fun_{/\ZZ_{\geq 0}^{\op}}(\ZZ_{\geq 0}^{\op}, (C^{\infty})^{\vee})).
\end{equation}

Note that under our assumptions, for all $0 \leq k \leq n$ we have equivalences
\[ \begin{tikzcd}[row sep=4ex, column sep=6ex, text height=1.5ex, text depth=0.5ex]
C^{n+1}_k \ar{d}{r_{n,k}} \ar{r}{\simeq} & C^{n}_{k-1} \ar{r}{\simeq} \ar{d}{r_{n-1,k-1}} & \cdots \ar{r}{\simeq} & C^{n-k+2}_{1} \ar{r}{\simeq} \ar{d}{r_{n-k+1,1}} & C^{n-k+1}_0 \ar{d}{r_{n-k,0}} \\
C^n_k \ar{r}{\simeq} &  C^{n-1}_{k-1} \ar{r}{\simeq} & \cdots \ar{r}{\simeq} & C^{n-k+1}_{1} \ar{r}{\simeq} & C^{n-k}_0,
\end{tikzcd} \]
hence we have equivalences $C^{\infty}_{n+1} \simeq C^{\infty}_{n}$ for all $n \geq 0$, under which $t^{n+2}_{n+1} \simeq t^{n+1}_n$. Therefore, if we let $C = C^{\infty}_0$ and $F = t^1_0: C \to C^{\infty}_1 \simeq C$ as an endofunctor of $C$, then with $\widehat{C} \to B \NN$ defined as in \cref{cnstr:SetupForEndofunctorFibration}, we have a homotopy pullback square
\[ \begin{tikzcd}[row sep=4ex, column sep=4ex, text height=1.5ex, text depth=0.25ex]
(C^{\infty})^{\vee} \ar{r} \ar{d} & \widehat{C} \ar{d} \\
\ZZ_{\geq 0}^{\op} \ar{r} & B \NN
\end{tikzcd} \]
and hence $\Fun_{/\ZZ_{\geq 0}^{\op}}(\ZZ_{\geq 0}^{\op}, (C^{\infty})^{\vee}) \simeq \Fun_{/B \NN}(\ZZ_{\geq 0}^{\op}, \widehat{C})$. \cref{lem:LaxEqualizerGenericEquivalence} then implies the equivalence
\begin{equation} \label{eq:laxequalizer_and_equalizer}
\Eq_{\id:\sh}(\Fun_{/\ZZ_{\geq 0}^{\op}}(\ZZ_{\geq 0}^{\op}, (C^{\infty})^{\vee})) \simeq \LEq_{\id:F}(C).
\end{equation}
Changing the target of the functor \eqref{eq:abstract_precomparison_functor} by the equivalence \eqref{eq:laxequalizer_and_equalizer}, we thereby obtain the generic comparison functor
\begin{equation} \label{eqn:AbstractComparisonFunctor} \Eq_{\id:\sh}(\Fun^{\cocart}_{/\ZZ_{\geq 0}}(\sd(\ZZ_{\geq 0}), C^{\infty})) \to \LEq_{\id:F}(C).
\end{equation}

Let us now return to our two situations of interest. In \cref{cnstr:FamilyOfLocallyCocartesianFibrations}, we may take either
\begin{enumerate}
\item $C^n = \Sp^{C_{p^n}}_{\locus{\phi}}$, the maps $r_n$ as in \cite[Constr.~2.54]{QS21a}, and the maps $i_n$ as in \cite[Constr.~2.55]{QS21a}.
\item $C^n = \Sp^{D_{2p^n}}_{\locus{\phi,\zeta}}$, the maps $r_n$ as in \cref{vrn:DihedralRestrictionGeometricLoci}, and the maps $i_n$ as in \cref{vrn:DihedralGeometricFixedPoints}.
\end{enumerate}

Let $\Sp^{C_{p^{\infty}}}_{\locus{\phi}}$ and $\Sp^{D_{2p^{\infty}}}_{\locus{\phi,\zeta}}$ be the resulting inverse limits as locally cocartesian fibrations over $\ZZ_{\geq 0}$, so we have equivalences
\begin{align*} \Theta: \Fun^{\cocart}_{/\ZZ_{\geq 0}}(\sd(\ZZ_{\geq 0}), \Sp^{C_{p^{\infty}}}_{\locus{\phi}}) \xto{\simeq} \Sp^{C_{p^{\infty}}}, \\
\Theta[\zeta]: \Fun^{\cocart}_{/\ZZ_{\geq 0}}(\sd(\ZZ_{\geq 0}), \Sp^{D_{2p^{\infty}}}_{\locus{\phi,\zeta}}) \xto{\simeq} \Sp^{D_{2p^{\infty}}}.
\end{align*}

By the identification of $\Phi^{C_p}$, resp. $\Phi^{\mu_p}$ as $i_n^{\ast}$ as observed in \cite[Constr.~2.54]{QS21a} and \cref{vrn:DihedralGeometricFixedPoints}, together with \cref{lm:GeometricFixedPointsAsShift}, we may identify the endofunctors $\Phi^{C_p}$ of $\Sp^{C_{p^{\infty}}}$ and $\Phi^{\mu_p}$ of $\Sp^{D_{2p^{\infty}}}$ with the shift endofunctors under the equivalences $\Theta$ and $\Theta[\zeta]$. Consequently, we obtain equivalences\footnote{Here, we implicitly use the equivalence of \cref{rem:EqualizerSwap}.}
\begin{align*} \widehat{\Theta}: \Eq_{\id:\sh}(\Fun^{\cocart}_{/\ZZ_{\geq 0}}(\sd(\ZZ_{\geq 0}), \Sp^{C_{p^{\infty}}}_{\locus{\phi}})) &\xto{\simeq} \CycSp_p^{\mr{gen}}, \\
\widehat{\Theta}[\zeta]: \Eq_{\id:\sh}(\Fun^{\cocart}_{/\ZZ_{\geq 0}}(\sd(\ZZ_{\geq 0}), \Sp^{D_{2p^{\infty}}}_{\locus{\phi,\zeta}})) &\xto{\simeq} \RCycSp_p^{\mr{gen}}.
\end{align*}

Moreover, being defined as the inverse limit over the restriction functors, the fibers of $\Sp^{C_{p^{\infty}}}_{\locus{\phi}}$ and $\Sp^{D_{2p^{\infty}}}_{\locus{\phi,\zeta}}$ are $\Fun(B C_{p^{\infty}}, \Sp)$ and $\Fun_{C_2}(B^t_{C_2} \mu_{p^{\infty}}, \underline{\Sp}^{C_2})$, and the adjacent pushforward endofunctors are $t^{C_p}$ and $t_{C_2} \mu_p$. Choosing inverses to $\widehat{\Theta}$ and $\widehat{\Theta}[\zeta]$, the functor \eqref{eqn:AbstractComparisonFunctor} then yields comparison functors
\begin{align*} \sU: \CycSp^{\mr{gen}}_p \to \CycSp_p, \\
\sU_{\RR}: \RCycSp^{\mr{gen}}_p \to \RCycSp_p.
\end{align*}

In general, for a functor $X: \sd(\ZZ_{\geq 0}) \to C^{\infty}$ over $\ZZ_{\geq 0}$, let $X_n \in C^{\infty}_n$ be the object given by evaluating $X$ at $n$. Then we may describe $\sU$ and $\sU_{\RR}$ on objects by the formulas
\begin{align*} \sU(X, X \simeq \Phi^{C_p} X) = (X_0, X_0 \simeq X_1 \to (X_0)^{t C_p}), \\
\sU_{\RR}(X,X \simeq \Phi^{\mu_p} X) = (X_0, X_0 \simeq X_1 \to (X_0)^{t_{C_2} \mu_p}).
\end{align*}

It is then clear that $\sU$ is equivalent to the functor of \cite[Prop.~II.3.2]{NS18} considered by Nikolaus and Scholze, and $\sU_{\RR}$ is equivalent to the functor defined at the beginning of this section. The main motivation behind our somewhat roundabout reformulation of the comparison functors is to leverage \cref{prp:EquivalenceOnBoundedBelowAtFiniteLevel} to prove analogous statements for $\sU$ and $\sU_{\RR}$. 

\begin{thm} \label{thm:MainTheoremRestated} $\sU$ and $\sU_{\RR}$ restrict to equivalences on the full subcategories of bounded-below, resp. underlying bounded-below objects.
\end{thm}
% Because the comparison functors in both cases are implemented by restriction along $\sd_1(\ZZ_{\geq 0}) \subset \sd(\ZZ_{\geq 0})$,
\begin{proof} Let $\widehat \Sp{}^{h C_{p^{\infty}}} \to B \NN$ and $\widehat \Sp {}^{h_{C_2} \mu_{p^{\infty}}} \to B \NN$ be the cartesian fibrations classified by the endofunctors $(-)^{t C_p}$ on $\Fun(B C_{p^{\infty}}, \Sp)$ and $(-)^{t_{C_2} \mu_p}$ on $\Fun_{C_2}(B_{C_2} \mu_{p^{\infty}}, \underline{\Sp}^{C_2})$. By taking the inverse limit of the equivalences of \cref{prp:EquivalenceOnBoundedBelowAtFiniteLevel}, we obtain equivalences
\begin{align*} \Fun^{\cocart}_{/\ZZ_{\geq 0}}(\sd(\ZZ_{\geq 0}), \Sp^{C_{p^{\infty}}}_{\locus{\phi}})_{bb} & \xto{\simeq} \Fun_{/B \NN}(\ZZ_{\geq 0}^{\op},\widehat \Sp{}^{h C_{p^{\infty}}})_{bb}, \\
\Fun^{\cocart}_{/\ZZ_{\geq 0}}(\sd(\ZZ_{\geq 0}), \Sp^{D_{2p^{\infty}}}_{\locus{\phi,\zeta}})_{ubb} & \xto{\simeq} \Fun_{/B \NN}(\ZZ_{\geq 0}^{\op},\widehat \Sp {}^{h_{C_2} \mu_{p^{\infty}}})_{ubb}.
\end{align*}
The functors $\sU$ and $\sU_{\RR}$ are induced by these functors through taking equalizers of the identity and shift functors on both sides, so the theorem follows.
\end{proof}

\subsection{Variant: \texorpdfstring{$\mathrm{O}(2)$}{O(2)}-spectra at the prime \texorpdfstring{$p$}{p}} \label{sec:variant}

We explain a variant of \cref{thm:MainTheoremRestated} involving $S^1$ resp. $\mathrm{O}(2)$ ``at the prime $p$'' rather than $C_{p^{\infty}}$ resp. $D_{2p^{\infty}}$.

\begin{ntn} \label{ntn:pfinite_family}
Let $\sF_p$ be the family of finite subgroups of $S^1$ contained in $C_{p^{\infty}}$, so $\sF_p \cong \ZZ_{\geq 0}$. Let $\sF'_p$ be the family of subgroups of $\mathrm{O}(2)$ contained in $D_{2p^{\infty}}$ and let $\zeta: \sF'_p \to \ZZ_{\geq 0}$ be as in \cref{def_zeta_map}. 
\end{ntn}

Let $\Sp^{S^1}_{\sF_p}$ be the $\infty$-category of $\sF_p$-complete $S^1$-spectra and let $\Sp^{\mathrm{O}(2)}_{\sF'_p}$ be the $\infty$-category of $\sF'_p$-complete $\mathrm{O}(2)$-spectra. To understand $\Sp^{S^1}_{\sF_p}$ and $\Sp^{\mathrm{O}(2)}_{\sF'_p}$, we may use the extension of the Ayala--Mazel-Gee--Rozenblyum reconstruction theorem to compact Lie groups \cite[Thm.~E]{AMGRb}. For any compact Lie group $G$, we again have a locally cocartesian fibration $\Sp^G_{\locus{\phi}} \to \fS[G]$ (where $\fS[G]$ is the nerve of the preordered set of \emph{closed} subgroups of $G$ ordered by subconjugacy) and a comparison functor
\[ \Theta: \Fun^{\cocart}_{/\fS[G]}(\sd(\fS[G]), \Sp^G_{\locus{\phi}}) \to \Sp^G \]
which is \emph{not} an equivalence unless $\fS[G]$ is down-finite. However, restricting to a down-finite sieve $\sF \subset \fS[G]$ yields the reconstruction equivalence
\begin{equation} \label{eqn:reconstruction_equivalence_LieGroup} \Theta: \Fun^{\cocart}_{/\sF}(\sd(\sF), \Sp^G_{\locus{\phi}}|_{\sF}) \xto{\simeq} \Sp^G_{\sF}. \end{equation}
Indeed, $\Sp^G_{\locus{\phi}}|_{\sF} \simeq (\Sp^G_{\sF})_{\locus{\phi}} \to \sF$ is the locally cocartesian fibration associated to the restricted stratification of $\Sp^G$ over $\sF$ in the sense of \cite[Thm.~B(1)]{AMGRb}. In particular, this applies to $\sF_p$ and $\sF'_p$. Moreover, using the pushforward stratification of $\Sp^{\mathrm{O}(2)}_{\sF'_p}$ along $\zeta$ in the sense of \cite[Thm.~B(4)]{AMGRb}, we obtain a locally cocartesian fibration $(\Sp^{\mathrm{O}(2)}_{\sF'_p})_{\locus{\phi,\zeta}} \to \ZZ_{\geq 0}$ whose right-lax limit again reconstructs $\Sp^{\mathrm{O}(2)}_{\sF'_p}$.

Unwinding the definitions, we see that $(\Sp^{S^1}_{\sF_p})_{\locus{\phi}}$ is almost the same as $\Sp^{C_{p^{\infty}}}_{\locus{\phi}}$, but where the fibers $\Fun(B C_{p^{\infty}}, \Sp)$ have been replaced by $\Fun(B S^1, \Sp)$. In particular, the cocartesian pushforward functors remain $(-)^{\tau C_{p^k}}$ over each $[n < n+k]$ in $\ZZ_{\geq 0}$, so the Tate orbit lemma still applies to simplify the subcategory of fiberwise bounded-below objects as in \cref{prp:EquivalenceOnBoundedBelowAtFiniteLevel}. Furthermore, the geometric fixed points endofunctor $\Phi^{C_p}$ of $\Sp^{S^1}_{\sF_p}$ again transports to the shift endofunctor of $\Fun^{\cocart}_{/\ZZ_{\geq 0}}(\sd(\ZZ_{\geq 0}), (\Sp^{S^1}_{\sF_p})_{\locus{\phi}})$ over the reconstruction equivalence. Running the same argument as above, we see that the forgetful functor
$$\sU: \Eq_{\id:\Phi^{C_p}}(\Sp^{S^1}_{\sF_p}) \to \LEq_{\id:t C_p}(\Sp^{h S^1})$$
restricts to an equivalence on bounded-below objects. Now, a similar unwinding of the definitions shows likewise that $(\Sp^{\mathrm{O}(2)}_{\sF'_p})_{\locus{\phi,\zeta}} \to \ZZ_{\geq 0}$ is given by replacing the fibers of $\Sp^{D_{2p^{\infty}}}_{\locus{\phi,\zeta}}$ with $\Fun_{C_2}(B^t_{C_2} S^1, \underline{\Sp}^{C_2})$ but preserving the cocartesian pushforward functors. Therefore, the same argument shows that the forgetful functor
$$\sU_{\RR}: \Eq_{\id:\Phi^{\mu_p}}(\Sp^{\mathrm{O}(2)}_{\sF'_p}) \to \LEq_{\id:t_{C_2} \mu_p}(\Sp^{h_{C_2} S^1})$$
restricts to an equivalence on underlying bounded-below objects.

\section{Extension to integral theories}\label{SS:Integral}

In this section, we define integral versions of genuine and Borel real $p$-cyclotomic spectra and extend \cref{thm:MainTheoremRestated} to a comparison of integral theories (\cref{thm:integral_comparison}).

\subsection{Borel real cyclotomic spectra}

\begin{ntn}
Given a $G$-$\infty$-category $C$ and a collection $\{ F_i \}_{i \in I}$ of $G$-endofunctors of $C$, let
$$\underline{\LEq}_{\id: (F_i)_{i\in I}}(C) \coloneq C \times_{(\id, F_i), {\prod}_{G,i \in I} (C \times_{\sO^{\op}_G} C), (\ev_0,\ev_1)} {\prod}_{G,i \in I} \Ar_{G}(C).$$
where ${\prod}_{G,i \in I}$ denote the $I$-indexed product of $G$-$\infty$-categories.\footnote{This is the $I$-indexed fiber product over $\sO^{\op}_G$.}
\end{ntn}

Let $\PP$ denote the set of prime numbers.

\begin{dfn}\label{dfn:realintegral}
The \emph{$C_2$-$\infty$-category of real cyclotomic spectra} is 
$$\underline{\RCycSp} := \underline{\LEq}_{\id: (\underline{t}_{C_2}\mu_p)_{p \in \PP}}(\underline{\Sp}^{h_{C_2}S^1}).$$
\end{dfn}

As usual, we write $\RCycSp$ for the fiber of $\underline{\RCycSp}$ over $C_2/C_2$ and observe that $\underline{\RCycSp}_{C_2/1}$ identifies with the Nikolaus--Scholze $\infty$-category $\CycSp$ of integral cyclotomic spectra. Note that by definition, $\RCycSp$ identifies with the fiber product
\begin{equation} \label{eq:pullback_integralrcyc}
\Sp^{h_{C_2} S^1} \times_{\prod_{p \in \PP} \Sp^{h_{C_2} S^1}} \prod_{p \in \PP} \LEq_{\id:t_{C_2} \mu_p}(\Sp^{h_{C_2} S^1}).
\end{equation}

\begin{rem} \label{rem:padic_equivalence}
Recall that for an inclusion $H \subset G$ of groups that is a $p$-adic equivalence, the restriction functor $\Fun(B G, \Sp)_p^{\wedge} \to \Fun(B H, \Sp)_p^{\wedge}$ is fully faithful. Since $D_{2p^\infty} \subset \mathrm{O}(2)$ is a $p$-adic equivalence and the induced map on Weyl groups of the subgroup $C_2$ is also a $p$-adic equivalence, this in particular implies that $(\Sp^{h_{C_2}S^1})_p^{\wedge} \to (\Sp^{h_{C_2}\mu_{p^\infty}})^{\wedge}_p$ is fully faithful. It follows that the forgetful functor
\[ \RCycSp \to \Sp^{h_{C_2}S^1} \times_{\prod_{p \in \PP} \Sp^{h_{C_2}\mu_{p^\infty}}} \prod_{p \in \PP} \RCycSp_p \]
restricts to an equivalence on the full subcategories of underlying bounded-below objects (using that for $p$-complete $X \in \Sp^{h_{C_2}S^1}$, the underlying spectrum of $X^{t_{C_2} \mu_p}$ is $p$-complete).\footnote{Beware that the statement of \cite[Prop.~II.3.4]{NS18} contains the erroneous assertion that the analogous statement for $\CycSp$ holds unconditionally.}
\end{rem}

We next enumerate the basic properties of $\underline{\RCycSp}$. By the same reasoning as for the $C_2$-$\infty$-category $\underline{\RCycSp}_p$, we see that:

\begin{enumerate}[leftmargin=*]
\item $\underline{\RCycSp}$ is $C_2$-stable and fiberwise presentable (hence $C_2$-bicomplete).
\item The forgetful $C_2$-functor $\underline{\RCycSp} \to \underline{\Sp}^{C_2}$ is fiberwise conservative, $C_2$-exact, and creates $C_2$-colimits and finite $C_2$-limits.
\item As a $C_2$-lax equalizer of the identity and the collection of $C_2$-symmetric monoidal endofunctors $\{ (-)^{t_{C_2} \mu_p}\}_{p \in \PP}$ (where we endow $\Fun_{C_2} (B^t_{C_2} S^1, \underline{\Sp}^{C_2})$ with the pointwise $C_2$-symmetric monoidal structure), $\underline{\RCycSp}$ acquires a $C_2$-symmetric monoidal structure such that the forgetful $C_2$-functor to $\underline{\Sp}^{C_2}$ is $C_2$-symmetric monoidal. By (2) it then follows that the $C_2$-symmetric monoidal structure on $\underline{\RCycSp}$ is $C_2$-distributive.
\item We have a $C_2$-adjunction
\[ \adjunct{\underline{\mr{triv}}_{\RR}}{\underline{\Sp}^{C_2}}{\underline{\RCycSp}}{\underline{\TCR}} \]
in which the $C_2$-right adjoint $\underline{\TCR}$ is $C_2$-corepresentable by the unit. Moreover, $\underline{\mr{triv}}_{\RR}$ uniquely acquires the structure of a $C_2$-symmetric monoidal functor as the unit map in $\CAlg_{C_2}(\underline{\Pr}^{L, \st}_{C_2})$, and its $C_2$-right adjoint $\underline{\TCR}$ is then lax $C_2$-symmetric monoidal.
\item For a real cyclotomic spectrum $[X \in \Sp^{h_{C_2}S^1}, \{ \varphi_p: X \to X^{t_{C_2} \mu_p} \}_{p \in \PP} ]$, we have a fiber sequence
\begin{equation} \label{eq:fundamental_fiber_sequence}
\TCR(X) \to X^{h_{C_2}S^1} \xtolong{\prod_{p \in \PP} (\varphi_p^{h_{C_2} S^1} - \can_p)}{2.8} \prod_{p \in \PP} (X^{t_{C_2}\mu_p})^{h_{C_2}S^1}.
\end{equation}
\end{enumerate}

When $X$ is $C_2$-bounded-below, we further claim that the final term in our fiber sequence can be identified with the profinite completion of $X^{t_{C_2} S^1}$. This is an immediate consequence of the following proposition (cf. \cite[Rem. II.4.3]{NS18} for the analogous statement regarding cyclotomic spectra):

\begin{prp}\label{Prp:ProfiniteCompletion}
If $X \in \Sp^{h_{C_2}S^1}$ is $C_2$-bounded-below, then $(X^{t_{C_2}\mu_p})^{h_{C_2}S^1}$ is $p$-complete and the canonical morphism $X^{t_{C_2}S^1} \to (X^{t_{C_2}\mu_p})^{h_{C_2}S^1}$ exhibits it as the $p$-completion of $X^{t_{C_2}S^1}$. 
\end{prp}

\begin{proof}
By \cref{rem:Tate_pcomplete}, $X^{t_{C_2} \mu_{p^n}}$ is $p$-complete for all $n$. By replacing homotopy fixed points and Tate constructions with their parametrized analogs in the proof of \cite[Lem. II.4.2]{NS18}, it is straightforward to reduce to showing the canonical map $X^{t_{C_2}S^1} \to \lim_n X^{t_{C_2}\mu_{p^n}}$ is a $p$-adic equivalence. This is the content of \cref{prp:S1vsLimit}. 
\end{proof}

\begin{prp}\label{prp:S1vsLimit}
If $X \in \Sp^{h_{C_2}S^1}$ is underlying bounded-below, then
\[
    X^{t_{C_2}S^1} \to \lim_n X^{t_{C_2}\mu_{p^n}}
\]
is a $p$-adic equivalence. 
\end{prp}

Before proving the proposition in general, we prove a special case:

\begin{lem}\label{lem:LimitUnderlyingZero}
\cref{prp:S1vsLimit} holds when the underlying spectrum of $X \in \Sp^{h_{C_2}S^1}$ is trivial.
\end{lem}

\begin{proof}
The $C_2$-spectra in question are trivial at the underlying level, so by \cref{lem:pcompletion} it suffices to check on $C_2$-geometric fixed points. We begin by analyzing the lefthand side. Applying $\Phi^{C_2}$ to the cofiber sequence of $C_2$-spectra (\cite[Exm.~5.57]{QS21a})
\[
    \Sigma^{\sigma} X_{h_{C_2}S^1} \to X^{h_{C_2}S^1} \to X^{t_{C_2}S^1}
\]
yields
\[
    (X^{\phi C_2})_{h\mu_2} \to (X^{\phi C_2})^{h\mu_2} \to (X^{\phi C_2})^{t\mu_2},
\]
using $\Phi^{C_2}(S^{\sigma}) \simeq S^0$, \cref{exm:geometric_fixed_points_and_parametrized_colimits}, and \cref{rem:UpperShriekPreservesParametrizedLimits}.

To complete the proof, we consider separately the cases of when $p$ is odd or even:
\begin{enumerate}
\item $p$ is an odd prime: Note that $(X^{\phi C_2})^{t\mu_2}$ is $2$-local, so if $p$ is odd, its $p$-completion is zero. On the other hand, if $p$ is odd then by \cite[Exm.~2.53]{QS21a} and \cref{lem:BoundedBelowEquivalence} we have that $X^{t_{C_2}\mu_{p^n}} \simeq 0$ for all $n$.
\item $p=2$: Note that for a tower $\{ X_n \}$ of spectra,
$$i^*(\lim_n i_* X_n) \simeq i^* i_* \lim_n X_n \simeq \lim_n X_n$$
since $i^*i_* = \id$. We then have
$$(\lim_n X^{t_{C_2}\mu_{2^n}})^{\phi C_2} \simeq \lim_n (X^{\phi C_2 t \mu_2})^{\oplus 2}.$$ 
The maps $B_{C_2}^t \mu_{2^n} \to B_{C_2}^t \mu_{2^{n+1}}$ send the distinct orbits $D_{2^{n+1}}/\mu_{2^n}$ and $D_{2^{n+1}}/\Delta$ to the same orbit $D_{2^{n+2}}/\mu_{2^{n+1}}$. Therefore, in the inverse limit we have
$$\lim_n (X^{\phi C_2 t \mu_2})^{\oplus 2} \simeq X^{\phi C_2 t \mu_2}.$$
\end{enumerate}
\end{proof}

% \footnote{Beware that if $X^{\phi C_2}$ is not bounded-below, then $X^{\phi C_2 t \mu_2}$ may not be $2$-complete.}

% First note that the proposition at the underlying level is justified in the proof of \cite[Lem. II.4.2]{NS18}.
\begin{proof}[Proof of \cref{prp:S1vsLimit}]
We begin with a series of reductions. First, consider the recollement
\[ \begin{tikzcd}
\Fun(B C_2, \Sp) \ar[shift left=3, hookrightarrow]{r}{j_!}  \ar[shift right=3, hookrightarrow]{r}[swap]{j_*} & \Sp^{C_2} \ar[shift left=3]{r}{i^*} \ar[shift right=3]{r}[swap]{i^!} \ar{l}[description]{j^*} \ar[hookleftarrow]{r}[description]{i_*} & \Sp. \end{tikzcd} \]
and the associated fiber sequence of $C_2$-spectra
$$j_!j^* X \to X \to i_* i^* X.$$
The proposition holds for $i_*i^*X$ by \cref{lem:LimitUnderlyingZero}, so it suffices to prove the proposition for $j_!j^*X$. This $C_2$-spectrum is $C_2$-bounded-below, so we may reduce to the case $X = H\underline{M}$ for some $C_2$-Mackey functor $\underline{M}$  by \cref{lem:TateConvergence}. 

We may further assume that $H\underline{M}$ has a trivial $C_2$-parametrized $S^1$-action.\footnote{By this, we mean that the $C_2$-functor $H\underline{M}: B^t_{C_2} S^1 \to \underline{\Mack}^{C_2}$ factors through $\sO^{\op}_{C_2}$, so the $\mathrm{O}(2)$-action on $\uM^e$ factors through the determinant map $\mathrm{O}(2) \to C_2$ and the $\mu_2$-action on $\uM^{C_2}$ is trivial. We don't mean to assert that the $C_2$-action on $\uM^e$ is trivial.} Indeed, consider the fiber sequence of $C_2$-spectra
$$j_! j^* H \uM \to H \uM \to i_*i^* H \uM.$$
Invoking \cref{lem:LimitUnderlyingZero} once more, it suffices to consider $j_! j^* H \uM$. But since the $S^1$-action on $\uM^e$ is necessarily trivial, the $C_2$-parametrized $S^1$-action on $j_! j^* H \uM$ is also trivial. We then apply \cref{lem:TateConvergence} to reduce the claim for $j_! j^* \uM$ to its Postnikov slices.

Any $C_2$-Mackey functor $\underline{M}$ with trivial $C_2$-parametrized $S^1$-action can be expressed as the cokernel of a map between free Mackey functors with trivial $C_2$-parametrized $S^1$-action $\uA_S \to \uA_T$ for some (possibly infinite) $C_2$-sets $S$ and $T$, so it suffices to prove the proposition for $X = H\uA_S$. 

The short exact sequence of $C_2$-Mackey functors
$$0 \to \uI_S \to \uA_S \to \uZ_S \to 0$$
implies that if the proposition holds for $H\uZ_S$ and $H\uI_S$, then it holds for $H\uA_S$. The proposition holds for $H\uI_S$ since $(H\uI_S)^e = 0$ by \cref{lem:LimitUnderlyingZero}. The proposition holds for $H\uZ_I$ by \cref{Lem:EquivForTorsionFree}. 
\end{proof}

\begin{lem}\label{Lem:EquivForTorsionFree}
Let $\uM = \underline{\ZZ}_S$, where $S$ is a $C_2$-set, equipped with trivial $C_2$-parametrized $S^1$-action. Then the map
$$\uM^{t_{C_2}S^1} \to \lim_n \uM^{t_{C_2}\mu_{p^n}}$$
is a $p$-adic equivalence. 
\end{lem}

\begin{proof}
The cofiber sequences
$$H\uM_{h_{C_2}\mu_{p^n}} \to H\uM^{h_{C_2}\mu_{p^n}} \to H\uM^{t_{C_2}\mu_{p^n}},$$
$$\Sigma^{\sigma} H\uM_{h_{C_2}S^1} \to H\uM^{h_{C_2}S^1} \to H\uM^{t_{C_2}S^1}$$
allow us to analyze the parametrized Tate construction in terms of parametrized homotopy fixed points and orbits. 

We begin by analyzing the parametrized homotopy fixed points and orbits with respect to $S^1$. Since the $C_2$-parametrized $S^1$-action on $H\uM$ is trivial, we have 
% \simeq F(E^t_{C_2}S^1_+, H\uM)^{S^1}
$$H\uM^{h_{C_2}S^1} \simeq F(B^t_{C_2}S^1_+, H\uM) \simeq F(\CC\PP^\infty_+, H\uM),$$
where we have used that $B^t_{C_2}S^1 \simeq \CC\PP^\infty$ with $C_2$-action given by complex conjugation in the last equivalence. The cofiber sequences
$$\CC\PP^{n-1}_+ \to \CC\PP^n_+ \to S^{2n,n}$$
% split after applying $F(-, H\uM)$ since $H\uM$ is a module over the real oriented spectrum $H\underline{\ZZ}$
split after tensoring with the real oriented spectrum $H\underline{\ZZ}$. Since $\CC\PP^\infty \simeq \colim_n \CC\PP^n$, we find that 
\begin{align*}
H\uM^{h_{C_2}S^1} & \simeq F(\CC\PP^\infty_+, H\uM) \simeq F_{H\underline{\ZZ}}(\CC\PP^\infty_+ \otimes H\underline{\ZZ}, H\uM) \\ 
& \simeq F_{H\underline{\ZZ}} (\colim_n \bigoplus_{i=0}^n \Sigma^{2i,i}H\underline{\ZZ},H\uM) \simeq \lim_n \bigoplus_{i=0}^n \Sigma^{-2i,-i} H\uM.
\end{align*}
Similarly, we find that
$$H\uM_{h_{C_2}S^1} \simeq H\uM \otimes \CC\PP^\infty_+ \simeq  \colim_n \bigoplus_{i=0}^n \Sigma^{2i,i} H\uM.$$
Now observe that the $C_2$-parametrized $S^1$-norm is trivial for degree reasons. More precisely, the norm is trivial for $\uM = \underline{\ZZ}$ since $\pi_{**}^{C_2}(H\underline{\ZZ})$ vanishes in the relevant bidegrees \cite{Greenlees_FourApproaches}.\footnote{One can also read this off from the computation for $\pi_{**}^{C_2}(H \ZZ_2)$ presented in \cite[\S 9]{BehrensShah} by replacing all instances of $\ZZ_2$ with $\ZZ$ in \cite[Fig.~9.2]{BehrensShah}.} Since $\uM$ is a $C_2$-indexed coproduct of copies of $\uZ$, the same vanishing holds in $\pi_{**}^{C_2}(H\underline{M})$. Thus
$$H\uM^{t_{C_2}S^1} \simeq \colim_{m \rightarrow -\infty} \lim_{n \rightarrow \infty} \bigoplus_{i=m}^n \Sigma^{-2i,-i} H\uM.$$

We now analyze the parametrized homotopy fixed points and orbits with respect to the $\mu_{p^n}$-action. The short exact sequence of abelian groups (with $C_2$-action given by inversion)
$$0 \to \mu_{p^n} \to S^1 \xrightarrow{\cdot p^n} S^1 \to 0$$
% gives rise to a cofiber sequence of $C_2$-spaces
% $$B^t_{C_2}\mu_{p^n} \to B^t_{C_2}S^1 \xrightarrow{\cdot p^n} B_{C_2}^t S^1.$$
gives rise to a fibration of $C_2$-spaces
\[ S^{\sigma} \to B^t_{C_2}\mu_{p^n} \to B^t_{C_2}S^1. \]
Since $\uM$ is torsion-free, the associated Gysin sequence\footnote{The Gysin sequence for real oriented theories is discussed, for example, in the paragraph after \cite[Lem.~3.1]{KW08}. The discussion there generalizes to our situation by replacing $2$ by $p^n$ everywhere.} implies that 
$$H\uM^{h_{C_2}\mu_{p^n}} \simeq \uM \oplus \lim_n \bigoplus_{i=1}^n \Sigma^{-2i,-i} H\uM/p^n,$$
and similarly,
$$H\uM_{h_{C_2}\mu_{p^n}} \simeq \uM \oplus \colim_n \bigoplus_{i=1}^n \Sigma^{2i-1,i} H\uM/p^n.$$
As in the $S^1$-case, the norm is trivial for degree reasons except on the unshifted copy of $\uM$ in each expression, where it is given by multiplication by $p^n$. Therefore
$$H\uM^{t_{C_2}\mu_{p^n}} \simeq \colim_{m \rightarrow -\infty} \lim_{n \rightarrow \infty} \bigoplus_{i=m}^n \Sigma^{-2i,-i} H\uM/p^n.$$

The map $H\uM^{t_{C_2}S^1} \to H\uM^{t_{C_2}\mu_{p^n}}$ is the evident quotient map on each summand. Therefore, the map $H\uM^{t_{C_2}S^1} \to \lim_n H\uM^{t_{C_2}\mu_{p^n}}$ is a $p$-adic equivalence. 
\end{proof}

\begin{rem}
If $\uM = \underline{\ZZ}$, then we may identify the homotopy groups of $H\uM^{t_{C_2}\mu_{p^n}}$ as a ring: 
$$\pi_{**}^{C_2}(H\underline{\ZZ}^{t_{C_2}\mu_{p^n}}) \simeq H\underline{\ZZ}_{**}^{C_2}((x))/(p^n x).$$
The case $n=1$ can be proven using the Gysin sequence and the fact that $H\underline{\ZZ}$ is real orientable; see, for instance, \cite[Thm. 5.7]{LLQ19}. 
\end{rem}

\subsection{\texorpdfstring{$\cF$}{F}-genuine real cyclotomic spectra}

% We next introduce the $C_2$-$\infty$-category of $\cF$-genuine real cyclotomic spectra, where $\cF$ generically denotes the family of finite subgroups.

Let $\cF$ generically denote the family of finite subgroups.

\begin{dfn} \label{dfn:O2_spectra}
Let $\Sp^{\mathrm{O}(2)}_{C_2} \coloneq [\res: \Sp^{\mathrm{O}(2)} \to \Sp^{S^1}]$ denote the $C_2$-$\infty$-category defined by the $C_2$-equivariant restriction functor $\res$ (with respect to the trivial action on the source and the residual $C_2$-conjugation action on the target). Let $\Sp^{\mathrm{O}(2)}_{C_2,\cF}$ be the fiberwise $\cF$-completion of $\Sp^{\mathrm{O}(2)}_{C_2}$.\footnote{In \cite[Def.~II.2.15]{NS18}, Nikolaus and Scholze define $\mathbb{T}\Sp_{\cF}$ as the localization of the $\infty$-category of $S^1$-equivariant spectra at those maps $f: X \to Y$ such that the restriction of $f$ along any finite subgroup $C_n \subset S^1$ is an equivalence. Note that this yields precisely the $\infty$-category $\Sp^{S^1}_{\cF}$ since this condition on the map $f$ is equivalent to demanding that $f^{\phi C_n}: X^{\phi C_n} \xto{\simeq} Y^{\phi C_n}$ for all $C_n \subset S^1$.}
\end{dfn}

In \cref{dfn:O2_spectra}, we implicitly used that if $f:X \to Y \in \Sp^{\mathrm{O}(2)}$ is such that for all finite subgroups $H \leq \mathrm{O}(2)$, $f^{\phi H}: X^{\phi H} \xto{\simeq} Y^{\phi H}$, we have that $\res(f)^{\phi K}$ is an equivalence for all finite subgroups $K \leq S^1$, hence $r$ descends to the $\cF$-completion. Similarly, the geometric fixed points $C_2$-endofunctors $\{ \underline{\Phi}^{\mu_p} \}_{p \in \PP}$ of $\Sp^{\mathrm{O}(2)}_{C_2}$ descend to the $\cF$-completion. They also pairwise commute as $C_2$-endofunctors of $\Sp^{\mathrm{O}(2)}_{C_2,\cF}$ and so define an action of the multiplicative monoid $\NN_{>0}$ on $\Sp^{\mathrm{O}(2)}_{C_2,\cF}$ in $\Cat_{\infty, C_2}$.

\begin{dfn}\label{dfn:realintegralgen}
The \emph{$C_2$-$\infty$-category of $\cF$-genuine real cyclotomic spectra} is
$$\underline{\RCycSp}^{\mr{gen}} \coloneq (\Sp^{\mathrm{O}(2)}_{C_2,\cF})^{h \NN_{> 0}}.$$
\end{dfn}

We enumerate the basic properties of $\underline{\RCycSp}^{\mr{gen}}$ that are established by parallel reasoning to that used for $\underline{\RCycSp}^{\mr{gen}}_p$:

\begin{enumerate}[leftmargin=*]
\item $\Sp^{\mathrm{O}(2)}_{C_2}$, $\Sp^{\mathrm{O}(2)}_{C_2, \cF}$, and $\underline{\RCycSp}^{\mr{gen}}$ are $C_2$-stable and fiberwise presentable (hence $C_2$-bicomplete).
\item The forgetful $C_2$-functor $\underline{\RCycSp}^{\mr{gen}} \to \Sp^{\mathrm{O}(2)}_{C_2,\cF}$ is fiberwise conservative, $C_2$-exact, and creates $C_2$-colimits and finite $C_2$-limits.
\item The endofunctors $\{ \Phi^{\mu_p} \}_{p \in \PP}$ of $\Sp^{\mathrm{O}(2)}_{\cF}$ are symmetric monoidal and colimit preserving, so the action of $\NN_{>0}$ on $\Sp^{\mathrm{O}(2)}_{\cF}$ lifts to $\CAlg(\Pr^{L, \st})$. Therefore, $\RCycSp^{\mr{gen}}$ acquires a distributive symmetric monoidal structure such that the forgetful functor to $\Sp^{\mathrm{O}(2)}_{\cF}$ is symmetric monoidal.
\item We have a $C_2$-adjunction
\[ \adjunct{\underline{\mr{triv}}^{\mr{gen}}_{\RR}}{\underline{\Sp}^{C_2}}{\underline{\RCycSp}^{\mr{gen}}}{\underline{\TCR}^{\mr{gen}}} \]
in which $\underline{\TCR}^{\mr{gen}}$ is $C_2$-corepresentable by the unit. Here, $\underline{\mr{triv}}^{\mr{gen}}_{\RR}$ is defined to be the unique $C_2$-colimit preserving $C_2$-exact functor that selects the unit in $\RCycSp^{\mr{gen}}$. Furthermore, $\mr{triv}^{\mr{gen}}_{\RR}$ canonically acquires a symmetric monoidal structure such that the composite functor to $\Sp^{\mathrm{O}(2)}_{\cF}$ is the $\cF$-completion of the symmetric monoidal inflation functor $\inf^{\mathrm{O}(2)}_{C_2}: \Sp^{C_2} \to \Sp^{\mathrm{O}(2)}$.
\end{enumerate}

\noindent To improve upon (4) and define the $C_2$-symmetric monoidal structure on $\underline{\RCycSp}^{\mr{gen}}$, we need to introduce a few more constructions.

\begin{cnstr} \label{con:C2_smc_O2_spectra}
We equip $\Sp^{\mathrm{O}(2)}_{C_2}$ with a $C_2$-distributive $C_2$-symmetric monoidal structure by imitating the construction in \cite[\S 9.2]{BachmannHoyoisNorms} of the $G$-symmetric monoidal structure on $\underline{\Sp}^G$ for a finite group $G$. First, recall that for a compact Lie group $G$, by Elmendorf's theorem the $\infty$-category $\Spc^G$ of $G$-spaces is described by the $\infty$-category of $\Spc$-valued presheaves on the (topological) orbit category $\sO_G$. Furthermore, the $\infty$-category $\Sp^G$ of $G$-spectra is obtained by inverting the set of real representation spheres $\{ S^V \}_{V \in \cU}$ in $(\Spc^G)_{\ast}$, where $\cU$ denotes a complete $G$-universe; more precisely, if we let $\mr{Sub}_{\cU}$ be the poset of finite-dimensional subrepresentations of $\cU$, then one forms $\Sp^G$ as the (filtered) colimit of the functor $\mr{Sub}_{\cU} \to \Cat^{\omega}_{\infty}$ (valued in the $\infty$-category of compactly generated $\infty$-categories and compact-object preserving left adjoints) that sends each $V$ to $\Spc^G_{\ast}$ and each inclusion $V \subset W$ to $\Sigma^{W - V}: \Spc^G_{\ast} \to \Spc^G_{\ast}$ for $W-V$ the orthogonal complement of $V$ in $W$. See \cite[App.~C]{gepnermeier} for a comparison of this $\infty$-category to that which underlies orthogonal $G$-spectra.

Returning to our situation, let $\Spc^{\mathrm{O}(2)}_{C_2} \coloneq [\res: \Spc^{\mathrm{O}(2)} \to \Spc^{S^1}]$ denote the $C_2$-$\infty$-category defined by the unstable $C_2$-equivariant restriction functor $\res$, and consider the right Kan extension of $\Spc^{\mathrm{O}(2)}_{C_2}$ to the domain $\FF_{C_2}^{\op}$ (equating $\Cat_{\infty}$-valued functors with cocartesian fibrations). Since $r$ admits a right adjoint given by the coinduction functor $\Coind$ such that the Beck--Chevalley property is satisfied, via Barwick's unfurling construction \cite[\S 11]{M1} $\Spc^{\mathrm{O}(2)}_{C_2}$ extends to a product-preserving functor
$$(\Spc^{\mathrm{O}(2)}_{C_2})^{\times} : \Span(\FF_{C_2}) \to \Cat_{\infty}$$
that encodes the $C_2$-cartesian $C_2$-symmetric monoidal structure on $\Spc^{\mathrm{O}(2)}_{C_2}$. Moreover, this $C_2$-symmetric monoidal structure is $C_2$-distributive. Indeed, since $\Spc^{\mathrm{O}(2)}_{C_2}$ is fiberwise cartesian closed, it suffices to check that $\Coind$ preserves sifted colimits and satisfies the binomial formula
$$\Coind(X \sqcup Y) \simeq \Coind(X) \sqcup \Coind(Y) \sqcup \Ind(X \sqcup Y).$$
But under the equivalence $\Spc^G \simeq \categ{P}_{\Sigma}(\FF^G)$,\footnote{Here $\FF^G$ denotes the finite coproduct completion of $\sO_G$ and we refer to its objects as \emph{orbit-finite $G$-spaces}.} note that the restriction, induction, and coinduction functors are obtained by left Kan extension of the same functors defined at the level of orbit-finite $\mathrm{O}(2)$ and $S^1$-spaces. Therefore, $\Coind$ preserves sifted colimits and the verification of the binomial formula reduces to that for coproducts of orbit-finite $S^1$-spaces, which is clear.

Note also that $(\Spc^{\mathrm{O}(2)}_{C_2})^{\times}$ may be defined as the composite of the functor
$$(\FF^{\mathrm{O}(2)}_{C_2})^{\times}: \Span(\FF_{C_2}) \to \Cat_{\infty}$$
with $\categ{P}_{\Sigma}$, where $(\FF^{\mathrm{O}(2)}_{C_2})^{\times} \subset (\Spc^{\mathrm{O}(2)}_{C_2})^{\times}$ is the product-preserving subfunctor that sends $C_2/1$ to $\FF^{S^1}$ and $C_2/C_2$ to $\FF^{\mathrm{O}(2)}$. To now define the ``smash product'' $C_2$-symmetric monoidal structure on $\Spc^{\mathrm{O}(2)}_{C_2, \ast}$, we may proceed as in \cite[\S 9.2]{BachmannHoyoisNorms}. Namely, observe that for any compact Lie group $G$, $\FF^G_*$ embeds as a wide subcategory of $\Span(\FF^G)$ under the functor that sends a pointed orbit-finite $G$-space $X_+$ to $X$ and a pointed $G$-equivariant map $f: X_+ \to Y_+$ to the span $[X \ot f^{-1}(Y) \to Y]$. Now define $(\FF^{\mathrm{O}(2)}_{C_2, \ast})^{\otimes}$ to be the product-preserving subfunctor of $\Span \circ (\FF^{\mathrm{O}(2)}_{C_2})^{\times}$ that sends $C_2/1$ to $\FF^{S^1}_{\ast}$ and $C_2/C_2$ to $\FF^{\mathrm{O}(2)}_{\ast}$,\footnote{As in the proof of \cite[Lem.~9.5(4)]{BachmannHoyoisNorms}, for $p: C_2/1 \to C_2/C_2$, $p_{\otimes} = \Span(p_*): \Span(\FF^{S^1}) \to \Span(\FF^{\mathrm{O}(2)})$ restricts to a functor $p_{\otimes}: \FF^{S^1}_* \to \FF^{\mathrm{O}(2)}_*$ since $p_*$ preserves monomorphisms as a left-exact functor.} and let $(\Spc^{\mathrm{O}(2)}_{C_2, \ast})^{\otimes} \coloneq \categ{P}_{\Sigma} \circ (\FF^{\mathrm{O}(2)}_{C_2, \ast})^{\otimes}$. Using the same reasoning as with $(\FF^{\mathrm{O}(2)}_{C_2})^{\times}$ but applied to orbit-finite pointed $S^1$ and $\mathrm{O}(2)$-spaces, we see that $(\Spc^{\mathrm{O}(2)}_{C_2, \ast})^{\otimes}$ is $C_2$-distributive.

For any compact Lie group $G$, since for every orthogonal real $G$-representation $V$ the associated representation sphere $S^V$ is a symmetric object in pointed $G$-spaces \cite[Lem.~C.5]{gepnermeier}, the stabilization functor $\Sigma^{\infty}: \Spc^G_{\ast} \to \Sp^G$ refines to a symmetric monoidal functor of presentably symmetric monoidal $\infty$-categories (cf. \cite[Cor.~C.7]{gepnermeier} or \cite[Lem.~4.1]{BachmannHoyoisNorms}). To see that $(\Spc^{\mathrm{O}(2)}_{C_2, \ast})^{\otimes}$ in addition $C_2$-stabilizes to define a product-preserving functor
\[ (\Sp^{\mathrm{O}(2)}_{C_2})^{\otimes}: \Span(\FF_{C_2}) \to \Cat_{\infty} \]
endowing $\Sp^{\mathrm{O}(2)}_{C_2}$ with its $C_2$-symmetric monoidal structure, it suffices to show that the unstable pointed norm $\bigwedge_{C_2}: \Spc^{S^1}_* \to \Spc^{\mathrm{O}(2)}_*$ prolongs to a functor $N^{C_2}: \Sp^{S^1} \to \Sp^{\mathrm{O}(2)}$.\footnote{Once this is shown, we will have a specific collection of stable norm functors $f_{\otimes}$ defined for all maps $f$ of finite $C_2$-sets. The coherence problem of assembling these norm functors to define a functor out of $\Span(\FF_{C_2})$ is then solved via the methods of \cite[\S 6.1]{BachmannHoyoisNorms}; cf. the discussion immediately prior to \cite[Rem.~6.3]{BachmannHoyoisNorms} or prior to \cite[Rem.~9.10]{BachmannHoyoisNorms}.} Now by \cite[Rem.~4.2]{BachmannHoyoisNorms}, $\Sp^G \simeq \Spc^G_*[(S^V)^{-1}]_{V \in \cU}$ in the larger $\infty$-category of symmetric monoidal $\infty$-categories that admit sifted colimits and whose tensor product distributes over sifted colimits, together with those symmetric monoidal functors that preserve sifted colimits. Therefore, it suffices to show that for every finite-dimensional real representation $V$ of $S^1$, the composite symmetric monoidal and sifted-colimit preserving functor $\Sigma^{\infty} \circ \bigwedge_{C_2}: \Spc^{S^1}_* \to \Sp^{\mathrm{O}(2)}$ sends $S^V$ to an invertible object of $\Sp^{\mathrm{O}(2)}$. As with any finite-index subgroup $H$ of a compact Lie group $G$, this follows from the formula $\bigwedge_{C_2} (S^V) \simeq S^{\Ind V}$.

Finally, observe that $\Sigma^{\infty}: \Spc^{\mathrm{O}(2)}_{C_2,*} \to \Sp^{\mathrm{O}(2)}_{C_2}$ is $C_2$-symmetric monoidal and $C_2$-colimit preserving (being fiberwise a left adjoint and intertwining with the induction functors). Therefore, since $\Sp^{\mathrm{O}(2)}_{C_2}$ is generated fiberwise under sifted colimits by desuspensions of $\Sigma^{\infty} X$, the binomial formula for the norm functor $N^{C_2}: \Sp^{S^1} \to \Sp^{\mathrm{O}(2)}$ holds by virtue of $(\Spc^{\mathrm{O}(2)}_{C_2,*})^{\otimes}$ being $C_2$-distributive, and it follows that $(\Sp^{\mathrm{O}(2)}_{C_2})^{\otimes}$ is $C_2$-distributive.
\end{cnstr}

\begin{rem} \label{rem:topstacks}
Let $G$ be a compact Lie group and $H \leq G$ a finite-index normal subgroup. We then have the $G/H$-stable $G/H$-$\infty$-category $\Sp^G_{G/H}$ whose fiber over the orbit $G/K$, $H \leq K$ is given by $\Sp^K$, with the functoriality that of restriction and conjugation. An elaboration of \cref{con:C2_smc_O2_spectra} endows $\Sp^G_{G/H}$ with its canonical $G/H$-distributive $G/H$-symmetric monoidal structure.

More generally, these structures should all come packaged into a single functor
\[ \SH^{\otimes}: \Span(\categ{TopStk}, \mr{all}, \mr{fcov}) \to \CAlg(\Cat^{\sift}_{\infty}) \]
where $\categ{TopStk}$ is an $\infty$-category of ``nice'' topological stacks and the covariant maps are restricted to be $0$-truncated finite covering maps; cf. \cite[Rem.~11.8]{BachmannHoyoisNorms}. We note that the various coherence problems encountered below (e.g., in \cref{con:geometric_fixedpoints_C2sm}) are also solvable as instances of the compatibilities recorded by the putative functor $\SH^{\otimes}$.
\end{rem}

\begin{cnstr} \label{con:pass_to_completion}
The $C_2$-symmetric monoidal structure $\Sp^{\mathrm{O}(2)}_{C_2}$ descends to $\Sp^{\mathrm{O}(2)}_{C_2, \cF}$ such that the $\cF$-completion $C_2$-functor is $C_2$-symmetric monoidal. We can use the multiplicative theory of the parametrized Verdier quotient \cite[\S 5.2]{QS21a} to show this. Let $I \subset \Sp^{\mathrm{O}(2)}_{C_2}$ be the full $C_2$-subcategory given fiberwise by those $X$ whose $\cF$-completion is $0$, and note that $X \in I$ if and only if $\res^H X \simeq 0 \in \Sp^H$ for all finite subgroups $H$ (interpreted in $\Sp^{\mathrm{O}(2)}$ or $\Sp^{S^1}$). Then $I$ is $C_2$-stable and is a fiberwise $\otimes$-ideal such that $\Sp^{\mathrm{O}(2)}_{C_2, \cF}$ is the $C_2$-Verdier quotient $(\Sp^{\mathrm{O}(2)}_{C_2})/I$ \cite[Def.~5.21]{QS21a}. We claim moreover that $I$ is a $C_2$-$\otimes$-ideal \cite[Def.~5.24]{QS21a}, for which it suffices to show that for every $X \in I_{C_2/1}$, we have that $N^{C_2} X \in I_{C_2/C_2}$. But to see that $\res^{H} N^{C_2} X \simeq 0$ for all finite $H \leq \mathrm{O}(2)$, we can use the base-change formula
\begin{equation} \label{eqn:base_change_restriction_norm} \res^H N^{C_2} \simeq \bigotimes_{S^1 \setminus \mathrm{O}(2) / H} N^{H}_{H \cap S^1} \res^{H \cap S^1}
\end{equation}
associated to the pullback square of spaces
\[ \begin{tikzcd}
\coprod_{S^1 \setminus \mathrm{O}(2) / H} B(H \cap S^1) \ar{r} \ar{d} & B H \ar{d} \\ 
B S^1 \ar{r} & B \mathrm{O}(2),
\end{tikzcd} \]
where the formula is proven as the genuine stabilization of the evident unstable base-change formula involving the restriction and coinduction functors. Then by \cite[Thm.~5.28]{QS21a} (and noting fiberwise accessibility of the inclusion $I \subset \Sp^{\mathrm{O}(2)}_{C_2}$), the claim follows.
\end{cnstr}

\begin{cnstr} \label{con:geometric_fixedpoints_C2sm}
For every prime $p$, we may promote the $C_2$-endofunctor $\underline{\Phi}^{\mu_p}$ of $\Sp^{\mathrm{O}(2)}_{C_2}$ to a $C_2$-symmetric monoidal endofunctor as follows. Consider the commutative square of compact Lie groups
\[ \begin{tikzcd}
S^1 \ar[->>]{r} \ar[hookrightarrow]{d} & S^1/\mu_p \cong S^1 \ar[hookrightarrow]{d} \\ 
\mathrm{O}(2) \ar[->>]{r} & \mathrm{O}(2)/\mu_p \cong \mathrm{O}(2).
\end{tikzcd} \]
Associated to this one has the commutative square of restriction functors
\[ \begin{tikzcd}
\Spc^{S^1} & \Spc^{S^1} \ar{l} \\ 
\Spc^{\mathrm{O}(2)} \ar{u} & \Spc^{\mathrm{O}(2)}. \ar{u} \ar{l}
\end{tikzcd} \]
Passing to right adjoints and then genuine stabilizations, one obtains a commutative square
\[ \begin{tikzcd}
\Sp^{S^1} \ar{r}{\Phi^{\mu_p}} \ar{d}{N^{C_2}} & \Sp^{S^1} \ar{d}{N^{C_2}} \\ 
\Sp^{\mathrm{O}(2)} \ar{r}{\Phi^{\mu_p}} & \Sp^{\mathrm{O}(2)}.
\end{tikzcd} \]
By similar methods to those used to construct $(\Sp^{\mathrm{O}(2)}_{C_2})^{\otimes}$ above, we may elaborate this commutative diagram to define $(\underline{\Phi}^{\mu_p})^{\otimes}$ as a natural transformation $(\Sp^{\mathrm{O}(2)}_{C_2})^{\otimes} \Rightarrow (\Sp^{\mathrm{O}(2)}_{C_2})^{\otimes}$ over $\Span(\FF_{C_2})$.

Thereafter, $\underline{\Phi}^{\mu_p}$ descends to a $C_2$-symmetric monoidal endofunctor of $\Sp^{\mathrm{O}(2)}_{C_2, \cF}$ in view of the universal property of the $C_2$-Verdier quotient \cite[Thm.~5.28]{QS21a}.
\end{cnstr}

Under \cref{con:geometric_fixedpoints_C2sm}, we may form $\underline{\RCycSp}^{\mr{gen}} = (\Sp^{\mathrm{O}(2)}_{C_2, \cF})^{h \NN_{>0}}$ as the limit in the $\infty$-category $\CAlg_{C_2}(\underline{\Pr}^{L, \st}_{C_2})$ of $C_2$-stable $C_2$-presentably symmetric monoidal $\infty$-categories, and thereby endow $\underline{\RCycSp}^{\mr{gen}}$ with its $C_2$-distributive $C_2$-symmetric monoidal structure. As the unit map in $\CAlg_{C_2}(\underline{\Pr}^{L, \st}_{C_2})$, $\underline{\mr{triv}}^{\mr{gen}}_{\RR}$ then uniquely acquires the structure of a $C_2$-symmetric monoidal functor, and thus its $C_2$-right adjoint $\underline{\TCR}^{\mr{gen}}$ obtains the structure of a lax $C_2$-symmetric monoidal functor.
% As we did with $\underline{\mr{triv}}^{\mr{gen}}_{\RR,p}$ in \cref{sec:C2_smc}, we may then uniquely refine $\underline{\mr{triv}}^{\mr{gen}}_{\RR}$ to a $C_2$-symmetric monoidal functor, which then equips its $C_2$-right adjoint $\underline{\TCR}^{\mr{gen}}$ with its lax $C_2$-symmetric monoidal structure. 
% Here, instead of using the functor $\SH^{\otimes}$ to construct the $C_2$-symmetric monoidal structure as was done with $\underline{\inf}^{\mu_{p^{\infty}}}: \Sp^{C_2} \to \Sp^{D_{2p^{\infty}}}_{C_2}$, we may directly construct the $C_2$-symmetric monoidal structure on $\underline{\inf}^{\mathrm{O}(2)}: \underline{\Sp}^{C_2} \to \Sp^{\mathrm{O}(2)}_{C_2}$ as the genuine stabilization of its unstable counterpart. Alternatively, $\underline{\mr{triv}}^{\mr{gen}}_{\RR}$ obtains as the unit map of $\underline{\RCycSp}^{\mr{gen}}$ as a $C_2$-commutative algebra in $\underline{\Pr}^{L, \st}_{C_2}$.

% Explain this
\subsection{Integral comparison} \label{sec:integral_comparison}

We have a forgetful $C_2$-symmetric monoidal functor
\begin{equation}\label{Eqn:ForgetIntegral}
\underline{\sU}_{\RR}: \underline{\RCycSp}^{\mr{gen}} \to \underline{\RCycSp} \simeq \Sp^{\underline{h}_{C_2} S^1} \times_{\prod_{p \in \PP} \Sp^{\underline{h}_{C_2} S^1}} \prod_{p \in \PP} \underline{\LEq}_{\id:\underline{t}_{C_2} \mu_p}(\Sp^{\underline{h}_{C_2} S^1})
\end{equation}
(where $\Sp^{\underline{h}_{C_2} S^1} \coloneq \underline{\Fun}_{C_2}(B^t_{C_2} S^1, \underline{\Sp}^{C_2})$) that is defined as follows:

\begin{itemize}[leftmargin=4ex]
\item[($\ast$)] For every prime $p$, let $\Sp^{\mathrm{O}(2)}_{C_2, \cF_p} \coloneq [\res: \Sp^{\mathrm{O}(2)}_{\cF'_p} \to \Sp^{S^1}_{\cF_p}]$ (cf. \cref{ntn:pfinite_family}). By the same reasoning as in \cref{con:pass_to_completion}, the $C_2$-symmetric monoidal structure on $\Sp^{\mathrm{O}(2)}_{C_2}$ along with its $C_2$-symmetric monoidal endofunctor $\underline{\Phi}^{\mu_p}$ descends to $\Sp^{\mathrm{O}(2)}_{C_2, \cF_p}$, so we obtain a $C_2$-symmetric monoidal structure on $\underline{\Eq}_{\id: \underline{\Phi}^{\mu_p}}(\Sp^{\mathrm{O}(2)}_{C_2, \cF_p})$. Now consider the setup of \cref{rem:C2symm_forgetful} but with the $C_2$-recollement
\[ \begin{tikzcd}[row sep=4ex, column sep=8ex, text height=1.5ex, text depth=0.5ex]
\underline{\Fun}_{C_2}(B^t_{C_2} S^1, \ul{\Sp}^{C_2}) \ar[shift right=1,right hook->]{r}[swap]{\underline{\sF}^{\vee}_b } & \Sp^{\mathrm{O}(2)}_{C_2, \cF_p} \ar[shift right=2]{l}[swap]{\underline{\sU}_b} \ar[shift left=2]{r}{\underline{\Phi}^{\mu_p}} & \Sp^{\mathrm{O}(2)}_{C_2, \cF_p}. \ar[shift left=1,left hook->]{l}{i_{\ast}}
\end{tikzcd} \]
To see that the forgetful $C_2$-functor\footnote{We now disambiguate among the different primes by writing the subscript $p$ in $\underline{\sU}_{\RR,p}$, reserving $\underline{\sU}_{\RR}$ for the integral comparison $C_2$-functor.}
$$\underline{\sU}_{\RR,p}: \underline{\Eq}_{\id: \underline{\Phi}^{\mu_p}}(\Sp^{\mathrm{O}(2)}_{C_2, \cF_p}) \to \underline{\LEq}_{\id: \underline{t}_{C_2} \mu_p}(\Sp^{\underline{h}_{C_2} \mu_p})$$
is $C_2$-symmetric monoidal, it suffices to show that $i_*$ is the inclusion of a $C_2$-$\otimes$-ideal. For this, since we already know $i_*$ to be fiberwise the inclusion of a $\otimes$-ideal, it remains to observe using the formula \eqref{eqn:base_change_restriction_norm} that for all $X \in \Sp^{S^1}_{C_2}$ such that the underlying spectrum $X^e$ is trivial, $\res^{C_2} N^{C_2} X \simeq 0$.

By functoriality of limits, the restriction $C_2$-functor $\underline{\RCycSp}^{\mr{gen}} \to \underline{\Eq}_{\id: \underline{\Phi}^{\mu_p}}(\Sp^{\mathrm{O}(2)}_{C_2, \cF_p})$ is $C_2$-symmetric monoidal. Postcomposing this with $\underline{\sU}_{\RR,p}$ and taking the product over the set of primes $\PP$, we define the projection of $\underline{\sU}_{\RR}$ to the second factor. The projection of $\underline{\sU}_{\RR}$ to the first factor is then taken to be the composite $\underline{\RCycSp}^{\mr{gen}} \to \Sp^{\mathrm{O}(2)}_{C_2, \cF} \to \Sp^{\underline{h}_{C_2} S^1}$, and this defines $\underline{\sU}_{\RR}$ itself as a $C_2$-symmetric monoidal functor.
\end{itemize}

\noindent By \cite[Thm.~II.6.9]{NS18}, $\sU \coloneq (\underline{\sU}_{\RR})_{C_2/1}$ restricts to an equivalence on the full subcategories of bounded-below objects. Over the fiber $C_2/C_2$, we have:

\begin{thm} \label{thm:integral_comparison}
$\sU_{\RR}$ restricts to an equivalence on the full subcategories of underlying bounded-below objects.
\end{thm}

% Along with the equivalence of \cref{Thm:IntegralEquivalence}, $C_2$-corepresentability also implies an integral analog of \cref{cor:TCRFormulasEquivalent}:

\begin{cor}
Let $X$ be a genuine real cyclotomic spectrum that is underlying bounded-below. Then we have a canonical equivalence
$$\TCR^{\mr{gen}}(X) \simeq \TCR(\mathscr{U}_{\RR}X)$$
that moreover respects $\cO$-algebra structures for any (reduced) $C_2$-$\infty$-operad $\sO^{\otimes}$.
\end{cor}

% parametrizing $\cF$-genuine $\mathrm{O}(2)$-spectra in terms of their geometric fixed points
One may prove \cref{thm:integral_comparison} by a straightforward adaptation of the arguments of \cite[II.5-6]{NS18}. To illustrate a different technique, we will again present an alternative argument that leverages the reconstruction theorem of Ayala--Mazel-Gee--Rozenblyum. We first discuss arithmetic recollements of stable presentable symmetric monoidal $\infty$-categories. Recall that the arithmetic fracture square for spectra arises from a stable symmetric monoidal recollement
\[ \begin{tikzcd}
\prod_{p \in \PP} \Sp_p^\wedge \simeq \Sp^{\wedge} \arrow[r,shift right=1] & \Sp \arrow[r,shift left=1] \arrow[l, shift right=1] & \Sp_{\QQ}. \ar[l,shift left=1]
\end{tikzcd} \]
More generally, if $C$ is a stable presentable symmetric monoidal $\infty$-category, the unique colimit-preserving symmetric monoidal functor $\Sp \to C$ gives rise to a stable symmetric monoidal recollement
\[ \begin{tikzcd}
\prod_{p \in \PP} C_p^\wedge \simeq C^\wedge \arrow[r,shift right=1] & C \arrow[r,shift left=1] \arrow[l, shift right=1] & C_\QQ \ar[l,shift left=1]
\end{tikzcd} \]
in view of the correspondence between such data and idempotent $E_{\infty}$-algebras \cite[Obs.~2.36]{Sha21}. Here, we have the splitting $C^\wedge \simeq \prod_{p \in \PP} C_p^\wedge$ in view of the following lemma:

\begin{lem}
Let $B$ be a stable presentable $\infty$-category and let $\{ L_n: B \to B \}_{n \in \NN}$ be a collection of localization functors such that the $L_n$ are jointly conservative and for all $n \neq m$, $L_n L_m \simeq 0$. Let $B_n \subset D$ denote the essential image of $L_n$. Then the functor $L = \prod_n L_n: B \to \prod_{n} B_n$ is an equivalence.
\end{lem}
\begin{proof}
Let $i_n$ denote the right adjoint of $L_n: B \to B_n$. Note that $L$ admits a right adjoint $i$ that sends an object $(x_n)$ to $\prod_n i_n x_n$. Since $L_m i_n \simeq 0$ if $n \neq m$, we get that $(L i)(x_n) \simeq (L_n(\prod_n i_n x_n)) \simeq (x_n)$ and hence $i$ is fully faithful. Since the $L_n$ are jointly conservative, we also get that $L$ is conservative. The conclusion follows.
\end{proof}

% Further, if $\mathscr{C}$ is a $G$-stable, presentable, symmetric monoidal $G$-$\infty$-category, then this is a $G$-stable $G$-recollement. These arithmetic stable recollements are functorial in colimit-preserving functors which preserve the unit. 

Suppose now that $F: C \to D$ is a colimit-preserving functor that preserves the unit. Then $F$ preserves the defining idempotent $E_{\infty}$-algebra $1_{\QQ}$ of the arithmetic recollement, so we obtain a morphism of recollements
\[ \begin{tikzcd}
\prod_{p \in \PP} C_p^{\wedge} \arrow{d}[swap]{\prod_p F^{\wedge}_{p}} & C \ar[shift right=1]{l} \arrow{d}{F} \ar[shift left=1]{r} & C_{\QQ} \arrow{d}{F_{\QQ}} \\
\prod_{p \in \PP} D_p^{\wedge} & D \ar[shift right=1]{l} \ar[shift left=1]{r} & D_{\QQ}
\end{tikzcd} \]
where we have retained in the diagram those horizontal functors that necessarily commute with the vertical functors. To show that $F$ is an equivalence, by \cite[Rem.~2.7]{Sha21} it then suffices to show that
\begin{enumerate}
\item For all primes $p$, $F^{\wedge}_p$ is an equivalence.
\item $F_{\QQ}$ is an equivalence.
\item For all primes $p$ and $p$-complete $x \in C$, the canonical map $F(x) \to F(x)^{\wedge}_p$ is an equivalence.
\end{enumerate}

In particular, this shows:
\begin{prp} \label{prp:p_localization_equivalence_criterion}
Let $C$ and $D$ be stable presentable symmetric monoidal $\infty$-categories, and let $F: C \to D$ be a colimit-preserving functor that preserves the unit. Suppose that for all primes $p$, the $p$-localization $F_{(p)}$ is an equivalence. Then $F$ is an equivalence.
\end{prp}
\begin{proof}
By assumption, $F_{\QQ}$ is an equivalence and $F^{\wedge}_p$ is an equivalence for all primes $p$. Moreover, since $F_{(p)}$ is an equivalence, $F_{(p)}$ in particular preserves limits. Therefore, if $x \in C$ is $p$-complete, then $F(x) \simeq F(x)_{(p)}$ is $p$-complete, which shows that $F$ is a \emph{strict} morphism of arithmetic recollements. The claim now follows from \cite[Rem.~2.7]{Sha21}.
\end{proof}

We now turn to the proof of \cref{thm:integral_comparison}.

\begin{rem}
Let $\cF[G]$ denote the family of finite subgroups of a compact Lie group $G$ and let $\zeta: \cF = \cF[\mathrm{O}(2)] \to \cF[S^1] \cong \NN_{>0}$ be the map of posets given by $H \mapsto H \cap S^1$. In the following, we will consider the pushforward stratification of $\Sp^{\mathrm{O}(2)}_{\cF}$ along $\zeta$ and the associated equivalence
\begin{equation} \label{eqn:rel_stratification}
\Theta[\zeta]: \Fun^{\cocart}_{/\NN_{>0}}(\sd(\NN_{>0}), (\Sp^{\mathrm{O}(2)}_{\cF})_{\locus{\phi,\zeta}}) \xto{\simeq} \Sp^{\mathrm{O}(2)}_{\cF} 
\end{equation}
in which $(\Sp^{\mathrm{O}(2)}_{\cF})_{\locus{\phi,\zeta}}$ is a locally cocartesian fibration over $\NN_{>0}$ whose fiber over every object identifies with $\Sp^{h_{C_2}S^1}$. We also will use the formula for the generalized Tate functors given in \cite[Rem.~5.2.5]{AMGRb}: for every pair of finite subgroups $H,K \leq \mathrm{O}(2)$ with $H$ subconjugate to $K$, if we let $C \coloneq \{ g \in \mathrm{O}(2): H \leq g K g^{-1} \leq N_{\mathrm{O}(2)} H \}$, then the generalized Tate functor
$$\tau^K_H: \Sp^{h W_{\mathrm{O}(2)} H} \to \Sp^{h W_{\mathrm{O}(2)} K}$$
is given in terms of proper Tate constructions by the formula
\begin{equation} \label{eqn:generalized_Tate_formula} E \mapsto \bigoplus_{[g] \in N(H) \setminus C / N(K)} \Ind^{W(K)}_{(N(H) \cap N(g K g^{-1}))/(g K g^{-1})} E^{\tau((g K g^{-1})/H)}
\end{equation}
where we have suppressed subscripts on the normalizers and Weyl groups for brevity. In particular, if $K \not\leq N_{\mathrm{O}(2)} H$, then $\tau^K_H = 0$.
\end{rem}

\begin{prp}\label{Prop:p-complete-p-typical}
For every prime $p$, the $p$-localized functor $(\sU_{\RR})_{(p)}$ participates in a commutative square
\[ \begin{tikzcd}
(\RCycSp^{\mr{gen}})_{(p)} \ar{r}{(U^{\mr{gen}}_p)_{(p)}} \ar{d}{(\sU_{\RR})_{(p)}} & \Eq_{\id: \Phi^{\mu_p}} (\Sp^{\mathrm{O}(2)}_{\sF'_p})_{(p)} \ar{d}{(\sU_{\RR,p})_{(p)}} \\
\RCycSp_{(p)} \ar{r}{(U_p)_{(p)}} & \LEq_{\id: t_{C_2} \mu_p} (\Sp^{h_{C_2} S^1})_{(p)}
\end{tikzcd} \]
in which the horizontal functors are equivalences and the vertical functors restrict to equivalences on full subcategories of underlying bounded-below objects.
% There are equivalences of $C_2$-$\infty$-categories
% $$(\underline{\RCycSp}^{\mr{gen}}_{\mr{bb}})_p^{\wedge} \simeq (\underline{\RCycSp}^{\mr{gen}}_{p,\mr{bb}})^{\wedge}_p,$$
% $$(\underline{\RCycSp})^{\wedge}_p \simeq (\underline{\RCycSp}_p)^{\wedge}_p.$$
\end{prp}
\begin{proof} Here $U^{\mr{gen}}_p$ and $U_p$ denote the evident forgetful functors. We have proven the claim for $(\sU_{\RR, p})_{(p)}$ (in fact, prior to $p$-localization) at the end of \cref{sec:main_theorem} as a variant of \cref{thm:MainTheoremRestated}. To see that $(U_p)_{(p)}$ is an equivalence, first note that for every $[X, \{ \varphi_q: X \to X^{t_{C_2} \mu_q} \}] \in \RCycSp_{(p)}$, $X$ is $p$-local as a $C_2$-spectrum by \cref{prp:RCycSpPresentable}. Therefore, for all primes $q \neq p$ we have that the real cyclotomic structure maps $\varphi_q$ are trivial by \cref{cor:ParamTatePcomplete}(3), so $(U_p)_{(p)}$ is an equivalence.
% \begin{itemize}
% \item[($\ast$)] For every $p$-local $X \in \Sp^{h_{C_2}S^1}$, by \cref{cor:ParamTatePcomplete}(3) $(X^{t_{C_2} \mu_q}) \simeq 0$ for all primes $q \neq p$. Therefore, for every $[X, \{ \varphi_q \}] \in \RCycSp_{(p)}$, since $X$ is $p$-local by \cref{prp:RCycSpPresentable} its real cyclotomic structure maps $\varphi_q$ are trivial for all primes $q \neq p$.
% \end{itemize}

It remains to show $(U^{\mr{gen}}_p)_{(p)}$ is an equivalence. Let $\NN^{\setminus p}_{>0}$ be the subposet of the divisibility poset $\NN_{>0}$ on $m$ coprime to $p$ and let $\lambda'_p: \NN_{>0} \to \NN^{\setminus p}_{>0}$ be the map of posets $n = p^k m \mapsto m$ where $(m,p)=1$. Let $\lambda_p = \lambda'_p \circ \zeta$ and consider the pushforward stratification of $\Sp^{\mathrm{O}(2)}_{\cF}$ along $\lambda_p$. Then  under the isomorphism $\mathrm{O}(2)/C_m \cong \mathrm{O}(2)$ we have that
$$((\Sp^{\mathrm{O}(2)}_{\cF})_{\locus{\phi, \lambda_p}})_m \simeq \Sp^{\mathrm{O}(2)}_{\cF'_p}$$
for all $m \in \NN^{\setminus p}_{>0}$. Furthermore, for every $m|m'$, we have that the cocartesian pushforward functor $\tau^{m'}_{m}: \Sp^{\mathrm{O}(2)}_{\cF'_p} \to \Sp^{\mathrm{O}(2)}_{\cF'_p}$ vanishes on $p$-local objects. Indeed, first note that for every pair of finite subgroups $H,K$ of $\mathrm{O}(2)$ lying over $m, m'$ respectively, we have that $\tau^{K}_{H}$ vanishes on $p$-local spectra by the formula \eqref{eqn:generalized_Tate_formula} and the fact that for any prime $q \neq p$ and finite $q$-group $G$, the proper Tate construction $(-)^{\tau G}$ annihilates $p$-local spectra. Furthermore by \cite[Thm.~A]{Sha21}, for $X \in ((\Sp^{\mathrm{O}(2)}_{\cF})_{\locus{\phi, \lambda_p}})_m$ the geometric fixed points of $\tau^{m'}_{m} X$ at every $K \leq \mathrm{O}(2)$ over $m'$ is given by the limit of $\tau^{K}_{H_k} \tau^{H_k}_{H_{k-1}} ... \tau^{H_2}_{H_1} X^{\phi H_1}$ indexed over the finite subposet of strings $\{ H_1 < ... < H_k < K \}$ in $\sd(\cF)$ where the $H_i$ lie over $m$. Therefore, $\tau^{m'}_{m}$ vanishes on $p$-local objects.

Since the geometric fixed points functors commute with $p$-localization, it follows that in the equivalence
$$ \Theta[\lambda_p]: \Fun^{\cocart}_{/ \NN^{\setminus p}_{>0}}(\sd(\NN^{\setminus p}_{>0}), (\Sp^{\mathrm{O}(2)}_{\cF})_{\locus{\phi, \lambda_p}}) \xto{\simeq} \Sp^{\mathrm{O}(2)}_{\cF},$$
$p$-localization on the lefthand side is computed by the fiberwise $p$-localization. In view of the above vanishing result, we thus get an equivalence
\[ \Theta[\lambda_p]_{(p)}: \prod_{m \in \NN^{\setminus p}_{>0}} (\Sp^{\mathrm{O}(2)}_{\cF'_p})_{(p)} \xto{\simeq} (\Sp^{\mathrm{O}(2)}_{\cF})_{(p)} \]
under which the action of $\NN_{>0}$ factors as the translation action of $\NN^{\setminus p}_{>0}$ on the indexing set $\NN^{\setminus p}_{>0}$ and the termwise action of the submonoid $\underline{p} \coloneq \{1 < p < p^2 < ... \}$ on each component $(\Sp^{\mathrm{O}(2)}_{\cF'_p})_{(p)}$. Therefore, we obtain an equivalence
\[ (\Theta[\lambda_p]^{h \underline{p}})_{(p)}: \prod_{m \in \NN^{\setminus p}_{>0}} \Eq_{\id: \Phi^{\mu_p}}(\Sp^{\mathrm{O}(2)}_{\cF'_p})_{(p)} \xto{\simeq} ((\Sp^{\mathrm{O}(2)}_{\cF})^{h \underline{p}})_{(p)} \]
where the residual $\NN^{\setminus p}_{>0}$-action on the lefthand side is still the translation action. Applying \cref{lem:poset_action}, we deduce that
\[ (((\Sp^{\mathrm{O}(2)}_{\cF})^{h \underline{p}})_{(p)})^{h \NN^{\setminus p}_{>0}} \simeq \Eq_{\id: \Phi^{\mu_p}}(\Sp^{\mathrm{O}(2)}_{\cF'_p})_{(p)}. \]
Since the endofunctors $\Phi^{\mu_q}$ commute with $p$-localization, we get that
\[ (\RCycSp^{\mr{gen}})_{(p)} \simeq (((\Sp^{\mathrm{O}(2)}_{\cF})^{h \underline{p}})_{(p)})^{h \NN^{\setminus p}_{>0}}.  \]
Putting these equivalences together yields the claim.
\end{proof}

\begin{lem} \label{lem:poset_action}
Let $P$ be a poset with an initial object $x_0$ equipped with compatible monoid structure, and consider the action of $P$ on the set $P^{\delta}$ by translation. Suppose that for every $x \leq y \in P$, there exists a unique $z$ such that $y = x \cdot z$. Then for any $\infty$-category $C$, the diagonal functor $\delta: C \to \prod_P C$ is $P$-equivariant and induces an equivalence $C \xto{\simeq} (\prod_P C)^{h P}$, with inverse given by evaluation at $x_0$.
\end{lem}
\begin{proof}
We have $(\prod_P C)^{h P} = \Fun(P^{\delta}, C)^{h P} \simeq \Fun((P^{\delta})_{h P}, C)$ and $(P^{\delta})_{h P} \simeq |P^{\delta} // P|$ where $P^{\delta} // P$ is the action category for the monoid action. But in view of the divisibility assumption on $P$ we have an isomorphism of categories $P \cong P^{\delta} // P$, and since $x_0 \in P$ is an initial object the inclusion $\{ x_0\} \subset P$ is a weak homotopy equivalence.
\end{proof}

\begin{proof}[Proof of \cref{thm:integral_comparison}]
Combine \cref{Prop:p-complete-p-typical} and \cref{prp:p_localization_equivalence_criterion}.
\end{proof}

Finally, for the sake of completeness we record the observation that $(\sU_{\RR})_{\QQ}$ is an equivalence unconditionally.

\begin{prp}\label{Prop:rational}
The functor $(\sU_{\RR})_{\QQ}$ participates in a commutative triangle of equivalences
\[ \begin{tikzcd}
(\RCycSp^{\mr{gen}})_{\QQ} \ar{d}[swap]{(\sU_{\RR})_{\QQ}} \ar{r}{U^{\mr{gen}}_{\QQ}} & \Fun_{C_2}(B^t_{C_2}S^1, \underline{\Sp}^{C_2})_{\QQ}. \\
(\RCycSp)_{\QQ} \ar{ru}[swap]{U_{\QQ}}
\end{tikzcd} \]
% in which all functors restrict to equivalences on full subcategories of underlying bounded-below objects.
% There are equivalences of $C_2$-$\infty$-categories
% $$(\underline{\RCycSp}^{\mr{gen}})_{\QQ} \simeq (\underline{\Sp}^{h_{C_2}S^1})_{\QQ} \simeq (\underline{\RCycSp})_{\QQ}.$$
\end{prp}
\begin{proof} Here we let $U^{\mr{gen}}$ and $U$ denote the functors that take the underlying $C_2$-spectrum with parametrized $S^1$-action. It suffices to show that that $U^{\mr{gen}}_{\QQ}$ and $U_{\QQ}$ are equivalences. To see that $U_{\QQ}$ is an equivalence, note that since $U$ preserves colimits, given a rational real cyclotomic spectrum $[X, \{\varphi_p: X \to X^{t_{C_2} \mu_p} \}_{p \in \PP}]$, $X$ is rational as a $C_2$-spectrum. Therefore, all of the real cyclotomic structure maps $\varphi_p$ are trivial since $(-)^{t_{C_2} \mu_p}$ vanishes on rational $C_2$-spectra, so $U_{\QQ}$ is an equivalence.
 
To see that $U^{\mr{gen}}_{\QQ}$ is an equivalence, note first that for every pair of finite subgroups $H,K \leq \mathrm{O}(2)$ with $H$ subconjugate to $K$, the generalized Tate functor $\tau^K_H$ vanishes on rational spectra in view of the formula \eqref{eqn:generalized_Tate_formula} since the proper Tate constructions do. Consequently, after rationalization the locally cocartesian fibration $\Sp^{\mathrm{O}(2)}_{\locus{\phi}}|_{\sF} \to \sF$ becomes a cocartesian fibration with all pushforward functors identically zero. Taking the pushforward stratification along $\zeta$, we get the same for the rationalization of $(\Sp^{\mathrm{O}(2)}_{\cF})_{\locus{\phi,\zeta}} \to \NN_{>0}$. Therefore, the equivalence $\Theta[\zeta]$ of \eqref{eqn:rel_stratification} yields after rationalization an equivalence
\[ \prod_{n \in \NN_{>0}} (\Sp^{h_{C_2}S^1})_{\QQ} \xto{\simeq} (\Sp^{\mathrm{O}(2)}_{\cF})_{\QQ} \]
under which the $\NN_{>0}$-action on $(\Sp^{\mathrm{O}(2)}_{\cF})_{\QQ}$ identifies with the translation action of $\NN_{>0}$ on the indexing set $\NN_{>0}$. The claim now follows from \cref{lem:poset_action}.
\end{proof}

\section{Real cyclotomic structure on \texorpdfstring{$\THR$}{THR} of \texorpdfstring{$C_2$}{C2}-\texorpdfstring{$E_{\infty}$}{E infinity}-algebras} \label{sec:thr}

In this section, we define the real cyclotomic structure present on the real topological Hochschild homology of a $C_2$-$E_{\infty}$-algebra $A$ in $\underline{\Sp}^{C_2}$. This amounts to a straightforward extension of the construction presented in \cite[\S IV.2]{NS18} for $\THH$ of $E_{\infty}$-algebras.

\begin{rem}
All unattributed results in this section on the $C_2$-$\infty$-category of $C_2$-commutative algebras (e.g., that it is $C_2$-cocomplete such that $C_2$-indexed coproducts are computed as $C_2$-indexed tensor products) may be found in forthcoming work of Nardin and the second author \cite{paramalg} or by specializing and slightly adapting the results of Bachmann-Hoyois on normed $E_{\infty}$-algebras in \cite[\S 7]{BachmannHoyoisNorms}.
\end{rem}

\begin{dfn} \label{dfn:THR}
The \emph{real topological Hochschild homology} $\THR(A)$ is defined to be the $C_2$-tensor\footnote{For a $C_2$-category $C$ and an object $E \in C_{C_2/C_2}$, the $C_2$-tensor of $E$ with a $C_2$-space $X$ is defined to be the $C_2$-colimit of the $C_2$-functor $X \xto{\pi} \sO_{C_2}^{\op} \xto{E} C$, where $\pi$ is the structure map and $E$ also denotes the cocartesian section that it determines.} of $A$ with the $C_2$-space $S^{\sigma}$ in the $C_2$-$\infty$-category $\underline{\CAlg}_{C_2}(\underline{\Sp}^{C_2})$ of $C_2$-$E_{\infty}$-algebras in $\underline{\Sp}^{C_2}$. In symbols,
\[ \THR(A) \coloneq S^{\sigma} \odot A. \]
\end{dfn}

\begin{rem}
Consider the pushout square of $C_2$-spaces
\[ \begin{tikzcd}
C_2 = \{ \pm i \} \ar[hookrightarrow]{r} \ar[hookrightarrow]{d} & S^{\sigma} - \{ 1 \} \ar[hookrightarrow]{d} \\
S^{\sigma} - \{ -1 \} \ar[hookrightarrow]{r} & S^{\sigma}.
\end{tikzcd} \]
Using that $S^{\sigma} - \{ 1 \}$ and $S^{\sigma} - \{ -1 \}$ are $C_2$-equivariantly contractible to $\{ -1 \}$ and $\{ +1 \}$, respectively, we get a pushout square of $C_2$-spaces
\[ \begin{tikzcd}
C_2 = \{ \pm i \} \ar[hookrightarrow]{r} \ar[hookrightarrow]{d} & \{ -1 \} \ar[hookrightarrow]{d} \\
\{ +1 \} \ar[hookrightarrow]{r} & S^{\sigma}
\end{tikzcd} \]
and thus a pushout square of $C_2$-$E_{\infty}$-algebras
\[ \begin{tikzcd}
C_2 \odot A \ar{r} \ar{d} & A \ar{d} \\ 
A \ar{r} & S^{\sigma} \odot A.
\end{tikzcd} \]
Since $C_2$-indexed coproducts are computed as $C_2$-indexed tensor products in $\underline{\CAlg}_{C_2}(\underline{\Sp}^{C_2})$, we have that $C_2 \odot A \simeq A^{\otimes C_2} \coloneq N^{C_2} A^e$ where $A^e \coloneq \res^{C_2} A$. Moreover, pushouts are computed as relative tensor products.\footnote{To see this, one can e.g. replace the pushout diagram by a bar construction and use that coproducts are tensor products and the forgetful functor $U: \CAlg_{C_2}(\underline{\Sp}^{C_2}) \to \Sp^{C_2}$ creates sifted colimits.} Thus, we get that
\begin{equation} \label{eq:THR}
\THR(A) \simeq A \otimes_{A^{\otimes C_2}} A,
\end{equation}
which refines the equivalence $\THH(A^e) \simeq A^e \otimes_{A^e \otimes A^e} A^e$ of spectra. Moreover, upon taking $C_2$-geometric fixed points, one gets
\begin{equation} \label{eq:THR_geometric_fixed_points}
\THR(A)^{\phi C_2} \simeq A^{\phi C_2} \otimes_{A^e} A^{\phi C_2}.
\end{equation}
\end{rem}

\begin{rem}
In view of \cite[Cor.~4.5]{DMPR17}, the formula \eqref{eq:THR} shows that \cref{dfn:THR} agrees with the definition of $\THR$ in terms of the dihedral Bar construction as associative algebras in $\Sp^{C_2}$. This equivalence suffices for all of the calculational input regarding $\THR$ from \cite{DMPR17,DMP21} that we will use in \cref{Sec:TCRFp}.
\end{rem}

The rotation action of $S^{\sigma}$ on itself then extends $\THR(A)$ to a $C_2$-functor\footnote{Note that the delooping of $S^{\sigma}$ as an associative algebra in $C_2$-spaces is $B^t_{C_2} S^1$.}
$$\THR(A): B^t_{C_2} S^1 \to \underline{\CAlg}_{C_2}(\underline{\Sp}^{C_2}).$$
It remains to construct the real cyclotomic structure on $\THR(A)$. First, we need to introduce a dihedral refinement of the Tate diagonal $\Delta_p: X \to (X^{\otimes p})^{t C_p}$ for a spectrum $X$.

\begin{cnstr} Consider the functor $N^{D_{2p}}_{C_2}: \Sp^{C_2} \to \Sp^{D_{2p}}$. Given the recollement
\[ \begin{tikzcd}
\Fun_{C_2}(B^t_{C_2} \mu_p, \underline{\Sp}^{C_2}) \ar[shift left=3, hookrightarrow]{r}{j_!}  \ar[shift right=3, hookrightarrow]{r}[swap]{j_*} & \Sp^{D_{2p}} \ar[shift left=3]{r}{i^*} \ar[shift right=3]{r}[swap]{i^!} \ar{l}[description]{j^*} \ar[hookleftarrow]{r}[description]{i_*} & \Sp^{C_2} ,\end{tikzcd} \]
postcomposing $N^{D_{2p}}_{C_2}$ with the natural transformation $\eta: i^* \Rightarrow i^* j_* j^* = (-)^{t_{C_2} \mu_p} j^*$ yields the dihedral Tate diagonal
\[ \Delta_p: X \rightarrow (X^{\otimes \mu_p})^{t_{C_2} \mu_p} \] 
for all $X \in \Sp^{C_2}$. Here, we write $X^{\otimes \mu_p}$ for the $C_2$-indexed tensor product with $C_2$-action on $\mu_p$ given by complex conjugation (i.e., inversion). Indeed, a diagram chase shows that $\res^{D_{2p}}_{C_2} N^{D_{2p}}_{C_2} X \simeq X^{\otimes \mu_p}$.

Furthermore, using the functoriality of the Hill--Hopkins--Ravenel norms we have that $N^{D_{2p}}_{C_2}$ refines to a $C_2$-symmetric monoidal functor from $\underline{\Sp}^{C_2}$ to $\Sp^{D_{2p}}_{C_2} \coloneq [\Sp^{D_{2p}} \xto{\res} \Sp^{\mu_p}]$ endowed with the $C_2$-symmetric monoidal structure defined in \cref{sec:C2_smc}. By the proof of \cref{rem:C2symm_forgetful}, we also get that $\eta$ refines to a lax $C_2$-symmetric monoidal natural transformation. We conclude that for a $C_2$-$E_{\infty}$-algebra $A$, $(A^{\otimes \mu_p})^{t_{C_2} \mu_p}$ is a $C_2$-$E_{\infty}$-algebra and the dihedral Tate diagonal is a morphism of $C_2$-$E_{\infty}$-algebras.
\end{cnstr}

The construction of the real $p$-cyclotomic Frobenius
$$\varphi_p: \THR(A) \to \THR(A)^{t_{C_2} \mu_p}$$
now proceeds as follows: first note that $S^{\sigma} \odot -$ computes the left adjoint to the forgetful functor
\[ \Fun_{C_2}(B^t_{C_2} S^1, \underline{\CAlg}_{C_2}(\underline{\Sp}^{C_2})) \to \CAlg_{C_2}(\underline{\Sp}^{C_2}) \]
in view of the formula for $C_2$-left Kan extension. The unit map $A \to S^{\sigma} \odot A$ then exhibits $\THR(A)$ as the initial $C_2$-$E_{\infty}$-algebra with an action by $S^{\sigma}$ under $A$. Now recall that $\mu_p \odot - \simeq (-)^{\otimes \mu_p}$ in $\CAlg_{C_2}(\underline{\Sp}^{C_2})$, so the $C_2$-equivariant inclusion $\mu_p \subset S^{\sigma}$ induces a map
\[ (A^{\otimes \mu_p})^{t_{C_2} \mu_p} \to (S^{\sigma} \odot A)^{t_{C_2} \mu_p}. \]
Precomposing this with $\Delta_p$ and using the residual $S^{\sigma} \cong S^{\sigma}/\mu_p$-action on the target, we then obtain the $S^{\sigma}$-equivariant map $\varphi_p$ by universal property. Moreover, since $\varphi_p$ is a map of $C_2$-$E_{\infty}$-algebras, the resulting real $p$-cyclotomic spectrum is a $C_2$-$E_{\infty}$-algebra in $\underline{\RCycSp}_p$. Finally, the set $\{ \varphi_p \}_{p \in \PP}$ defines $\THR(A)$ as a $C_2$-$E_{\infty}$-algebra in $\underline{\RCycSp}$.

\begin{rem}
We have used that the equivalence $\mu_p \odot A \simeq A^{\otimes \mu_p}$ is compatible with the $\mu_p$-action: to see this, note that the comparison map $\chi$ is induced by the universal property of $\mu_p \odot (-)$ as the left adjoint to the forgetful functor
\[ \Fun_{C_2}(B^t_{C_2} \mu_p, \underline{\CAlg}_{C_2}(\underline{\Sp}^{C_2})) \to \CAlg_{C_2}(\underline{\Sp}^{C_2}) \]
and one may check that $\chi$ is an equivalence after forgetting action.
\end{rem}

\begin{rem}
By definition, we have a commutative diagram of $C_2$-$E_{\infty}$-algebras
\[ \begin{tikzcd}
A \ar{r}{\eta} \ar{d}{\Delta_p} & \THR(A) \ar{d}{\varphi_p} \\ 
(A^{\otimes \mu_p})^{t_{C_2} \mu_p} \ar{r} & \THR(A)^{t_{C_2} \mu_p}.
\end{tikzcd} \]
Moreover, the map of $C_2$-spaces $S^{\sigma} \to \ast$ induces a $S^{\sigma}$-equivariant map $\epsilon: \THR(A) \to A$ (where the target has trivial $S^{\sigma}$-action) such that the composite
\[ A \xto{\eta} \THR(A) \xto{\varphi_p} \THR(A)^{t_{C_2} \mu_p} \xto{\epsilon'} A^{t_{C_2} \mu_p} \]
is the dihedral Tate-valued Frobenius.
\end{rem}

% \subsection{\texorpdfstring{$\TCR$}{TCR} of \texorpdfstring{$H \ul{\FF_p}$}{HFp} for an odd prime}
\section{\texorpdfstring{$\TCR$}{TCR} of the constant mod \texorpdfstring{$p$}{p} Mackey functor}\label{Sec:TCRFp}

In this section, we compute the real topological cyclic homology of the constant $C_2$-Mackey functor on $\FF_p$. We split our analysis into the cases of $p$ odd (\cref{SS:HFpOdd}) and $p=2$ (\cref{SS:HFp2}). We also establish an extension to constant $C_2$-Mackey functors on perfect $\FF_p$-algebras in \cref{SS:Perf}. All spectral sequences will be indexed and displayed using homological Serre grading. 
% Throughout this section, using that $\TCR(H \mfp) \simeq \TCR(H \mfp, p)$ we may and will take all functors to be implicitly $p$-typical.

% When $p$ is odd, we closely follow the computation of $\TC(H\FF_p)$ given by Nikolaus--Scholze in \cite[\S IV.4]{NS18}. The key point is that just as for their non-parametrized analogs, the relevant parametrized homotopy fixed points and parametrized Tate spectral sequences all collapse on the $E_2$-page for degree reasons. However, when $p=2$ this is no longer the case: we find nontrivial $d_3$-differentials in the relevant spectral sequences (\cref{Prop:Differentials}) that serve to complicate our analysis. 
When $p$ is odd, the computation follows immediately from the known computation of $\TC(H\FF_p)$ given by Nikolaus--Scholze in \cite[\S IV.4]{NS18}. When $p=2$, this is no longer the case and we need to employ parametrized versions of the $S^1$-homotopy fixed points and Tate spectral sequences to compute. The reader should recall that in the case of $\TC^{-}(H \FF_p)$ and $\TP(H \FF_p)$, these spectral sequences all collapse on the $E_2$-page for degree reasons. In stark contrast, in our situation we will find nontrivial $d_3$-differentials (\cref{Prop:Differentials}) that serve to complicate our analysis.

We will use the following grading convention for $RO(C_2)$-graded homotopy groups: 

\begin{cvn}\label{Cvn:Grading} Let $S^{s,t} = S^{t\sigma} \otimes S^{s-t}$, and for a $C_2$-spectrum $X$, let $\pi_{s,t}^{C_2}(X) = [S^{s,t},X]_{C_2}$. For a $C_2$-Mackey functor $H \underline{M}$, we also write $\underline{M}_{**} = \pi^{C_2}_{**}(H \underline{M})$.
\end{cvn}
% We employ homological grading conventions throughout the paper. We will denote the $RO(C_2)$-graded homotopy groups of a genuine $C_2$-spectrum $X$ by $\pi_{s,t}^{C_2}(X)$, where $S^{s,t} = S^{t\sigma} \wedge S^{s-t}$ with $S^\sigma$ the one-point compactification of the sign representation and $S^1$ the one-point compactification of the trivial one-dimensional representation. 

We will also rely heavily on the following identification of the $C_2$-equivariant homotopy type of $\THR(H\mfp)$ by work of Dotto--Moi--Patchkoria--Reeh.
% , which holds for all primes $p$. 

\begin{thm}[{\cite[Thm.~5.18]{DMPR17}}] \label{Thm:THRFp}
There is an equivalence of associative ring $C_2$-spectra
$$T_{H\FF_p}(S^{2,1}) \overset{\simeq}{\to} \THR(H\mfp)$$
where $T_{H\FF_p}(S^{2,1}) = \bigoplus_{n=0}^\infty \Sigma^{2n,n} H \mfp$ is the free associative $H \mfp$-algebra on $S^{2,1}$. In particular, there is an isomorphism of bigraded rings 
$$H\mfp_{{**}}^{C_2}[{x}] \cong \pi_{{**}}^{C_2}\THR(H\mfp)$$
where $|{x}| = (2,1)$. 
\end{thm}

Finally, we will use the following terminology.

\begin{dfn}
Let $X \in \Sp^{h_{C_2} S^1}$. The \emph{negative real topological cyclic homology} is
\[ \TCR^-(X) \coloneq X^{h_{C_2}S^1}.  \]
The \emph{periodic real topological cyclic homology} is
\[ \TPR(X) \coloneq X^{t_{C_2}S^1}. \]
\end{dfn}

% Our strategy is then to use the formula \eqref{eq:fundamental_fiber_sequence} for $\TCR(H\underline{\FF_2})$ as the fiber of a map between the parametrized homotopy fixed points and parametrized Tate construction on $\THR(H\underline{\FF_2})$.
We then have the fiber sequence \eqref{eq:fundamental_fiber_sequence} of $C_2$-spectra in the form
\begin{equation}\label{Eqn:FiberSeq}
%substituted \THR(H\underline{\FF_2})^{h_{C_2} S^1} and \THR(H\underline{\FF_2})^{t_{C_2}S^1}
\TCR(H\underline{\FF_p}) \to \TCR^{-}(H\underline{\FF_p}) \xrightarrow{\varphi_p^{h_{C_2}S^1} - \can} \TPR(H\underline{\FF_p}),
\end{equation}
where use that $\THR(H\underline{\FF_p})^{t_{C_2}S^1}$ is $p$-complete\footnote{This will be explicitly shown in all cases.} and \cref{prp:S1vsLimit} to replace the righthand term in \eqref{eq:fundamental_fiber_sequence} with the $C_2$-parametrized $S^1$-Tate construction.

\subsection{Computation for \texorpdfstring{$p$}{p} odd}\label{SS:HFpOdd}

% \textcolor{red}{Geometric fixed points computation. Note missing reference in remark in next section for triviality of action, as well as missing reference for odd-primary computation in Theorem in Intro. }

Fix an odd prime $p$. Recall that both the $C_2$-geometric fixed points $\Phi^{C_2} \mfp$ and the $C_2$-Tate construction $(H \FF_p)^{t C_2}$ are trivial, so the same holds for $\THR(H \mfp)$ (using that the Tate construction commutes with Postnikov limits \cite[Lem.~I.2.6]{NS18}). In other words, $\THR(H \mfp)$  is both Borel torsion and Borel complete.

We will also use as input the computation of $\TC$, $\TC^{-}$, and $\TP$ of $H \FF_p$ as recorded in \cite[Prop.~IV.4.6, Cor.~IV.4.8, and Cor.~IV.4.10]{NS18}.

\begin{lem}
$\TCR^{-}(H\mfp)$, $\TPR(H\mfp)$, and $\TCR(H\mfp)$ are Borel torsion and Borel complete.
\end{lem}
\begin{proof}
Since $\TC^{-}(H \FF_p)$ is $p$-complete, $\TC^{-}(H \FF_p)^{t C_2} = 0$. $\TCR^{-}(H\mfp)^{\phi C_2} = 0$ then follows from \cref{rem:UpperShriekPreservesParametrizedLimits}.

Similarly, $\TP^{-}(H \FF_p)$ is $p$-complete, so $\TP(H \FF_p)^{t C_2} = 0$. By \cref{exm:geometric_fixed_points_and_parametrized_colimits} and using that $\THR(H\mfp)$ is Borel torsion, we get that $(\THR(H\mfp)_{h_{C_2} S^1})^{\phi C_2} = 0$. It follows that $\TPR(H\mfp)^{\phi C_2} = 0$.

Now by the fiber sequence \eqref{Eqn:FiberSeq}, we deduce that $\TCR(H\mfp)$ is Borel torsion and Borel complete.
\end{proof}

\begin{thm} \label{thm:odd_primary_computation}
There is an equivalence of $C_2$-spectra
$$\TCR(H\mfp) \simeq H\uzp \oplus \Sigma^{-1} H\uzp.$$
\end{thm}
\begin{proof}
Since both sides are Borel torsion and Borel complete, it suffices to prove an equivalence of underlying spectra with $C_2$-action. By \cite[Cor.~IV.4.10]{NS18}, $\TC(H \FF_p) \simeq H \ZZ_p \oplus \Sigma^{-1} H \ZZ_p$ as spectra. Moreover, since the $C_2$-action on $\TC(H \FF_p)$ is via a map of ring spectra, it must act trivially on $\pi_0 \TC(H \FF_p)$ and by $1$ or $-1$ on $\pi_{-1} \TC(H \FF_p)$.

%To conclude, it suffices to rule out the possibility that $C_2$ acts by $-1$ on the summand $\Sigma^{-1} H \ZZ_p$. Suppose for a contradiction that this was the case. Then we would have that $\TC(H \FF_p)^{h C_2} \simeq H \ZZ_p$.
These two cases are distinguished by their effect on homotopy $C_2$-fixed points, and to conclude it suffices to show that $\pi_{-1} \TC(H \FF_p)^{h C_2} = \ZZ_p$. Consider now the fiber sequence
\[ \TC(H \FF_p)^{h C_2} \to \TC^{-}(H \FF_p)^{h C_2} \to \TP (H \FF_p)^{h C_2}. \]
Since $\TC^{-}(H \FF_p)$ and $\TP (H \FF_p)$ are concentrated in even degrees, and $C_2$-cohomology of $\ZZ_p$ vanishes in positive degrees irrespective of the action, we see that $\pi_{-1}(\TC(H \FF_p)^{h C_2})$ is the cokernel of the map $\can^{h C_2} - \varphi^{h C_2}$. But since $C_2$ acts via ring maps and both $\can$ and $\varphi$ are ring maps, we must have that $\pi_0 \TC^{-}(H \FF_p)^{h C_2} = \ZZ_p$, $\pi_0 \TP (H \FF_p)^{h C_2} = \ZZ_p$, and $\can^{h C_2} - \varphi^{h C_2}$ is the zero map on $\pi_0$. We conclude that $\pi_{-1}(\TC(H \FF_p)^{h C_2}) = \ZZ_p$, as desired.
\end{proof}

\subsection{Computation for \texorpdfstring{$p=2$}{p=2}}\label{SS:HFp2}

Our goal in this subsection is to prove:

\begin{thm}\label{Thm:TCRF2}
There is an equivalence of $C_2$-spectra
$$\TCR(H\underline{\FF_2}) \simeq H\underline{\ZZ_2} \oplus \Sigma^{-1} H\underline{\ZZ_2}.$$
\end{thm}

\begin{rem}
The result is already known at the underlying level $\TCR(H\underline{\FF_2})^e \simeq \TC(\FF_2)$ by \cite[Thm.~B]{HM97}. It also follows from \cite[Thm. D]{DMP21}, which was proven using different methods. 
\end{rem}

To compute $\TCR^{-}(H\underline{\FF_2})$ and $\TPR(H\underline{\FF_2})$, we will introduce a new spectral sequence arising from the `$x$-adic' filtration of $\THR(H\mft)$, where $x$ is as in \cref{Thm:THRFp}.

% Before we start, we comment on the necessity of this new approach:
%\begin{rem}
%In the odd-primary case, we studied the parametrized homotopy fixed points and parametrized Tate construction using the parametrized homotopy fixed points and Tate spectral sequences. In the $2$-primary case, this approach encounters technical difficulties which we are unable to handle. In particular, there is some ambiguity in the parametrized $S^1$-action appearing in the $E_2$-terms of these spectral sequences. This ambiguity arises from the density of the $RO(C_2)$-graded coefficients $\pi_{**}^{C_2} H\mft$ compared to $\pi_{**}^{C_2} H\underline{\FF_p}$ for $p$ odd. 
%\end{rem}

\begin{dfn}
We define the \emph{$x$-adic filtration} of $\THR(H\mft)$ by 
$$F_{2k} := H\mft[x]\{x^k\} \simeq \bigoplus_{n \geq k} \Sigma^{2n,n} H\mft,$$
$$F_{2k+1} := F_{2k},$$
where $|x| = (2,1)$. 

The canonical inclusions define a decreasing filtration
$$\cdots \to F_2 \to F_1 \to F_0 = \THR(H\mft)$$
with filtration quotients
$$F_{2k}/F_{2k-1} \simeq H\mft\{x^k\} \simeq \Sigma^{2k,k} H\mft,$$
$$F_{2k+1}/F_{2k} \simeq *.$$
\end{dfn}

The filtration quotients $\Sigma^{2k,k} H\mft$ each have a $C_2$-parametrized $S^1$-action. This action can be understood using the following two results.

\begin{lem}\label{Lem:MapsEM}
Let $G$ be a finite group. Let $A$ and $B$ be Mackey functors for $G$, and let $HA$ and $HB$ be their associated Eilenberg--MacLane $G$-spectra. The map
$${\pi}_0^{(-)} : \Map_{\Sp^G}(HA,HB) \to \Map_{\Mack^G}(A,B)$$
is an isomorphism.
\end{lem}

\begin{proof}
Since
$$\Sp^G \simeq \Fun^{\oplus}(\Span(\FF_G), \Sp),$$
we can write
$$\Map_{\Sp^G}(HA,HB) \simeq \int_{U \in \Span(\FF_G)} \Map_{\Sp}(HA(U),HB(U))$$
using \cite[Prop. 2.3]{Gla16}. We then have the classical equivalence
$$\pi_0: \Map_{\Sp}(HA(U),HB(U)) \xrightarrow{\simeq} \Map_{\Ab}(A(U),B(U)).$$
Inserting this into the end formula above, we have
$$\Map_{\Sp^G}(HA,HB) \simeq \int_{U \in \Span(\FF_G)} \Map_{\Ab}(A(U),B(U)) \simeq \Map_{\Mack^G}(A,B).$$
\end{proof}

\begin{prp} \label{prp:trivial_action}
% The $C_2$-parametrized $S^1$-action on $\Sigma^V H\underline{\FF_2}$ is trivial for any $V \in RO(C_2)$. 
Let $\underline{M}$ be any constant $C_2$-Mackey functor and $V \in RO(C_2)$. Then any $C_2$-parametrized $S^1$-action on $\Sigma^V H \underline{M}$ is necessarily trivial. 
\end{prp}
\begin{proof}
First, we note that for any $C_2$-spectrum $E$, we have
$$\Aut(\Sigma^V E) \simeq \Aut(E).$$
Therefore, it suffices to show the $C_2$-parametrized $S^1$-action is trivial on $H\underline{M}$. 

We will show that any $C_2$-functor
$$B_{C_2}^t S^1 \to \underline{\Sp}^{C_2}$$
with value $H \underline{M}$ necessarily factors through $\ast_{C_2}$. Unpacking definitions, it suffices to show that the pair of horizontal morphisms
\[
\begin{tikzcd}
B\mu_2 \arrow{d} \arrow{r} & \Sp^{C_2} \arrow{d} \\
B\mathrm{O}(2) \arrow{r} & \Sp
\end{tikzcd}
\]
factor through $*$ whenever the value of the top map is $H \underline{M}$ and the bottom map is $H M$.\footnote{This is stronger than trivializing the $C_2$-parametrized $S^1$-action since we also have that the $C_2$-action on the underlying spectrum is trivial.}

By \cref{Lem:MapsEM}, the $\mu_2$-action on $H \underline{M}$ is through maps of Mackey functors. But the $\mu_2$-action on the underlying spectrum $H M$ is restricted from the necessarily trivial $S^1$-action, so is trivial, and hence the $\mu_2$-action on $H \underline{M}$ is trivial since the restriction map is the identity. This also forces the $C_2$-action on $H M$ to be trivial, so the entire $\mathrm{O}(2)$-action is trivial.
\end{proof}

Using the $x$-adic filtration, we obtain spectral sequences for computing the $RO(C_2)$-graded homotopy groups of $\TCR^-(H\mft)$ and $\TPR(H\mft)$. 

\begin{prp}
There are strongly convergent multiplicative spectral sequences
\begin{equation}\label{Eqn:HFPSS}
E_2^{s,2t,w} = \pi_{2t+s,w}^{C_2}(\Sigma^{2t,t} H\mft^{h_{C_2}S^1}) \cong \pi_{s,w-t}^{C_2}(H\mft^{h_{C_2}S^1}) \Rightarrow \pi_{2t+s,w-t}^{C_2}\TCR^-(H\mft),
\end{equation}
\begin{equation}\label{Eqn:TSS}
E_2^{s,2t,w} = \pi_{2t+s,w}^{C_2}(\Sigma^{2t,t} H\mft^{t_{C_2}S^1}) \cong \pi_{s,w-t}^{C_2}(H\mft^{t_{C_2}S^1}) \Rightarrow \pi_{2t+s,w-t}^{C_2}\TPR(H\mft).
\end{equation}
In both $E_2$-terms, we assume that $t \geq 0$ and that the $C_2$-parametrized $S^1$-action on the lefthand side is trivial. The differentials in these spectral sequence have the form
$$d_r: E_r^{s,t,w} \to E_r^{s-r,t+r-1,w}.$$
\end{prp}

\begin{proof}
These are the spectral sequences which arise from applying $\pi_{**}^{C_2}(-)$ to the $x$-adic filtrations of $\TCR^-(H\mft)$ and $\TPR(H\mft)$. 

Multiplicativity follows from the observation that the product on $\THR(H\mft)$, the free associative $H\mft$-algebra on $x$, descends to a pairing on the $x$-adic filtrations. This pairing then descends to a pairing on the resulting spectral sequences, for instance via the discussion in \cite[\S II.3.2]{Hed20}. 

Strong convergence follows from \cite[Thm. 7.1]{Boa99}. In more detail, for each fixed $w \in \ZZ$, this is a half-plane spectral sequence with entering differentials. Conditional convergence (\cite[Def. 5.10]{Boa99}) follows from the increasing connectivity of the filtration. Finally, we will show in \cref{Prop:Differentials} that both spectral sequences collapse at $E_4$, so $RE_\infty$ vanishes (cf. \cite[Remark on Page 20]{Boa99}). 
\end{proof}

\begin{rem}
The pairing in the spectral sequences above has the form
$$E_r^{s,2t,w} \otimes E_r^{s', 2t', w'} \longrightarrow E_r^{s+s', 2(t+t'), w+w'}.$$
In particular, if $x \in E_r^{s,2t,0}$ and $y \in E_r^{s',2t',0}$, then $xy \in E_r^{s+s',2(t+t'),0}$, so the products of elements in the figures in this section (which are always for tridegrees with $w=0$) appear in the same figures. 
\end{rem}

We can identify the $E_2$-terms of these spectral sequences more explicitly using the following two propositions:

\begin{prp}\cite[Prop. 6.2]{HK01}
The $RO(C_2)$-graded homotopy groups of $H\mft$ are 
$$\pi_{**}^{C_2}(H\mft) \cong \FF_2[\tau, \rho] \oplus \dfrac{\FF_2[\tau,\rho]}{(\tau^\infty, \rho^\infty)} \{ \theta \},$$
where $|\tau| = (0,-1)$, $|\rho| = (-1,-1)$, and $|\theta| = (0,2)$. 
\end{prp}

%Note: changed to lower ** on righthand side
\begin{prp} \label{prp:identify_E2_page}
There are isomorphisms
$$\pi_{**}^{C_2}(H\mft^{h_{C_2}S^1}) \cong H\mft^{**}(B_{C_2}S^1) \cong H\mft_{**}[[u]],$$
$$\pi_{**}^{C_2}(H\mft^{t_{C_2}S^1}) \cong H\mft^{**}(B_{C_2}S^1)[u^{-1}] \cong H\mft_{**}((u)),$$
where $|u| = (-2,-1).$
\end{prp}

\begin{proof}
The first line of isomorphisms follows from the identification 
% Took out F(E_{C_2}S^1, H\mft)^{S^1} and added in basepoint to equation
$$H\mft^{h_{C_2}S^1} \simeq F((B_{C_2}S^1)_+, H\mft) \simeq \bigoplus_{i \geq 0} \Sigma^{-2i,-i} H\mft$$
along with \cite[Thm. 2.10]{HK01}. The second line follows from the cofiber sequence
$$\Sigma^{1,1} (H\mft)_{h_{C_2}S^1} \xrightarrow{\operatorname{Nm}} (H\mft)^{h_{C_2}S^1} \to (H\mft)^{t_{C_2}S^1}$$
by the same argument used to analyze $H\underline{M}^{t_{C_2}S^1}$ for $\uM$ a $C_2$-indexed coproduct of copies of $\underline{\ZZ}$ in the proof of \cref{prp:S1vsLimit}. More precisely, we have that
$$\Sigma^{1,1} (H\mft)_{h_{C_2}S^1} \simeq \bigoplus_{i \geq 0} \Sigma^{2i+1,i+1} H\mft,$$
so $\Nm$ is determined by a collection of maps
$$\Sigma^{2i+1,i+1} H\mft \to \Sigma^{-2j,-j} H\mft$$
with $i,j \geq 0$. Any such map is adjoint to a map
$$\Sigma^{2i+2j+1, i+j+1} \SS \to H\mft,$$
since the relevant norm map $\Nm$ is an $H\mft$-module map and $H\mft$-module maps out of $H\mft$ correspond to $\SS$-module maps out of $\SS$. Inspection of $\pi_{**}^{C_2}H\mft$ (cf. \cite[Thm. 6.41]{HK01}) shows that any such map is nullhomotopic. Therefore $\Nm$ is nullhomotopic and we obtain 
\begin{align*}
(H\mft)^{t_{C_2}S^1} & \simeq H\mft^{h_{C_2}S^1} \oplus \Sigma^{2,1} H\mft_{h_{C_2}S^1}  \\
				&\simeq \lim_{N \rightarrow -\infty} \left( \bigoplus_{n = N}^\infty \Sigma^{2n,n}H\mft \right) \\
				&\simeq \prod_{n=-\infty}^\infty \Sigma^{2n,n}H\mft.
\end{align*}
Let 
% $\{\ldots, t^2, t, 1, x, x^2, \ldots\}$, $|t| = (-2,-1)$ and $|x| = (2,1)$,
\[ \{\ldots, t^2, t, 1, x, x^2, \ldots\}, \: |t| = (-2,-1) \text{ and } |x| = (2,1) \]
be an additive $H\mft_{**}$-basis for $\pi_{**}^{C_2}H\mft^{t_{C_2}S^1}$. We deduce that $\pi_{**}^{C_2}H\mft^{t_{C_2}S^1}$ has the claimed ring structure as follows. In this basis, we have
$$tx = \sum_{i \geq 0} a_i t^i + a_0 \cdot 1 + \sum_{j \geq 0} b_j x^j$$
with $a_i \in H\mft_{2i,i}$ and $b_j \in H\mft_{-2j,-j}$. Again, inspection of $H\mft_{**}$ shows that $a_i = 0$ for $i >0$, $b_j=0$ for $j>0$, and $a_0 \in \FF_2$. Therefore $tx=a_0 \cdot 1$ with $a_0$ either zero or one. Classical results (e.g. \cite[Thm. 16.3]{GM95}) imply that the claimed ring structure holds on the underlying level, i.e. $\res_e^{C_2}(t)\res_e^{C_2}(x) = \res_e^{C_2}(1)$ in $\pi_* H\FF_2^{tS^1}$. Since restriction is a ring homomorphism, we must have that $a_0 = 1$, so the claimed ring structure holds. 
\end{proof}

\begin{cor}
The $E_2$-term of the spectral sequence \eqref{Eqn:HFPSS} is given by
$$E_2^{***} \cong \bigoplus_{i \geq 0} \left( \FF_2[\tau, \rho] \oplus \dfrac{\FF_2[\tau,\rho]}{(\tau^\infty, \rho^\infty)} \{ \theta \} \right)[[u]] \{x^i\}.$$
The $E_2$-term of the spectral sequence \eqref{Eqn:TSS} is given by
$$E_2^{***} \cong \bigoplus_{i \geq 0} \left( \FF_2[\tau, \rho] \oplus \dfrac{\FF_2[\tau,\rho]}{(\tau^\infty, \rho^\infty)} \{ \theta \} \right)((u)) \{ x^i\}.$$
The tridegrees of the generators are $|\tau| = (0,0,-1)$, $|\rho| = (-1, 0, -1)$, $|\theta| = (0,0,2)$, $|u| = (-2,0,-1)$, and $|x| = (0,2,1)$. 
\end{cor}

The $E_2$-terms of the spectral sequences \eqref{Eqn:HFPSS} and \eqref{Eqn:TSS} are depicted in \cref{Fig:HFPSSE2} and \cref{Fig:TSSE2}, respectively, in the tridegrees where $w=0$.

\begin{figure}
\includegraphics{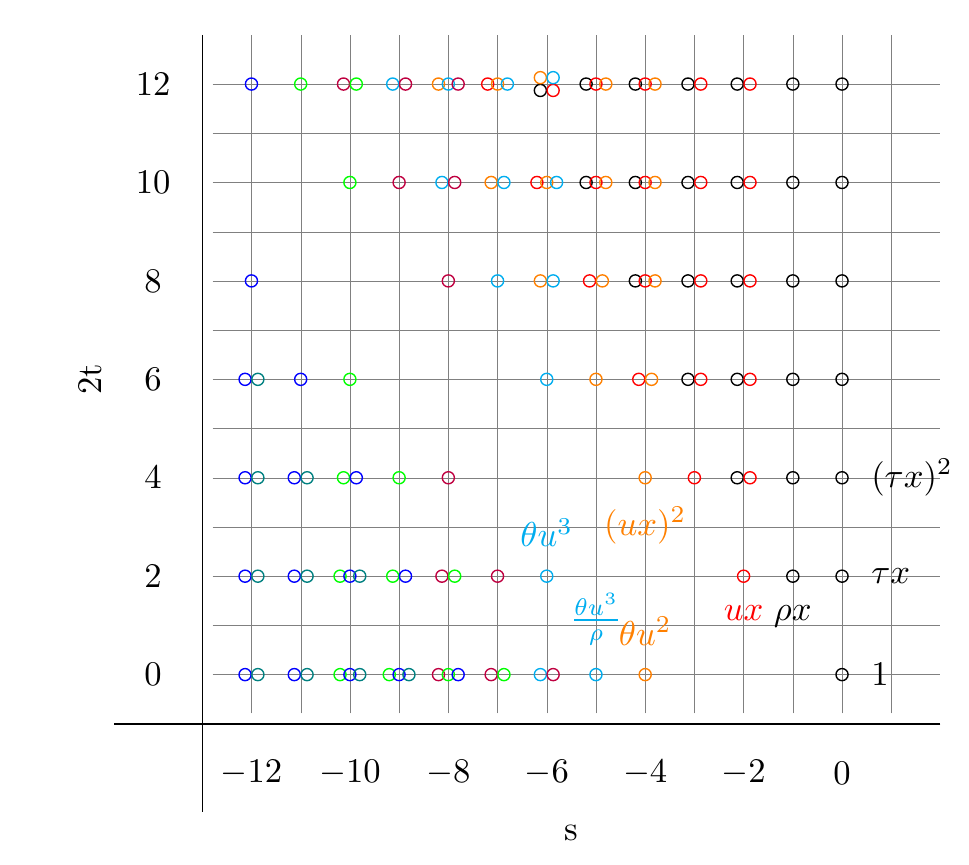}
\caption{The $E_2$-term of the spectral sequence \eqref{Eqn:HFPSS} when $w=0$.}\label{Fig:HFPSSE2}
\end{figure}

\begin{figure}
\includegraphics{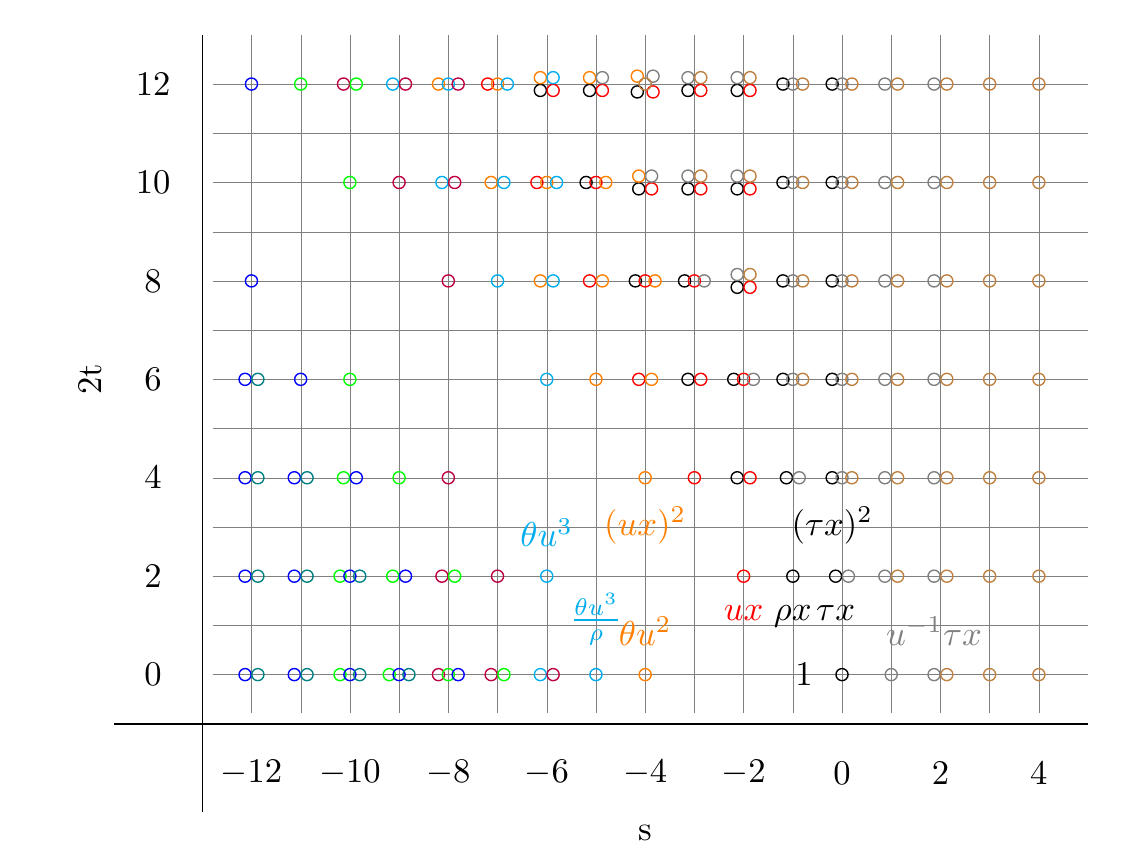}
\caption{The $E_2$-term of the spectral sequence \eqref{Eqn:TSS} when $w=0$.}\label{Fig:TSSE2}
\end{figure}

\begin{lem}
The $d_2$-differentials in the spectral sequences \eqref{Eqn:HFPSS} and \eqref{Eqn:TSS} are all trivial. 
\end{lem}

\begin{proof}
This happens for degree reasons: the $E_2$-term of each spectral sequence is concentrated in tridegrees $(*,2t,*)$ and $d_2$ changes the parity of the second index. 
\end{proof}

\begin{lem}
We have
$$d_r(u) = 0, \quad d_r(x)=0, \quad  d_r(\rho) = 0, \quad \text{ and } d_r(\theta)=0$$
for all $r \geq 2$. 
\end{lem}

\begin{proof}
We begin with $u$. We have $|u| = (-2,0,-1)$, so $|d_r(u)| = (-2-r,r-1,-1)$. Inspection of the $E_2$-term reveals that the target tridegrees are all zero, so $d_r(u)=0$ for all $r \geq 2$. 

Next, we consider $x$. Observe that $|ux| = (-2,2,0)$ restricts to a permanent cycle in the analogous spectral sequences for $\TC^-(H\FF_2)$ and $\TP(H\FF_2)$, so $d_r(ux) = 0$ for all $r \geq 2$. But then
$$d_r(ux) = d_r(u) x + u d_r(x) = u d_r(x) = 0,$$
and since the targets of $d_r(x)$ are not $u$-torsion, we have $d_r(x)=0$. 

We now turn to $\rho$. We have $|\rho x| = (-1,2,0)$, so $|d_r(\rho)| = (-1-r,2+r-1,0)$. If $r$ is even, the target tridegree is zero, but if $r$ is odd, the target tridegree is a copy of $\FF_2$ generated by $(ux)^i$ for some $i \geq 2$. The restrictions of $(ux)^i$ in the analogous spectral sequences for $\TC^-(H\FF_2)$ and $\TP(H\FF_2)$ survive, so $(ux)^i$ must survive and thus $d_r(\rho x) = 0$. Using the Leibniz rule and the fact that $d_r(x)=0$ proven above, we then conclude that $d_r(\rho)=0$. 

Finally, we consider $\theta$. By the Leibniz rule, it suffices to show that $d_r(\theta x^2)=0$. In fact, $d_r(\theta) = 0$ for $r = 2$ and $r \geq 4$ for degree reasons, so it suffices to show that $d_3(\theta u^2) =0$. The only possible target of this differential is $u^4 \frac{\theta}{\rho}$, so suppose that $d_3(\theta u^2) = u^4 \frac{\theta}{\rho}$. Then the Leibniz rule implies that
$$d_3(\theta u^2 \rho) = d_3(\theta u^2) \rho + \theta u^2 d_3(\rho) = u^4 \theta,$$
but on the other hand, $\theta u^2 \rho = 0$ since $\theta \rho = 0$, and thus
$$d_3(\theta u^2 \rho) = d_3(0) = 0,$$
which is a contradiction. Therefore $d_3(\theta u^2)=0$, and thus $d_3(\theta) = 0$. 
\end{proof}

\begin{prp}\label{Prop:Differentials}
In both spectral sequences, we have
$$d_3(\tau) = u\rho x.$$ 
This differential, along with the fact that $u$, $\rho$, $\theta$, and $x^i$ are permanent cycles, determine all remaining $d_3$-differentials via the Leibniz rule. 
\end{prp}

\begin{proof}
The map from categorical to geometric fixed points induces a map
$$\pi_*^{C_2}(\THR(H\mft)^{h_{C_2}S^1}) \to \pi_*(\THR(H\mft)^{h_{C_2}S^1 \phi C_2}).$$
Composing this with the canonical map interchanging parametrized homotopy fixed points and geometric fixed points, we get
\begin{equation}\label{Eqn:Comparison}
\pi^{C_2}_*(\THR(H\mft)^{h_{C_2}S^1}) \to \pi_*(\THR(H\mft)^{\phi C_2 h \mu_2}).
\end{equation}

% Recall (cf. \cite[Thm. 2.26]{DMPR17}) that
% $$\THR(H\mft)^{\phi C_2} \simeq H\mft^{\phi C_2} \otimes_{H\FF_2} H\mft^{\phi C_2},$$
% and thus (cf. \cite[Prop. 5.19]{DMPR17})
By \cref{prp:THRF2_geometric_fixed_points}, we have an isomorphism
$$\pi_{*}^{C_2}(\THR(H\mft)^{\phi C_2}) \cong \FF_2[w_1,w_2],$$
where $|w_1| = |w_2| = 1$ and $\mu_2$ acts by swapping $w_1$ and $w_2$.  

Let $x_1 := w_1 + w_2$ and $x_2 := w_2$. Note that $x$ is fixed under the $\mu_2$-action since $\pi_{2,1}^{C_2}(\THR(H\mft)) \cong \FF_2\{x\}$. Since $x_1$ is the only nontrivial $\mu_2$-fixed element in our presentation of $\pi_1(\THR(H\mft)^{\phi C_2})$, the $x$-adic filtration on $\THR(H\mft)$ induces the $x_1$-adic filtration on $\THR(H\mft)^{\phi C_2}$: 
\[
\begin{tikzcd}
 \FF_2[x_1, x_2]\{1\} \arrow{d}{p} & \FF_2[x_1,x_2]\{x_1\} \arrow[swap,l,"i"] \arrow{d}{p} & \FF_2[x_1,x_2]\{x_1^2\} \arrow[swap,l,"i"] \arrow{d}{p} & \cdots \arrow[swap,l,"i"] \\
\FF_2[\bar{x}_2]\{1\} \arrow[swap, dashed,ur,"\partial"] & \Sigma \FF_2[\bar{x}_2]\{x_1\} \arrow[swap,dashed,ur,"\partial"] & \Sigma^2 \FF_2[\bar{x}_2]\{x_1^2\}
\end{tikzcd}
\]
Applying $\pi_*(-)^{h \mu_2}$, we obtain a strongly convergent spectral sequence of the form
\begin{equation}\label{Eqn:GFPHmu2}
E_2^{s,t} = \pi_{s+t}((\FF_2[\bar{x}_2]\{x_1^t\})^{h\mu_2}) \cong \pi_s ( \FF_2[\bar{x}_2][[v]]) \Rightarrow \pi_{s+t} \THR(H\mft)^{\phi C_2 h \mu_2},
\end{equation}
where $|\bar{x}_2| = (1,0)$ and $|v| = (0,-1)$.\footnote{We emphasize that this spectral sequence is \emph{not} the usual $\mu_2$-homotopy fixed points spectral sequence.}

The map \eqref{Eqn:Comparison} induces a map from the $w=-1$ piece of the spectral sequence \eqref{Eqn:HFPSS} to the spectral sequence \eqref{Eqn:GFPHmu2} sending $\tau$ to $\bar{x}_2$. Indeed, the $\mu_2$-action on $\pi_{**}^{C_2} \THR(H\mft)$ is given by the right unit of the Hopf algebroid structure of $(H\mft_{**}, \THR(H\mft)_{**})$ described in \cite[Thm. 2.3]{HW21}, so it sends $\tau$ to $\tau + \rho^2 x$. In $\pi_{*}\THR(H\mft)^{\phi C_2}$, the elements which are not fixed by $\mu_2$ are $x_2$ and $x_1+x_2$. Therefore we have that $\tau$ maps to $\bar{x}_2$, as claimed. We will show that 
\begin{equation}\label{Eqn:KeyDiff}
d_2(\bar{x}_2) \neq 0,
\end{equation}
from which it will follow that $d_3(\tau) \neq 0$. Inspection of the $E_3$-term of the spectral sequence \eqref{Eqn:HFPSS} reveals that $u \rho x^2$ is the only nontrivial target of a $d_3$-differential on $\tau x$, so we will have proven the proposition. 

It remains to establish the differential \eqref{Eqn:KeyDiff}, i.e., to show that
$$p^h(\partial^h(\bar{x}_2)) \neq 0,$$
where $p^h$, $\partial^h$, and $i^h$ are the induced maps on $\pi_*(-)^{h\mu_2}$. 

We first show that $\partial^h(\bar{x}_2) \neq 0$. By exactness, we have $\ker(\partial) = \im(p)$. Since we will compute group cohomology, we change basis again to
$$\FF_2[x_1,x_2] \cong \FF_2[w_1,w_2].$$
In this basis, we have 
$$\FF_2[w_1,w_2]^{h\mu_2} \cong \FF_2\{w_1^iw_2^j + w_1^j w_2^i\}_{j>i\geq 0} \oplus \FF_2[[v]]\{(w_1w_2)^i\}_{i \geq 0}$$
where $|v| = -1$. Now, note that $p^h(w_1^i w_2^j) = \bar{x}_2^{i+j}$ since $p^h(x_1)=0$ and $p^h(x_2)=\bar{x}_2$. In particular, we see that $\bar{x}_2 \notin \im(p^h)$ and thus $\partial^h(\bar{x}_2) \neq 0$. 

We now consider the possible values of $\partial^h(\bar{x}_2)$. Again by exactness, we have $\im(\partial^h) = \ker(i^h)$. Since $x_1$ is fixed by the $\mu_2$-action, we have
$$(\FF_2[w_1,w_2]\{x_1\})^{h\mu_2} \cong \FF_2\{(w_1^iw_2^j + w_1^j w_2^i)x_1\}_{j>i\geq 0} \oplus \FF_2[[v]]\{(w_1w_2)^i x_1\}_{i \geq 0}.$$
The first summand is mapped injectively into the corresponding piece of $\FF_2[w_1,w_2]^{h\mu_2}$ under $i^h$. On the other hand, we have
$$i^h( (w_1w_2)^i x_1) = w_1^{i+1}w_2^i + w_1^i w_2^{i+1},$$ 
which is an induced class. Therefore
$$i^h( (w_1 w_2)^i x_1 v^j) = 0$$
for $j > 0$. Restricting to terms in total degree $1 = |\bar{x}_2|$, we conclude
$$\partial^h(\bar{x}_2) \neq 0\in \ker(i^h|_{\pi_0}) \cong \FF_2\{(w_1w_2)^i x_1 v^{2i+1}\}_{i \geq 0}.$$

Finally, we check that $p^h(\partial^h(\bar{x}_2)) \neq 0$. Recalling that $(w_1w_2)^i = (x_2(x_1+x_2))^i$, we have
$$p^h((w_1 w_2)^i x_1 v^{2i+1}) = x_1 \bar{x}_2^{i+1} v^{2i+1} \neq 0.$$
Therefore $p^h(\partial^h(\bar{x}_2)) \neq 0$, i.e. $d_2(\bar{x}_2) \neq 0$ as claimed.
\end{proof}

\begin{figure}
\includegraphics{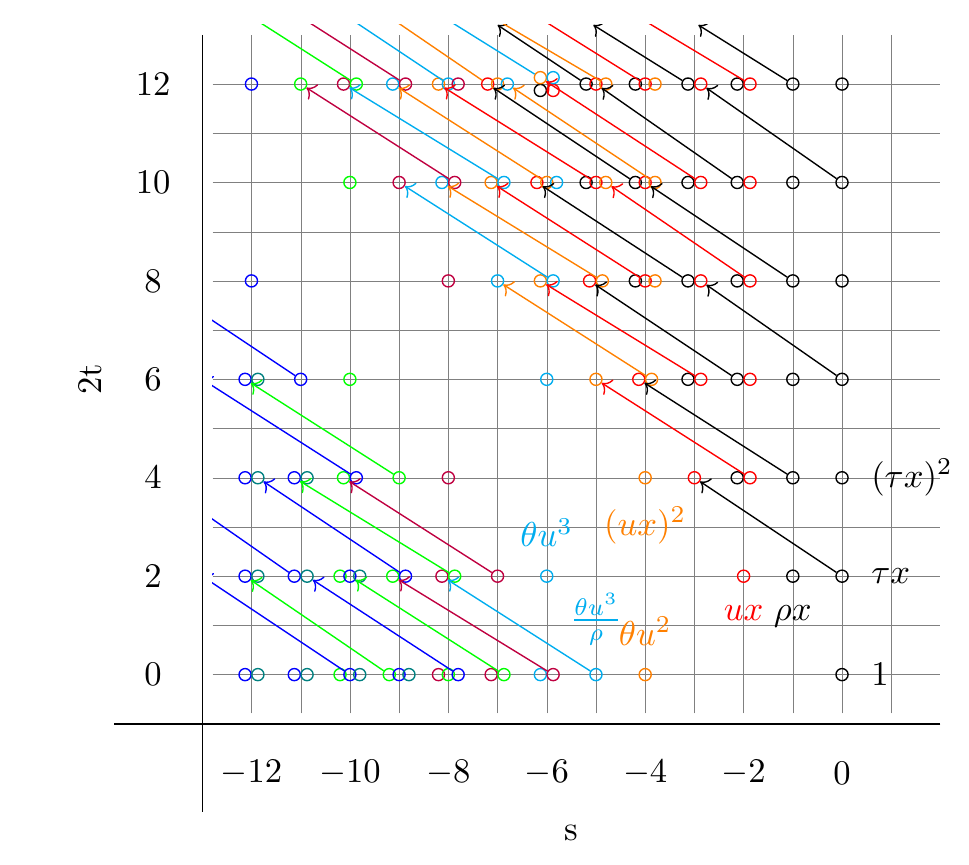}
\caption{The $E_3$-term of the spectral sequence \eqref{Eqn:HFPSS} when $w=0$. The color of a differential matches the color of its source. }\label{Fig:HFPSSE3}
\end{figure}

\begin{figure}
\includegraphics{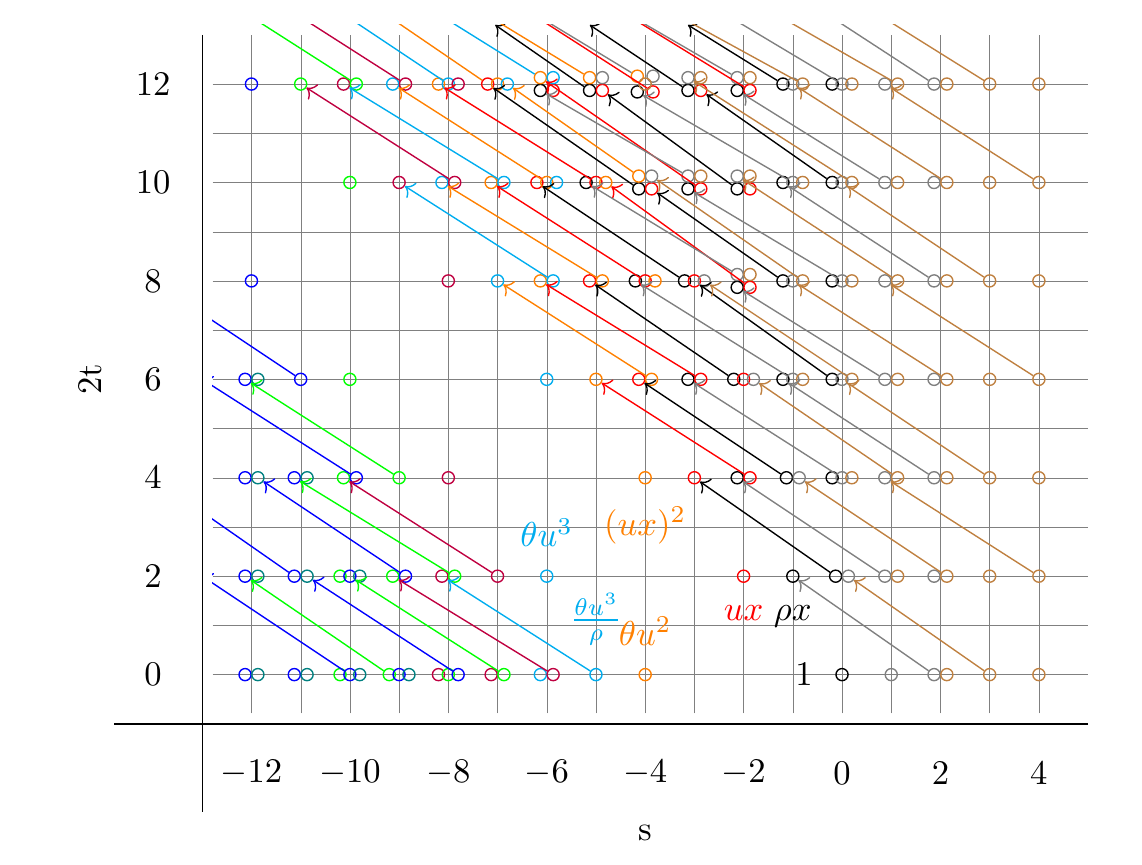}
\caption{The $E_3$-term of the spectral sequence \eqref{Eqn:TSS} when $w=0$. The color of a differential matches the color of its source. }\label{Fig:TSSE3}
\end{figure}

\begin{cor}
The spectral sequences \eqref{Eqn:HFPSS} and \eqref{Eqn:TSS} collapse at $E_4$. 
\end{cor}

\begin{proof}
This happens for degree reasons, along with the fact that $u$, $x$, $\rho$, and $\theta$ are permanent cycles. The situation is depicted in \cref{Fig:HFPSSE4} and \cref{Fig:TSSE4}. 
\end{proof}

\begin{figure}
\includegraphics{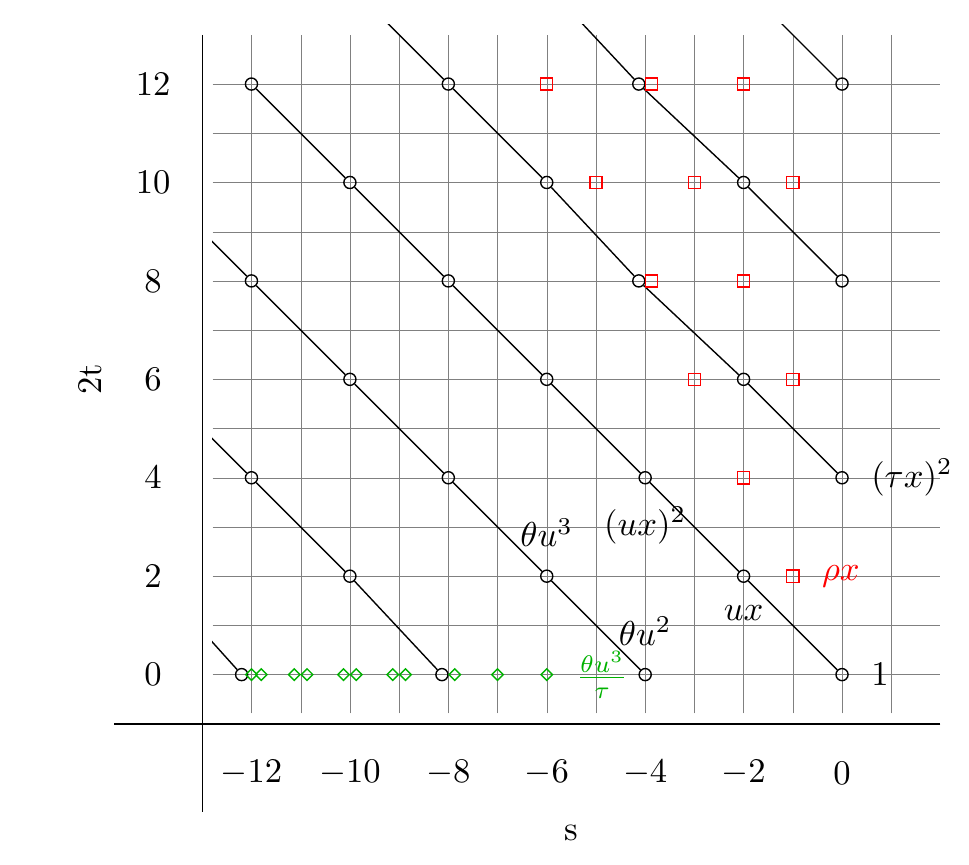}
\caption{The $E_4=E_\infty$-term of the spectral sequence \eqref{Eqn:HFPSS} when $w=0$.}\label{Fig:HFPSSE4}
\end{figure}

\begin{figure}
\includegraphics{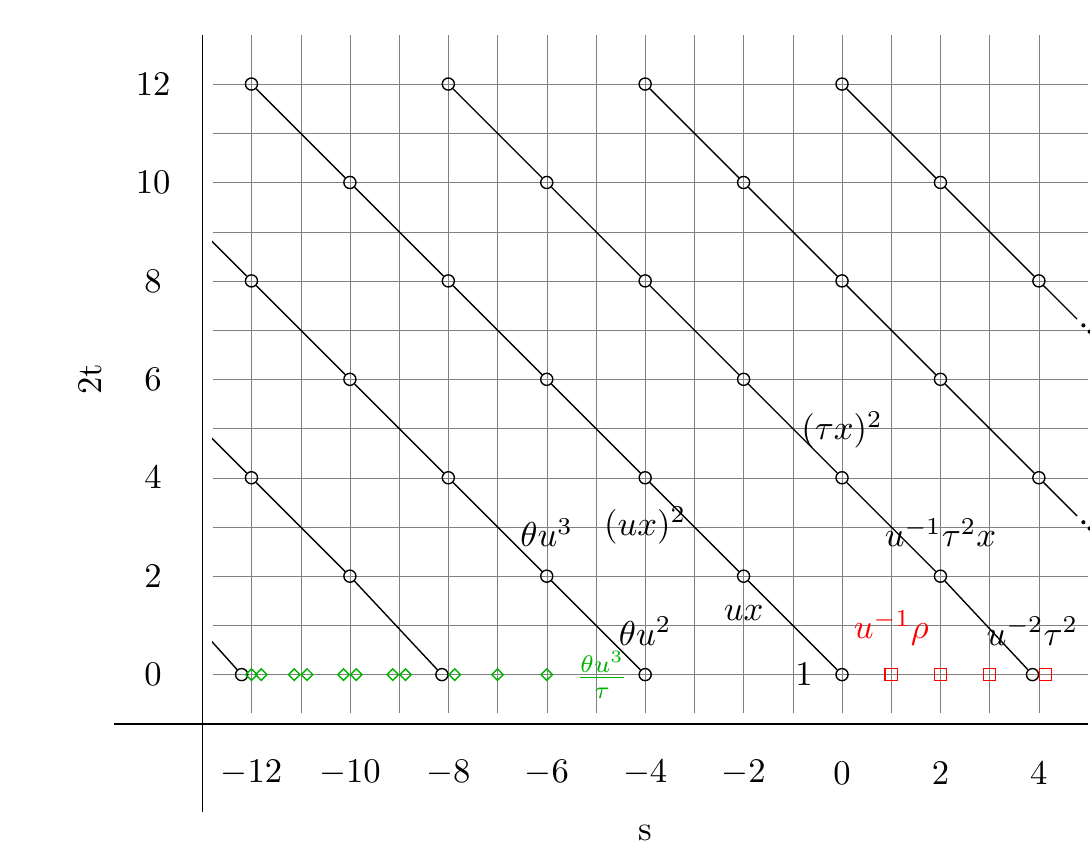}
\caption{The $E_4=E_\infty$-term of the spectral sequence \eqref{Eqn:TSS} when $w=0$. \newline Observe how the red classes in \cref{Fig:HFPSSE4} get ``squashed'' onto the 0-line.}\label{Fig:TSSE4}
\end{figure}

%The $E_4=E_\infty$-page of both spectral sequence can be identified with more familiar terms.

%\begin{prp}\cite[Fig. 9.2]{BS20}
%There is an isomorphism of $RO(C_2)$-graded rings
%$$\pi_{**}^{C_2} H\underline{\ZZ_2} \cong \dfrac{\ZZ_2[\tau^2, \rho, 2 \tau^{-2k}]}{(2\rho)} \oplus \dfrac{\FF_2[\tau^2,\rho]}{(\tau^\infty, \rho^\infty)} \{ \gamma \},$$
%where $|\tau^2| = (0,-2)$, $|\rho| = (-1,-1)$, $|\gamma| = (0,3)$.
%\end{prp}

\begin{thm}\label{Thm:TCR-TPR}
There are isomorphisms of graded abelian groups
$$\pi_*^{C_2} \TCR^-(H\mft) \cong \dfrac{\ZZ_2[\tau^2 x^2, \rho x]}{2 \cdot \rho x}\left\{1\right\} \oplus \ZZ_2\left[\frac{u^2}{\tau^2}\right]\{\theta u^2\} \oplus \FF_2\left\{ \dfrac{\theta}{\rho^{i}\tau^{2j+1}} u^{i+2j+3} : i,j \geq 0 \right\},$$
$$\pi_*^{C_2} \TPR(H\mft) \cong \dfrac{\ZZ_2[\tau^2 u^{-2}, \rho u^{-1}]}{2 \cdot \rho u^{-1}}\{1\} \oplus \ZZ_2\left[\frac{u^2}{\tau^2}\right]\{\theta u^2\} \oplus \FF_2\left\{ \dfrac{\theta}{\rho^{i}\tau^{2j+1}} u^{i+2j+3} : i,j \geq 0 \right\}.$$
\end{thm}

\begin{proof}
This follows from inspection of the $E_\infty$-pages above, except for the claimed multiplication by $2$ extensions. For these, we observe that restriction from $\TCR^-(H\mft)$ to $\TC^-(H\FF_2)$ implies that $ux$ detects $2$ above. The remaining extensions for $\TCR^-(H\mft)$, as well as $\TPR(H\mft)$, follow similarly. 
\end{proof}

\begin{prp}\label{Prop:can}
The canonical map
$$\can: \pi_{*}^{C_2} \TCR^-(H\mft) \to \pi_{*}^{C_2} \TPR(H\mft)$$
is as follows\footnote{In this statement, we identify classes in $\pi_*^{C_2}$ with the class detecting them in the spectral sequences above.}:
\begin{enumerate}
\item The black class $\circ$ in $(s,2t)$ maps to the corresponding black class $\circ$ in $(s,2t)$.
\item Each green class \textcolor{green}{$\diamond$} in $(s,2t)$ maps to the corresponding green class \textcolor{green}{$\diamond$} in $(s,2t)$. 
\item Each red class \textcolor{red}{{\scriptsize $\square$}} in $(s,2t)$ maps to zero. 
\end{enumerate}
\end{prp}

\begin{proof}
The canonical map induces an inclusion from the $E_2$-page of the spectral sequence \eqref{Eqn:HFPSS} into the $E_2$-page of the spectral sequence \eqref{Eqn:TSS}. The proposition then follows by examining the induced map on $E_4 = E_\infty$. 
\end{proof}

\begin{prp}\label{Prop:phi}
The real cyclotomic Frobenius 
$$\varphi := \varphi_2^{h_{C_2}S^1} : \pi_{**}^{C_2} \TCR^-(H\mft) \to \pi_{**}^{C_2} \TPR(H\mft)$$
is determined by
$$\varphi(x) = u^{-1}, \quad \varphi(u) = 2u, \quad \varphi(\tau^2) \equiv \tau^2 \mod \rho, \quad \varphi(\rho) = \rho, \quad  \varphi(\theta) = \theta.$$
In particular, for $w=0$:
\begin{enumerate}
\item The black class $\circ$ in $(s,2t)$ maps to the black class $\circ$ in $(2s+2t,-s)$ when $2t \geq -s$; otherwise it maps to the corresponding black class $\circ$ in $(s,2t)$.
\item Each green class \textcolor{green}{$\diamond$} in $(s,2t)$ maps to zero.
\item Each red class \textcolor{red}{{\scriptsize $\square$}} in $(s,2t)$ maps to a corresponding red class \textcolor{red}{{\scriptsize $\square$}} in $(s+2t,0)$. 
\end{enumerate}
\end{prp}

\begin{proof}
We begin by showing $\varphi(x) = u^{-1}$. Since $\varphi$ is a map of ring $C_2$-spectra, we obtain a map of $C_2$-Green functors 
\begin{equation}\label{Eqn:PhiRes}
\begin{tikzcd}
\pi_{i,j}^{C_2}\TCR^-(H\mft) \arrow{r}{\varphi} \arrow[d, swap, "\res"] & \pi_{i,j}^{C_2}\TPR(H\mft) \arrow{d}{\res} \\
\pi_i \TC^-(H\FF_2) \arrow{r}{\varphi} & \pi_i \TP(H\FF_2)
\end{tikzcd}
\end{equation}
with
\[
\begin{tikzcd}
x \arrow[r, mapsto, "\varphi"] \arrow[d, mapsto, swap, "\res"] & z \arrow[d, mapsto, "\res"] \\
x \arrow[r, mapsto, "\varphi"] & t^{-1}.
\end{tikzcd}
\]
We \emph{define} $u^{-1} := z \in \pi_{**}^{C_2}\TPR(H\mft)$ to be the element making this diagram commute. 

Note that our choice of $u^{-1}$ determines $u^k \in \pi_{**}^{C_2}\TPR(H\mft)$ for all $k \in \ZZ$. Chasing $u$ around the diagram \eqref{Eqn:PhiRes}, along with the fact that $\varphi(t) = 2t$ in $\TP(H\FF_2)$ (cf. \cite[Prop. IV.4.9]{NS18}), proves that $\varphi(u) = 2u$ in $\TPR(H\mft)$. 

We now show that $\varphi(\tau^2) \equiv \tau^2 \mod \rho$. We first consider the map of $C_2$-Green functors
\[
\begin{tikzcd}
\pi_{i,j}^{C_2} \TCR^-(H\mft) \arrow{r} \arrow[d, swap, "\res"] & \pi_{i,j}^{C_2} \THR(H\mft) \arrow{d}{\res} \\
\pi_i \TC^-(H\FF_2) \arrow{r} & \pi_i \THH(H\FF_2).
\end{tikzcd}
\]
When $i=0$ and $j=-2$, we have
\[
\begin{tikzcd}
\tau^2 \arrow[r, mapsto] \arrow[d, mapsto, swap, "\res"] & \tau^2 \arrow[d, mapsto, "\res"] \\
\alpha \arrow[r, mapsto] & 1
\end{tikzcd}
\]
since $\res(\tau^2) = 1 \in \pi_* H\FF_2$ and thus also in $\pi_* \THH(H\FF_2)$. Since the map from $\TC^-(H\FF_2)$ to $\THH(H\FF_2)$ is a ring map, we have that $\alpha \in \ZZ_2^\times$. 

Now, we apply this in the diagram \eqref{Eqn:PhiRes} to obtain
\[
\begin{tikzcd}
\tau^2 \arrow[r, mapsto, "\varphi"] \arrow[d, mapsto, swap, "\res"] & w \arrow[d, mapsto, "\res"] \\
\alpha \arrow[r, mapsto, "\varphi"] & \alpha.
\end{tikzcd}
\]
Inspection of the relevant bidegrees reveals that the righthand map can be identified with 
$$\res: \ZZ_2 \{ \tau^2 \} \oplus \FF_2 \{ \rho^2 t^{-1}\} \to \ZZ_2\{\alpha\}.$$
In particular, this map is nonzero, so we must have that $w \equiv \tau^2 \mod \rho$. 

Finally, we have $\varphi(\rho) = \rho$ and $\varphi(\theta) = \theta$ since $\rho$ and $\theta$ are spherical. 
\end{proof}

%\begin{rem}\label{Rmk:phi}
%As with the canonical map, closer inspection of the $E_\infty$-pages reveals that $\phi$ sends the $2$-torsion summand of $\pi_{**}^{C_2}(\TCR^-(H\mft))$ to zero. It sends the second summand consisting of $\rho$-torsion free $2$-torsion classes in $\pi_{**}^{C_2}(H\mft)$ isomorphically onto the second summand of $u^{-1}\rho$-torsion free $2$-torsion classes in $\pi_{**}^{C_2}(\TPR(H\mft))$. 
%\end{rem}

\begin{proof}[Proof of \cref{Thm:TCRF2}]
Applying $\pi_{*}^{C_2}$ to the fiber sequence \eqref{Eqn:FiberSeq} gives rise to a long exact sequence
$$\cdots \to \pi_{*}^{C_2} \TCR(H\mft) \to \pi_{*}^{C_2} \TCR^-(H\mft) \xrightarrow{\varphi - \can} \pi_{*}^{C_2} \TPR(H\mft) \xrightarrow{\partial} \pi_{*-1}^{C_2} \TCR(H\mft) \to \cdots.$$
\cref{Thm:TCR-TPR} describes the homotopy groups of the middle two terms, and  \cref{Prop:can} and \cref{Prop:phi} describe the map between them. A relatively straightforward computation of kernels and cokernels then shows that
$$\pi_{*}^{C_2}\TCR(H\mft) \cong \ZZ_2 \oplus \Sigma^{-1} \ZZ_2,$$
where $2 \in \ZZ_2$ is detected by $ux$ and $\Sigma^{-1} 2 \in \Sigma^{-1} \ZZ_2$ is detected by $\partial( ux)$. 

Now, since $\TCR(H\mft)^e \simeq \TC(H\FF_2)$, we have that 
$$\pi_*^e \TCR(H\mft) \cong \ZZ_2 \oplus \Sigma^{-1} \ZZ_2,$$
where $\ZZ_2$ and $\Sigma^{-1} \ZZ_2$ are generated by the restrictions of the $C_2$-equivariant generators above. In particular, this implies the restriction
$$\res_e^{C_2}: \pi_*^{C_2} \TCR(H\mft) \to \pi_*^e \TCR(H\mft)$$
is the identity map $\ZZ_2 \to \ZZ_2$ for $* = 0,-1$. This forces the $C_2$-action on $\TC(H \FF_2)$ to be trivial, and we must then have $\res_e^{C_2} \circ \tr_e^{C_2} = \cdot 2$. The transfer 
$$\tr_e^{C_2} : \pi_*^{e}(\TCR(H\mft)) \to \pi_*^{C_2}(\TCR(H\mft))$$
is then forced to be multiplication by $2$. Altogether, this implies that we have an isomorphism of Mackey functors
$$\pi_*^{(-)} (\TCR(H\mft)) \cong \underline{\ZZ_2} \oplus \Sigma^{-1} \underline{\ZZ_2}.$$

To obtain a spectrum-level splitting, we observe that the inclusion of the connective cover
$$H\underline{\ZZ_2} \simeq \tau_{\geq 0} \TCR(H\mft) \to \TCR(H\mft)$$
equips $\TCR(H\mft)$ with the structure of an $H\underline{\ZZ_2}$-module. It therefore splits as claimed into a wedge of Eilenberg--MacLane spectra. 
\end{proof}

% As a consequence of \cref{Thm:TCRF2}, we can determine the $RO(C_2)$-graded homotopy groups of $\TCR^-(H\mft)$ and $\TPR(H\mft)$. 

% \begin{thm}
% There are isomorphisms of $RO(C_2)$-graded rings 
% $$\TCR^-(H\mft) \simeq \pi_{**}^{C_2}(H\underline{\ZZ_2})[x,t]/(xt-2),$$
% $$\TPR(H\mft) \simeq \pi_{**}^{C_2}(H\underline{\ZZ_2})[t^{\pm 1}].$$
% \end{thm}

% \begin{proof}
% The composite
% $$H\underline{\ZZ_2} \simeq \tau_{\geq 0} \TCR(H\mft) \to \TCR^-(H\mft)$$
% induces a map of filtered $RO(C_2)$-graded rings
% $$\pi_{**}^{C_2}(H\underline{\ZZ_2})[x,t]/(xt-2) \to \pi_{**}^{C_2}(\TCR^-(H\mft))$$
% sending $t$ to $u$ and $x$ to $x$. Using the known computation of $\pi_{**}^{C_2}(H\underline{\ZZ_2})$ (cf. \cite[\S 9]{BehrensShah}), we see that this map induces an isomorphism on associated gradeds, so the first isomorphism holds. The second isomorphism follows similarly by composing with the canonical map. 
% \end{proof}

\begin{rem}
By inspection of the $RO(C_2)$-graded homotopy of $H \underline{\ZZ_2}$ (cf. \cite[\S 9]{BehrensShah}) and running the spectral sequences for all weights, we observe abstract isomorphisms
\begin{subequations}
\begin{align}
\pi_{**}(\TCR^-(H\mft)) & \cong (H\underline{\ZZ_2})_{**}[x,u]/(x u-2) \label{eq:tcr_minus} \\ 
\pi_{**}(\TPR(H\mft)) & \cong (H\underline{\ZZ_2})_{**}[u^{\pm 1}] \label{eq:tpr}
\end{align}
\end{subequations}
of $RO(C_2)$-graded abelian groups. In fact, a finer analysis of the multiplicative structure should show that the map $H\underline{\ZZ_2} \simeq \tau_{\geq 0} \TCR(H\mft) \to \TCR^-(H\mft)$ of $C_2$-$E_{\infty}$-algebras extends on homotopy groups to define \eqref{eq:tcr_minus} as an isomorphism of $RO(C_2)$-graded rings, and likewise for $\TPR(H\mft)$. As we do not need this to establish \cref{Thm:TCRF2}, we will leave the details to the interested reader.
\end{rem}

\begin{rem}
Our computation of $\TCR(H\underline{\mathbb{F}_2})$ using the $x$-adic filtration can be modified to give an alternative computation of $\TCR(H\underline{\mathbb{F}_p})$ for $p$ odd. The relevant changes are as follows:
\begin{enumerate}
% we employ an odd primary version of \cref{prp:trivial_action} to ensure that the $C_2$-parametrized $S^1$-action is trivial on each filtration quotient $\Sigma^V H\underline{\mathbb{F}_p}$.
\item We use \cref{prp:trivial_action} with respect to the filtration quotients $\Sigma^{2k,k} H \mfp$ and then the odd primary version of \cref{prp:identify_E2_page}.

\item The $RO(C_2)$-graded homotopy groups of $H\mfp$ for $p$ odd are substantially sparser than for $H\mft$. In particular, we have 
$$\pi_{**}^{C_2}(H\mfp) \cong \FF_p[\tau^{\pm 2}]$$
where $|\tau| = (0,-2)$. This follows, for instance, from \cite{Sta16} and the techniques in \cite{BehrensShah}. This implies that the $E_2$-term of the $x$-adic spectral sequence for $\TCR^-(H\mfp)$ has the form
$$E_2^{***} \cong \bigoplus_{i \geq 0} \left( \FF_p[\tau^2][[u]] \right) \{x^i\},$$
and similarly for the spectral sequence for $\TPR(H\mfp)$. 

\item For degree reasons, the spectral sequences for $\TCR^-(H\mfp)$ and $\TPR(H\mfp)$ both collapse at $E_2$. This drastically simplifies the computation. Extensions can be handled as in the $2$-primary case. 

\item The canonical and Frobenius maps can both be identified as in the $2$-primary case to show that $\TCR(H\mfp)$ has the desired form. 

\end{enumerate}
\end{rem}

\subsection{Extension to perfect \texorpdfstring{$\FF_p$}{Fp}-algebras}\label{SS:Perf}

When $k$ is a perfect field of characteristic $p>0$, Dotto--Moi--Patchkoria \cite{DMP21} have computed $\TCR(H\underline{k})$ using a genuine real cyclotomic approach. We now discuss the extension of our computations above using the Borel real cyclotomic approach to perfect $\FF_p$-algebras. Throughout this section, we fix a prime $p$ and let $k$ be a perfect $\FF_p$-algebra. Our arguments will closely follow the computation of $\TC(Hk)$ presented by Nikolaus in \cite{KNYouTube}. 

We begin by recalling the $RO(C_2)$-graded coefficients of the relevant Eilenberg--MacLane spectra. 

\begin{lem}
There are isomorphisms of $RO(C_2)$-graded rings
$$\pi_{**}^{C_2} H\underline{k} \cong k \otimes_{\FF_p} \pi_{**}^{C_2} H\underline{\FF_p},$$
$$\pi_{**}^{C_2} H\underline{W(k)} \cong W(k) \otimes_{\ZZ_p} \pi_{**}^{C_2} H\underline{\ZZ_p}.$$
\end{lem}

\begin{proof}
The $RO(C_2)$-graded coefficients of any Eilenberg--MacLane spectrum, along with a recipe for analyzing their multiplicative structure, can be found in \cite{Sik21}. Alternatively, the coefficients of $H\underline{k}$ (resp. $H\underline{W(k)}$) can be obtained from the coefficients of $H\underline{\FF_p}$ (resp. $H\underline{\ZZ_p}$) by base-changing along the map $\FF_p \to k$ (resp. $\ZZ_p \to W(k)$). This can be proven using the universal coefficient theorem and the Tate diagram. 
\end{proof}

We have the following identification of $\THR(H\underline{k})$:

\begin{prp}[{cf. \cite[Rem. 5.14]{DMP21}}]
There is an equivalence of $C_2$-spectra
$$\THR(H\underline{k}) \simeq H\underline{k} \otimes_{H \underline{\FF_p}} \THR(H\underline{\FF_p}) \simeq \bigoplus_{n \geq 0} \Sigma^{2n,n}H\underline{k}.$$
\end{prp}

\begin{proof}
Dotto--Moi--Patchkoria justify the equivalence for $k$ a perfect field of characteristic $2$ in \cite[Rem. 5.14]{DMP21}, but their proof holds under the more general hypothesis that $k$ is a perfect $\FF_p$-algebra. 
\end{proof}

The fact that $\THR(H\underline{k})$ is obtained via base-change from $\THR(H\underline{\FF_p})$ extends to $\TCR^-(H\underline{k})$ and $\TPR(H\underline{k})$ as follows:

\begin{prp} \label{prp:tcrminus_and_tpr}
There are isomorphisms of graded abelian groups
$$\pi_*^{C_2}\TCR^-(H\underline{k}) \cong W(k) \otimes_{\ZZ_p} \pi_*^{C_2}\TCR^-(H\underline{\FF_p}),$$
$$\pi_*^{C_2}\TPR(H\underline{k}) \cong W(k) \otimes_{\ZZ_p} \pi_*^{C_2} \TPR(H\underline{\FF_p}).$$
\end{prp}

\begin{proof}
In fact, one can compute the $RO(C_2)$-graded homotopy groups of $\TCR^-(H\underline{k})$ and $\TPR(H\underline{k})$ using the spectral sequences arising from the $x$-adic filtration. 

The $E_2$-terms of all of the relevant spectral sequences can be obtained by base-changing in each degree from $\FF_p$ to $k$. Indeed, this follows from the fact that $\THR(H\underline{k}) \simeq H\underline{k} \otimes_{H\underline{\FF_p}} \THR(H\underline{\FF_p})$ and the $RO(C_2)$-graded coefficients of $H\underline{k}$ are obtained by base-change of the coefficients of $H\underline{\FF_p}$. 

We claim that the differentials over $H\underline{k}$ can be obtained by extending the differentials over $H\underline{\FF_p}$ $k$-linearly. For $p$ odd, this is clear since all the differentials are zero for degree reasons. For $p=2$, we argue as follows. First, naturality implies that $d_3(\tau) = u\rho x$ in the $x$-adic spectral sequence for $H\underline{k}$. The Leibniz rule then implies that the differentials are $k$-linear, since each element in the copy of $k$ in tridegree $(0,0,0)$ of these spectral sequence is a permanent cycle for degree reasons, so $d_3(z \tau) = z u \rho x$ for any $z \in k$. This gives us the $E_3$-page for $H\underline{k}$ when $p=2$, and the spectral sequence collapses at this point for degree reasons as in the case $k=\FF_p$.

Finally, we claim that the extensions which give the copies of $\ZZ_p$ for $k = \FF_p$ give copies of $W(k)$ more generally. This can be justified as in the nonequivariant setting: inspection of the relevant spectral sequences shows that $\pi_0^{C_2}\TCR^-(H\underline{k})$ and $\pi_0^{C_2}\TPR(H\underline{k})$ are $p$-torsion free, $p$-adically complete, and that their mod $p$ reduction is isomorphic to $k$. Thus $\pi_0^{C_2} \TCR^-(H\underline{k}) \cong W(k)$ and $\pi_0^{C_2} \TPR(H\underline{k}) \cong W(k)$.\footnote{For further discussion of this point, we refer the reader to \cite{KNYouTube}.}
\end{proof}

To pass from $\TCR^-$ and $\TPR$ to $\TCR$, we need to identify the maps
$$\can, \ \varphi \coloneq \varphi_p^{h_{C_2}S^1} : \TCR^-(H\underline{k}) \to \TPR(H\underline{k}).$$

\begin{prp} \label{prp:effect_of_can_and_phi}
The maps $\can$ and $\varphi$ are as follows:
\begin{enumerate}
\item The canonical map for $H\underline{k}$ is given by precisely the same formulas as it is for $H\underline{\FF_p}$. It is $W(k)$-linear. 
\item For 
$$a \otimes b \in W(k) \otimes_{\ZZ_p} \pi_*^{C_2}\TCR^-(H\mfp),$$
the map $\varphi$ is given by 
$$\varphi(x \otimes y) = F(a) \otimes \varphi(b) \in W(k) \otimes_{\ZZ_p} \pi_*^{C_2}\TPR(H\mfp),$$
where $F: W(k) \to W(k)$ is the Witt vector Frobenius. 
\end{enumerate}
\end{prp}

\begin{proof}
\begin{enumerate}[leftmargin=*]
\item  As for $H\underline{\FF_p}$, this follows by inspection of the relevant spectral sequences. 

\item The claim is clear for $a=1$ since the horizontal maps in the commutative diagram
\[
\begin{tikzcd}
\TCR^-(H\mfp) \ar{r} \ar{d}{\varphi} & \TCR^-(H\underline{k}) \ar{d}{\varphi} \\
\TPR(H\mfp) \ar{r} & \TPR(H\underline{k})
\end{tikzcd}
\]
are given by $b \mapsto 1 \otimes b$. More generally, the structure of $\pi_{*}^{C_2}\TCR^-(H\underline{k})$ as a module over $\pi_0^{C_2} \TCR^-(H\underline{k})$ allows us to reduce the study of $\varphi$ to its behavior on $\pi_0^{C_2}$. Indeed, an element $a \otimes b \in W(k) \otimes_{\ZZ_p} \pi_{*}^{C_2}\TCR^-(H\underline{k})$ can be expressed as $a \cdot (1 \otimes b)$, where $a \in W(k) \cong \pi_0^{C_2}\TCR^-(H\underline{k})$. 

Note that the values of $\pi_0^{(-)}\TCR^-(H\underline{k})$ on $C_2/C_2$ and $C_2/e$ are both $W(k)$. Moreover, the restriction map is the identity; this can be seen by considering the composite
$$\TCR^-(H\underline{k}) \to \THR(H\underline{k}) \to H\underline{k}.$$
Altogether, this implies that $\pi_0^{(-)}\TCR^-(H\underline{k}) \cong \underline{W(k)}$ as a Green functor. Similarly, we have that $\pi_0^{(-)}\TPR(H\underline{k}) \cong \underline{W(k)}$. 

We are thus reduced to considering the map $\pi^{(-)}_0 \varphi : \underline{W(k)} \to \underline{W(k)}$. This is the Witt vector Frobenius at the underlying level, as shown in \cite{KNYouTube}. Since the restriction map is the identity, it follows that $\pi^{C_2}_0 \varphi$ is also the Witt vector Frobenius, which completes the proof.
\end{enumerate}
\end{proof}

% In summary, the canonical map is the `usual' inclusion and the real cyclotomic Frobenius can be identified with the Witt vector Frobenius. 

\begin{prp}
There is an isomorphism of Mackey functors
$$\pi_*^{(-)}\TCR(H\underline{k}) \cong \pi_*^{(-)}(H\underline{\ker(1-F)}) \oplus \pi_*^{(-)}(\Sigma^{-1}H\underline{\coker(1-F)}).$$
\end{prp}

\begin{proof}
We use Propositions~\ref{prp:tcrminus_and_tpr} and \ref{prp:effect_of_can_and_phi} together with the long exact sequence in Mackey functor homotopy groups associated to the fiber sequence
$$\TCR(H\underline{k}) \to \TCR^-(H\underline{k}) \xrightarrow{\can - \varphi} \TPR(H\underline{k}).$$
The argument is analogous to the case $k=\FF_p$, with the key point being that the middle map in 
$$\cdots \to \pi_0^{(-)}\TCR(H\underline{k}) \to \pi_0^{(-)} \TCR^-(H\underline{k}) \xrightarrow{\can - \varphi} \pi_0^{(-)} \TPR(H\underline{k}) \to \pi_{-1}^{(-)} \TCR(H\underline{k}) \to \cdots$$
is now identified with
$$1-F : \underline{W(k)} \to \underline{W(k)}.$$
Also note that when $p=2$, $\varphi$ now acts by the Frobenius on the analogous red classes (which are now copies of $k$), and since the Frobenius map is an isomorphism by the perfectness assumption on $k$, the red classes continue to not contribute to the homotopy of $\TCR(H \underline{k})$.
\end{proof}

As for $\FF_p$, this isomorphism on homotopy Mackey functors promotes to an equivalence of $C_2$-spectra by observing that the composite
$$H\underline{\ZZ_p} \to H\underline{\ker(1-F)} \simeq \tau_{\geq 0} \TCR(H\underline{k}) \to \TCR(H\underline{k})$$
equips $\TCR(H\underline{k})$ with an $H\underline{\ZZ_p}$-module structure. It therefore splits as follows:

\begin{thm} \label{thm:extn_perfect_algebras}
Let $p$ be a prime and $k$ a perfect $\FF_p$-algebra. There is an equivalence of $C_2$-spectra
$$\TCR(H\underline{k}) \simeq H\underline{\ker(1-F)} \oplus \Sigma^{-1} H\underline{\coker(1-F)}.$$
\end{thm}

\appendix

\section{On the geometric fixed points of real topological cyclic homology}\label{SS:GFPF2}

We determined the $C_2$-equivariant homotopy type of $\TCR(H\underline{\FF_2})$ above using a fiber sequence formula involving parametrized homotopy fixed points and Tate constructions. From this identification, one can then determine the geometric fixed points. In this appendix, we present an alternative computation of the $C_2$-geometric fixed points of $\TCR(H\underline{\FF_2})$. Our computation recovers \cite[Thm.~C]{DMP21} in the case $k=\FF_2$, which was computed using genuine real cyclotomic methods. The computations in this section are not necessary for any results in the previous sections. 

We begin by recalling our expected answer, cf. \cite{BehrensShah}:

\begin{prp}
We have
$$\pi_*(H\underline{\ZZ_2} \oplus \Sigma^{-1} H\underline{\ZZ_2})^{\phi C_2} \cong \FF_2[\tau^2] \oplus \Sigma^{-1} \FF_2[\tau^2],$$
where $|\tau^2| = 2$. 
\end{prp}

Given a real cyclotomic spectrum $X$, let $\psi: X^{\phi C_2} \to \bigoplus_{\mu_2} X^{\phi C_2 t \mu_2}$ be the $\mu_2$-equivariant map given by the projection of $(\varphi_2)^{\phi C_2}$ with respect to the description \eqref{eq:geometric_fixed_points_formula} of $(X^{t_{C_2} \mu_2})^{\phi C_2}$. We will compute the $C_2$-geometric fixed points of $\TCR(H\underline{\FF_2})$ using the following theorem, which will appear in forthcoming work of Harpaz, Nikolaus, and the second author \cite{HarpazNikolausShah}:

\begin{thm} \label{thm:HNS_equalizer_formula}
Let $X$ be a real cyclotomic spectrum that is underlying bounded-below. There is an equivalence of spectra
$$\TCR(X)^{\phi \mu_2} \simeq \fib ( X^{\phi C_2 h\mu_2} \xrightarrow{\psi^{h\mu_2} - \can} X^{\phi C_2 t\mu_2}).$$
\end{thm}

% The $\mu_2$-action on $X^{\phi C_2}$ can be made explicit when $X = \THR(R)$ using the following result of Dotto--Moi--Patchkoria--Reeh:

% \begin{thm}\cite[Thm. 2.26]{DMPR17}
% There is an equivalence of spectra
% $$\THR(R)^{\phi C_2} \simeq R^{\phi C_2} \otimes_{i^*_e R} R^{\phi C_2}.$$
% \end{thm}

Using the formula \eqref{eq:THR_geometric_fixed_points} for $\THR(R)^{\phi C_2}$, we may identify the $\mu_2$-action on $\THR(R)^{\phi C_2}$ as follows:

\begin{prp}\label{Prop:Swap}
The $\mu_2$-action on $\THR(R)^{\phi C_2} \simeq R^{\phi C_2} \otimes_{R^e} R^{\phi C_2}$ is given by swapping factors. 
\end{prp}
\begin{proof}
This holds since the residual $\mu_2$-action is determined by the rotation action of $\mu_2$ on $S^\sigma$ that swaps the points of $S^0 = \{\pm 1\}$.
% This is easily seen using the definition of $\THR(R)$ as $C_2$-equivariant factorization homology
% $$\THR(R) \simeq \int_{S^\sigma} R$$
% developed in \cite{HKZ20}. More precisely, we have an equivalence
% $$\THR(R)^{\phi C_2} \simeq \left( \int_{S^\sigma} R \right)^{\phi C_2} \simeq \int_{(S^\sigma)^{C_2}} R^{\phi C_2} \simeq \int_{S^0} R^{\phi C_2}$$
% under which the residual $\mu_2$-action is determined by the rotation action of $\mu_2$ on $S^\sigma$ that swaps the points of $S^0 = \{\pm 1\}$.
\end{proof}

The geometric fixed points of $\THR(H\underline{\FF_2})$ were analyzed by Dotto--Moi--Patchkoria--Reeh:

\begin{prp} \label{prp:THRF2_geometric_fixed_points}
There is an isomorphism
$$\pi_*\THR(H\mft)^{\phi C_2} \cong \FF_2[w_1,w_2]$$
where $|w_1| = |w_2| = 1$. The residual $\mu_2$-action on $\pi_*\THR(H\mft)^{\phi C_2} \cong \FF_2[w_1,w_2]$ swaps $w_1$ and $w_2$. 
\end{prp}

\begin{proof}
The equivalence of ring spectra
$$\THR(H\underline{\FF}_2)^{\phi C_2} \simeq H\underline{\FF}_2^{\phi C_2} \otimes_{H\FF_2} H\underline{\FF}_2^{\phi C_2}$$
gives the claimed computation of homotopy groups \cite[Prop. 5.19]{DMPR17}. The identification of the residual action follows from inspection of the proof of \cite[Prop. 5.19]{DMPR17} along with \cref{Prop:Swap}. 
\end{proof}

We are now ready to compute the $\mu_2$-homotopy fixed points and Tate construction. 

\begin{prp} \label{prp:superfluous_homotopy_fixed_points_computation}
There is an isomorphism
$$
\pi_*\THR(H\mft)^{\phi C_2 h \mu_2} \cong \FF_2[w,y] \oplus \bigoplus_{0 \leq a < b < \infty} \FF_2\{w_1^aw_2^b + w_1^bw_2^a\},
$$
where $|w| = 2$, $|y| = -1$, and $|w_1^a w_2^b + w_1^b w_2^a| = a+b$. 
\end{prp}

\begin{proof}
We consider the homotopy fixed points spectral sequence
\begin{equation}\label{Eqn:HFPSSTHR}
E_2^{s,t} = H^s(\mu_2; \pi_t \THR(H\mft)^{\phi C_2}) \Rightarrow \pi_{t-s} \THR(H\mft)^{\phi C_2 h \mu_2}.
\end{equation}
For each $t \geq 0$, we may decompose the $\ZZ[\mu_2]$-module $\pi_t \THR(H\mft)^{\phi C_2}$ as follows. We have
$$\pi_t \THR(H\mft)^{\phi C_2} \cong \bigoplus_{a+b=t} \FF_2\{w_1^a w_2^b\}$$
where $\mu_2$ sends $w_1^a w_2^b$ to $w_1^b w_2^a$. Therefore as a $\ZZ[\mu_2]$-module, we have
\begin{equation}\label{Eqn:THRmu2Mod}
\pi_t \THR(H\mft)^{\phi C_2} \cong 
\begin{cases}
\bigoplus_{a+b=t, \ 0 \leq a< b} \FF_2[\mu_2]\{w_1^aw_2^b + w_1^b w_2^a\} \quad & \text{ if } t \text{ is odd,} \\
\bigoplus_{a+b=t, \ 0 \leq a < b} \FF_2[\mu_2]\{w_1^aw_2^b + w_1^b w_2^a\} \oplus \FF_2\{w^{t/2}\} \quad & \text{ if } t \text{ is even.}
\end{cases}
\end{equation}
In the above expression, $w^{t/2} = w_1^{t/2}w_2^{t/2}$ is the $\mu_2$-fixed generator when $t$ is even. Therefore
$$H^*(\mu_2, \pi_* \THR(H\mft)^{\phi C_2}) \cong \FF_2[y,w] \oplus \bigoplus_{0 \leq a < b < \infty} \FF_2\{w_1^a w_2^b + w_1^b w_2^a\},$$
where $|y| = (1,0)$, $|w| = (0,2)$, and $|w_1^a w_2^b + w_1^b w_2^a| = (0,a+b)$. 

We claim the spectral sequence collapses at $E_2$. To see this, we use the augmentation map $\THR(H\mft) \to H\mft$. The homotopy fixed points spectral sequence for $H\mft^{\phi C_2}$ has the form
$$E_2^{s,t} = H^s(\mu_2; \pi_t H\mft^{\phi C_2}) \Rightarrow \pi_{t-s} H\mft^{\phi C_2 h\mu_2}.$$
There is an isomorphism
$$\pi_*H\mft^{\phi C_2} \cong \FF_2[\tau]$$
where $|\tau| = 1$. This is the shadow of a $\mu_2$-equivariant equivalence of spectra
$$H\mft^{\phi C_2} \simeq \bigoplus_{i \geq 0} \Sigma^i H\FF_2,$$
where $\mu_2$ acts trivially on the righthand side. We thus have
$$H^*(\mu_2; \pi_* H\mft^{\phi C_2}) \cong \FF_2[z,\tau]$$
where $|z| = (1,0)$ and $|\tau| = (0,1)$. The homotopy fixed points spectral sequence collapses since 
$$\pi_* HM^{h\mu_2} \cong H^{-*}(\mu_2; M)$$
for any $\ZZ[\mu_2]$-module $M$.

The augmentation map
$$\THR(H\mft)^{\phi C_2} \to H\mft^{\phi C_2}$$
is $\mu_2$-equivariant, where $\mu_2$ acts trivially on $H\mft^{\phi C_2}$, so it induces a map between homotopy fixed points spectral sequences. This map sends the class $y^i \in H^i(\mu_2; \pi_0 \THR(H\mft)^{\phi C_2})$ to $z^i \in H^i(\mu_2; \pi_0 H\mft^{\phi C_2})$. Since $z^i$ survives for all $i \geq 0$, $y^i$ must as well. 

The homotopy fixed points spectral sequence is multiplicative since $\THR(H\mft)^{\phi C_2}$ is a ring spectrum. It follows that $w$ is a permanent cycle, since the only potential target for a nontrivial differential on $w$ is $y$. By the Leibniz rule, we then have that $w^i y^j$ is a permanent cycle for all $i,j \geq 0$. Finally, the class $w_1^a w_2^b + w_1^b w_2^a$ is a permanent cycle for all $0 \leq a < b < \infty$ because the class is simple $y$-torsion while the possible targets of differentials on it are $y$-torsion free. 
\end{proof}

\begin{prp}
There is an isomorphism
$$\pi_* \THR(H\mft)^{\phi C_2 t \mu_2} \cong \FF_2[w, y^{\pm 1}].$$
\end{prp}

\begin{proof}
Consider the Tate spectral sequence
$$E_2^{s,t} = \widehat{H}^s(\mu_2;\pi_t \THR(H\mft)^{\phi C_2}) \Rightarrow \pi_{t-s} \THR(H\mft)^{\phi C_2 t \mu_2}.$$
In our description of $\pi_t \THR(H\mft)^{\phi C_2}$ as a $\ZZ[\mu_2]$-module \eqref{Eqn:THRmu2Mod}, the generators $w_1^a w_2^b + w_1^b w_2^a$ are in the image of the additive norm map from $\mu_2$-coinvariants to $\mu_2$-invariants. Therefore in Tate cohomology, we have
$$E_2^{**} = \widehat{H}^s(\mu_2;\pi_t \THR(H\mft)^{\phi C_2}) \cong \FF_2[w,y^{\pm 1}]$$
where $|w| = (0,2)$ and $|y| = (1,0)$. 

As in the homotopy fixed points computation, the spectral sequence can be shown to collapse by naturality with respect to the augmentation $\THR(H\mft) \to H\mft$. The result follows. 
\end{proof}

We have therefore described the relevant homotopy fixed points and Tate construction. It remains to study the maps $\can$ and $\psi^{h\mu_2}$. We begin with the canonical map. 

\begin{lem}\label{Lem:can2}
On homotopy groups, the canonical map
$$\can: \pi_* \THR(H\mft)^{\phi C_2 h \mu_2} \to \pi_* \THR(H\mft)^{\phi C_2 t\mu_2}$$
is determined by multiplicativity and
$$w \mapsto w, \quad y \mapsto y, \quad w_1^a w_2^b + w_1^b w_2^a \mapsto 0.$$
\end{lem}

\begin{proof}
This follows immediately from our analysis of the homotopy fixed points and Tate spectral sequences above. 
\end{proof}

The map $\psi^{h\mu_2}$ is more delicate. We need the following lemma:

\begin{lem}\label{Lem:PsiFactorization}
Let $X$ be a real cyclotomic spectrum. The map $\psi^{h\mu_2} : X^{\phi C_2 h\mu_2} \to X^{\phi C_2 t\mu_2}$ factors as
\[
X^{\phi C_2 h\mu_2} \to X^{\phi C_2} \xrightarrow{\bar{\psi}} X^{\phi C_2 t \mu_2}.
\]
\end{lem}

\begin{proof}
More generally, let $X \xrightarrow{f} Y$ be a nonequivariant map with $X \in \Sp^{h\mu_2}$. We may promote $f$ to a $\mu_2$-equivariant map $X \xrightarrow{\tilde{f}} \bigoplus_{\mu_2} Y$ which factors as 
$$X \xrightarrow{\Delta} \bigoplus_{\mu_2} X \xrightarrow{\bigoplus_{\mu_2} f} \bigoplus_{\mu_2} Y.$$
Applying $(-)^{h\mu_2}$, we obtain a factorization
$$X^{h\mu_2} \to X \to Y.$$
The lemma follows by considering the map $\psi: X^{\phi C_2} \to \bigoplus_{\mu_2} X^{\phi C_2 t\mu_2}$. 
\end{proof}

To compute $\psi^{h\mu_2}$, we will actually study the nonequivariant map 
$$\THR(H\mft)^{\phi C_2} \xrightarrow{\bar{\psi}} \THR(H\mft)^{\phi C_2 t \mu_2}$$
of \cref{Lem:PsiFactorization}. 

\begin{prp}
Up to swapping $w_1$ and $w_2$, the map
$$\bar{\psi}: \pi_* \THR(H\mft)^{\phi C_2} \cong \FF_2[w_1,w_2] \to \FF_2[w,y^{\pm 1}] \cong \pi_* \THR(H\mft)^{\phi C_2 t\mu_2}$$
is determined by multiplicativity and
$$w_1 \mapsto wy, \quad w_2 \mapsto y^{-1}.$$
\end{prp}

\begin{proof}
We have a commutative diagram
\[
\begin{tikzcd}
\THR(H\mft)^{\phi C_2} \ar{r}{\bar{\psi}} \ar{d} & \THR(H\mft)^{\phi C_2 t\mu_2} \ar{d} \\
H\mft^{\phi C_2} \ar{r} & H\mft^{\phi C_2 t\mu_2}
\end{tikzcd}
\]
where the vertical maps are induced by the geometric fixed points of the augmentation of $\THR(H\mft)$. Applying homotopy groups, we obtain the commutative diagram
\[
\begin{tikzcd}
\FF_2[w_1,w_2] \ar{r}{\bar{\psi}} \ar{d} & \FF_2[w,y^{\pm 1}] \ar{d} \\
\FF_2[\tau] \ar{r} & \FF_2[\tau,y^{\pm 1}].
\end{tikzcd}
\]
Evidently, the bottom map sends $\tau$ to $\tau$, and the righthand map sends $y$ to $y$. 

We claim that the lefthand map sends $w_1$ and $w_2$ to $\tau$. To see this, consider the diagram
\[
\begin{tikzcd}
(H\mft^{\phi C_2})^{\otimes 2} \ar{r} \ar{dr} & \THR(H\mft)^{\phi C_2} \ar{d} \\
										& H\mft^{\phi C_2}.
\end{tikzcd}
\]
obtained by applying $(- \otimes H\mft)^{\phi C_2}$ to the diagram of $C_2$-spaces
\[
\begin{tikzcd}
S^0 \ar{r} \ar{dr}   & S^\sigma \ar{d} \\
					& *.
\end{tikzcd}
\]
In homotopy groups, the classes $w_1, w_2 \in \pi_1 \THR(R)^{\phi C_2}$ are the image of the generators of $\pi_1H\mft^{\phi C_2}$ in the lefthand side, both of which map to $\tau \in \pi_1 H\mft^{\phi C_2}$. This proves the claim. 

It follows that $w_1w_2$ maps to $\tau^2$ along the lefthand composite. Since $y \mapsto y$, we must have $w \mapsto \tau^2$ in the righthand map and thus $\bar{\psi}(w_1 w_2) = \bar{\psi}(w_1) \bar{\psi}(w_2)= w$. Write
$$\bar{\psi}(w_1) = a_0 y^{-1} + a_1 w y + a_2 w^2 y^3 + a_3 w^3 y^5 + \cdots,$$
$$\bar{\psi}(w_2) = b_0 y^{-1} + b_1 w y + b_2 w^2 y^3 + b_3 w^3 y^5 + \cdots.$$
Since $w$ is not invertible, we must have $a_i = b_i = 0$ for $i \geq 2$, i.e.
$$\bar{\psi}(w_1) = a_0 y^{-1} + a_1 w y, \quad \bar{\psi}(w_2) = b_0 y^{-1} + b_1 w y.$$
Further, we must have
$$a_0 b_0 = 0, \quad a_1 b_1 = 0, \quad a_0 b_1 + a_1 b_0 = 1.$$
It follows that either 
$$a_0 = b_1 = 0 \ \text{ and } \ a_1 = b_0 = 1 \quad \text{ or } \quad  a_1 = b_0 = 0 \ \text{ and } \  a_0 = b_1 = 1.$$
This completes the proof. 
\end{proof}

\begin{cor}\label{Cor:psihmu2}
The map
$$\psi^{h \mu_2} : \pi_* \THR(H\mft)^{\phi C_2 h\mu_2} \to \pi_* \THR(H\mft)^{\phi C_2 t \mu_2}$$
is determined by
$$w \mapsto w, \quad y \mapsto 0, \quad w_1^a w_2^b + w_1^b w_2^a \mapsto w^ay^{a-b} + w^b y^{b-a}.$$
\end{cor}

\begin{proof}
The images of $w$ and $w_1^a w_2^b + w_1^b w_2^a$ follow from the previous discussion. We have $y \mapsto 0$ since $\psi^{h\mu_2}$ factors through the inclusion of fixed points. 
\end{proof}

This allows us to compute the homotopy groups of the geometric fixed points:

\begin{prp}
There is an isomorphism
$$\pi_* \TCR(H\mft)^{\phi C_2} \cong \pi_*(H\underline{\ZZ_2} \oplus \Sigma^{-1} H\underline{\ZZ}_2)^{\phi C_2}.$$
\end{prp}

\begin{proof}
Consider the long exact sequence in homotopy groups associated to the fiber sequence of spectra
$$\TCR(H\mft)^{\phi C_2} \to \THR(H\mft)^{\phi C_2 h\mu_2} \xrightarrow{\psi^{h\mu_2} - \can} \THR(H\mft)^{\phi C_2 t\mu_2}.$$

It is straightforward to check that
$$\ker(\psi^{h\mu_2} - \can) \cong \FF_2[w]$$
using \cref{Lem:can2} and \cref{Cor:psihmu2}. 

We claim that
$$\coker(\psi^{h\mu_2}-\can) \cong \FF_2[w].$$
To see this, observe that each class $w^a y^b$ with $a, b \geq 0$ is in the image of $\can$, while the classes $w^a y^b$ with $b<0$ are not in the image of $\can$. For $b>0$, the class $w^a y^b$ is not in the image of $\psi^{h\mu_2}$, so each class $w^a y^b$, $b>0$, is in the image of $\psi^{h\mu_2} - \can$. For each $d>c \geq 0$, we have
$$\psi^{h\mu_2}(w_1^c w_2^d + w_1^d w_2^c) = w^c y^{c-d} + w^d y^{d-c},$$
with $c-d <0$ and $d-c > 0$. Therefore each class $w^a y^b$, $b<0$, lies in the image of $\psi^{h\mu_2} - \can$. The only classes not in the image of $\psi^{h\mu_2} - \can$ are $w^a$, $a \geq 0$, which gives the claimed cokernel.

We therefore have from the long exact sequence in homotopy groups that
$$\pi_*\TCR(H\mft)^{\phi C_2} \cong \FF_2[w] \oplus \Sigma^{-1} \FF_2[w] \cong \pi_* H\underline{\ZZ_2}^{\phi C_2} \oplus \pi_*  \Sigma^{-1} H\underline{\ZZ_2}^{\phi C_2}.$$
\end{proof}

\bibliographystyle{amsalpha}
\bibliography{master}

\end{document}